\documentclass[a4paper,12pt]{article}
\usepackage{amsfonts,amsmath,amssymb}
\newtheorem{theorem}{Theorem}[section]
\newtheorem{lemma}[theorem]{Lemma}
\newtheorem{proposition}[theorem]{Proposition}
\newtheorem{corollary}[theorem]{Corollary}
\newtheorem{remark}[theorem]{Remark}
\newtheorem{definition}[theorem]{Definition}
\numberwithin{equation}{section}
\newenvironment{proof}{{\bf Proof\ }}{\QED\\}
\newcommand{\QED}{\hspace*{\fill}\rule{2.5mm}{2.5mm}}

\newcommand{\be}{\begin{equation}}
\newcommand{\ee}{\end{equation}}
\newcommand{\ba}{\begin{array}}
\newcommand{\ea}{\end{array}}
\newcommand{\bea}{\begin{eqnarray}}
\newcommand{\eea}{\end{eqnarray}}
\newcommand{\bee}{\begin{eqnarray*}}
\newcommand{\eee}{\end{eqnarray*}}

\newcommand{\lab}{\label}

\newcommand{\ds}{\displaystyle}
\newcommand{\nn}{\nonumber}
\providecommand{\norm}[1]{\lVert#1\rVert}
\providecommand{\normm}[1]{\left\lVert#1\right\rVert}
\renewcommand{\c}{\cdot}
\newcommand{\les}{\lesssim}

\newcommand{\R}{\mathbb{R}}
\newcommand{\N}{\mathbb{N}}

\newcommand{\La}{\Lambda}
\newcommand{\trt}{\textrm{tr}\theta}
\newcommand{\uo}{u(0,x,\o)}
\renewcommand{\gg}{{\bf g}}
\newcommand{\dd}{{\bf D}}
\renewcommand{\o}{\omega}
\renewcommand{\th}{\theta}
\newcommand{\ep}{\varepsilon}
\newcommand{\hth}{\widehat{\theta}}
\newcommand{\s}{\Sigma}
\renewcommand{\S}{\mathbb{S}^2}
\newcommand{\muu}{\mu_u}
\newcommand{\p}{P_u}
\newcommand{\h}{H^1(S)}
\newcommand{\li}[2]{L^{#1}_uL^{#2}(P_u)}
\newcommand{\lli}[1]{L^{#1}(\Sigma)}
\renewcommand{\l}[2]{L^{#1}_{[-2,2]}L^{#2}(P_u)}
\renewcommand{\ll}[1]{L^{#1}(S)}
\newcommand{\lp}[1]{L^{#1}(P_u)}
\newcommand{\hs}[1]{H^{#1}(P_u)}
\newcommand{\lhs}[2]{L^{#1}_uH^{#2}(P_u)}
\newcommand{\nabb}{\mbox{$\nabla \mkern-13mu /$\,}}
\newcommand{\lap}{\mbox{$\Delta \mkern-13mu /$\,}}
\newcommand{\lapa}{a^{-1}\mbox{$\Delta \mkern-13mu /$\,($a$)}}
\newcommand{\divb}{\mbox{$\textrm{div} \mkern-13mu /$\,}}
\newcommand{\curlb}{\mbox{$\textrm{curl} \mkern-13mu /$\,}}
\newcommand{\ana}{\mbox{$a^{-1}\nabla \mkern-13mu /\,a$\,}}
\newcommand{\nabn}{\nabla_N}
\newcommand{\nabna}{\nabla_{aN}}
\renewcommand{\lg}{ a}
\newcommand{\po}{\partial_{\omega}}
\renewcommand{\a}{\alpha}
\renewcommand{\b}{\beta}
\newcommand{\ga}{\gamma}
\newcommand{\de}{\delta}

\begin{document}

\begin{center}
\Large{\bf Parametrix for wave equations on a rough background I: regularity of the phase at initial time}
\end{center}

\vspace{0.5cm}

\begin{center}
\large{J\'er\'emie Szeftel}
\end{center}

\vspace{0.2cm}

\begin{center}
\large{DMA, Ecole Normale Sup\'erieure,\\
45 rue d'Ulm, 75005 Paris,\\
jeremie.szeftel@ens.fr}
\end{center}

\vspace{0.5cm}

{\bf Abstract.} This is the first of a sequence of four papers  \cite{param1}, \cite{param2}, \cite{param3}, \cite{param4} dedicated to the construction and the control of a parametrix to the homogeneous wave equation $\square_{\bf g} \phi=0$, where ${\bf g}$ is a rough metric satisfying the Einstein vacuum equations. Controlling such a parametrix as well as its error term when one only assumes $L^2$ bounds on the curvature tensor ${\bf R}$ of ${\bf g}$ is a major step of the proof of the bounded $L^2$ curvature conjecture proposed in \cite{Kl:2000}, and solved jointly with S.  Klainerman and I. Rodnianski in \cite{boundedl2}. On a more general level, this sequence of papers deals with the control of the eikonal equation on a rough background, and with the  derivation of $L^2$ bounds for Fourier integral operators on manifolds with rough phases and symbols, and as such is also of independent interest.

\vspace{0.2cm}

\section{Introduction}

We consider the Einstein vacuum equations,
\be\lab{eq:I1}
{\bf R}_{\alpha\beta}=0
\end{equation}
where ${\bf R}_{\alpha\beta}$
denotes the  Ricci curvature tensor  of  a four dimensional Lorentzian space time  $(\mathcal{M},\,  {\bf g})$. The Cauchy problem consists in finding a metric ${\bf g}$ satisfying \eqref{eq:I1} such that the metric induced by ${\bf g}$ on a given space-like hypersurface $\Sigma_0$ and the second fundamental form of $\Sigma_0$ are prescribed. The initial data then consists of a Riemannian three dimensional metric $g_{ij}$ and a symmetric tensor $k_{ij}$ on the space-like hypersurface $\Sigma_0=\{t=0\}$. Now, \eqref{eq:I1} is an overdetermined system and the initial data set $(\Sigma_0,g,k)$ must satisfy the constraint equations
\be\lab{const}
\left\{\begin{array}{l}
\nabla^j k_{ij}-\nabla_i \textrm{Tr}k=0,\\
 R-|k|^2+(\textrm{Tr}k)^2=0, 
\end{array}\right.
\ee
where the covariant derivative $\nabla$ is defined with respect to the metric $g$, $R$ is the scalar curvature of $g$, and $\textrm{Tr}k$ is the trace of $k$ with respect to the metric $g$.

The fundamental problem in general relativity is to 
 study the long term regularity and asymptotic 
properties
 of the Cauchy developments of general, asymptotically flat,  
initial data sets $(\Sigma_0, g, k)$. As far as local regularity is concerned it 
is natural to ask what are the minimal regularity properties of
the initial data which guarantee the existence and 
uniqueness of local developments. In \cite{boundedl2}, we obtain the 
following result which solves bounded $L^2$ curvature conjecture proposed 
in \cite{Kl:2000}:

\begin{theorem}[Theorem 1.10 in \cite{boundedl2}]\lab{th:mainbl2}
Let $(\mathcal{M}, {\bf g})$ an asymptotically flat solution to the Einstein vacuum equations \eqref{eq:I1} together with a maximal foliation by space-like hypersurfaces $\Sigma_t$ defined as level hypersurfaces of a time function $t$. Let $r_{vol}(\Sigma_t,1)$ the volume radius on scales $\leq 1$ of $\Sigma_t$\footnote{See Remark \ref{rem:volrad} below for a  definition}. Assume that the initial slice $(\Sigma_0,g,k)$ is such that:
$$\norm{R}_{L^2(\Sigma_0)}\leq \ep,\,\norm{k}_{L^2(\Sigma_0)}+\norm{\nabla k}_{L^2(\Sigma_0)}\leq \ep\textrm{ and }r_{vol}(\Sigma_0,1)\geq \frac{1}{2}.$$
Then, there exists a small universal constant $\ep_0>0$ such that if $0<\ep<\ep_0$, then the following control holds on $0\leq t\leq 1$:
$$\norm{\R}_{L^\infty_{[0,1]}L^2(\Sigma_t)}\lesssim \ep,\,\norm{k}_{L^\infty_{[0,1]}L^2(\Sigma_t)}+\norm{\nabla k}_{L^\infty_{[0,1]}L^2(\Sigma_t)}\lesssim \ep\textrm{ and }\inf_{0\leq t\leq 1}r_{vol}(\Sigma_t,1)\geq \frac{1}{4}.$$
\end{theorem}

\begin{remark}
While  the first   nontrivial improvements  for well posedness for quasilinear  hyperbolic systems (in spacetime dimensions greater than $1+1$), based on Strichartz estimates,  were obtained in   \cite{Ba-Ch1}, \cite{Ba-Ch2}, \cite{Ta1}, \cite{Ta}, \cite{KR:Duke}, \cite{KR:Annals}, \cite{SmTa}, Theorem \ref{th:mainbl2}, is the first  result in which the  full nonlinear structure of the quasilinear system, not just its principal part,  plays  a  crucial  role.  We note that  though  the result is not  optimal with respect to the  standard  scaling  of the Einstein equations, it is  nevertheless critical   with respect to its causal geometry,  i.e. $L^2 $ bounds on the curvature is the minimum requirement necessary to obtain lower bounds on the radius of injectivity of null hypersurfaces. We refer the reader to section 1 in \cite{boundedl2} for more motivations and historical perspectives concerning Theorem \ref{th:mainbl2}. 
\end{remark}

\begin{remark}
The regularity assumptions on $\Sigma_0$ in Theorem \ref{th:mainbl2} - i.e. $R$ and $\nabla k$ bounded in $L^2(\Sigma_0)$ - correspond to an initial data set $(g,\, k )\in H^2_{loc}(\Sigma_0)\times H^1_{loc}(\Sigma_0)$.
\end{remark}

\begin{remark}\lab{rem:reducsmallisok}
In \cite{boundedl2}, our main result is stated for corresponding large data. We then reduce the proof to the  small data statement of Theorem \ref{th:mainbl2} relying on a truncation and rescaling procedure, the control of the harmonic radius of $\Sigma_0$ based on Cheeger-Gromov convergence of Riemannian manifolds together with the assumption on the lower bound of the volume radius of $\Sigma_0$, and the gluing procedure in \cite{Co}, \cite{CoSc}. We refer the reader to section 2.3 in \cite{boundedl2} for the details.
\end{remark}

\begin{remark}\lab{rem:volrad}
We recall for the convenience of the reader the definition of the volume radius of the Riemannian manifold $\Sigma_t$. Let $B_r(p)$ denote the geodesic ball of center $p$ and radius $r$. The volume radius $r_{vol}(p,r)$ at a point $p\in \Sigma_t$ and scales $\leq r$ is defined by
$$r_{vol}(p,r)=\inf_{r'\leq r}\frac{|B_{r'}(p)|}{r^3},$$
with $|B_r|$ the volume of $B_r$ relative to the metric $g_t$ on $\Sigma_t$. The volume radius $r_{vol}(\Sigma_t,r)$ of $\Sigma_t$ on scales $\leq r$ is the infimum of $r_{vol}(p,r)$ over all points $p\in \Sigma_t$.
\end{remark}

The proof of Theorem \ref{th:mainbl2}, obtained in the sequence of papers \cite{boundedl2}, \cite{param1}, \cite{param2}, \cite{param3}, \cite{param4}, \cite{bil2}, relies on the following ingredients\footnote{We also need trilinear estimates and an $L^4(\mathcal{M})$ Strichartz estimate (see the introduction in \cite{boundedl2})}: 
{\em\begin{enumerate}
\item[{\bf A}] Provide  a system of coordinates relative to which \eqref{eq:I1} exhibits a null structure.

\item[{\bf B}] Prove  appropriate bilinear estimates for solutions to $\square_{\bf g} \phi=0$, on
 a fixed Einstein vacuum  background\footnote{Note that the first bilinear estimate of this type was obtained in \cite{BIL}}.

\item[{\bf C}] Construct a parametrix for solutions to the homogeneous wave equations $\square_{\bf g} \phi=0$ on a fixed Einstein vacuum  background, and obtain control of the parametrix and of its error term only using the fact that the curvature tensor is bounded in $L^2$. 
\end{enumerate}
}

Steps {\bf A} and {\bf B} are carried out in \cite{boundedl2}. In particular, the proof of the bilinear estimates rests on a representation formula for the solutions of the wave equation using the following plane wave parametrix\footnote{\eqref{param} actually corresponds to a half-wave parametrix. The full parametrix corresponds to the sum of two half-parametrix. See \cite{param2} for the construction of the full parametrix}:
\be\lab{param}
Sf(t,x)=\int_{\S}\int_{0}^{+\infty}e^{i\lambda u(t,x,\o)}f(\lambda\o)\lambda^2 d\lambda d\o,\,(t,x)\in\mathcal{M} 
\ee
where $u(.,.,\o)$ is a solution to the eikonal equation ${\bf g}^{\alpha\beta}\partial_\alpha u\partial_\beta u=0$ on $\mathcal{M}$ such that $u(0,x,\o)\sim x.\o$ when $|x|\rightarrow +\infty$ on $\Sigma_0$\footnote{The asymptotic behavior for $u(0,x,\o)$ when $|x|\rightarrow +\infty$ will be used in \cite{param2} to generate with the parametrix any initial data set for the wave equation}. Therefore, in order to complete the proof of the bounded $L^2$ curvature conjecture, we need to carry out step {\bf C} with the parametrix defined in \eqref{param}. 

\begin{remark}
Note that the parametrix \eqref{param} is invariantly defined\footnote{Our choice is reminiscent of the one used in \cite{SmTa} in the context of $H^{2+\epsilon}$ solutions of quasilinear wave equations. Note however that the construction in that paper is coordinate dependent}, i.e. without reference to any coordinate system. This is crucial since coordinate systems consistent with $L^2$ bounds on the curvature would not be regular enough to control a parametrix. 
\end{remark}

\begin{remark}
In addition to their relevance to the resolution of the bounded $L^2$ curvature conjecture, the methods and results  of step {\bf C} are  also of independent interest. Indeed, they deal on the one hand with the control of the eikonal equation ${\bf g}^{\alpha\beta}\partial_\alpha u\partial_\beta u=0$ at a critical level\footnote{We need at least $L^2$ bounds on the curvature to obtain a lower bound on the radius of injectivity of the null level hypersurfaces of the solution $u$ of the eikonal equation, which in turn is necessary to control the local regularity of $u$ (see \cite{param3})}, and on the other hand with the derivation of $L^2$ bounds for Fourier integral operators with significantly lower differentiability  assumptions both for the corresponding phase and symbol compared to classical methods (see for example \cite{stein} and references therein). 
\end{remark}

In view of the energy estimates for the wave equation, it suffices to control the parametrix at $t=0$ (i.e. restricted to $\Sigma_0$)
\be\lab{parami}
Sf(0,x)=\int_{\S}\int_{0}^{+\infty}e^{i\lambda u(0,x,\o)}f(\lambda\o)\lambda^2 d\lambda d\o,\,x\in\Sigma_0 
\ee
 and the error term
\be\lab{err} 
Ef(t,x)=\square_{\bf g}Sf(t,x)=\int_{\S}\int_{0}^{+\infty}e^{i\lambda u(t,x,\o)}\square_{\bf g}u(t,x,\o)f(\lambda\o)\lambda^3 d\lambda d\o,\,(t,x)\in\mathcal{M}. 
\ee
This requires the following ingredients, the two first being related to the control of the parametrix restricted to $\Sigma_0$ \eqref{parami}, and the two others being related to the control of the error term \eqref{err}:
{\em\begin{enumerate}
\item[{\bf C1}] Make an appropriate choice for the equation satisfied by $u(0,x,\o)$ on $\Sigma_0$, and control the geometry of the foliation generated by the level surfaces of $u(0,x,\o)$ on $\Sigma_0$.

\item[{\bf C2}] Prove that the parametrix at $t=0$ given by \eqref{parami} is bounded in $\mathcal{L}(L^2(\mathbb{R}^3),\lli{2})$ using the estimates for $u(0,x,\o)$ obtained in {\bf C1}.

\item[{\bf C3}] Control the geometry of the foliation generated by the level hypersurfaces of $u$ on $\mathcal{M}$.

\item[{\bf C4}] Prove that the error term \eqref{err} satisfies the estimate $\norm{Ef}_{L^2(\mathcal{M})}\leq C\norm{\lambda f}_{L^2(\mathbb{R}^3)}$ using the estimates for $u$ and $\square_{\gg}u$ proved in {\bf C3}.
\end{enumerate}
}

Step {\bf C3} was initiated in the sequence of papers \cite{FLUX}, \cite{LP}, \cite{STT} where the authors prove the estimate $\square_{\gg}u\in L^{\infty}(\mathcal{M})$, which is crucial for step {\bf C3} and {\bf C4}. In the present paper, we focus on step {\bf C1}. Remember that $u$ is a solution to the eikonal equation $\gg^{\alpha\beta}\partial_\alpha u\partial_\beta u=0$ on $\mathcal{M}$. To define $u$ in a unique manner, we still have to prescribe $u$ on $\Sigma_0$. Having in mind steps {\bf C2} and {\bf C3}, we look for $u(0,x,\o)$ satisfying the three following conditions:
{\em\begin{enumerate}
\item[{\bf C1a}] $u(0,x,\o)\sim x.\o$ when $|x|\rightarrow +\infty$ on $\Sigma_0$.

\item[{\bf C1b}] $\square_{\gg}u(0,x,\o)$ is in $\lli{\infty}$. In fact, the estimate $\square_{\gg}u\in L^{\infty}(\mathcal{M})$ is obtained in \cite{FLUX} using a transport equation (the Raychadhouri equation) so that one needs the corresponding estimate on $\Sigma_0$ (i.e. at $t=0$).

\item[{\bf C1c}] $u(0,x,\o)$ has enough regularity in $x$ and $\o$ to achieve step {\bf C2}, i.e. to control the parametrix at $t=0$ given by \eqref{parami}.
\end{enumerate}
}
Such a choice turns out to be a difficult task. This is due to the fact that the initial data set $(\Sigma_0,g,k)$ has very little regularity. In fact, to be consistent with the bounded $L^2$ curvature conjecture, one should only assume that the curvature tensor $R$ of $g$ and $\nabla k$ are in $\lli{2}$. Together with ${\bf C1b}$, this drastically limits the regularity in $x$ of $u(0,x,\o)$. Although $(\Sigma_0,g,k)$ is independent of $\o$ (which only intervenes in {\bf C1a} to prescribe the asymptotic behavior of $u(0,x,\o)$), the function $u(0,x,\o)$ has also very limited regularity in $\o$. We will thus have to make a very careful choice of $u(0,x,\o)$ to be able to satisfy the three conditions {\bf C1a C1b C1c} at the same time.

Let us note that the typical choice $u(0,x,\o)=x\c\o$ in a given coordinate system would not work for us, since we don't have enough control on the regularity of a given coordinate system within our framework. Instead, we need to find a geometric definition of $u(0,x,\o)$. A natural choice would be
$$\square_{\gg} u=0\textrm{ on }\Sigma_0$$
which by a simple computation turns out to be the following simple variant of the minimal surface equation\footnote{In the time symmetric case $k=0$, this is exactly the minimal surface equation}
$$\textrm{div}\left(\frac{\nabla u}{|\nabla u|}\right)=k\left(\frac{\nabla u}{|\nabla u|}, \frac{\nabla u}{|\nabla u|}\right)\textrm{ on }\Sigma_0.$$
Unfortunately, this choice does not allow us to have enough control of the derivatives of $u$ in the normal direction to the level surfaces of $u$. This forces us to look for an alternate equation for $u$:
$$\textrm{div}\left(\frac{\nabla u}{|\nabla u|}\right)=1-\frac{1}{|\nabla u|}+k\left(\frac{\nabla u}{|\nabla u|}, \frac{\nabla u}{|\nabla u|}\right)\textrm{ on }\Sigma_0.$$
In the time symmetric case, i.e. $k=0$, this choice simply means that the mean curvature of the level surfaces of $u$ is equal to 1 minus the lapse of $u$. In this context, this construction has not appeared in the literature. It is closest
in spirit to the mean curvature flow equation, as it can be recast in an alternative
form  
$$\frac{dx}{du}=(1+H+k_{NN}) N,$$
where $N$ is the mean curvature of the level surface of $u$. Its main advantage is that it  turns out to be parabolic in the normal direction to the level surfaces of $u$. Consequently, this construction retains the regularity of the leaves of the foliation of the minimal surface choice, but also additionally gives stronger control in the normal direction to the leaves.\\

The rest of the paper is as follows. In section 2, we motivate our choice for $u(0,x,\o)$ and we state the main results. In section 3, we assume the existence of $u(0,x,\o)$ and prove calculus inequalities with respect to the foliation generated by $\uo$ on $\Sigma_0$, which will be needed in the sequel. In section 4, we investigate the regularity of $\uo$ with respect to $x$. In section 5, we recall the properties of the geometric Littlewood-Paley decompositions established in \cite{LP}, and we derive useful commutator estimates, product estimates, as well as parabolic estimates. In section 6, we derive additional regularity for $\uo$ with respect to $x$. In section 7, we investigate the regularity of $\uo$ with respect to $\o$. In section 7, we construct a global coordinate system on the leaves of the foliation generated by $\uo$ on $\Sigma_0$. Finally, we derive additional estimates for $\uo$ in section 8.\\

\noindent{\bf Acknowledgments.} The author wishes to express his deepest gratitude to Sergiu Klainerman and Igor Rodnianski for stimulating discussions and constant encouragements during the long years where this work has matured. He also would like to stress that  the basic strategy of the construction of the parametrix and how it fits  into the whole proof of the bounded $L^2$ curvature conjecture has been done in collaboration with them. The author is supported by ANR jeunes chercheurs SWAP.\\

\section{Main results}

From now on, there will be no further reference to $\Sigma_t$ for $t>0$. Since there is no confusion, we will denote $\Sigma_0$ simply by $\Sigma$ in the rest of the paper. 

\subsection{Modification of $R$ and $k$ near the asymptotic end}\label{reducsmall}

Recall from Theorem \ref{th:mainbl2} that our assumptions on the initial data set $(\Sigma, g, k)$ are the following
\be\lab{small1}
\norm{R}_{\lli{2}}+\norm{k}_{L^2(\s)}+\norm{\nabla k}_{\lli{2}}\leq\ep,
\ee
where $\ep>0$ is small enough. Now, as a byproduct of the reduction to these small initial data outlined in Remark \ref{rem:reducsmallisok} and performed in section 2.3 of \cite{boundedl2}, we may also assume the existence of a global coordinate system on $(\s, g, k)$ relative to which we have
\be\lab{globalcoorsigmareduc}
\frac{1}{2}|\xi|^2\leq g_{ij}\xi^i\xi^j\leq 2|\xi|^2,
\ee
and $(\s, g, k)$ is smooth in $|x|\geq 1$.

In order to construct $\uo$ satisfying the asymptotic behavior {\bf C1a}, we need to modify $(\s,g,k)$ outside of $|x|\leq 1$. We can glue it to $(\R^3,\delta,0)$ so that the new initial data set is still smooth outside of $|x|\leq 1$, satisfies \eqref{small1}, and coincides with $(\R^3,\delta,0)$ outside of a slightly larger neighborhood. We still denote this initial data set $(\s,g,k)$. Of course, $(\s,g,k)$ does not satisfies the constraint equations in the annulus where the gluing takes place. However, for the construction of $\uo$, we only require $(\s,g,k)$ to satisfy the constraint equations in $|x|\leq 1$. Outside of $|x|\leq 1$, $(\s,g,k)$ is smooth, so things are much easier.

Finally, in order to be consistent with the statement of Theorem \ref{th:mainbl2}, we consider a maximal foliation, i.e.
$$\textrm{Tr}k=0.$$

\subsection{Geometry of the foliations generated by $u$ on $\mathcal{M}$ and by  $u_{|_\s}$ on $\s$}\label{sec:foliation}

Let $u$ a solution to the eikonal equation $\gg^{\alpha\beta}\partial_\alpha u\partial_\beta u=0$ on $\mathcal{M}$. Let $L=-\gg^{\alpha\beta}\partial_\alpha u\, \partial_\beta$ be the corresponding null generator vectorfield and $s$ its affine parameter, i.e. $L(s)=1$. Let us introduce the level hypersurfaces of $u$ 
$$\mathcal{H}_{u_0}=\{(t,x)\textrm{ in }\mathcal{M}\textrm{ such that }u=u_0\}$$  
which generate a foliation on $\mathcal{M}$. The level surfaces $P_{s,u}$ of $s$ generate the geodesic foliation on $\mathcal{H}_u$. 

The geometry of $\mathcal{H}_u$ depends in particular of the null second fundamental
form
\be\lab{eq:I4}
\chi(X,Y)=\gg(\dd_{X}L,Y)
\end{equation}
with $X,Y$ arbitrary vectorfields tangent to the  $s$-foliation $P_{s,u}$ and where $\dd$ is the covariant differentiation with respect to $\gg$. We denote by $\textrm{tr}\chi$ the trace of $\chi$,
i.e. $\textrm{tr}\chi=\delta^{AB}\chi_{AB}$ where $\chi_{AB}$ are the components of $\chi$ relative to an
orthonormal frame $(e_A)_{A=1,2}$ on the leaves of the $s$-foliation. An easy computation yields:
\begin{equation}\label{easycomp}
\square_\gg u=\textrm{tr}\chi
\end{equation}
so that ones needs to prove enough regularity for $\textrm{tr}\chi$ to control the error term \eqref{err} of the parametrix \eqref{param}. $\textrm{tr}\chi$ satisfies the well known Raychadhouri equation
\be\lab{eq:I5}
\frac{d}{ds}\textrm{tr}\chi+\frac{1}{2} (\textrm{tr}\chi)^2=-|\widehat{\chi}|^2
\end{equation}
with $\widehat{\chi}_{AB}=\chi_{AB}-1/2\textrm{tr}\chi\delta_{AB}$ the traceless part of $\chi$. This transport equation is used in \cite{FLUX} to prove the crucial estimate $\textrm{tr}\chi\in L^{\infty}(\mathcal{M})$ provided that $\textrm{tr}\chi$ is in $\lli{\infty}$ at $t=0$.

Let us now recall the link between $u_{|_\s}$ and  $\textrm{tr}\chi_{|_\s}$. We define the lapse $a=|\nabla u|^{-1}$, and the unit vector $N$ such that $\nabla u=a^{-1}N$. We also define the level surfaces 
$$P_{u_0}=\{x\textrm{ in }\s\textrm{ such that }u=u_0\},$$ 
so that $N$ is the normal to $\p$ in $\s$. The second fundamental form $\th$ of $\p$ is defined by 
\be\lab{eq:I6}
\th(X,Y)=g(\nabla_{X}N,Y)
\end{equation}
with $X,Y$ arbitrary vectorfields tangent to the  $u$-foliation $P_{u}$ on $\s$ and where $\nabla$ denotes the covariant differentiation with respect to $g$. We extend $\th$ as a tensor on $\s$ by setting
\be\lab{thsurN}
\th(N, .)=\th(., N)=0.
\ee
We denote by $\trt$ the trace of $\th$,
i.e. $\trt=\delta^{AB}\th_{AB}$ where $\th_{AB}$ are the components of $\th$ relative to an
orthonormal frame $(e_A)_{A=1,2}$ on $P_u$. We then have the following equality on $\s$:
\be\lab{rap0}
\textrm{tr}\chi=\trt + \textrm{tr}k.
\ee
Now, $\textrm{Tr}k=\textrm{tr}k+k_{NN}$. Recall from section \ref{reducsmall} that we impose $\textrm{Tr}k=0$ which corresponds to a maximal foliation. Thus, we obtain the following relation between $u$ and  $\textrm{tr}\chi$ on $\s$:
\be\lab{rap}
\textrm{tr}\chi=\trt - k_{NN}\textrm{ on }\s.
\ee
Finally, using \eqref{easycomp} and \eqref{rap}, we may reformulate {\bf C1b} as:
\be\lab{cond}
\trt - k_{NN}\in\lli{\infty}.
\ee

\subsection{Structure equations of the foliation generated by a function $u$ on $\s$}

We recall the structure equations of the foliation generated by a scalar function $u$ on $\s$ (see for example \cite{ChKl}). 
\begin{proposition}
The orthonormal frame frame $N, e_A, A=1, 2$ of $\s$ satisfies the following system:
\be\lab{frame}
\left\{\begin{array}{l}
\nabla_Ne_A=\nabb_Ne_A+a^{-1}(\nabb_Aa)N,\\[1mm]
\nabla_AN=\th_{AB}e_B,\\[1mm]
\nabla_Be_A=\nabb_Be_A-\th_{AB}N,\\[1mm]
\nabn N=-a^{-1}\nabb a.
\end{array}\right.
\ee
Also, the lapse $a$ and the second fundamental form $\th$ satisfy the following system:
\be\lab{struct}
\left\{\begin{array}{l}
\lapa=-\nabn\trt-|\th|^2+R_{NN},\\[1mm]
\nabb^B\hth_{AB}=\frac{1}{2}\nabb_A\trt+R_{NA},\\[1mm]
a^{-1}\nabb_A\nabb_Ba+\nabn\th_{AB}+2\th_A^C\th_{CB}-\trt\th_{AB}+K\gamma_{AB}=R_{AB},
\end{array}\right.
\ee
where $\hth_{AB}=\th_{AB}-1/2\trt\delta_{AB}$ is the traceless part of $\th$, $K$ is the Gauss curvature of $\p$, $\gamma$ is the metric on $\p$ induced by $g$, and $\nabb$ is the intrinsic covariant derivative on $\p$. Finally, we have:
\be\lab{gauss}
2K-\trt^2+|\th|^2=R-2R_{NN}.
\ee
\end{proposition}

\begin{proof}
We start with \eqref{frame}. Note that the second equality in \eqref{frame} follows from the definition of the second fundamental form $\th$. Also, the first and the third equality follow from the second and the fourth equality and the fact that the frame is orthonormal. Thus, it remains to prove the fourth equality in \eqref{frame}.

Since $\nabla u=a^{-1}N$, we have $N(u)=a^{-1}$. Thus, using $e_A(u)=0$ using the fact that the frame is orthonormal, we obtain:
\begin{eqnarray*}
\nabla_A(a^{-1})& = & \nabla_A(N(u))\\
& = & [e_A,N](u)\\
& = & \nabla_AN(u)-\nabla_Ne_A(u)\\
& = & a^{-1}g(N,\nabla_AN-\nabla_Ne_A)\\
& = & a^{-1}g(\nabla_NN,e_A)
\end{eqnarray*}
which concludes the proof of \eqref{frame}.

We now turn to the proof of \eqref{struct} starting with the first equation. Using the definition of the curvature tensor $R$, we have:
\begin{eqnarray*}
g([\nabla_A,\nabla_N]N,e_B)&=&g(\nabb_A\nabla_NN,e_B)-g(\nabla_N\nabb_AN,e_B)+g(\nabla_{\nabla_Ne_A}N,e_B)\\
&=& g(\nabla_A\nabla_NN,e_B)-g(\nabla_N\nabla_AN,e_B)+g(\nabla_{\nabla_Ne_A}N,e_B)\\
&=&-R_{ANBN}+g(\nabla_{\nabla_AN-\nabla_Ne_A}N,e_B)+g(\nabla_{\nabla_Ne_A}N,e_B)\\
&=&-R_{ANBN}+\th_{AC}\th_{CB}
\end{eqnarray*}
where we used \eqref{frame} in the last inequality. Taking the trace yields:
$$[\textrm{div} ,\nabla_N]N=-R_{NN}+|\th|^2,$$
which together with \eqref{frame} implies:
\begin{equation}\label{proofstruct1}
\textrm{div}(\nabla_NN)=\nabla_N(\textrm{div}(N))+[\textrm{div},\nabla_N]N=\nabla_N\trt-R_{NN}+|\th|^2.
\end{equation}
Using \eqref{frame}, we have:
$$\textrm{div}(\nabla_NN)=-\textrm{div}(a^{-1}\nabb a)=-\divb(a^{-1}\nabb a)-|a^{-1}\nabb a|^2=-\lapa$$
which together with \eqref{proofstruct1} proves the first equality of \eqref{struct}.

Next, we turn to the second equality of \eqref{struct}. Using the definition of the curvature tensor $R$, we have:
\begin{eqnarray*}
\nabb_A\th_{BC}-\nabb_B\th_{AC}&=&e_A(g(\nabla_BN,e_C))-\th(\nabb_Ae_B,e_C)-\th(e_B,\nabb_Ae_C)\\
&& -e_B(g(\nabla_AN,e_C))+\th(\nabb_Be_A,e_C)+\th(e_A,\nabb_Be_C)\\
&=& g((\nabla_A\nabla_B-\nabla_B\nabla_A)N,e_C)+g(\nabla_BN,\nabla_Ae_C-\nabb_Ae_C)\\
&& -\th(\nabb_Ae_B,e_C)-g(\nabla_AN,\nabla_Be_C-\nabb_Be_C)+\th(\nabb_Be_A,e_C)\\
&=& R_{ABNC}+g(\nabla_{\nabla_Ae_B-\nabb_Ae_B-\nabla_Be_A+\nabb_Be_A}N,e_C)\\
&=& R_{ABNC}
\end{eqnarray*}
where we used \eqref{frame}, the fact that $\th$ is symmetric, and the fact that the frame is orthonormal. Taking the trace yields:
$$\divb(\th)_A=\nabb_A\trt+R_{ABNB}=\nabb_A\trt+R_{AN}$$
which together with the definition of $\hth$ proves the second equality of \eqref{struct}.

We now turn to the last equality of \eqref{struct}. Using the definition of the curvature tensor $R$ and the property \eqref{thsurN} of $\th$, we have:
\begin{eqnarray*}
\nabla_N\th_{AB}&=&\nabla_N(g(\nabla_AN,e_B))-\th(\nabla_Ne_A,e_B)-\th(e_A,\nabla_Ne_B)\\
&=& g(\nabla_N\nabla_AN,e_B)-\th(\nabb_Ne_A,e_B)\\
&=& g(\nabla_A\nabla_NN,e_B)+R_{ANBN}+g(\nabla_{\nabla_Ne_A-\nabla_AN}N,e_B)-\th(\nabb_Ne_A,e_B)\\
&=& g(\nabla_A\nabla_NN,e_B)+R_{ANBN}+g(\nabla_{\nabla_Ne_A-\nabb_Ne_A-\nabla_AN}N,e_B)
\end{eqnarray*}
which together with \eqref{frame} yields:
\begin{equation}\label{proofstruct2}
\nabla_N\th_{AB}=-a^{-1}\nabb_A\nabb_Ba-\th_{AC}\th_{CB}+R_{ANBN}.
\end{equation}
Now, the Gauss equation of the foliation generated by $u$ on $\s$ reads:
\begin{equation}\label{proofstruct3}
R_{AB}=R_{ANBN}+K\ga_{AB}+\th_{AC}\th_{CB}-\trt\th_{AB},
\end{equation}
which together with \eqref{proofstruct2} proves the last equation of \eqref{struct}.

Finally, we turn to \eqref{gauss}. This follows from taking the trace of the Gauss equation \eqref{proofstruct3}. Note that it also follows form taking the trace of the last equality of \eqref{struct} and using the first equality. 
\end{proof}

\subsection{Commutation formulas}

Let $\Pi$ the projection operator from the tangent space of $\s$ to the tangent space $\p$, which is defined in an arbitrary orthonormal frame on $\s$ by
$$\Pi^i_j=\delta^i_j-N^iN_j.$$
Then, for any $\p$-tangent tensor $F$, we define $\nabb_NF$ as the projection of $\nabla_NF$ on $\p$:
$$\nabb_NU_{i_1\cdots i_n}=\Pi^{i_1}_{j_1}\cdots\Pi^{i_n}_{j_n}\nabla_NU_{j_1\cdots j_n}.$$
We have the following useful commutation formulas between $\nabb$ and $\nabb_N$ (see \cite{ChKl} page 64).
\begin{lemma}
For any $\p$-tangent tensor $F$ on $\Sigma$, we have schematically:
\be\lab{commut}
[\nabb_N,\nabb]F=\ana\c\nabn F-\theta\c\nabb F+R_{N.}\c F+\theta\c\ana\c F.
\ee
In particular, we obtain for any scalar $f$ on $\Sigma$:
\be\lab{scommut}
[\nabb_N,\nabb]f=\ana\nabn f-\th\c\nabb f,
\ee
and:
\begin{equation}\label{ad12}
\begin{array}{ll}
\ds [\nabn,\lap]f = & \ds -\trt\lap f -2\hth\c\nabb^2f+2a^{-1}\nabb a\c\nabb\nabn f+a^{-1}\lap a \nabn f-2R_{N.}\c\nabb f\\
&\ds -\nabb\trt\c\nabb f -2\hth\c a^{-1}\nabb a\c\nabb f.
\end{array}
\end{equation}
\end{lemma}

We will use some variants of the commutator formulas \eqref{commut}, \eqref{scommut} and \eqref{ad12}. In particular, for any scalar function $f$ on $\Sigma$, \eqref{ad12} yields:
\bea\label{commut1}
\nn a[\nabn,a^{-1}\lap]f &=& -(\trt+a^{-1}\nabn a)\lap f-2\hth\c\nabb^2f+2a^{-1}\nabb a\c\nabb\nabn f+\lapa\nabn f\\
&&-2R_{N.}\c\nabb f-\nabb\trt\c\nabb f-2\hth\c a^{-1}\nabb a\c\nabb f.
\eea

Also, for some applications we have in mind, we would like to get rid of the terms containing $\nabb_N$ in the right-hand side of \eqref{commut}, \eqref{scommut} and \eqref{ad12}. This is achieved by considering the commutators with $\nabb_{aN}$ instead of $\nabb_N$. \eqref{commut} implies for any $\p$-tangent tensor $F$ on $\Sigma$, schematically:
\be\lab{commutnabna1}
[\nabb_{aN},\nabb]F=-a\theta\c\nabb F+aR_{N.}\c F+\theta\c \nabb(a)\c F.
\ee
Using twice the commutator formula \eqref{commutnabna1}, we obtain, schematically:
\be\lab{commutnabna2}
[\nabb_{aN},\lap]F=\nabb\c (-\theta\c\nabb F+R_{N.}\c F+\theta\c \nabb(a)F)-\nabb\theta\c \nabb F+R_{N.}\c \nabb F+\theta\c\nabb(a)\c\nabb F.
\ee
In view of \eqref{commutnabna1}, we also have for any scalar function $f$ on $\Sigma$:
\begin{equation}\label{dj3}
[\nabla_{aN},\lap]f=-a\trt\lap f-2a\hth\c\nabb^2f+(-2aR_{N.}-a\nabb\trt+2\hth\c \nabb a)\c\nabb f.
\end{equation}

Finally, we conclude this section with the following commutator formula on $\p$. For any scalar function $f$ on $\p$, we have:
\begin{equation}\label{commut2}
[\nabb,\lap]f=K\nabb f.
\end{equation}

\subsection{The choice of $\uo$}\label{sec:choice}

In view of \eqref{cond}, we may reformulate {\bf C1a C1b C1c}. We look for $u(0,x,\o)$ satisfying the three following conditions:
{\em\begin{enumerate}
\item[{\bf C1a}] $u(0,x,\o)\sim x.\o$ when $|x|\rightarrow +\infty$ on $\s$

\item[{\bf C1b}] $\trt - k_{NN}\in\lli{\infty}$

\item[{\bf C1c}] $u(0,x,\o)$ has as enough regularity in $x$ and $\o$ to achieve step {\bf C2}, i.e. to control the parametrix at $t=0$ given by \eqref{parami}
\end{enumerate}
}
\noindent where the initial data set $(\s,g,k)$ satisfies:
\be\lab{const1}
\left\{\begin{array}{l}
\nabla^j k_{ij}=0,\\
 R=|k|^2,\\
\textrm{Tr}k=0, 
\end{array}\right.
\ee
and where $R$ and $\nabla k$ are in $\lli{2}$ and satisfy the smallness assumption \eqref{small1}.

In order to motivate our choice of $\uo$, we investigate the regularity of the lapse $a$, which by \eqref{struct} satisfies the following equation:
\be\lab{eqlapse}
\lapa=-\nabn\trt-|\th|^2-R_{NN}.
\ee
Since $R$ is in $\lli{2}$, \eqref{eqlapse} implies that $a$ has at most two derivatives in $\lli{2}$. Thus, $\uo$ has at most three derivatives with respect to $x$ in $\lli{2}$. This is not enough to satisfy {\bf C1c} (i.e. to obtain the boundedness of the parametrix at $t=0$ in $L^2$). In fact, the classical $T^*T$ argument (see for example \cite{stein}) relies on integrations by parts in $x$ and would require at least one more derivative since $\s$ has dimension 3. 

Alternatively, we could try to use the $TT^*$ argument which relies on integrations by parts in $\o$. Indeed, $R$ being independent of $\o$, one would expect the regularity of $\uo$ with respect to $\o$ to be better. Differentiating \eqref{eqlapse} with respect to $\o$, we obtain:
\be\lab{diffom}
a^{-1}\lap(\partial_\o a)=2\nabb\nabn a+\cdots,
\ee
where the term on the right-hand side comes from the commutator $[\partial_\o,\lap]$ (see section \ref{regomega}). Thus, obtaining an estimate for $\partial_\o a$ from \eqref{diffom} requires to control $\nabn a$. Unfortunately, \eqref{eqlapse} seems to give control of tangential derivatives of $a$ only. This is where the specific choice of $\uo$ comes into play. 

Having in mind the equation of minimal surfaces (i.e. $\trt=0$), condition {\bf C1b} suggest the choice $\trt - k_{NN}=0$. Unfortunately, this equation together with \eqref{eqlapse} does not provide any control of $\nabn a$. We might propose as a second guess natural guess to take instead $\trt -k_{NN}=\nabn a$. Plugging in \eqref{eqlapse} yields an elliptic equation for $a$: $\nabn^2a+\lapa=-|\th|^2-\nabn (k_{NN})-R_{NN}$. This allows us to control $\nabn^2a$ in $\lli{2}$. However, $\nabn a$ is at most in $H^1(\s)$ which does not embed in $\lli{\infty}$ - since $\s$ has dimension 3 - so that condition {\bf C1b} is not satisfied. To sum up, the first guess $\trt-k_{NN}=0$ satisfies {\bf C1b}, but not {\bf C1c}, whereas the second guess $\trt-k_{NN}=\nabn a$ might satisfy {\bf C1c}, but does not satisfy {\bf C1b}. 

The correct choice is the intermediate one:
\be\lab{choice}
\trt-k_{NN}=1-a.
\ee
We will see in section \ref{regx} that $a-1$ belongs to $\lli{\infty}$ so that {\bf C1b} is satisfied. Also, plugging \eqref{choice} in \eqref{eqlapse} yields:
\be\lab{eqlapse1}
\nabn a-\lapa=|\th|^2+\nabn (k_{NN})+R_{NN}.
\ee
This parabolic equation will allow us to control normal derivatives of $a$. In turn, we will control derivatives of $a$ with respect to $\o$ using \eqref{diffom}. Ultimately, we will prove enough regularity with respect to both $x$ and $\o$ for {\bf C1c} to be satisfied.

\subsection{Main results}\label{sec:mainres}

From now on, we will not make any further reference to the space-time $\mathcal{M}$. Instead, we will work only with the initial data set $(\s,g,k)$. Thus, since there can be no more confusion, we will denote $\uo$ simply by $u(x,\o)$. To $u$, we associate $\p$, $a$, $N$, $\th$ and $K$ as in section \ref{sec:foliation}. For $1\leq p,q\leq +\infty$, we define the spaces $\li{p}{q}$ for tensors $F$ on $\s$ using the norm: 
$$\norm{F}_{\li{p}{q}}=\left(\int_u\norm{F}^p_{L^q(\p)}du\right)^{1/p}.$$

\begin{remark}
In the rest of the paper, all inequalities hold for any $\o\in\S$ with the constant in the right-hand side being independent of $\o$. Thus, one may take the supremum in $\o$ everywhere. To ease the notations, we do not explicitly write down this supremum. 
\end{remark}

We first state a result of existence and regularity with respect to $x$ for $u$.
\begin{theorem}\label{thregx}
Let $(\s,g,k)$ chosen as in section \ref{reducsmall}. There exists a scalar function $u$ on $\s\times\S$ satisfying assumption {\bf C1a} and such that:
\begin{equation}\label{thregx1} 
\begin{array}{l}
\norm{a-1}_{\li{\infty}{2}}+\norm{\nabla a}_{\li{\infty}{2}}+\norm{a-1}_{\lli{\infty}}+\norm{\nabb\nabla a}_{\lli{2}}\lesssim\ep,\\
\norm{\trt-k_{NN}}_{\lli{\infty}}+\norm{\nabla\th}_{\lli{2}}+\norm{K}_{\lli{2}}\lesssim\ep,
\end{array}
\end{equation}
where $\p$, $a$, $N$, $\th$ and $K$ are associated to $u$ as in section \ref{sec:foliation}. 
\end{theorem}
Notice that condition {\bf C1b} is implied by \eqref{thregx1}. In order to state our second result, we introduce fractional Sobolev spaces $H^b(\p)$ on the surfaces $\p$ for any $b\in\mathbb{R}$ (see section \ref{def:Hs} for their definition). We have the following estimate for $\nabn^2a$, and improved estimate for $\nabn a$.  
\begin{theorem}\label{thnabn2a}
Let $(\s,g,k)$ chosen as in section \ref{reducsmall}. Let $u$ the scalar function on $\s\times\S$ constructed in theorem \ref{thregx}, and let $\p$, $a$, $N$, $\th$ and $K$ be associated to $u$ as in section \ref{sec:foliation}. We have:
\be\lab{nabn2a1}
\norm{\nabn a}_{\li{\infty}{4}}+\norm{\nabn^2a}_{\lhs{2}{-\frac{1}{2}}}\les \ep.
\ee
\end{theorem}

The third theorem investigates the regularity of $u$ with respect to $\o$:
\begin{theorem}\label{thregomega}
Let $(\s,g,k)$ chosen as in section \ref{reducsmall}. Let $u$ the scalar function on $\s\times\S$ constructed in theorem \ref{thregx}, and let $\p$, $a$, $N$, $\th$ and $K$ be associated to $u$ as in section \ref{sec:foliation}. We have:
\bea\label{threomega1}
\nn\norm{\po a}_{\lli{\infty}}+\norm{\nabb\po a}_{\li{\infty}{2}}+\norm{\nabb^2\po a}_{\lli{2}}+\norm{\nabn\po a}_{L^2_u H^{\frac{1}{2}}(\p)}&&\\
\nn +\norm{\nabn\po a}_{L^\infty_u H^{-\frac{1}{2}}(\p)}+\norm{\nabn^2\po a}_{L^2_u H^{-\frac{3}{2}}(\p)}+\norm{\nabla\po\th}_{\lli{2}}&\lesssim&\ep,\\
\norm{\po N}_{\lli{\infty}}&\lesssim& 1,
\eea
\bea\label{threomega2}
\nn\norm{\po^2a}_{L^2_u H^{\frac{3}{2}}(\p)}+\norm{\po^2a}_{L^\infty_u H^{\frac{1}{2}}(\p)}+\norm{\nabn\po^2 a}_{L^2_u H^{-\frac{1}{2}}(\p)}+\norm{\nabla\po^2\th}_{\lli{2}}&\lesssim&\ep,\\
\norm{\po^2 N}_{\lli{\infty}}&\lesssim& 1,
\eea
and
\begin{equation}\label{threomega3}
\norm{\po^3u}_{L^\infty_{\textrm{loc}}(\Sigma)}\lesssim 1.
\end{equation}
\end{theorem}

\begin{remark}
In order to prove Theorem \ref{thregx}, Theorem \ref{thnabn2a} and Theorem \ref{thregomega}, we will rely in a fundamental way on the choice \eqref{choice} for $u$, and on the structure of the constraint equations in the maximal foliation \eqref{const1}.
\end{remark}

\subsection{Coordinate systems on $\p$ and $\s$}

In order to prove Theorem \ref{thregx}, Theorem \ref{thnabn2a} and Theorem \ref{thregomega}, we will use embeddings on the level surfaces $\p$ of $u$. These embeddings are discussed in section \ref{sec:ineq}, 
and their proof will require in particular, the existence of a suitable coordinate system. The following proposition establishes the existence of a global coordinate system on $\p$.
\begin{proposition}\label{gl0}
Let $\o\in\S$. Let $\Phi_u:\p\rightarrow T_\o\S$ defined by:
\begin{equation}\label{gl1}
\Phi_u(x):=\po u(x,\o),
\end{equation}
where $T_\o\S$ is the tangent space to $\S$ at $\o$. Then $\Phi_u$ is a global $C^1$ diffeomorphism from $\p$ to $T_\o\S$.
\end{proposition}

The following proposition establishes the existence of a global coordinate system on $\Sigma$ and provides the control of the determinant of the corresponding Jacobian. This will turn out to be useful to control 
the parametrix at $t=0$ given by \eqref{parami}, which corresponds to step {\bf C2} (see \cite{param2}).
\begin{proposition}\label{gl20}
Let $\o\in\S$. Let $\Phi:\s\rightarrow\R^3$ defined by:
\begin{equation}\label{gl21}
\Phi(x):=u(x,\o)\o+\po u(x,\o)=u(x,\o)\o+\Phi_u(x),
\end{equation}
where $\Phi_u$ has been defined in \eqref{gl1}. Then $\Phi$ is a bijection, and the determinant of its Jacobian satisfies the following estimate:
\begin{equation}\label{gl22}
\norm{|\det(\textrm{Jac}\,\Phi)|-1}_{\lli{\infty}}\lesssim\ep.
\end{equation}
\end{proposition}

\subsection{Additional estimates}

Below, we provide several additional estimates. These are consequences of  Theorem \ref{thregx}, Theorem \ref{thnabn2a} and Theorem \ref{thregomega} that will be needed in steps {\bf C2} and {\bf C3} (see respectively \cite{param2} and \cite{param3}). We start with a first proposition.  
\begin{proposition}\lab{prop:estimatesadded}
Let $(\s,g,k)$ chosen as in section \ref{reducsmall}. Let $u$ the scalar function on $\s\times\S$ constructed in theorem \ref{thregx}, and let $N$ be associated to $u$ as in section \ref{sec:foliation}. For all $x\in\s$ and $\o\in\S$, we have:
 \begin{equation}\label{gogol1}
 |N(x,\o)+N(x,-\o)|\lesssim\ep.
 \end{equation}
Also, we have:
\begin{equation}\label{threomega1ter}
||N(x,\o)-N(x,\o')|-|\o-\o'||\lesssim |\o-\o'|(\ep+|\o-\o'|),\,\forall x\in\s, \o,\o'\in\S.
\end{equation}
Finally, let $\nu\in\S$ and $\Phi_\nu$ the map defined in \eqref{gl21}. Then, we have:
\begin{equation}\label{gogol2}
\begin{array}{l}
\ds u(x,\o)-\Phi_\nu(x)\c\o=O(\ep|\o-\nu|^2),\\
\ds \po u(x,\o)-\po(\Phi_\nu(x)\c\o)=O(\ep|\o-\nu|),\\
\ds \po^2 u(x,\o)-\po^2(\Phi_\nu(x)\c\o)=O(\ep).
\end{array}
\end{equation}
\end{proposition}

We introduce the family of intrinsic Littlewood-Paley projections
$P_j$ which have been constructed in \cite{LP} using the heat flow 
on the surfaces $\p$ (see section \ref{sec:LP} for their main properties). 
This allows us to define the following Besov space 
$\mathcal{B}$ for tensors $F$ on $\s$:
\begin{equation}\label{defbesov}
\norm{F}_{\mathcal{B}}=\sum_{j\ge 0} 2^{j}\norm{P_jF}_{\li{\infty}{2}} +
\norm{P_{<0}F}_{\li{\infty}{2}},
\end{equation}
where $P_{<0}=\sum_{j<0} P_j$. In particular, one can show that a scalar function belonging to $\mathcal{B}$ also belongs to $\lli{\infty}$ (see \cite{LP}). Now, as recalled in the introduction, the reason for requiring condition {\bf C1b} for $u$ is that a crucial space-time quantity has been proved to be in $L^\infty$ in \cite{FLUX} relying on a transport equation (the Raychadhouri equation) so that the corresponding quantity at $t=0$ should be in $\lli{\infty}$. However, pseudodifferential operators of order 0 do not map $L^\infty$ to $L^\infty$ which forces the authors in \cite{FLUX} to actually prove a stronger estimate. In fact, they work with a Besov space which both embeds in $L^\infty$ and is stable relative to operators of order 0. In turn, this forces us to obtain a stronger version of condition {\bf C1b}. This is the aim of the following proposition:
\begin{proposition}\label{propbesov}
Let $(\s,g,k)$ chosen as in section \ref{reducsmall}. Let $u$ the scalar function on $\s\times\S$ constructed in theorem \ref{thregx}, and let $\p$, $N$ and $\th$ be associated to $u$ as in section \ref{sec:foliation}. We have:
\begin{equation}\label{propbesov1}
\norm{\trt-k_{NN}}_{\mathcal{B}}\lesssim\ep.
\end{equation}
\end{proposition}

Using the geometric Littlewood Paley projections $P_j$ together with the estimates for $\nabn a$ in \eqref{thregx1}, and the estimate for $\nabn^2a$ in \eqref{nabn2a1}, we obtain the following proposition:
\begin{proposition}\label{cordecfr}
Let $(\s,g,k)$ chosen as in section \ref{reducsmall}. Let $u$ the scalar function on $\s\times\S$ constructed in theorem \ref{thregx}, and let $a$ and $N$ be associated to $u$ as in section \ref{sec:foliation}. For all $j\geq 0$, there are scalar functions $a^j_1$ and $a^j_2$ such that:
\begin{equation}\label{cordecfr1}
\nabn a=a^j_1+a^j_2\textrm{ where }\norm{a^j_1}_{\lli{2}}\lesssim 2^{-\frac{j}{2}}\ep\textrm{ and }\norm{\nabn a^j_2}_{\lli{2}}\lesssim 2^{\frac{j}{4}}\ep.
\end{equation}
\end{proposition} 

\begin{remark}\label{paslinfty}
Recall from section \ref{sec:choice} that we do not have enough regularity in $x$ to apply the $T^*T$ method. Alternatively, we could try the $TT^*$ method which relies on integration by parts in $\o$. But $\po^3u\in L^\infty_{\textrm{loc}}(\s)$ is also not enough and we would need at least one more derivative in $\o$ (see also Remark \ref{loeb2011}).  Nevertheless, we will prove in a subsequent paper that the regularity of $u$ both with respect to $x$ and $\o$ obtained in this paper is enough to show that condition {\bf C1c} is satisfied.
\end{remark}

The rest of the paper is as follows. In section \ref{sec:ineq}, we prove various embeddings and estimates on $\p$ and $\Sigma$ which are compatible with the regularity for $u$ obtained in Theorem \ref{thregx}. In section \ref{regx}, we prove Theorem \ref{thregx}. In section \ref{sec:estimatesLP}, we recall the properties of the geometric Littlewood-Paley projections $P_j$ introduced in \cite{LP}. We then prove several commutator and product estimates, as well as estimates for some parabolic equations on $\Sigma$. In section \ref{sec:nabla2a}, we prove Theorem \ref{thnabn2a}.  In section \ref{regomega}, we prove Theorem \ref{thregomega}. In  section \ref{sec:globalcoord}, we prove Proposition \ref{gl0} and Proposition \ref{gl20}. Finally, Proposition \ref{prop:estimatesadded}, Proposition \ref{propbesov} and Proposition \ref{cordecfr} are proved in section \ref{sec:addition}.



\section{Calculus inequalities}\label{sec:ineq}

\subsection{The Sobolev embedding on $\s$}

Recall from section \ref{reducsmall} that there is a global coordinate system on $(\s, g, k)$ relative to which we have
\be\lab{globalcoorsigmareducbis}
\frac{1}{2}|\xi|^2\leq g_{ij}\xi^i\xi^j\leq 2|\xi|^2.
\ee

\begin{lemma}\lab{sitl1}
Let $f$ a real scalar function on $\s$. Then:
\be\lab{eq:sitl1}
\norm{f}_{L^{\frac{3}{2}}(\s)}\lesssim \norm{\nabla f}_{L^1(\s)}.
\ee 
\end{lemma}

\begin{proof}
We may assume that $f$ has compact support in $\s$. In the global coordinate system $x=(x_1,x_2,x_3)$ on $\s$ satisfying \eqref{globalcoorsigmareducbis}, we have:
\bee
|f(x_1,x_2,x_3)|^{\frac{3}{2}}=\left|\int_{-\infty}^{x_1}\partial_1f(y,x_2,x_3)dy\int_{-\infty}^{x_2}\partial_2f(x_1,y,x_3)dy\int_{-\infty}^{x_3}\partial_3f(x_1,x_2,y)dy\right|^{\frac{1}{2}}\\
\lesssim\left(\int_{\mathbb{R}}|\partial_1f(y,x_2,x_3)|dy\right)^{\frac{1}{2}}\left(\int_{\mathbb{R}}|\partial_2f(x_1,y,x_3)|dy\right)^{\frac{1}{2}}\left(\int_{\mathbb{R}}|\partial_3f(x_1,x_2,y)|dy\right)^{\frac{1}{2}}.
\eee
Hence,
\bee
&&\int_{\mathbb{R}^3}|f(x_1,x_2,x_3)|^{\frac{3}{2}}dx_1dx_2dx_3\\
&\lesssim&\left(\int_{\mathbb{R}^3}|\partial_1f(x_1,x_2,x_3)|dx_1dx_2dx_3\right)^{\frac{1}{2}}\left(\int_{\mathbb{R}^3}|\partial_2f(x_1,x_2,x_3)|dx_1dx_2dx_3\right)^{\frac{1}{2}}\\
&&\left(\int_{\mathbb{R}^3}|\partial_3f(x_1,x_2,x_3)|dx_1dx_2dx_3\right)^{\frac{1}{2}}\\
&\lesssim& \left(\int_{\mathbb{R}^3}|\nabla f(x_1,x_2,x_3)|dx_1dx_2dx_3\right)^{\frac{3}{2}}.
\eee
Now in view of the coordinates system property \eqref{globalcoorsigmareducbis},  we deduce from the previous estimate:
$$\left(\int_{\mathbb{R}^3}|f(x)|^{\frac{3}{2}}\sqrt{|g_t|}dx_1dx_2dx_3\right)^{\frac{2}{3}}\lesssim \int_{\mathbb{R}^3}|\nabla f(x)|\sqrt{|g_t|}dx_1dx_2dx_3$$
as desired.
\end{proof}

As a corollary of the estimate \eqref{eq:sitl1}, we may derive the following Sobolev embeddings.
\begin{corollary}
Given an arbitrary tensorfield $F$ on $\s$, we have
\be\lab{sobineqs}
\norm{F}_{L^{6}(\s)}\lesssim \norm{\nabla F}_{L^{2}(\s)}.
\ee 
\end{corollary}

\begin{proof}
We use \eqref{eq:sitl1} with $f=|F|^4$:
$$\norm{F}^4_{L^{6}(\s)}=\norm{|F|^4}_{L^{\frac{3}{2}}(\s)}\lesssim \norm{|F|^2F\nabla F}_{L^{1}(\s)}\lesssim \norm{\nabla F}_{L^{2}(\s)}\norm{F}^3_{L^{6}(\s)}$$
which yields \eqref{sobineqs}.
\end{proof}

\subsection{Embeddings compatible with the foliation generated by $u$ on $\s$}

We assume the existence of a real function $u$ on $\s$. We define the lapse $a=|\nabla u|^{-1}$, and the unit vector $N$ such that $\nabla u=a^{-1}N$. We also define the level surfaces $\p=\{x\,/\,u(x)=u\}$ so that $N$ is the normal to $\p$. In this section we establish some basic calculus inequalities with respect to the foliation generated by $u$ on $\s$ in the strip $S$ defined by: 
$$S=\{x\textrm{ such that }-2< u(x)< 2\}.$$
These calculus inequalities will be used in all subsequent sections of the present paper. We will use the following assumptions, which are consistent with our assumption on $R$ and our choice of bootstrap assumptions (see \eqref{boot}, \eqref{boot1}, \eqref{appboot}, \eqref{appboot1}):
\begin{equation}\label{t0}
\begin{array}{rr}
\ds\norm{R}_{\ll{2}}+\norm{a-1}_{\ll{\infty}}+\norm{\nabb\nabla a}_{\ll{2}}+\norm{\nabla a}_{\ll{2}}+\norm{\trt}_{\ll{6}} & \\
\ds +\norm{\nabla\theta}_{\ll{2}}+\norm{\ana}_{\l{\infty}{4}}+\norm{\th}_{\l{\infty}{4}}+\norm{K}_{\ll{2}} & \ds\leq\delta
\end{array}
\end{equation}
for some small enough constant $\delta>0$.

Let $\muu$ denote the area element of $\p$. Then, for all integrable function $f$ on $S$, the coarea formula implies:
\begin{equation}\label{coarea}
\int_{S} fd\s =\int_{-2}^{2}\int_{\p} fad\muu du.
\end{equation}
It is also well-known that for a scalar function $f$:
\begin{equation}\label{du}
\frac{d}{du}\left(\int_{\p}fd\muu\right) =\int_{\p}\left(\frac{df}{du}+\trt f\right)d\muu.
\end{equation}
For $1\leq p,q\leq +\infty$, we define the spaces $\l{p}{q}$ using the norm 
$$\norm{F}_{\l{p}{q}}=\left(\int_{-2}^{2}\norm{F}^p_{L^q(\p)}du\right)^{1/p}.$$
In particular, in view of the assumptions \eqref{t0} for $a$, $\l{p}{p}$ coincides with $L^p(S)$ for all $1\leq p\leq +\infty$. We denote by $\gamma$ the metric induced by $g$ on $\p$, and by $\nabb$ the induced covariant derivative. We define the space $\h$ for tensors $F$ on $S$ using the norm 
$$\norm{F}_{\h}=(\norm{F}_{\ll{2}}^2+\norm{\nabla F}^2_{\ll{2}})^{1/2}.$$

A coordinate chart $U\subset \p$  with coordinates $x^1, x^2$ is admissible if,
relative to  these coordinates, there exists  a constant $c>0$ such that,
\be\lab{eq:coordchart}
c^{-1}|\xi|^2\le \gamma_{AB}(p)\xi^A\xi^B\le c|\xi|^2, \qquad \mbox{uniformly for  all }
\,\, p\in U.
\ee
We assume that $\p$ can be covered by a finite number
of admissible coordinate charts, i.e., charts satisfying 
the conditions \eqref{eq:coordchart}. Furthermore, we assume that the 
constant $c$ in \eqref{eq:coordchart} and the number of charts is independent of $u$. 
\begin{remark}
The existence of a covering of $\p$ by coordinate charts satisfying \eqref{eq:coordchart} with a constant $c>0$ and the number of charts independent of $u$ follows from Proposition \ref{gl0}.
\end{remark}
Under these assumptions, the following calculus inequality has been proved in \cite{LP}:
\begin{proposition} Let $f$ be a  real scalar function. Then, 
\be\lab{eq:isoperimetric}
\|f\|_{L^2(\p)}\lesssim\|\nabb f\|_{L^1(\p)}+\|f\|_{L^1(\p)}.
\ee
\end{proposition}
As a corollary of the estimate \eqref{eq:isoperimetric}, 
 the following Gagliardo-Nirenberg inequality is derived in \cite{LP}:
\begin{corollary}
Given an arbitrary tensorfield  $F$ on $\p$ and any $2\le p<\infty$, we have:
\be\lab{eq:GNirenberg}
\|F\|_{L^p(\p)}\lesssim \|\nabb F\|_{L^2(\p)}^{1-\frac{2}{p}}\|F\|_{L^2(\p)}^{\frac 2 p}+\|F\|_{L^2(\p)}.
\end{equation}
\end{corollary}
As a corollary to \eqref{eq:isoperimetric} it is also classical to derive the following inequality (for a proof, see for example \cite{GT} page 157):
\begin{corollary} For any tensorfield $F$ on $\p$ and any $p>2$,
\be\lab{Linfty}
\|F\|_{L^\infty(\p)}\lesssim \|\nabb F\|_{L^p(\p)}+\norm{F}_{L^p(\p)}.
\end{equation}
\end{corollary}

Below, we state and prove several embeddings with respect to the foliation generated by $u$ on $\s$. The difficulty 
is to obtain these estimates while using only assumptions that are compatible with the regularity for $u$ obtained in Theorem \ref{thregx}.
\begin{proposition}\label{p1}
Let $F$ be a tensorfield on $S$ such that $F\in\h$. Assume also \eqref{t0}. Then $F$ belongs to $\l{\infty}{4}$.
\end{proposition}

\begin{proof}
\bea\label{p1e1}
\norm{F(u,.)}_{\lp{4}}^4&=&\norm{F(-2,.)}_{L^4(P_{-2})}^4\\
\nn& & +4\int_{-2}^u\int_{P_{u'}}\nabn F(u',x')\c F(u',x')|F(u',x')|^2du'd\mu_{u'}\\
\nn& &+\int_{-2}^u\int_{P_{u'}}\trt |F(u',x')|^4du'd\mu_{u'}\\
\nn&\lesssim & \ds\norm{F(-2,.)}_{L^4(P_{-2})}^4+\norm{\nabn F}_{\ll{2}}\norm{F}^3_{\ll{6}}\\
\nn& &  +\norm{\trt}_{\ll{6}}\norm{F}^4_{\ll{24/5}}\\
\nn&\lesssim & \norm{F(-2,.)}_{L^4(P_{-2})}^4+\norm{\nabn F}_{\ll{2}}\norm{F}^3_{\ll{6}}+\norm{F}^4_{\l{12}{24/5}},
\eea
where we used the assumption \eqref{t0} for $\trt$ in the last inequality. Replacing $F$ with $\varphi(u)F$ where $\varphi$ is a smooth function such that $\varphi(-2)=1$ and $\varphi(2)=0$, and proceeding as in \eqref{p1e1}, we obtain:
\begin{equation}\label{p1e2}
\norm{F(-2,.)}_{L^4(P_{-2})}^4\lesssim \norm{\nabn F}_{\ll{2}}\norm{F}^3_{\ll{6}}+\norm{F}^4_{\l{12}{24/5}}+\norm{F}_{\ll{4}}^4,
\end{equation}
which together with \eqref{p1e1} yields:
\begin{equation}\label{p1e3}
\norm{F(u,.)}_{\lp{4}}^4\lesssim \ds\norm{\nabn F}_{\ll{2}}\norm{F}^3_{\ll{6}}+\norm{F}^4_{\l{12}{24/5}}+\norm{F}_{\ll{2}}\norm{F}_{\ll{6}}^3.
\end{equation}
This concludes the proof by taking the supremum in $u$ on the left-hand side, and by using the Sobolev embedding \eqref{sobineqs} and the following estimate:
$$\norm{F}_{\l{12}{24/5}}\lesssim \norm{F}_{\ll{6}}^{\frac{1}{2}}\norm{F}_{\l{\infty}{4}}^{\frac{1}{2}}.$$
\end{proof}

In Proposition \ref{p1}, we can get rid of the assumption that $F\in\ll{2}$. This is done in the following corollary.
\begin{corollary}\label{c0}
Let $F$ be a tensorfield on $S$ such that $\nabla F\in \ll{2}$ and $F(-2,.)\in L^4(P_{-2})$. Assume also \eqref{t0}. Then $F$ belongs to $\l{\infty}{4}$ and $L^6(S)$. Moreover, if $F(-2,.)\in L^2(P_{-2})$, then $F$ also belongs to $\l{\infty}{2}$ and $H^1(S)$.
\end{corollary}

\begin{proof}
The proof of Proposition \ref{p1} yields:
\begin{displaymath}
\begin{array}{ll}
& \ds\norm{F}_{\l{\infty}{4}}\\
\lesssim & \ds\norm{F(-2,.)}_{L^4(P_{-2})}+\norm{\nabn F}_{\ll{2}}^{\frac{1}{4}}\norm{F}^{\frac{3}{4}}_{\ll{6}}+\norm{\trt}_{\ll{6}}(\norm{F}_{\l{\infty}{4}}+\norm{F}_{\ll{6}}),
\end{array}
\end{displaymath}
which together with the Sobolev embedding \eqref{sobineqs}, and the assumption \eqref{t0} for $\trt$, yields for $\delta$ small enough:
\begin{equation}\label{eqc0}
\norm{F}_{\l{\infty}{4}}\lesssim \norm{F(-2,.)}_{L^4(P_{-2})}+\norm{\nabla F}_{\ll{2}}.
\end{equation}
This proves the first statement of the corollary.

Now, we also assume that $F(-2,.)\in L^2(P_{-2})$.
\begin{equation}\label{c0e1}
\begin{array}{lll}
\ds\norm{F(u,.)}_{\lp{2}}^2 &=&\norm{F(-2,.)}_{L^2(P_{-2})}^2\\
& &\ds +2\int_{-2}^u\int_{P_{u'}}\nabn F(u',x')\c F(u',x')du'd\mu_{u'}\\
& &\ds+\int_{-2}^u\int_{P_{u'}}\trt |F(u',x')|^2du'd\mu_{u'}\\
&\lesssim & \ds\norm{F(-2,.)}_{L^2(P_{-2})}^2+\norm{\nabn F}_{\ll{2}}\norm{F}_{\ll{2}}\\
& & \ds +\norm{\trt}_{\ll{6}}\norm{F}^2_{\ll{12/5}}\\
&\lesssim & \ds\norm{F(-2,.)}_{L^2(P_{-2})}^2+\norm{\nabn F}_{\ll{2}}\norm{F}_{\l{\infty}{2}}\\
& & \ds +\norm{\trt}_{\ll{6}}\norm{F}^{3/2}_{\l{\infty}{2}}\norm{F}^{1/2}_{\ll{6}},
\end{array}
\end{equation}
which proves that $F\in\l{\infty}{2}$ by taking the supremum in $u$ on the left-hand side and using the Sobolev embedding \eqref{sobineqs} and the assumption \eqref{t0} for $\trt$. This concludes the proof of the corollary.
\end{proof}

\begin{proposition}\label{p2}
Let $F$ be a tensorfield on $S$ such that $F\in \l{\infty}{2}$ and $\nabb F\in\ll{2}$. Then $F$ belongs to $\ll{4}$.
\end{proposition}

\begin{proof}
\begin{equation}
\begin{array}{lll}
\ds\norm{F}_{\ll{4}}^4=\int_{-2}^2\norm{F}_{\lp{4}}^4adu & \lesssim & \ds\int_{-2}^2\norm{F}_{\lp{4}}^4du\\
& \lesssim & \ds\int_{-2}^2(\norm{F}_{\lp{2}}^2\norm{\nabb F}_{\lp{2}}^2+\norm{F}_{\lp{2}}^4)du\\
& \lesssim & \ds\norm{F}_{\l{\infty}{2}}^2\norm{\nabb F}_{\ll{2}}^2+\norm{F}_{\l{4}{2}}^4
\\
& \lesssim & \ds\norm{F}_{\l{\infty}{2}}^2(\norm{\nabb F}_{\ll{2}}^2+\norm{F}_{\l{\infty}{2}}^2)
\end{array}
\end{equation}
where we have used \eqref{eq:GNirenberg} with $p=4$.
\end{proof}

\begin{proposition}\label{p3}
Let $F$ be a tensorfield on $S$ such that $F\in\h$ and $\nabb\nabla F\in \ll{2}$. Assume also \eqref{t0}. Then $F$ belongs to $\ll{\infty}$ and $\nabn\nabb F$ belongs to $\ll{2}$. Moreover, the conclusion still holds if instead of $F\in\h$ we assume $\nabla F\in\ll{2}$ and $F(-2,.)\in L^4(P_{-2})$.
\end{proposition}

\begin{proof}
Using \eqref{Linfty} with $p=4$ and Proposition \ref{p1}, we obtain:
\be\lab{e1}
\norm{F}_{\ll{\infty}}\lesssim\norm{\nabb F}_{\l{\infty}{4}}+\norm{F}_{\l{\infty}{4}}\lesssim\norm{F}_{\h}+\norm{\nabla\nabb F}_{\ll{2}}.
\ee
Thus, we just need to prove that $\nabn\nabb F$ belongs to $\ll{2}$ to conclude the proof. Since $\nabb\nabn F$ belongs to $\ll{2}$, it remains to prove that $[\nabb,\nabn]F$ is in $\ll{2}$. The commutation formula \eqref{commut} yields:
\begin{equation}
\begin{array}{l}
\ds\norm{[\nabn,\nabb]F}_{\ll{2}}\leq (\norm{\nabb a}_{\l{\infty}{4}}+\norm{\theta}_{\l{\infty}{4}})\norm{\nabla F}_{\l{2}{4}}\\
\hspace{3.5cm}\ds +(\norm{R}_{\ll{2}}+\norm{\nabb a}_{\l{2}{4}}\norm{\theta}_{\l{\infty}{4}})\norm{F}_{\ll{\infty}}.
\end{array}
\end{equation}
Using the Gagliardo-Nirenberg inequality \eqref{eq:GNirenberg} and Proposition \ref{p1} to bound the norm in $\l{2}{4}$ and $\l{\infty}{4}$ of $\nabla F$ and $\nabb a$, together with the estimate \eqref{e1} and the estimate \eqref{t0}, we finally obtain:
\bea\label{t1}
&&\norm{\nabn\nabb F}_{\ll{2}}\\
\nn&\lesssim & (1+\norm{\nabb a}_{\h}+\norm{\nabb a}_{\ll{2}}+\norm{\nabb^2 a}_{\ll{2}})(\norm{F}_{\h}+\norm{\nabla\nabb F}_{\ll{2}}).
\eea
Next, we evaluate $\nabn\nabb a$. The commutation formula for scalars \eqref{scommut} yields:
$$\norm{[\nabn,\nabb]a}_{\ll{2}}\leq (\norm{\nabb a}_{\l{\infty}{4}}+\norm{\theta}_{\l{\infty}{4}})\norm{\nabla a}_{\l{2}{4}},$$
which together with the Gagliardo-Nirenberg inequality \eqref{eq:GNirenberg} and Proposition \ref{p1} to bound the norm in $\l{2}{4}$ of $\nabb a$, and the estimate \eqref{t0}, implies
\begin{equation}\label{t2}
\ds\norm{\nabn\nabb a}_{\ll{2}}\lesssim  \norm{\nabb\nabn a}_{\ll{2}}+(\norm{\nabb a}_{\h}+\delta)(\norm{\nabb\nabla a}_{\ll{2}}+\norm{\nabla a}_{\ll{2}}).
\end{equation}
Using again \eqref{t0}, we deduce for $\delta>0$ small enough:
\begin{equation}\label{t3}
\norm{\nabn\nabb a}_{\ll{2}}\lesssim\delta.
\end{equation}
Finally, we conclude the proof in the case where $F\in\h$ using \eqref{t1} together with the smallness assumption \eqref{t0} and \eqref{t3}. In the case where $\nabla F\in\ll{2}$ and $F(-2,.)\in L^4(P_{-2})$, we proceed in the same way except that we use Corollary \ref{c0} to bound $F$ in $\l{\infty}{4}$.
\end{proof}

\begin{proposition}\lab{p2bis}
Let $F$ be a tensorfield on $S$ such that $\nabb^2F\in\ll{2}$, $\nabn F\in \ll{2}$ and $\nabb F(-2,.)\in L^2(P_{-2})$. Assume also \eqref{t0}. Then $\nabb F$ belongs to $\l{\infty}{2}$ and to $\ll{4}$. 
\end{proposition}

\begin{proof}
We start with the estimate of $\nabb F$ in $\l{\infty}{2}$. We have:
\bea\label{bulg1}
&&\norm{\nabb F(u,.)}_{\lp{2}}^2\\
\nn&=& \norm{\nabb F(-2,.)}_{L^2(P_{-2})}^2+2\int_{-2}^u\int_{P_{u'}}\nabn\nabb F(u',x')\c \nabb F(u',x') du'd\mu_{u'}\\
\nn&&+\int_{-2}^u\int_{P_{u'}}\trt |\nabb F(u',x')|^2du'd\mu_{u'}\\
\nn&\lesssim & \norm{\nabb F(-2,.)}_{L^2(P_{-2})}^2+\left|\int_{-2}^u\int_{P_{u'}}\nabb\nabn F(u',x')\c \nabb F(u',x') du'd\mu_{u'}\right|\\
\nn&&+\left|\int_{-2}^u\int_{P_{u'}}[\nabn,\nabb] F(u',x')\c \nabb F(u',x') du'd\mu_{u'}\right|+\norm{\trt}_{\l{\infty}{4}}\norm{\nabb F}_{\l{2}{\frac{8}{3}}}^2\\
\nn&\lesssim & \norm{\nabb F(-2,.)}_{L^2(P_{-2})}^2+\norm{\lap F}_{\ll{2}}\norm{\nabn F}_{\ll{2}}\\
\nn&&+\norm{[\nabb,\nabn]F}_{\l{2}{\frac{4}{3}}}\norm{\nabb F}_{\l{2}{4}}+\norm{\trt}_{\l{\infty}{4}}\norm{\nabb^2F}_{\ll{2}}^2,
\eea
where we used in the last inequality an integration by parts and the Gagliardo-
Nirenberg inequality \eqref{eq:GNirenberg}. Now, using the commutator formula \eqref{commut}, we have:
\bea\label{bulg2}
&&\norm{[\nabb,\nabn]F}_{\l{2}{\frac{4}{3}}}\\
\nn&\les& \norm{\ana\nabn F-\theta\nabb F+R_{N.}F+\theta\ana F}_{\l{2}{\frac{4}{3}}}\\
\nn&\lesssim & \norm{\ana}_{\l{\infty}{4}}\norm{\nabn F}_{\ll{2}}+\norm{\th}_{\l{\infty}{4}}\norm{\nabb F}_{\ll{2}}\\
\nn &&+(\norm{R}_{\ll{2}}+\norm{\th \ana}_{\ll{2}})\norm{F}_{\l{\infty}{4}}\\
\nn &\les& (\norm{R}_{\ll{2}}+\norm{\th}_{\l{\infty}{4}}+\norm{\ana}_{\l{\infty}{4}})(\norm{\nabn F}_{\ll{2}}+\norm{\nabb F}_{\l{\infty}{2}}),
\eea
where we used in the last inequality the Gagliardo-Nirenberg inequality \eqref{eq:GNirenberg}. \eqref{bulg1} and \eqref{bulg2} yields:
\bee
&&\norm{\nabb F(u,.)}_{\lp{2}}^2\\
&\lesssim&  \norm{\nabb F(-2,.)}_{L^2(P_{-2})}^2+(1+\norm{R}_{\ll{2}}+\norm{\th}_{\l{\infty}{4}}+\norm{\ana}_{\l{\infty}{4}})\\
&&\times(\norm{\nabn F}^2_{\ll{2}}+\norm{\nabb F}_{\l{\infty}{2}}\norm{\nabb^2F}_{\ll{2}}+\norm{\nabb^2F}_{\ll{2}}^2).
\eee
Finally, taking the supremum in $u$ and using the assumption \eqref{t0} implies:
\be\lab{bulg3}
 \norm{\nabb F}_{\l{\infty}{2}}\les  \norm{\nabb F(-2,.)}_{L^2(P_{-2})}+\norm{\nabb^2 F}_{\ll{2}}+\norm{\nabn F}_{\ll{2}}.
 \ee

Next, we estimate of $\nabb F$ in $\ll{4}$. In view of Proposition \ref{p2}, we have:
$$\norm{\nabb F}_{\ll{4}}\les \norm{\nabb^2 F}_{\ll{2}}+\norm{\nabb F}_{\l{\infty}{2}}.$$
Together with \eqref{bulg3}, this concludes the proof of the proposition.
\end{proof}

\begin{proposition}\label{p4}
Let $F$ be a vectorfield on $S$ such that $F(-2,.)\in L^2(P_{-2})$, $\nabb F\in\ll{2}$, and $\nabn F=\nabb f_1+F_2$ where $f_1$ is a scalar function on $S$ such that $f_1\in\ll{2}$ and $F_2$ is a vectorfield on $S$ such that $F_2\in\ll{4/3}$. Assume also \eqref{t0}. Then $F$ belongs to $\l{\infty}{2}$. 
\end{proposition}

\begin{proof}
\begin{equation}\label{p4e1}
\begin{array}{lll}
\ds\norm{F(u,.)}_{\lp{2}}^2 & \ds\lesssim & \ds\norm{F(-2,.)}^2_{L^2(P_{-2})}+\int_{-2}^u\int_{P_{u'}}\nabn F\c Fd\mu_{u'} du'\\
& & \ds +\int_{-2}^u\int_{P_{u'}} \trt |F|^2d\mu_{u'} du'\\
& \ds\lesssim & \ds\norm{F(-2,.)}^2_{L^2(P_{-2})}+\int_{-2}^u\int_{P_{u'}}(\nabb f_1+F_2)\c Fd\mu_{u'} du'\\
& & \ds +\norm{\trt}_{\ll{6}}\norm{F}_{\ll{\frac{12}{5}}}^2\\
& \ds\lesssim & \ds\norm{F(-2,.)}^2_{L^2(P_{-2})}+\int_{-2}^u\int_{P_{u'}}f_1\divb Fd\mu_{u'} du'\\
& & \ds +\norm{F_2}_{\ll{4/3}}\norm{F}_{\ll{4}}+\norm{\trt}_{\ll{6}}\norm{F}_{\ll{\frac{12}{5}}}^2\\
& \ds \lesssim & \ds\norm{F(-2,.)}^2_{L^2(P_{-2})}+(\norm{f_1}_{\ll{2}}+\norm{F_2}_{\ll{4/3}})\\
& & \ds \times(\norm{\nabb F}_{\ll{2}}+\norm{F}_{\l{\infty}{2}})\\
& & +\norm{\trt}_{\ll{6}}(\norm{\nabb F}_{\ll{2}}^{\frac{1}{3}}\norm{F}_{\l{\infty}{2}}^{\frac{5}{3}}+\norm{F}_{\l{\infty}{2}})
\end{array}
\end{equation}
where we have used Proposition \ref{p2} to bound $\norm{F}_{\ll{4}}$, and \eqref{eq:GNirenberg} with $p=12/5$ to bound $\norm{F}_{\ll{12/5}}$. This concludes the proof by taking the supremum in $u$ on the left-hand side and using the assumption \eqref{t0}.
\end{proof}

\subsection{The Bochner identity and consequences}

We recall  the Bochner identity on $\p$ (which has dimension 2). This allows us  to 
control  the $L^{2}$ norm of the second derivatives of a tensorfield 
in terms of the $L^{2}$ norm of the laplacian and 
geometric quantities associated with $\p$ (see for example \cite{LP} for a proof).
\begin{proposition}\label{prop:Bochner}
Let  $K$ denote the  Gauss curvature of $\p$. Then

\noindent 
{\bf i})\quad For a scalar function $f$:
\begin{equation}
\label{sboch}
\int_{\p} |\nabb^{2} f|^{2}\muu = \int_{\p} |\lap f|^{2}\muu - 
\int_{\p} K |\nabb f|^{2}\muu.
\end{equation}

\noindent
{\bf ii)}\quad For a vectorfield $F_{a}$:
\begin{equation}
\label{vboch}
\int_{\p} |\nabb^{2} F|^{2}\muu = \int_{\p} |\lap F|^{2}\muu -
\int_{\p} K (2\,|\nabb F|^{2}-|\divb F|^{2}-|\curlb F|^2)\muu + \int_{\p} K^{2} |F|^{2}\muu,
\end{equation}
where $\divb F=\gamma^{ab}\nabb_b F_a$, $\curlb F=\divb( ^*F)=\in_{ab}\nabb_a F_b$.
\end{proposition}

\begin{remark}
As a consequence of \eqref{vboch} together with a $\lp{\infty}$ estimate for tensors, we have the following Bochner inequality for tensors $F$ on $\p$ (see \cite{LP} for a proof):
\be\lab{Bochtensorineq}
\norm{\nabb^2F}_{\lp{2}}\les \norm{\lap F}_{\lp{2}}+\norm{K}_{\lp{2}}\norm{\nabb F}_{\lp{2}}+\norm{K}^2_{\lp{2}}\norm{F}_{\lp{2}}.
\ee
\end{remark}

Using Proposition \ref{prop:Bochner}, we obtain the following proposition:
\begin{proposition}\label{p5}
Let $f$ be a scalar function on $S$ such that $\nabb f(-2,.)\in L^2(P_{-2})$, $\nabn f\in \ll{2}$ and $\lap f\in\ll{2}$. Assume also \eqref{t0}. Then $\nabb^2 f$ belongs to $\ll{2}$ and $\nabb f$ belongs to $\l{\infty}{2}$.
\end{proposition}

\begin{proof}
The Bochner identity \eqref{sboch} implies:
\begin{equation}\label{e2}
\begin{array}{lll}
\ds\norm{\nabb^2 f}_{\ll{2}} & \lesssim & \ds\norm{\lap f}_{\ll{2}}+\norm{K|\nabb f|^2}_{\ll{1}}^{1/2}\\
& \ds\lesssim & \ds\norm{\lap f}_{\ll{2}}+\norm{K}_{\ll{2}}^{1/2}\norm{\nabb f}_{\ll{4}}\\
& \ds\lesssim & \ds\norm{\lap f}_{\ll{2}}+\norm{K}_{\ll{2}}^{1/2}(\norm{\nabb f}_{\l{\infty}{2}}^{1/2}\norm{\nabb^2 f}_{\ll{2}}^{1/2}\\
& & \hspace{5.5cm}\ds +\norm{\nabb f}_{\l{\infty}{2}})
\end{array}
\end{equation}
where we have used Proposition \ref{p2}. Thus, it just remains to prove that $\nabb f$ belongs to $\l{\infty}{2}$. In order to use Proposition \ref{p4} , we have first to estimate $[\nabn,\nabb]f$, which is 
given by the commutator formula \eqref{scommut}. We estimate $[\nabn,\nabb]f$ in $\l{2}{4/3}$:
$$\norm{[\nabn,\nabb]f}_{\l{2}{4/3}}\lesssim\norm{\ana}_{\l{\infty}{4}}\norm{\nabn f}_{\ll{2}}+\norm{\th}_{\l{\infty}{4}}\norm{\nabb f}_{\l{\infty}{2}}.$$
Thus, $\nabn\nabb f=\nabb f_1+F_2$ where $f_1=\nabn f$ belongs to $\ll{2}$ and $F_2=[\nabn,\nabb]f$ belongs to $\l{2}{4/3}$. According to Proposition \ref{p4}, and using assumption \eqref{t0}, this implies:
$$\norm{\nabb f}_{\l{\infty}{2}}\lesssim \norm{\nabla f}_{\ll{2}}+\norm{\nabb^2 f}_{\ll{2}}+\delta\norm{\nabb f}_{\l{\infty}{2}}.$$
For $\delta>0$ small enough, this yields:
$$\norm{\nabb f}_{\l{\infty}{2}}\lesssim \norm{\nabla f}_{\ll{2}}+\norm{\nabb^2 f}_{\ll{2}}.$$
Together with \eqref{e2}, this implies:
$$\norm{\nabb^2 f}_{\ll{2}}\lesssim \norm{\nabla f}_{\ll{2}}+\norm{K}_{\ll{2}}^{1/2}\norm{\nabb^2 f}_{\ll{2}}$$
which concludes the proof since $\norm{K}_{\ll{2}}\leq\delta$ for a small $\delta>0$ in view of assumption \eqref{t0}.
\end{proof}

\subsection{Parabolic and elliptic estimates}

In the proof of Theorem \ref{thregx} and Theorem \ref{thregomega}, we will often encounter parabolic equations of the following type:
$$(\nabn-a^{-1}\lap)f=h\textrm{ on }-2<u<2,$$
(see for example \eqref{eqlapse1}). In Proposition \ref{p7} and Proposition \ref{p8} below, we obtain estimates 
for such equations.
\begin{proposition}\label{p7}
Let $f$ be a scalar function on $S$ such that:
\begin{equation}\label{p7e1}
(\nabn-a^{-1}\lap)f=h\textrm{ on }-2<u<2,
\end{equation}
where $h$ is in $\ll{2}$. Assume also that $f(-2,.)$ and $\nabb f(-2,.)$ both belong to $L^2(P_{-2})$.   
Finally, assume \eqref{t0}. Then, we have:
\begin{equation}\label{p7e2}
\begin{array}{r}
\ds\norm{f}_{\l{\infty}{2}}+\norm{\nabb f}_{\l{\infty}{2}}+\norm{\nabn f}_{\ll{2}}+\norm{\nabb^2f}_{\ll{2}}\\
\ds\lesssim\norm{h}_{\ll{2}}+\norm{f(-2,.)}_{L^2(P_{-2})}+\norm{\nabb f(-2,.)}_{L^2(P_{-2})}.
\end{array}
\end{equation}
\end{proposition}

\begin{proof}
We multiply \eqref{p7e1} by $f$ and integrate on $-2<u'<u$ where $u\leq 2$. Using integration by parts together with \eqref{coarea} and \eqref{du}, we obtain:
\begin{equation}\label{p7e3}
\begin{array}{ll}
& \ds\frac{1}{2}\norm{f(u,.)}_{L^2(P_{u})}^2+\norm{a^{-1/2}\nabb f}_{\ll{2}}^2\\[2mm]
\ds = & \ds\frac{1}{2}\norm{f(-2,.)}_{L^2(P_{-2})}^2+\frac{1}{2}\int_{-2}^u\int_{P_{u'}}a^{-1}\trt f^2d\mu_{u'}du'+\int_{-2}^u\int_{P_{u'}} hf d\mu_{u'}du'\\[1mm]
\ds\lesssim & \ds\norm{f(-2,.)}_{L^2(P_{-2})}^2+\norm{\trt}_{\l{\infty}{4}}\norm{f}^2_{\l{2}{\frac{8}{3}}}\\
& \ds +\norm{h}_{\ll{2}}\norm{f}_{\l{\infty}{2}}.
\end{array}
\end{equation}
Together with \eqref{eq:GNirenberg} and  \eqref{t0}, we get:
\begin{equation}\label{p7e4}
\norm{f}_{\l{\infty}{2}}^2+\norm{\nabb f}_{\ll{2}}^2\lesssim \norm{h}^2_{\ll{2}}+\norm{f(-2,.)}^2_{L^2(P_{-2})}.
\end{equation}

We multiply \eqref{p7e1} by $\lap f$ and integrate on $-2<u'<u$ where $u\leq 2$: 
\begin{equation}\label{p7e5}
\begin{array}{ll}
& \ds\int_{-2}^{u}\int_{P_{u'}}\nabb(a\nabn f)\nabb fd\mu_{u'}du'+\norm{a^{-1/2}\lap f}_{\ll{2}}^2\\
\ds = & \ds\int_{-2}^{u}\int_{P_{u'}}h\lap fad\mu_{u'}du'\lesssim\norm{h}_{\ll{2}}\norm{\lap f}_{\ll{2}}.
\end{array}
\end{equation}
Using integration by parts together with \eqref{coarea} and \eqref{du}, we obtain:
\begin{equation}\label{p7e6}
\begin{array}{l}
\ds\frac{1}{2}\norm{\nabb f(u,.)}_{L^2(P_{u})}^2+\norm{a^{-1/2}\lap f}_{\ll{2}}^2\\
\ds\leq\frac{1}{2}\norm{\nabb f(-2,.)}_{L^2(P_{-2})}^2+\int_{-2}^{u}\int_{P_{u'}}[\nabn,\nabb]f\nabb fd\mu_{u'}du'\\
\ds +\int_{-2}^{u}\int_{P_{u'}}|\nabb f|^2\trt d\mu_{u'}du'-\int_{-2}^{u}\int_{P_{u'}}\nabb a\nabb f\nabn fd\mu_{u'}du'+\norm{h}_{\ll{2}}\norm{\lap f}_{\ll{2}}.
\end{array}
\end{equation}
Using the commutator formula \eqref{scommut}, we get:
\begin{equation}\label{p7e7}
\begin{array}{ll}
& \ds\norm{\nabb f}_{\l{\infty}{2}}^2+\norm{a^{-1/2}\lap f}_{\ll{2}}^2\\
\ds\lesssim & \ds\norm{\nabb f(-2,.)}_{L^2(P_{-2})}^2+\norm{\th}_{\l{\infty}{4}}\norm{\nabb f}_{\l{2}{8/3}}^2+\norm{h}_{\ll{2}}\norm{\lap f}_{\ll{2}},
\end{array}
\end{equation}
which together with \eqref{eq:GNirenberg} and  \eqref{t0} yields:
\begin{equation}\label{p7e8}
\begin{array}{ll}
& \ds\norm{\nabb f}_{\l{\infty}{2}}^2+\norm{\lap f}_{\ll{2}}^2\\
\ds\lesssim & \ds\delta(\norm{\nabb f}_{\ll{2}}^2+\norm{\nabb^2f}_{\ll{2}}^2)+\norm{\nabb f(-2,.)}_{L^2(P_{-2})}^2+\norm{h}^2_{\ll{2}}.
\end{array}
\end{equation}
Since $\nabn f=a^{-1}\lap f+h$, \eqref{p7e8} yields:
\begin{equation}\label{p7e9}
\norm{\nabn f}_{\ll{2}}^2\lesssim\delta(\norm{\nabb f}_{\ll{2}}^2+\norm{\nabb^2f}_{\ll{2}}^2)+\norm{\nabb f(-2,.)}_{L^2(P_{-2})}^2+\norm{h}^2_{\ll{2}},
\end{equation}
which together with Proposition \ref{p5},  \eqref{t0} and \eqref{p7e8} implies:
\begin{equation}\label{p7e10}
\norm{\nabb^2f}_{\ll{2}}^2\lesssim\delta\norm{\nabb f}_{\ll{2}}^2+\norm{\nabb f(-2,.)}_{L^2(P_{-2})}^2+\norm{h}^2_{\ll{2}}.
\end{equation}
Finally, \eqref{p7e4}, \eqref{p7e8}, \eqref{p7e9} and \eqref{p7e10} yield \eqref{p7e2} for $\delta>0$ small enough.
\end{proof}

\begin{proposition}\label{p8}
Let $f$ be a scalar function on $S$ such that:
\begin{equation}\label{p8e1}
(\nabn-a^{-1}\lap)f=h\textrm{ on }-2<u<2.
\end{equation}
Assume that there exists a vectorfield $H$ on $S$ tangent to $\p$ and a scalar function $h_1$ on $S$ such that:
\begin{equation}\label{p8e2}
h=\divb(H)+h_1\textrm{ with }H\in\ll{2}\textrm{ and }h_1\in\ll{\frac{4}{3}}.
\end{equation}
Assume also that $f(-2,.)$ belongs to $L^2(P_{-2})$. Finally, assume \eqref{t0}. Then, we have:
\begin{equation}\label{p8e4}
\norm{f}_{\l{\infty}{2}}+\norm{\nabb f}_{\ll{2}}\lesssim\norm{H}_{\ll{2}}+\norm{h_1}_{\ll{\frac{4}{3}}}+\norm{f(-2,.)}_{L^2(P_{-2})}.
\end{equation}
\end{proposition}

\begin{proof}
We multiply \eqref{p8e1} by $f$ and integrate on $-2<u'<u$ where $u\leq 2$. Using integration by parts together with \eqref{coarea} and \eqref{du}, we obtain:
\begin{equation}\label{p8e5}
\begin{array}{ll}
& \ds\frac{1}{2}\norm{f(u,.)}_{L^2(P_{u})}^2+\norm{a^{-1/2}\nabb f}_{\ll{2}}^2\\[2mm]
\ds = & \ds\frac{1}{2}\norm{f(-2,.)}_{L^2(P_{-2})}^2+\frac{1}{2}\int_{-2}^u\int_{P_{u'}}a^{-1}\trt f^2d\mu_{u'}du'+\int_{-2}^u\int_{P_{u'}} hf d\mu_{u'}du'\\
\ds \lesssim & \ds\norm{f(-2,.)}_{L^2(P_{-2})}^2+\norm{\trt}_{\l{\infty}{4}}\norm{f}_{\l{2}{\frac{8}{3}}}^2+\int_{-2}^u\int_{P_{u'}} hf d\mu_{u'}du'.
\end{array}
\end{equation}
Taking \eqref{p8e2} into account, we have:
\begin{equation}\label{p8e6}
\begin{array}{lll}
\ds\int_{-2}^u\int_{P_{u'}}hf d\mu_{u'}du' & \ds = & \ds\int_{-2}^{u}\int_{P_{u'}}\divb(H)fad\mu_{u'} du'+\int_{-2}^u\int_{P_{u'}}fh_1d\mu_{u'}du'\\
& \ds = & \ds -\int_{-2}^u\int_{P_{u'}}H\nabb fd\mu_{u'}du' -\int_{-2}^u\int_{P_{u'}}\ana Hf d\mu_{u'}du'\\
& & \ds +\int_{-2}^u\int_{P_{u'}}fh_1d\mu_{u'}du'\\
& \ds\lesssim & \ds\norm{\nabb f}_{\ll{2}}\norm{H}_{\ll{2}}+\norm{f}_{\ll{4}}(\norm{\ana}_{\ll{4}}\norm{H}_{\ll{2}}\\
& & \ds +\norm{h_1}_{\ll{\frac{4}{3}}}),
\end{array}
\end{equation}
which together with Proposition \ref{p2} and  \eqref{t0} yields:
\begin{equation}\label{p8e7}
\int_{-2}^u\int_{P_{u'}}hf d\mu_{u'}du'\lesssim (\norm{f}_{\l{\infty}{2}}+\norm{\nabb f}_{\ll{2}})(\norm{H}_{\ll{2}}+\norm{h_1}_{\ll{\frac{4}{3}}}).
\end{equation}
Finally, \eqref{eq:GNirenberg},  \eqref{t0}, \eqref{p8e5} and \eqref{p8e7} imply \eqref{p8e4}.
\end{proof}

In section \ref{regx}, we will have among other things to control $\hth$ (the traceless part of $\th$). Now, 
according to the second equation of \eqref{struct}, $\hth$ satisfies an equation of the type $\divb(F)=h$. 
Thus, we conclude this section with an estimate that will allow us to control the solution to such equations.
\begin{proposition}\label{p6}
Let $F$ a symmetric 2-tensor such that tr$F=0$. Then:
\begin{equation}\label{e3}
\norm{\nabb F}_{\ll{2}}\lesssim \norm{\divb F}_{\ll{2}}+\norm{K}^{\frac{1}{2}}_{\ll{2}}\norm{F}_{\ll{4}}.
\end{equation}
\end{proposition}

\begin{proof}
This follows immediately from the following identity for Hodge systems (see for example \cite{FLUX}):
\begin{equation}\label{hodge}
\int_{\p}(|\nabb F|^2+2K|F|^2)=2\int_{\p}|\divb F|^2.
\end{equation}
\end{proof}

\section{Construction of the foliation and regularity with respect to $x$}\label{regx}

This section deals with the proof of Theorem \ref{thregx}. By section \ref{reducsmall}, we may assume that $(\s,g,k)$ coincides with $(\R^3,\delta,0)$ outside of a compact, say $|x|\geq 1$. Notice that in $|x|\geq 1$ and for all $\o\in\S$, the scalar function $x.\o$ satisfies the equation \eqref{choice} and the estimate \eqref{thregx1}, since $a\equiv 1, \th\equiv 0$ and $N\equiv\o$ in this region. Thus, we would like to construct a function $u$ solution of \eqref{choice} satisfying \eqref{thregx1} in a region containing $|x|\leq 2$ and to glue it to $x.\o$ in $1\leq |x|\leq 2$. Now, \eqref{choice} is of parabolic type - see \eqref{eqlapse1} - where $u$ plays the role of time. Therefore, for each $\o\in\S$, we will construct $u(.,\o)$ on a strip of type $S=\{x\in\s\textrm{ such that }-2< u(x,\o)< 2\}$ solution of:
\begin{equation}\label{choice1}
\left\{\begin{array}{l}
\trt-k_{NN}=1-a,\textrm{ on }-2<u<2,\\
u(.,\o)=-2\textrm{ on }x.\o=-2.
\end{array}\right.
\end{equation}

The rest of the section is as follows. We first prove a priori estimates consistent with the estimate \eqref{thregx1} and valid on $-2<u<2$ for the solution $u$ of \eqref{choice1}. We also prove on $-2<u<2$ a priori estimates for higher derivatives of the solution $u$ of \eqref{choice1}. We then recall 
the result obtained in \cite{papernash}, where we use a Nash-Moser procedure to obtain the existence of $u$ solution to:
\begin{equation}\label{choice2}
\left\{\begin{array}{l}
\trt-k_{NN}=1-a,\textrm{ on }\alpha<u<\alpha+T,\\
u=\alpha\textrm{ on }\underline{u}=\alpha,
\end{array}\right.
\end{equation}
where $-2\leq \alpha\leq 2$, $\underline{u}$ is smooth, and $T>0$ is small enough. Together with the a  priori estimates, this allows us to control the solution of \eqref{choice2} on $-2+kT<u<-2+(k+1)T$ uniformly with respect to $k=0,\dots,[4/T]$ in order to obtain a solution $u$ of \eqref{choice1} on $-2<u<2$. Finally, we conclude the proof of Theorem \ref{thregx} by showing how to glue the solution $u$ of \eqref{choice1} to $x.\o$ in $1\leq |x|\leq 2$ in order to obtain a solution on $\s$ satisfying \eqref{thregx1}.

\begin{remark}\label{highder}
In order to obtain higher order derivatives estimates for \eqref{choice1}, and in order to construct the solution of \eqref{choice2} using a Nash Moser procedure, we need to assume that $(\s,g,k)$ is smooth. We would like to insist on the fact that the smoothness is only assumed to obtain the existence of $u$ solution of \eqref{choice1}. On the other hand, we only rely on the control of $\norm{R}_{L^2(\s)}$ and $\norm{\nabla k}_{L^2(\s)}$ given by \eqref{small1} to prove the estimate \eqref{thregx1}. 
\end{remark}

\subsection{A priori estimates for lower order derivatives}\lab{sec:apriorilow}

Let $(\s,g,k)$ chosen as in section \ref{reducsmall}. In particular, we assume:
\begin{equation}\label{small2}
\norm{\nabla k}_{L^2(\s)}+\norm{R}_{L^2(\s)}\leq\ep.
\end{equation}
Let $u$ a scalar function on $\s\times\S$, and let $\p$, $a$, $N$, $\th$ and $K$ be associated to $u$ as in section \ref{sec:foliation}. Assume that $u$ satisfies the additional equation \eqref{choice}. The equations \eqref{frame} \eqref{struct} \eqref{gauss} may be rewritten:
\be\lab{frame1}
\left\{\begin{array}{l}
\nabla_AN=\th_{AB}e_B,\\[1mm]
\nabn N=-\nabb\lg,
\end{array}\right.
\ee
\be\lab{struct1}
\left\{\begin{array}{l}
\trt-k_{NN}=1-a,\\[1mm]
\nabn a-\lapa=|\th|^2+\nabn (k_{NN})+R_{NN},\\[1mm]
\nabb^B\hth_{AB}=\frac{1}{2}\nabb_A\trt+R_{NA},\\[1mm]
a^{-1}\nabb_A\nabb_Ba+\nabn\th_{AB}+\th_A^C\th_{CB}+K\gamma_{AB}=R_{AB},
\end{array}\right.
\ee
and
\be\lab{gauss1}
2K-\trt^2+|\th|^2=R-2R_{NN}.
\ee
In this section, we establish a priori estimates for $a$, $N$, $\th$ and $K$ corresponding to \eqref{thregx1} in the region $S$ of $\s$ between $P_{-2}$ and $P_2$ (i.e. $S=\{x\,/\,-2< u(x,\o)< 2\}$) where $u$ is initialized on $x.\o=-2$ by:
\begin{equation}\label{init}
u(x,\o)=-2\textrm{ on }x.\o=-2.
\end{equation}
Note that the first equation of \eqref{struct1}, \eqref{init} and the fact that $(g,k,\s)$ coincides with $(\delta,0,\R^3)$ for $|x|\geq 2$ yields:
\begin{equation}\label{init1}
\nabla^p(a-1)=0,\,\nabla^p\theta=0,\,\nabla^p(N-\o)=0\textrm{ for all }p\in\N\textrm{ on }u=-2,
\end{equation} 
so that the subsequent integrations by parts will not create boundary terms at $u=-2$. We will assume:
\begin{equation}\label{boot} 
\norm{a-1}_{\l{\infty}{2}}+\norm{\nabla a}_{\l{\infty}{2}}+\norm{a-1}_{\ll{\infty}}+\norm{\nabb\nabla a}_{\ll{2}}+\norm{K}_{\ll{2}}\leq D\ep,
\end{equation}
and
\begin{equation}\label{boot1} 
\norm{\nabla\theta}_{\ll{2}}\leq D^2\ep,
\end{equation}
where $D$ is a large enough constant. We will then try to improve on these estimates. Let us note that \eqref{init1}, \eqref{boot} and \eqref{boot1} together with Corollary \ref{c0}, Proposition \ref{p2} and 
Proposition \ref{p3} yield:
\begin{equation}\label{appboot} 
\norm{\nabn a}_{\ll{4}}+\norm{\nabb a}_{\l{\infty}{4}}+\norm{\nabb a}_{\ll{6}}\leq D^2\ep,
\end{equation}
and
\begin{equation}\label{appboot1} 
\norm{\th}_{\l{\infty}{2}}+\norm{\th}_{\l{\infty}{4}}+\norm{\th}_{\ll{6}}\leq D^3\ep.
\end{equation}
Also, using Corollary \ref{c0}, \eqref{small2}, and the fact that $k\equiv 0$ on $x.\o=-2$ by section 
\ref{reducsmall} yields:
\begin{equation}\label{appsmall2} 
\norm{k}_{\l{\infty}{2}}+\norm{k}_{\l{\infty}{4}}+\norm{k}_{\ll{6}}\leq D\ep.
\end{equation}


\subsubsection{Improvement of the bootstrap assumptions \eqref{boot1}}

We start by estimating $\th$. Since $\trt-k_{NN}=1-a$, we have from \eqref{boot}:
\begin{equation}\label{r0}
\norm{\trt-k_{NN}}_{\ll{\infty}}\leq D\ep.
\end{equation}
Also, the first equation of \eqref{struct1} together with \eqref{frame1} yields, schematically: 
$$\nabn\trt=\nabn k_{NN}-2k_{\nabb\lg N}-\nabn a,\,\nabb\trt=\nabb k_{NN}+2k_{N\c}\c\th-\nabb a,$$ 
so that: 
\begin{equation}\label{r3}
\begin{array}{lll}
\ds\norm{\nabla\trt}_{\ll{2}} & \lesssim & \ds\norm{\nabla a}_{\ll{2}}+\norm{\nabla k}_{\ll{2}}\\
& & \ds +\norm{k}_{\l{\infty}{4}}(\norm{\nabb a}_{\l{\infty}{4}}+\norm{\th}_{\l{\infty}{4}})\\
& \ds\lesssim & \ds (D+D^4\ep)\ep,
\end{array}
\end{equation}
where we have used the bootstrap assumption \eqref{boot}, \eqref{appboot}, \eqref{appboot1} and \eqref{appsmall2} to obtain the last inequality. We continue with the estimates for $\hth$. The third equation in \eqref{struct1} and Proposition \ref{p6} yield:
\begin{equation}\label{r4}
\norm{\nabb\hth}_{\ll{2}}\lesssim \norm{\nabb\trt}_{\ll{2}}+\norm{R_{N.}}_{\ll{2}}+
\norm{K}_{\ll{2}}^{\frac{1}{2}}\norm{\hth}_{\ll{4}}
\end{equation}
which together with \eqref{small2}, \eqref{boot}, \eqref{appboot1} and \eqref{r3} yields:
\begin{equation}\label{r5}
\norm{\nabb\hth}_{\ll{2}}\lesssim (D+D^4\ep^{\frac{1}{2}})\ep.
\end{equation}
Also, using the last equation of \eqref{struct1}, we have:
\begin{equation}\label{r6}
\norm{\nabn\th}_{\ll{2}}\lesssim \norm{\nabb^2a}_{\ll{2}}+\norm{K}_{\ll{2}}+\norm{R}_{\ll{2}}+
\norm{\th}^2_{\ll{4}}
\end{equation}
which together with \eqref{small2}, \eqref{boot} and \eqref{appboot1} yields:
\begin{equation}\label{r7}
\norm{\nabn\th}_{\ll{2}}\lesssim (D+D^6\ep)\ep.
\end{equation}
Finally, \eqref{r3}, \eqref{r5} and \eqref{r7} yield:
\begin{equation}\label{r8}
\norm{\nabla\th}_{\ll{2}}\lesssim (D+D^{6}\ep^{\frac{1}{2}})\ep,
\end{equation}
which is an improvement of \eqref{boot1}.

\subsubsection{Improvement of the bootstrap assumptions \eqref{boot}}

We now try to improve \eqref{boot}. Note first that \eqref{gauss1} yields:
\begin{equation}\label{r9}
\norm{K}_{\ll{2}}\lesssim \norm{\trt}^2_{\ll{4}}+\norm{\th}^2_{\ll{4}}+\norm{R}_{\ll{2}}.
\end{equation}
Together with \eqref{small2} and \eqref{appboot1}, this yields:
\begin{equation}\label{r10}
\norm{K}_{\ll{2}}\lesssim (1+D^6\ep)\ep.
\end{equation}
We rewrite the second equation of \eqref{struct1} as:
\begin{equation}\label{r11}
(\nabn-a^{-1}\lap)(a-1)=h,
\end{equation}
where $h$ is given by:
\begin{equation}\label{r12}
h=|\th|^2+\nabn (k_{NN})+R_{NN}.
\end{equation}
Using the second equation of \eqref{frame1} implies:
\begin{equation}\label{r12bis}
\nabn(k_{NN})=\nabn k_{NN}+k(\nabn N,N)=\nabn k_{NN}-k(\nabb a,N),
\end{equation}
which together with \eqref{r12} yields:
\begin{equation}\label{r12ter}
h=|\th|^2+\nabn k_{NN}-k(\nabb a,N)+R_{NN}.
\end{equation}
Using \eqref{small2}, \eqref{appboot}, \eqref{appboot1}, \eqref{appsmall2} and \eqref{r12ter}, we obtain:
\begin{equation}\label{r13}
\norm{h}_{\ll{2}}\lesssim (1+D^6\ep)\ep.
\end{equation}
Using Proposition \ref{p7}, \eqref{init1}, \eqref{boot}, \eqref{appboot}, \eqref{appboot1}, \eqref{r11} and  \eqref{r13} we obtain:
\begin{equation}\label{r19}
\ds\norm{a-1}_{\l{\infty}{2}}+\norm{\nabb a}_{\l{\infty}{2}}+\norm{\nabn a}_{\ll{2}}+\norm{\nabb^2a}_{\ll{2}}\lesssim (1+D^6\ep)\ep.
\end{equation}

In order to obtain estimates for $\nabb\nabn a$ and $\nabn^2a$, we differentiate the second equation of \eqref{struct1} by $\nabn$:
\begin{equation}\label{r20}
(\nabn - a^{-1}\lap)\nabn a= [\nabn, a^{-1}\lap]a+2\th\nabn\th+\nabn^2k_{NN}+\nabn R_{NN}, 
\end{equation}
Using \eqref{r12bis}, we have:
\begin{equation}\label{r201}
\nabn^2(k_{NN})=\nabn(\nabn k_{NN}-k(\nabb a,N)).
\end{equation}
The commutator formula \eqref{scommut} and the second equation of \eqref{frame1} yield:
\begin{equation}\label{r202}
\begin{array}{lll}
\ds\nabn(k(\nabb a,N)) & = & \ds -\nabn k(\nabb a,N)
-k(\nabn\nabb a,N)\\
& & \ds +k(\nabb a,\nabb a)\\
& = & \ds -\nabn k(\nabb a,N)
-k(\nabb\nabn a,N)\\
& & \ds -\nabn a k(\nabb a,N)+\th(\nabb a,e_A)k_{AN}\\
& & \ds +k(\nabb a,\nabb a).
\end{array}
\end{equation}
Using the constraint equations \eqref{const} and the fact that we have a maximal foliation yields:
\bea\label{r21}
\nabn k(NN)&=&-\nabla_Ak_{AN}\\
\nn&=&-\divb (k_{N.})-\trt k_{NN}+\th_{AB}k_{AB}, 
\eea
which together with the commutator formula \eqref{commut}, the second equation of \eqref{frame1}, \eqref{r201} and \eqref{r202} implies, schematically:
\begin{equation}\label{r22}
\begin{array}{ll}
\ds \nabn^2(k_{NN})= & \ds -\divb(\nabn k_{N.})+\ana\nabn k_{N.}+\theta\nabb k_{N.}+R_{N.}k+\theta\ana k_{N.}\\
& \ds +\nabn\th k+\th\nabn k+\th k \nabb a+\nabn k(\nabb a,N)\\
&\ds +k(\nabb\nabn a,N) +\nabn a k(\nabb a,N)\\
&\ds +\th k+k(\nabb a,\nabb a).
\end{array}
\end{equation}
We use the twice-contracted Bianchi identity on $\s$ 
\begin{equation}\label{bianchi}
\nabla^jR_{ij}=\frac{1}{2}\nabla_iR,
\end{equation}
together with the constraint equations \eqref{const} to express $\nabn R_{NN}$:
\bea\label{bianchi1}\label{r22ter}
\nabn R_{NN}&=& -\nabla_AR_{AN}+k\c\nabn k\\
\nn&=&-\divb(R_{.N})+\trt R_{NN}-\th_{AB}R_{AB}+k\c\nabn k.
\eea
Finally, we use the commutator formula \eqref{commut1} for a scalar $f$: 
\bea\label{commut1conseq}
 a[\nabn,a^{-1}\lap]a &=& -\trt\lap a-2\hth\c\nabb^2a+2a^{-1}\nabb a\c\nabb\nabn a-2R_{N.}\c\nabb a\\
&&\nn-\nabb\trt\c\nabb a-2\hth\c a^{-1}\nabb a\c\nabb a.
\eea
\eqref{r20}, \eqref{r22}, \eqref{r22ter} and \eqref{commut1conseq} yield:
\begin{equation}\label{r23}
(\nabn - a^{-1}\lap)\nabn a=\divb(H)+h_1, 
\end{equation}
where the tensor $H$ is given by
\begin{equation}\label{r23bis}
H=-\nabn k_{.N}-R_{.N}, 
\end{equation}
and where the scalar $h_1$ is given schematically by
\begin{equation}\label{r24}
\begin{array}{ll}
\ds h_1= &  -a^{-1}\trt\lap a-2a^{-1}\hth\nabb^2a+2a^{-2}\nabb a\nabb\nabn a-2R_{N.}\ana -\nabb\trt\ana\\
& \ds +2\hth |\ana|^2+2\th\nabn\th+\ana\nabn k_{N.}+\nabla\th k+\theta\nabla k+R_{N.}k \\
&\ds+\theta\ana k_{N.}+2k\nabn k+\nabn k(\nabb a,N)+k(\nabb\nabn a,N)\\
&\ds+\nabn a k(\nabb a,N)+\th k+ k(\nabb a,\nabb a)+\th R.
\end{array}
\end{equation}
We estimate $H$ in $\ll{2}$ using \eqref{small2}: 
\begin{equation}\label{r24bis}
\norm{H}_{\ll{2}}\leq \norm{\nabla k}_{\ll{2}}+\norm{R}_{\ll{2}}\leq 2\ep. 
\end{equation}
We estimate $h$ in $\l{2}{\frac{4}{3}}$:
\begin{equation}\label{r25}
\begin{array}{r}
\ds\norm{h}_{\l{2}{\frac{4}{3}}}\lesssim (\norm{\th}_{\l{\infty}{4}}+\norm{\nabb a}_{\l{\infty}{4}}+ \norm{k}_{\l{\infty}{4}})\\[3mm]
\ds \times(\norm{\nabb^2a}_{\ll{2}}+\norm{\nabb\nabn a}_{\ll{2}}+\norm{R}_{\ll{2}}+\norm{\nabla\th}_{\ll{2}}+\norm{\nabla k}_{\ll{2}}\\
\ds +\norm{\nabb a}^2_{\ll{4}}+\norm{\th}^2_{\ll{4}}+\norm{\nabn a}^2_{\ll{4}}+\norm{k}^2_{\ll{4}})
\end{array}
\end{equation}
which together with \eqref{small2}, \eqref{boot}, \eqref{boot1}, \eqref{appboot}, \eqref{appboot1} and \eqref{appsmall2} yields:
\begin{equation}\label{r27}
\ds\norm{h}_{\l{2}{\frac{4}{3}}}\lesssim D^9\ep^2.
\end{equation}
Using Proposition \ref{p8}, \eqref{init1}, \eqref{appboot}, \eqref{appboot1}, \eqref{r23}, \eqref{r24bis} and \eqref{r27} we obtain:
\begin{equation}\label{r28}
\begin{array}{r}
\norm{\nabn a}_{\l{\infty}{2}}+\norm{\nabb\nabn a}_{\ll{2}}\lesssim (1+D^9\ep)\ep.
\end{array}
\end{equation}
Now, Proposition \ref{p3} together with \eqref{r19} and \eqref{r28} yields:
\begin{equation}\label{r31}
\norm{a-1}_{\ll{\infty}}\lesssim (1+D^{9}\ep)\ep.
\end{equation}
Finally, \eqref{r10}, \eqref{r19}, \eqref{r28} and \eqref{r31} imply:
\begin{equation}\label{r32} 
\begin{array}{rr}
\ds\norm{a-1}_{\l{\infty}{2}}+\norm{\nabla a}_{\l{\infty}{2}}+\norm{a-1}_{\ll{\infty}} & \\
+\norm{\nabb\nabla a}_{\ll{2}}+\norm{K}_{\ll{2}} & \ds\lesssim (1+D^9\ep)\ep,
\end{array}
\end{equation}
which is an improvement of \eqref{boot}.

Thus, there is a universal constant $D$ such that \eqref{boot} and \eqref{boot1} hold. Together with \eqref{frame1} and \eqref{r0}, this yields \eqref{thregx1}. 

\subsection{A priori estimates for higher order derivatives}

In addition to \eqref{small2}, we assume the following control on $R$ and $k$:
\begin{equation}\label{h1}
\norm{\nabla^jR}_{\ll{2}}+\norm{\nabla^{1+j}k}_{\ll{2}}\leq M,\textrm{ for all }1\leq j\leq 300,
\end{equation}
where $M$ is a large constant. The goal of this section is to prove the following proposition:
\begin{proposition}\label{h2}
Let $(\s,g,k)$ chosen as in section \ref{reducsmall}, and satisfying \eqref{h1}. Let $u$ a scalar function defined on $S=\{x\,/\,-2< u(x,\o)< 2\}$, and let $\p$, $a$, $N$, $\th$ and $K$ be associated to $u$ as in section \ref{sec:foliation}. Assume that $u$ satisfies the additional equation \eqref{choice} and is initialized on $x.\o=-2$ by \eqref{init}. Then, $a$ and $\theta$ satisfy the following estimates:
\begin{equation}\label{h3}
\norm{\nabb^2\nabla^{j-1}a}_{\ll{2}}+\norm{\nabla^ja}_{\ll{2}}+\norm{\nabla^j\th}_{\ll{2}}\leq C(M),\textrm{ for all }1\leq j\leq 300.
\end{equation} 
\end{proposition}

\begin{remark}
In connection with Remark \ref{highder}, let us insist again on the fact that the assumption \eqref{h1} is only used to obtain the existence of $u$ solution to \eqref{choice1}. 
\end{remark}

The proof of Proposition \ref{h2} is postponed to Appendix \ref{sec:proofhigherregapp}.

\subsection{Construction of the foliation on a small strip using a Nash Moser procedure}

In this section, we recall the local existence result obtained for \eqref{choice2} in \cite{papernash}. 
Let $-2\leq \alpha\leq 2$, a function $\underline{u}$ on $\s$ which is smooth in $\underline{u}\leq\alpha$, and $T>0$. To a function $u$ on $\s$, we associate $\p$, $a$, $N$ and $\th$ as in section \ref{sec:foliation}. We then define the nonlinear map $\phi$: 
\begin{equation}\label{nlmap}
\phi(u)=\trt-1+a-k_{NN}.
\end{equation}
Using $\phi$, we may rewrite \eqref{choice2} as:
\begin{equation}\label{choice3}
\left\{\begin{array}{l}
\phi(u)=0\textrm{ on }\alpha<u<\alpha+T,\\
u=\alpha\textrm{ on }\underline{u}=\alpha.
\end{array}\right.
\end{equation}
In \cite{papernash}, we prove that for $T>0$ small enough, we can construct a solution $u$ of \eqref{choice3} using a Nash Moser procedure: 
\begin{theorem}\label{nashmoser}
Assume that $g$ and $k$ are smooth, and that $\underline{u}$ is smooth in a neighborhood of $\underline{u}\leq\alpha$ where it satisfies $\phi(\underline{u})\equiv 0$. Assume also that the lapse of $\underline{u}$ satisfies $|\underline{a}-1|\leq 1/4$ on $\underline{u}\leq\alpha$. Then, there exists a constant $\underline{T}>0$ and a solution $u$ of \eqref{choice3} which is smooth in $\alpha<u<\alpha+\underline{T}$. Furthermore, $\underline{T}$ only depends on the norm of $(g,k)$ in $H^{300}(\s)$, and on the norm of $\underline{u}$ in $H^{300}$ in $\underline{u}\leq\alpha$.
\end{theorem}

\begin{remark}
The linearized operator $\phi'(u)$ is explicitly computed in \cite{papernash} to be:
\begin{equation}\label{estimlin10}
\phi'(u)h=a(\lap h+3a^{-1}\nabb a.\nabb h)-a^2\nabn h-2ak(N,\nabb h).
\end{equation}
Thus, $\phi'(u)$ is of parabolic type, where $u$ plays the role of time, $\nabn$ the role of $\partial_t$, and $\nabb$ the role of $\partial_x$. This is due to the fact that the nonlinear problem is itself of parabolic nature as exhibited by the equation of the lapse \eqref{eqlapse1}.
 The estimates obtained for \eqref{estimlin10} in \cite{papernash} exhibit a loss of derivatives which prevents us from using a standard Picard iterative scheme. Fortunately, these estimates are tame so that one can prove Theorem \ref{nashmoser} by performing a Nash Moser type iterative scheme (see \cite{papernash}).
\end{remark}

\begin{remark}
We do not claim any sharpness in the Sobolev exponents appearing in the statement of Theorem \ref{nashmoser}. Our goal is to obtain an existence result with $\underline{T}>0$ depending only on a fixed number of derivatives of  $(g,k)$ and $\underline{u}$, no matter how large this fixed number is. 
\end{remark}

\begin{remark}
There are numerous existence results in the literature for quasilinear parabolic equations of the form: 
\begin{equation}\label{eq:parabex}
\partial_tu-F(t,x,u,\partial_xu,\partial^2_xu)=0, 
\end{equation}
where $F$ is a nonlinear map such that $\partial_t-F'$ is a parabolic operator (see for example \cite{LaSoUr}). The main difference between \eqref{choice3} and \eqref{eq:parabex} lies in the fact that the 'time' $u$ and the 'time derivative' $\nabn$ depend themselves on the solution $u$. This considerably complicates the analysis. Indeed, one 
may solve \eqref{eq:parabex} by a standard Picard iteration scheme, while \eqref{choice3}  
requires an intricate Nash Moser procedure. In particular, to prove that \eqref{estimlin10} satisfies tame 
estimates, one has to use inhomogeneous Sobolev spaces that depend on $u$. In turn, the norms in which the converging Nash Moser sequence of Theorem \ref{nashmoser} is evaluated depend on the sequence itself (see \cite{papernash}).
\end{remark}

\subsection{Proof of Theorem \ref{thregx}}

We apply here the strategy explained in the introduction of section \ref{regx}. Let $0<\alpha\leq 4$. We look for a solution $u(.,\o)$ to:
\begin{equation}\label{choice4}
\left\{\begin{array}{l}
\trt-k_{NN}=1-a,\textrm{ on }-2<u<-2+\alpha,\\
u(.,\o)=-2\textrm{ on }x.\o=-2.
\end{array}\right.
\end{equation}
Theorem \ref{nashmoser} ensures that $u(.,\o)$ solution of \eqref{choice4} exists as long as $|a-1|\leq 1/4$ and the norm of $u(.,\o)$ in $H^{300}$ in $-2\leq u\leq -2+\alpha$ stays under control. Now, the a priori estimates \eqref{boot} and \eqref{h3} yield $|a-1|\leq 1/4$ and the control of the norm of $u(.,\o)$ in $H^{300}$ in $-2\leq u\leq 2$. Thus, we deduce the existence of $u(.,\o)$ solution of:
\begin{equation}\label{choice5}
\left\{\begin{array}{l}
\trt-k_{NN}=1-a,\textrm{ on }-2<u<2,\\
u(.,\o)=-2\textrm{ on }x.\o=-2,
\end{array}\right.
\end{equation}
satisfying \eqref{boot}, \eqref{boot1} and \eqref{h3} on $-2<u<2$.

Now, we would like to glue the solution $u(.,\o)$ of \eqref{choice5} to $x.\o$ in the region $1\leq |x|\leq 2$ where $(\s,g,k)$ coincides with $(\R^3,\delta,0)$ by section \ref{reducsmall}. We will use the following lemma.
\begin{lemma}\lab{lemma:uprochexo}
Let $u(.,\o)$ the solution of \eqref{choice5} satisfying \eqref{boot}, \eqref{boot1} and \eqref{h3} on $-2<u<2$. Then, we have:
\begin{equation}\label{imp9}
(1+|x|)^{-1}|u-x.\o|+|\nabla u-\o|\lesssim\ep,\textrm{ in }\{|x|\geq 1\}\cap \{-2<u<2\}.
\end{equation} 
\end{lemma}

The proof of Lemma \ref{lemma:uprochexo} is postponed to the end of the section. We now conclude the proof of Theorem \ref{thregx} by showing how to glue $u$ and $x.\o$ together in $\{1\leq |x|\leq 2\}$. Let $\varphi$ a smooth function with compact support which is equal to 1 on 
$|x|\leq 1$ and to 0 on $|x|\geq 2$. Let $\tilde{u}$ be defined on $\s$ by:
\begin{equation}\label{imp10}
\tilde{u}=\varphi u+(1-\varphi)x.\o.
\end{equation} 
Then, $\tilde{u}$ satisfies {\bf C1a}. Also, since $u$ satisfies \eqref{boot} and \eqref{boot1} in $\{-2<u<2\}$, since $x.\o$ satisfies the same estimates in $|x|\geq 1$, and since we 
have \eqref{imp9} on $1\leq |x|\leq 2$, $\tilde{u}$ satisfies \eqref{thregx1} on $\s$. 
This concludes the proof of Theorem \ref{thregx}. \\

{\bf Proof of Lemma \ref{lemma:uprochexo}} \ We first show that $u(.,\o)$ satisfies better 
estimates in this region due to the hypoellipticity of the parabolic-elliptic system \eqref{struct1}. In particular, we obtain the following improvement of \eqref{h3} for $j=2$:
\begin{equation}\label{imp1}
\norm{\nabb^2\nabla a}_{\ll{2}}+\norm{\nabla^2a}_{\ll{2}}+\norm{\nabla^2\th}_{\ll{2}}\lesssim\ep,\textrm{ in }\{|x|\geq 1\}\cap \{-2<u<2\}.
\end{equation} 
In fact, $C(M)$ in \eqref{h3} comes from the assumption \eqref{h1} on the norms of $R$ and $k$. However, since $R$ and $k$ vanish in $|x|\geq 1$, we may take $M=0$ in this region. Let us prove for example the estimate for 
$\norm{\nabb^2\nabn a}_{\ll{2}}$ in \eqref{imp1}, the others being similar. 
Let $\varphi$ a smooth function with compact support which is equal to 1 on 
$|x|\leq 1$. Using \eqref{h4}, we obtain an equation for $(1-\varphi)\nabn a$:
\begin{equation}\label{imp2}
(\nabn -a^{-1}\lap)[(1-\varphi)\nabn a]=(1-\varphi)h+\tilde{h}
\end{equation} 
where $h$ is given by \eqref{h5} and $\tilde{h}$ is given by:
\begin{equation}\label{imp3}
\tilde{h}=-\nabn\varphi\nabn a+a^{-1}\lap\varphi\nabn a+2a^{-1}\nabb\varphi\nabb\nabn a.
\end{equation} 
\eqref{h12} and the fact that $R$ and $k$ vanish on the support of $1-\varphi$ yield:
\begin{equation}\label{imp3bis}
\norm{(1-\varphi)h}_{\ll{2}}\lesssim \ep\norm{(1-\varphi)\nabla^2\th}_{\ll{2}}+\ep.
\end{equation} 
\eqref{boot} and the fact that $\varphi$ is smooth yields:
\begin{equation}\label{imp3ter}
\norm{\tilde{h}}_{\ll{2}}\lesssim \ep.
\end{equation} 
Proposition \ref{p7}, \eqref{imp2}, \eqref{imp3bis} and \eqref{imp3ter} yield:
\begin{equation}\label{imp6}
\norm{(1-\varphi)\nabn^2a}_{\ll{2}}+\norm{(1-\varphi)\nabb^2\nabn a}_{\ll{2}}\lesssim \sqrt{\ep}\norm{(1-\varphi)\nabla^2\th}_{\ll{2}}+\ep.
\end{equation} 
In the same fashion, we adapt the analysis of \eqref{hh1}-\eqref{h25} and we use the fact that $R$ and $k$ vanish on the support of $1-\varphi$ to obtain estimates for 
$\norm{(1-\varphi)\nabb^3a}_{\ll{2}}$ and $\norm{(1-\varphi)\nabla^2\th}_{\ll{2}}$ which yield 
\eqref{imp1}.

We now use \eqref{imp1} and the fact that $u=-2$ on $x.\o=-2$ to show that $u$ and $x.\o$ are close to 
each other in the region $\{|x|\geq 1\}\cap \{-2<u<2\}$. Proposition \ref{p3}, \eqref{frame1} and \eqref{imp1} yield:
\begin{equation}\label{imp7}
|\nabla N|\lesssim\ep,\textrm{ in }\{|x|\geq 1\}\cap \{-2<u<2\}.
\end{equation} 
Since $N=\o$ on $x.\o=-2$, \eqref{imp7} yields:
\begin{equation}\label{imp8}
|N-\o|\lesssim\ep,\textrm{ in }\{|x|\geq 1\}\cap \{-2<u<2\}.
\end{equation} 
$u=x.\o$ on $x.\o=-2$, so since $\nabla u=a^{-1}N$, \eqref{boot} and \eqref{imp8} yield the desired estimate \eqref{imp9}. This concludes the proof of Lemma \ref{lemma:uprochexo}. \hspace*{\fill}\rule{2.5mm}{2.5mm} 

\section{Littlewood-Paley theory on $\p$ and consequences}\lab{sec:estimatesLP}

In this section, we introduce several tools which will be needed to prove Theorem \ref{thnabn2a} and Theorem \ref{thregomega}. We introduce and recall the main properties of the family of intrinsic Littlewood-Paley projections $P_j$ which has been constructed in \cite{LP} using the heat flow on the surfaces $\p$. We then prove a crucial bound for $K$. This allows us to derive suitable commutator estimates, product estimates and estimates for parabolic equations.

\begin{remark}
Recall that $(\s,g,k)$ coincides with $(\R^3,\de,0)$ in $|x|\geq 2$. Also, $u(x,\o)$ coincides with $x.\o$ in $|x|\geq 2$, and so $a\equiv 1$, $N\equiv\o$, $\th\equiv 0$ and $K\equiv 0$ in this region. Therefore, $u$ clearly satisfies the estimates of Theorem \ref{thnabn2a}, Theorem \ref{thregomega} and of the propositions thereafter in the region $|x|\geq 2$. Thus, in the rest of the paper, we will restrict the proof all our estimates in the strip $S=\{x/\,-2<u<2\}$ where  $u(x,\o)$ is solution to: 
\begin{displaymath}
\left\{\begin{array}{l}
\trt-k_{NN}=1-a,\textrm{ on }-2<u<2,\\
u(.,\o)=-2\textrm{ on }x.\o=-2.
\end{array}\right.
\end{displaymath}
\end{remark}

\subsection{Properties of the geometric Littlewood-Paley projections $P_j$}\lab{sec:LP}

In this section, we introduce and recall the main properties of the family of intrinsic Littlewood-Paley projections $P_j$ which has been constructed in \cite{LP} using the heat flow on the surfaces $\p$. We recall the properties of the heat  equation  for arbitrary 
tensorfields $F$ on $\p$.
$$\partial_\tau U(\tau)F -\lap U(\tau) F=0, \,\, U(0)F=F.$$
The following $L^2$ estimates for the operator
$U(\tau)$ are proved in \cite{LP}.
\begin{proposition}
We have the following estimates for the operator $U(\tau)$:
\begin{align}
&\|U(\tau) F\|^2_{\lp{2}}+\int_0^\tau\norm{\nabb U(\tau')F}^2_{\lp{2}}d\tau'\lesssim \|F\|^2_{L^2(S)},\label{eq:l2heat1}\\
&\|\nabb U(\tau) F\|^2_{\lp{2}}+\int_0^\tau\norm{\lap U(\tau')F}^2_{\lp{2}}d\tau'\lesssim \|\nabb F\|^2_{L^2(S)},\label{eq:l2heatnab}\\
& \tau\|\nabb U(\tau) F\|^2_{\lp{2}}+\int_0^\tau{\tau'}\norm{\lap U(\tau')F}^2_{\lp{2}}d\tau'\lesssim \|F\|^2_{L^2(S)}.\label{eq:l2heat2}
\end{align}
\label{le:L2heat}
\end{proposition}
We also introduce the nonhomogeneous heat equation:
$$\partial_\tau V(\tau) -\lap V(\tau)=F(\tau), \,\, V(0)=0,$$
for which we easily derive the following estimates:
\begin{proposition}
Let $\b>0$. We have the following estimates for the operator $V(\tau)$:
\begin{align}
&\norm{\nabb V(\tau)}^2_{L^2(\p)}+\int_0^\tau\norm{\lap V(\tau')}^2_{L^2(\p)}d\tau'\lesssim \int_0^\tau\norm{F(\tau')}^2_{L^2(\p)}d\tau',\label{heatF1}\\
&\norm{V(\tau)}^2_{L^2(\p)}+\int_0^\tau\norm{\nabb V(\tau')}^2_{L^2(\p)}d\tau'\lesssim \int_0^\tau\int_{\p}V(\tau')F(\tau')d\mu_ud\tau',\label{heatF2}\\
&\tau\norm{V(\tau)}^2_{L^2(\p)}+\int_0^\tau{\tau'}\norm{\nabb V(\tau')}^2_{L^2(\p)}d\tau'\lesssim \int_0^\tau\int_{\p}{\tau'}V(\tau')F(\tau')d\mu_ud\tau',\label{eq:l2heat0}\\
&\tau^{2\b}\norm{V(\tau)}^2_{L^2(\p)}+\int_0^\tau{\tau'}^{2\b}\norm{\nabb V(\tau')}^2_{L^2(\p)}d\tau'\lesssim \int_0^\tau\int_{\p}{\tau'}^{2\b}V(\tau')F(\tau')d\mu_ud\tau',\nonumber \\
& \hspace{8.5cm} +\int_0^\tau{\tau'}^{2\b-1}\norm{V(\tau')}^2_{L^2(\p)}d\tau'.\label{eq:l2heat2bis}
\end{align}
\label{le:L2heatbis}
\end{proposition}

We now recall the definition of the geometric Littlewood-Paley projections $P_j$ constructed in \cite{LP}:
\begin{definition}\label{defLP}
Consider a smooth function $m$ on $[0,\infty)$,
vanishing sufficiently fast at $\infty$,
verifying the  vanishing  moments property:
\be\lab{eq:moments}
\int_0^\infty \tau^{k_1}\partial_\tau^{k_2} m(\tau) d\tau=0, \,\,\,\,
|k_1|+|k_2|\le N. 
\end{equation}
We set,
  $m_j(\tau)=2^{2j}m(2^{2j}\tau)$ 
and  define the geometric Littlewood -Paley (LP) 
projections $P_j$, for   arbitrary tensorfields  $F$ on $S$
to be 
\be\lab{eq:LP}P_j F=\int_0^\infty m_j(\tau) U(\tau) F d\tau.
\end{equation}
Given an interval $I\subset \Bbb Z$ we define $$P_I=\sum_{j\in I} P_j F.$$
In particular we shall use the notation $P_{<k}, P_{\le k}, P_{>k}, P_{\ge k}$.
\end{definition}
Observe that $P_j$ are selfadjoint, i.e., $P_j=P_j^*$, in the sense,
$$<P_jF, G>=<F,P_j G>,$$
where, for any given $m$-tensors $F,G$ 
$$<F,G>=\int_{\p}\ga^{i_1j_1}\ldots\ga^{i_mj_m}
F_{i_1\ldots i_m}G_{j_1\ldots j_m}d\mu_u    $$ 
denotes the usual $L^2$ scalar product. Recall also 
from \cite{LP} that there exists a function $m$ satisfying \eqref{eq:moments} 
such that the LP-projections associated to $m$ verify:
\be\lab{eq:partition}
\sum_jP_j=I.
\end{equation}


The following properties of the LP-projections $P_j$ have been proved in \cite{LP}:
\begin{theorem}\label{thm:LP}
 The LP-projections $P_j$ verify the following
 properties:

i)\quad {\sl $L^p$-boundedness} \quad For any $1\le
p\le \infty$, and any interval $I\subset \Bbb Z$,
\be\lab{eq:pdf1}
\|P_IF\|_{\lp{p}}\lesssim \|F\|_{\lp{p}}
\end{equation}

ii) \quad  {\sl Bessel inequality} 
$$\sum_j\|P_j F\|_{\lp{2}}^2\lesssim \|F\|_{\lp{2}}^2$$

iii)\quad {\sl Finite band property}\quad For any $1\le p\le \infty$.
\begin{equation}
\begin{array}{lll}
\|\lap P_j F\|_{\lp{p}}&\lesssim & 2^{2j} \|F\|_{\lp{p}}\\
\|P_jF\|_{\lp{p}} &\lesssim & 2^{-2j} \|\lap F \|_{\lp{p}}.
\end{array}
\end{equation}

In addition, the $L^2$ estimates
\begin{equation}
\begin{array}{lll}
\|\nabb P_j F\|_{\lp{2}}&\lesssim & 2^{j} \|F\|_{\lp{2}}\\
\|P_jF\|_{\lp{2}} &\lesssim & 2^{-j} \|\nabb F  \|_{\lp{2}}
\end{array}
\end{equation}
hold together with the dual estimate
$$\| P_j \nabb F\|_{\lp{2}}\lesssim 2^j \|F\|_{\lp{2}}$$

iv) \quad{\sl Weak Bernstein inequality}\quad For any $2\le p<\infty$
\begin{align*}
&\|P_j F\|_{\lp{p}}\lesssim (2^{(1-\frac 2p)j}+1) \|F\|_{\lp{2}},\\
&\|P_{<0} F\|_{\lp{p}}\lesssim \|F\|_{\lp{2}}
\end{align*}
together with the dual estimates 
\begin{align*}
&\|P_j F\|_{\lp{2}}\lesssim (2^{(1-\frac 2p)j}+1) \|F\|_{\lp{p'}},\\
&\|P_{<0} F\|_{\lp{2}}\lesssim \|F\|_{\lp{p'}}
\end{align*}
\end{theorem}

We use the Littlewood-Paley projections $P_j$ to define Sobolev spaces $H^b(\p)$.
\begin{definition}\lab{def:Hs}
Let $b\in\mathbb{R}$. Then, we define the Sobolev space $H^b(\p)$ as follows:
$$\norm{F}_{H^b(\p)}^2=\sum_{j\geq 0}2^{2jb}\norm{P_jF}^2_{\lp{2}}+\norm{P_{< 0}F}^2_{\lp{2}}.$$
\end{definition}

Let us state a lemma about the action of $\nabb$ on $H^b(\p)$. 
\begin{lemma}\lab{lemma:vacances}
Let $0<b<1$. Let $F$ a tensor on $\p$ such that $F\in\hs{b}$. Then, $\nabb F\in\hs{b-1}$.
\end{lemma}

\begin{proof}
We have:
\be\lab{vacances}
\norm{P_j\nabb F}_{\lp{2}}\les \sum_{l\geq 0}\norm{P_j\nabb P_lF}_{\lp{2}}.
\ee
If $l\leq j$, we use the boundedness of $P_j$ on $\lp{2}$ and the finite band property for $P_l$ to obtain:
\bea\lab{vacances1}
2^{j(b-1)}\norm{P_j\nabb P_lF}_{\lp{2}} &\les& 2^{j(b-1)}\norm{\nabb P_lF}_{\lp{2}}\\
\nn &\les& 2^{j(b-1)}2^l\norm{P_lF}_{\lp{2}}\\
\nn &\les& 2^{-|j-l|(1-b)}2^{bl}\norm{P_lF}_{\lp{2}},
\eea
where we used in the last inequality the fact that $l\leq j$ and $b<1$. 

If $l>j$, we use the finite band property for $P_j$ to obtain:
\bea\lab{vacances2}
2^{j(b-1)}\norm{P_j\nabb P_lF}_{\lp{2}} &\les& 2^{j(b-1)}2^j\norm{P_lF}_{\lp{2}}\\
\nn &\les& 2^{-|j-l|b}2^{bl}\norm{P_lF}_{\lp{2}},
\eea
where we used in the last inequality the fact that $l> j$ and $b>0$. Finally, \eqref{vacances}, \eqref{vacances1} and 
\eqref{vacances2} imply:
\bee
\sum_{j\geq 0}2^{2(b-1)j}\norm{P_j\nabb F}^2_{\lp{2}}&\les& \sum_{j\geq 0}\left(\sum_{l\geq 0}2^{-\min(b,1-b)|j-l|}2^{lb}\norm{P_lF}_{\lp{2}}\right)^2\\
&\les& \sum_{l\geq 0}2^{2lb}\norm{P_lF}^2_{\lp{2}}\\
&\les& \norm{F}_{\hs{b}}^2,
\eee
where we used the fact that $\min(b, 1-b)>0$. This concludes the proof of the lemma.
\end{proof}

We also recall the definition of the negative  fractional powers of $\La^2=I-\lap$ on any 
 smooth tensorfield $F$ on $\p$ used in \cite{LP}.
\be\lab{eq:defineLaa}
\La^{\a}F=\frac{1}{\Gamma(-\a/2)}\int_0^\infty \tau^{-\frac{\a}{2}-1}
e^{-\tau}U(\tau)F d\tau
\end{equation}
 where $\a$ is an arbitrary complex number with $\Re(\a)< 0$ and $\Gamma$ denotes the Gamma function. We extend the definition of fractional powers of $\La$ to the range of $\a$ with $\Re(\a)>0$, on smooth tensorfields $F$, by defining first 
 $$
 \La^\a F = \La^{\a-2} \c (I-\lap) F 
 $$
 for $0< \Re (\a) \le 2$ and then, in general, for 
 $0< \Re(\a) \le 2n$, with an arbitrary positive integer $n$, according 
 to the formula
 $$
 \La^\a F = \La^{\a-2n} \c (I-\lap)^n F.
 $$
With this definition, $\La^\a$ is symmetric and verifies the group property
$\La^{\a}\La^{\b} =\La^{\a+\b}$. We also have by standard complex interpolation the following inequality:
\begin{equation}\label{interpolLa}
\norm{\La^{\mu\a+(1-\mu)\b}F}_{\lp{2}}\lesssim\norm{\La^{\a}F}^{\mu}_{\lp{2}}\norm{\La^{\b}F}^{1-\mu}_{\lp{2}}.
\end{equation}
Using the operators $\La^\a$, we complete \eqref{eq:l2heat1}-\eqref{eq:l2heat2} with:
\begin{align}
&\|\La^{-1}U(\tau) F\|^2_{\lp{2}}+\int_0^\tau{\tau'}\norm{\nabb\La^{-1}U(\tau')F}^2_{\lp{2}}d\tau'\lesssim \|\La^{-1}F\|^2_{L^2(S)},\label{eq:l2heat5}\\
&\tau\| U(\tau) F\|^2_{\lp{2}}+\int_0^\tau{\tau'}\norm{\nabb U(\tau')F}^2_{\lp{2}}d\tau'\lesssim \|\La^{-1}F\|^2_{L^2(S)},\label{eq:l2heatna6}
\end{align}
for $\a\in\mathbb{R}$,
\be\label{eq:l2heat1bis}
\norm{\La^\a V(\tau)}^2_{L^2(\p)}+\int_0^\tau\norm{\nabb\La^\a V(\tau')}^2_{L^2(\p)}d\tau'\lesssim \int_0^\tau\int_{\p}\La^{2\a} V(\tau')F(\tau')d\mu_ud\tau',
\ee
and for $0<\eta<\de<1$:
\begin{align}
&\tau^{1+\de}\|\nabb U(\tau) F\|^2_{\lp{2}}+\int_0^\tau{\tau'}^{1+\de}\norm{\lap U(\tau')F}^2_{\lp{2}}d\tau'\lesssim \|\La^{-\eta}F\|^2_{L^2(S)},\label{eq:heat1}\\
&\tau^{1+\de}\|U(\tau) F\|^2_{\lp{2}}+\int_0^\tau{\tau'}^{1+\de}\norm{\nabb U(\tau')F}^2_{\lp{2}}d\tau'\lesssim \|\La^{-1-\eta}F\|^2_{L^2(S)},\label{eq:heat2}\\
&\tau^{\de}\|U(\tau) F\|^2_{\lp{2}}+\int_0^\tau{\tau'}^{\de}\norm{\nabb U(\tau')F}^2_{\lp{2}}d\tau'\lesssim \|\La^{-\eta}F\|^2_{L^2(S)},\label{eq:heat3}\\
&\tau^{\de}\|\La^{-1}U(\tau) F\|^2_{\lp{2}}+\int_0^\tau{\tau'}^{\de}\norm{U(\tau')F}^2_{\lp{2}}d\tau'\lesssim \|\La^{-1-\eta}F\|^2_{L^2(S)}.\label{eq:heat4}
\end{align}
\label{eq:heat0}


We now investigate the boundedness of $\La^{-\a}$ on $L^p(\p)$ spaces for $0\leq\a\leq 1$. For any tensor $F$ on $\p$ and any $\a\in\R$, integrating by parts and using the definition of $\La$, we get:
\begin{equation}\label{La1}
\begin{array}{ll}
\ds \norm{\La^\a F}^2_{\lp{2}}+\norm{\nabb\La^\a F}^2_{\lp{2}} & \ds =\int_{\p}\La^{\a}F\La^{\a}Fd\mu_u+\int_{\p}\nabb\La^{\a}F\nabb\La^{\a}Fd\mu_u\\
& \ds =\int_{\p}(1-\lap)\La^{\a}F\La^{\a}Fd\mu_u=\int_{\p}\La^2\La^{\a}F\La^{\a}Fd\mu_u\\
& \ds =\norm{\La^{\a+1}F}^2_{\lp{2}}.
\end{array}
\end{equation}
Taking $\a=-1$ in \eqref{La1}, we obtain:
\begin{equation}\label{La2}
\norm{\nabb\La^{-1}F}_{\lp{2}}\lesssim\norm{F}_{\lp{2}}.
\end{equation}
Below, we deduce several estimates from \eqref{La2}. 
Taking the adjoint of \eqref{La2}, we obtain for any vectorfield $F$:
\begin{equation}\label{La3}
\norm{\La^{-1}\divb F}_{\lp{2}}\lesssim\norm{F}_{\lp{2}}.
\end{equation}
Also, \eqref{eq:GNirenberg} and \eqref{La2} imply for any tensor $F$ on $\p$:
\begin{equation}\label{La4}
\norm{\La^{-1}F}_{\lp{p}}\lesssim\norm{F}_{\lp{2}}\textrm{ for all }2\leq p<+\infty.
\end{equation}
Taking the adjoint of \eqref{La4} yields:
\begin{equation}\label{La5}
\norm{\La^{-1}F}_{\lp{2}}\lesssim\norm{F}_{\lp{p}}\textrm{ for all }1<p\leq 2.
\end{equation}
Interpolating between the identity and $\La^{-1}$, we deduce form \eqref{La5}:
\begin{equation}\label{La6}
\norm{\La^{-\a}F}_{\lp{2}}\lesssim\norm{F}_{\lp{p}}\textrm{ for all }0<\a<1,\,\frac{2}{1+\a}<p\leq 2.
\end{equation}


Finally, we conclude this section by recalling the sharp Bernstein inequality for scalars obtained in 
\cite{LP}. It is derived under the additional assumption that the Christoffel symbols $\Gamma^A_{BC}$ of the coordinate system \eqref{eq:coordchart} on $\p$ verify:
\be\lab{eq:gammaL2}
\sum_{A,B,C}\int_U|\Gamma^A_{BC}|^2 dx^1dx^2\le c^{-1},
\end{equation}
with a constant $c>0$ independent of $u$ and where $U$ is a coordinate chart. 
\begin{remark}
The existence of a covering of $\p$ by coordinate charts satisfying \eqref{eq:coordchart} and \eqref{eq:gammaL2} with a constant $c>0$ and the number of charts independent of $u$ will be established in Proposition \ref{gl11}.
\end{remark}
Let $0\leq\ga<1$, and let $K_\ga$ be defined by:
\begin{equation}\label{La7}
K_\ga := \norm{\La^{-\ga}K}_{\lp{2}}.
\end{equation}
Then, we have the following sharp Bernstein inequality for any scalar function $f$ on $\p$,  $0\le \ga<1$,
 any $j\ge 0$, and an arbitrary $2\le p<\infty$ (see \cite{LP}):
\bea
\|P_j f\|_{\lp{\infty}}&\lesssim & 2^j\big(1+ 2^{-\frac{j}{p}} 
\big (K_\ga^{\frac{1}{p(1-\ga)}} + K_{\ga}^{\frac{1}{2p}}\big ) +
1\big)\|f\|_{\lp{2}}\label{eq:Pkf},\label{eq:strongbernscalar}\\
\|P_{<0} f\|_{\lp{\infty}}&\lesssim &  
\big (1  +K_\ga^{\frac{2}{p(1-\ga)}} + K_{\ga}^{\frac{1}{2p}}\big)
\|f\|_{\lp{2}}.\label{eq:strong-Bern-0}
\eea
Also, the Bochner identity \eqref{sboch} together with the properties of $\La$ implies the 
following inequality (see \cite{LP}):
\bea
\int_{\p} |\nabb^2 f|^2&\lesssim& \int_{\p} |\lap f|^2  + \big (K_\ga^{\frac{2}{1-\ga}} + K_{\ga}\big )\int_{\p} |\nabb f|^2.\label{eq:Bochconseq}
\eea
Thus, we need to bound $K_{\ga}$ in order to be able to use \eqref{eq:strongbernscalar}, \eqref{eq:strong-Bern-0}, and \eqref{eq:Bochconseq}. For $\Re(\a)<0$, we will use the fact that for any tensor $F$ on $\p$:
\begin{equation}\label{La8}
\norm{\La^{-\a}F}^2_{\lp{2}}\lesssim\norm{P_{<0}F}^2_{\lp{2}}+\sum_{j=0}^{+\infty}2^{-2\a j}\norm{P_jF}^2_{\lp{2}}.
\end{equation}
which follows from the methods in \cite{LP}. Therefore, we would like to control $K$ in $\lhs{\infty}{-\a}$ for some $\a<1$. This is the goal of the next section.

\subsection{Control of $K$ in $\lhs{\infty}{-\frac{1}{2}}$}

The goal of this section is to prove the following estimate.
\begin{proposition}\label{propK}
Let $(\s,g,k)$ chosen as in section \ref{reducsmall}. Let $u$ the scalar function on $\s\times\S$ constructed in theorem \ref{thregx}, and let $\p$, $N$, $\th$ and $K$ be associated to $u$ as in section \ref{sec:foliation}. We have:
\begin{equation}\label{propK1}
\sum_{j\geq 0}2^{-j}\norm{P_jK}^2_{\li{\infty}{2}}+\norm{P_{<0}K}^2_{\li{\infty}{2}}\lesssim\ep^2.
\end{equation}
\end{proposition}

\begin{proof}
Recall from \eqref{boot} that:
\begin{equation}\label{ad20}
\norm{K}_{\ll{2}}\lesssim\ep.
\end{equation}
Also, \eqref{const1} and \eqref{gauss1} yield:
$$\nabn K=\trt\nabn\trt-\th\nabn\th+k\nabn k -\nabn R_{NN},$$
which together with \eqref{r22ter} implies:
$$\nabn K=\divb(B)+b$$
where
$$B=R_{N.},\,b=\trt\nabn\trt-\th\nabn\th+R(\nabb a,N).$$
Multiplying by $a$, this implies:
\begin{equation}\label{ad22}
\nabna K=\divb(B_1)+b_1
\end{equation}
where
\begin{equation}\label{ad23}
B_1=aR_{N.},\,b_1=-\nabb a\cdot B+a\trt\nabn\trt-a\th\nabn\th+aR(\nabb a,N).
\end{equation} 
Using \eqref{small2}, \eqref{boot1}, \eqref{appboot} and \eqref{appboot1}, we obtain:
\begin{equation}\label{ad24}
\norm{B_1}_{\ll{2}}\leq\ep,
\end{equation}
and
\begin{equation}\label{ad25}
\norm{b_1}_{\l{2}{\frac{4}{3}}}\lesssim\norm{\th}_{\l{\infty}{4}}\norm{\nabn\th}_{\ll{2}}+(\norm{B_1}_{\ll{2}}\norm{R}_{\ll{2}})\norm{\nabb a}_{\l{\infty}{4}}\lesssim\ep.
\end{equation}
In particular, \eqref{La3}, \eqref{La5}, \eqref{ad22}, \eqref{ad24} and \eqref{ad25} yield:
\begin{equation}\label{ad25bis}
\norm{\La^{-1}\nabn K}_{\ll{2}}\lesssim\ep.
\end{equation}

We may assume the existence of $\widetilde{P_j}$ with the same properties than $P_j$ such that $P_j=\widetilde{P_j}^2$ (see \cite{LP}), and for simplicity we write $P_j=P_j^2$. Also, using the fact that $\La\La^{-1}=I$ and that $\La$ commutes with $P_j$, we obtain:
$$P_j\nabna K=\La P_j(P_j\La^{-1}\nabna K),$$
which together with property (iii) of Theorem \ref{thm:LP} yields:
$$\norm{P_j\nabna K}_{\ll{2}}\lesssim \norm{\La P_j(P_j\La^{-1}\nabna K)}_{\ll{2}}\lesssim 2^j\norm{P_j\La^{-1}\nabna K}_{\ll{2}}.$$
Using property (ii) of Theorem \ref{thm:LP}, we get:
\bee
\sum_{j\geq 0}2^{-2j}\norm{P_j\nabna K}^2_{\ll{2}} & \lesssim&\sum_{j\geq 0}\norm{P_j\La^{-1}\nabna K}^2_{\ll{2}}\\
&\lesssim&\norm{\La^{-1}\nabna K}^2_{\ll{2}}.
\eee
Together with \eqref{ad25bis}, we finally obtain:
\begin{equation}\label{ad35}
\sum_{j\geq 0}2^{-2j}\norm{P_j\nabna K}^2_{\ll{2}}\lesssim\ep^2.
\end{equation}

To prove Proposition \ref{propK}, we assume:
\begin{equation}\label{ad26}
\sum_{j\geq 0}2^{-j}\norm{P_jK}^2_{\l{\infty}{2}}+\norm{P_{<0}K}^2_{\l{\infty}{2}}\lesssim D^2\ep^2,
\end{equation}
where $D$ is a large enough constant. We will then try to improve \eqref{ad26}. Note that \eqref{eq:Bochconseq}, \eqref{La8} and \eqref{ad26} yield for any scalar function $f$ on $\p$:
\begin{equation}\label{ad26bis}
\norm{\nabb^2f}^2_{\lp{2}}\lesssim\norm{\lap f}^2_{\lp{2}}+(D\ep+D^4\ep^4)\norm{\nabb f}^2_{\lp{2}}.
\end{equation}

The term $\norm{P_{<0}K}_{\l{\infty}{2}}$ is easier to bound, so we concentrate on estimating the sum $\sum_{j\ge 0}2^{-j}\norm{P_jK}_{\ll{2}}$. We will use the following variant of \eqref{c0e1} where we 
do not yet use Cauchy-Schwarz in $u$ for the integral containing $\nabn F$:
\begin{equation}\label{ad27bis}
\begin{array}{ll}
\norm{F}^2_{\l{\infty}{2}}\lesssim & \ds\norm{F(-2,.)}^2_{L^2(P_{-2})}+\int_{-2}^{2}\norm{\nabn F}_{\lp{2}}\norm{F}_{\lp{2}}du\\
&\ds +\norm{\nabb F}_{\ll{2}}\norm{F}_{\ll{2}}.
\end{array}
\end{equation}
Using \eqref{ad27bis}, the fact that $P_jK\equiv 0$ on $u=-2$, and properties (ii) and (iii) of Theorem \ref{thm:LP}, we have:
\begin{equation}\label{ad27}
\begin{array}{ll}
&\ds\sum_{j\geq 0}2^{-j}\norm{P_jK}^2_{\l{\infty}{2}}\\
\ds\lesssim &\ds\sum_{j\ge 0} 2^{-j}\left(\int_{-2}^{2}\norm{P_jK}_{\lp{2}}\norm{\nabn P_jK}_{\lp{2}}du+\norm{P_jK}_{\ll{2}}\norm{\nabb P_jK}_{\ll{2}}\right)\\
\ds\lesssim &\ds\sum_{j\ge 0} 2^{-j}\left(\int_{-2}^{-2}\norm{P_jK}_{\lp{2}}\norm{\nabn P_jK}_{\lp{2}}du\right)+\sum_{j\ge 0}\norm{P_jK}^2_{\ll{2}}\\
\ds\lesssim &\ds\sum_{j\ge 0} 2^{-j}\left(\int_{-2}^{-2}\norm{P_jK}_{\lp{2}}\norm{\nabna P_jK}_{\lp{2}}du\right)+\ep^2,
\end{array}
\end{equation}
where we used in the last inequality the estimate \eqref{boot} for $a$ and the estimate \eqref{ad20} for $K$. We inject the estimate:
$$\norm{\nabna P_jK}_{\lp{2}}\les \norm{P_j\nabna K}_{\lp{2}}+\norm{[\nabna, P_j]K}_{\lp{2}}$$
in \eqref{ad27}. We obtain:
\bee
&&\sum_{j\geq 0}2^{-j}\norm{P_jK}^2_{\l{\infty}{2}}\\
\nn&\lesssim &\ds\sum_{j\ge 0}(\norm{P_jK}_{\ll{2}}^2+2^{-2j}\norm{P_j\nabna K}^2_{\ll{2}})\\
&&+\sum_{j\geq 0}2^{-j}\norm{P_jK}_{\l{\infty}{2}}\norm{[\nabna, P_j]K}_{\l{1}{2}}+\ep^2,
\eee
which together with the estimates \eqref{ad20} and \eqref{ad35} for $K$ implies:
\be\label{ad28}
\sum_{j\geq 0}2^{-j}\norm{P_jK}^2_{\l{\infty}{2}}\lesssim \ds\sum_{j\ge 0}2^{-j}\norm{[\nabna, P_j]K}^2_{\l{1}{2}}+\ep^2.
\ee
Now, we will prove:
\be\lab{ad29}
\norm{[\nabna, P_j]K}_{\l{1}{2}}\les 2^{\frac{j}{3}}(\ep+D\ep^2).
\ee
Together with \eqref{ad28}, this yields:
\bee
\sum_{j\geq 0}2^{-j}\norm{P_jK}^2_{\l{\infty}{2}}&\lesssim& \ep^2+\left(\sum_{j\geq 0}2^{-\frac{j}{3}}\right)(\ep^2+D^2\ep^4)\\
&\les& \ep^2+D^2\ep^4,
\eee
which is an improvement of \eqref{ad26}. Thus we have:
$$\sum_{j\geq 0}2^{-j}\norm{P_jK}^2_{\l{\infty}{2}}+\norm{P_{<0}K}^2_{\l{\infty}{2}}\lesssim \ep^2,$$
which concludes the proof of Proposition \ref{propK} provided \eqref{ad29} holds.

In the rest of the proof, we focus on obtaining \eqref{ad29}. We have:
\begin{equation}\label{ad30}
[\nabna, P_j]K=\int_0^\infty m_j(\tau)V(\tau) d\tau,
\end{equation}
where $V(\tau)$ is satisfies:
\begin{equation}\label{bis:clp22}
(\partial_{\tau}-\lap)V(\tau)=[\nabna,\lap]U(\tau)K,\, V(0)=0.
\end{equation}
In view of \eqref{ad30}, we have:
\be\lab{ad31}
\norm{[\nabna, P_j]K}_{\l{1}{2}}\les \int_0^\infty m_j(\tau)\norm{V(\tau)}_{\l{1}{2}} d\tau.
\ee
Now, using \eqref{interpolLa} and \eqref{La1}, we have:
\begin{equation}\lab{ad61}
\begin{array}{l}
\ds\int_0^\infty m_j(\tau)\norm{V(\tau)}_{\lp{2}} d\tau\lesssim\int_0^\infty m_j(\tau)\norm{\La^{-\frac{1}{3}}V(\tau)}^{\frac{2}{3}}_{\lp{2}}\norm{\nabb\La^{-\frac{1}{3}}V(\tau)}^{\frac{1}{3}}_{\lp{2}} d\tau\\
\ds\lesssim\left(\int_0^\infty m_j(\tau)\norm{\La^{-\frac{1}{3}}V(\tau)}_{\lp{2}} d\tau\right)^{\frac{2}{3}}\left(\int_0^\infty\norm{\nabb\La^{-\frac{1}{3}}V(\tau)}^{2}_{\lp{2}} d\tau\right)^{\frac{1}{6}}\\
\hspace{1cm}\ds\times\left(\int_0^\infty m_j(\tau)^2 d\tau\right)^{\frac{1}{6}}\\
\ds\lesssim 2^{\frac{j}{3}}\left(\sup_\tau\norm{\La^{-\frac{1}{3}}V(\tau)}_{\lp{2}} +\left(\int_0^\infty\norm{\nabb\La^{-\frac{1}{3}}V(\tau)}^{2}_{\lp{2}} d\tau\right)^{\frac{1}{2}}\right).
\end{array}
\end{equation}
Integrating in $u$ and using \eqref{ad31}, we obtain:
\bea\lab{ad32}
\norm{[\nabna, P_j]K}_{\l{1}{2}}&\lesssim& 2^{\frac{j}{3}}\Bigg(\sup_\tau\norm{\La^{-\frac{1}{3}}V(\tau)}_{\l{1}{2}} \\
\nn&&+\int_{-2}^2\left(\int_0^\infty\norm{\nabb\La^{-\frac{1}{3}}V(\tau)}^{2}_{\lp{2}} d\tau\right)^{\frac{1}{2}}du\Bigg).
\eea
Now, we will prove:
\be\lab{bis:clp23}
\norm{\La^{-\frac{1}{3}}V(\tau)}_{\l{1}{2}}+\int_{-2}^{2}\left(\int_0^\tau\norm{\nabb\La^{-\frac{1}{3}}V(\tau')}^2_{L^2(\p)}d\tau'\right)^{\frac{1}{2}}du\lesssim \ep+\ep^2D.
\ee
Together with \eqref{ad32}, this yields the wanted estimate \eqref{ad29}. 

In the rest of the proof, we focus on obtaining \eqref{bis:clp23}. In view of \eqref{bis:clp22} and the heat flow estimate \eqref{eq:l2heat1bis}, we have:
$$\norm{\La^{-\frac{1}{3}}V(\tau)}^2_{L^2(\p)}+\int_0^\tau\norm{\nabb\La^{-\frac{1}{3}} V(\tau')}^2_{L^2(\p)}d\tau'\lesssim \int_0^\tau\int_{\p}\La^{-\frac{2}{3}}V(\tau')[\nabna,\lap]U(\tau')d\mu_ud\tau'.$$
Injecting the commutator formula \eqref{dj3}, integrating by parts, we obtain the following estimate:
\bea\lab{bis:clp24}
&&\norm{\La^{-\frac{1}{3}}V(\tau)}^2_{L^2(\p)}+\int_0^\tau\norm{\nabb\La^{-\frac{1}{3}} V(\tau')}^2_{L^2(\p)}d\tau'\\
\nn&\lesssim& (\norm{a\nabb(\th)}_{\lp{2}}+\norm{\nabb(a)\th}_{\lp{2}}+\norm{aR}_{\lp{2}})\int_0^\tau\norm{\nabb U(\tau')}_{\lp{p}}\norm{\nabb\La^{-\frac{2}{3}} V(\tau')}_{\lp{2}}d\tau',
\eea
where 
$$2<p<3.$$ 
Now, we have in view of \eqref{La1} and \eqref{interpolLa}:
$$\norm{\nabb\La^{-\frac{2}{3}} V(\tau')}_{\lp{2}}\les \norm{\La^{-\frac{1}{3}} V(\tau')}^{\frac{1}{3}}_{\lp{2}}\norm{\nabb\La^{-\frac{2}{3}} V(\tau')}^{\frac{2}{3}}_{\lp{2}}$$
which together with \eqref{bis:clp24} implies:
\bea\lab{bis:clp25}
&&\norm{\La^{-\frac{1}{3}}V(\tau)}^2_{L^2(\p)}+\int_0^\tau\norm{\nabb\La^{-\frac{1}{3}} V(\tau')}^2_{L^2(\p)}d\tau'\\
\nn&\lesssim& (\norm{a\nabb(\th)}^2_{\lp{2}}+\norm{\nabb(a)\th}^2_{\lp{2}}+\norm{aR}^2_{\lp{2}})\int_0^\tau{\tau'}^{(\frac{1}{3})_-}\norm{\nabb U(\tau')}^2_{\lp{p}}d\tau'.
\eea
The Gagliardo-Nirenberg inequality \eqref{eq:GNirenberg} implies: 
\bee 
&&\int_0^\tau{\tau'}^{(\frac{1}{3})_-}\norm{\nabb U(\tau')}^2_{\lp{p}}d\tau'\\
&\les& \int_0^\tau{\tau'}^{(\frac{1}{3})_-}\norm{\nabb U(\tau')}^{\frac{4}{p}}_{\lp{2}}\norm{\nabb^2U(\tau')}^{2(1-\frac{2}{p})}_{\lp{2}}d\tau'\\
&\les& \left(\int_0^\tau\norm{\nabb U(\tau')}^2_{\lp{2}}d\tau'\right)^{\frac{2}{p}}\left(\int_0^\tau{\tau'}\norm{\nabb^2 U(\tau')}^2_{\lp{2}}d\tau'\right)^{1-\frac{2}{p}},
\eee
where we used in the last inequality the fact that:
$$\left(\frac{1}{3}\right)_--1+\frac{2}{p}>0$$
in view of the restriction $p<3$. Together with the Bochner inequality \eqref{ad26bis}, we obtain:
\bee 
&&\int_0^\tau{\tau'}^{(\frac{1}{3})_-}\norm{\nabb U(\tau')}^2_{\lp{p}}d\tau'\\
&\les& (1+D\ep+D^4\ep^4)^{1-\frac{2}{p}}\left(\int_0^\tau\norm{\nabb U(\tau')}^2_{\lp{2}}d\tau'+\int_0^\tau{\tau'}\norm{\lap U(\tau')}^2_{\lp{2}}d\tau'\right).
\eee
Thus, we obtain in view of the heat flow estimates \eqref{eq:l2heat1} and \eqref{eq:l2heat2}:
$$\int_0^\tau{\tau'}^{(\frac{1}{3})_-}\norm{\nabb U(\tau')}^2_{\lp{p}}d\tau'\les \norm{K}^2_{\lp{2}}.$$
Together with \eqref{bis:clp25}, this yields:
\bee
&&\norm{\La^{-\frac{1}{3}}V(\tau)}^2_{L^2(\p)}+\int_0^\tau\norm{\nabb\La^{-\frac{1}{3}} V(\tau')}^2_{L^2(\p)}d\tau'\\
\nn&\lesssim& (1+D\ep+D^4\ep^4)^{1-\frac{2}{p}}(\norm{a\nabb(\th)}^2_{\lp{2}}+\norm{\nabb(a)\th}^2_{\lp{2}}+\norm{aR}^2_{\lp{2}}) \norm{K}^2_{\lp{2}}.
\eee
Integrating in $u$, this yields:
\bee
&&\norm{\La^{-\frac{1}{3}}V(\tau)}_{\l{1}{2}}+\int_{-2}^{2}\left(\int_0^\tau\norm{\nabb\La^{-\frac{1}{3}}V(\tau')}^2_{L^2(\p)}d\tau'\right)^{\frac{1}{2}}du\\
\nn&\lesssim& (1+D\ep+D^4\ep^4)^{\frac{1}{2}-\frac{1}{p}}(\norm{a\nabb(\th)}_{\ll{2}}+\norm{\nabb(a)\th}_{\ll{2}}+\norm{aR}_{\ll{2}}) \norm{K}_{\ll{2}}\\
\nn&\les&(1+D\ep+D^4\ep^4)^{\frac{1}{2}-\frac{1}{p}}\ep^2,
\eee
where we used in the last inequality the estimate \eqref{thregx1} for $a$ and $\th$, the smallness assumption \eqref{small1} for $R$, and the estimate \eqref{ad20} for $K$. Now, since $2<p<3$, we obtain:
$$\norm{\La^{-\frac{1}{3}}V(\tau)}_{\l{1}{2}}+\int_{-2}^{2}\left(\int_0^\tau\norm{\nabb\La^{-\frac{1}{3}}V(\tau')}^2_{L^2(\p)}d\tau'\right)^{\frac{1}{2}}du \les \ep^2+D^{\frac{2}{3}}\ep^{\frac{8}{3}},$$
which implies \eqref{bis:clp23}. This concludes the proof of the proposition. 
\end{proof}

\begin{remark}
The following consequence of Proposition \ref{propK} will be useful in the next two sections. 
Proposition \ref{propK} and \eqref{La8} with the choice $\a=1/2$ imply:
\begin{equation}\label{ad65}
\norm{K_{\frac{1}{2}}}_{L^\infty(-2,2)}=\norm{\La^{-\frac{1}{2}}K}_{\l{\infty}{2}}\lesssim\ep,
\end{equation}
where $K_{1/2}$ has been defined in \eqref{La7}. Together with \eqref{eq:strongbernscalar} and \eqref{eq:strong-Bern-0}  with the choice $\gamma=1/2$, we obtain for any scalar function $f$ on $\p$ and any $j\ge 0$:
\bea
\|P_j f\|_{\lp{\infty}}&\lesssim & 2^j\|f\|_{\lp{2}},\label{eq:strongbernscalarbis}\\
\|P_{<0} f\|_{\lp{\infty}}&\lesssim &  
 \|f\|_{\lp{2}}.\label{eq:strong-Bern-0bis}
\eea
Also, \eqref{ad65} and \eqref{eq:Bochconseq} with the choice $\gamma=1/2$ imply:
\begin{equation}\label{eq:Bochconseqbis}
\int_{\p} |\nabb^2 f|^2\lesssim \int_{\p} |\lap f|^2  + \ep\int_{\p} |\nabb f|^2.
\end{equation}
\end{remark}

Using the Bochner inequality \eqref{eq:Bochconseqbis}, we may prove the following lemma. 
\begin{lemma}\lab{lemma:lbz5}
For any 1-form $F$ on $\p$, for any $1<p\leq 2$ and for all $j\geq 0$, we have:
\be\lab{lbz14bis}
\norm{P_j\divb(F)}_{\lp{2}}\les 2^{\frac{2}{p}j}\norm{F}_{\lp{p}}.
\ee
\end{lemma}

\begin{proof}
By duality, it suffices to prove for any scalar function $f$ on $\p$, for any $2\leq p<+\infty$ and for all $j\geq 0$ the following inequality:
\be\lab{lbz14ter}
\norm{\nabb P_jf}_{\lp{p}}\les 2^{2(1-\frac{1}{p})j}\norm{f}_{\lp{2}}.
\ee
Now, using the Gagliardo-Nirenberg inequality \eqref{eq:GNirenberg}, the Bochner inequality for scalar functions \eqref{eq:Bochconseqbis}, and the property iii) of Theorem \ref{thm:LP} for Littlewood-Paley projections, we have:
\bee
\norm{\nabb P_jf}_{\lp{p}}&\les& \norm{\nabb^2P_jf}^{1-\frac{2}{p}}_{\lp{2}}\norm{\nabb P_jf}^{\frac{2}{p}}_{\lp{2}}\\
&\les& (\norm{\lap P_jf}_{\lp{2}}+\norm{\nabb P_jf}_{\lp{2}})^{1-\frac{2}{p}}\norm{\nabb P_jf}^{\frac{2}{p}}_{\lp{2}}\\
&\les& 2^{2j(1-\frac{1}{p})}\norm{f}_{\lp{2}},
\eee
which is \eqref{lbz14ter}. This concludes the proof of Lemma \ref{lemma:lbz5}.
\end{proof}

Let us state another consequence of the Bochner inequality \eqref{eq:Bochconseqbis}. 
\begin{lemma}\lab{lemma:vacances:1}
Let $0<b<2$. Let $f$ a scalar on $\p$ such that $f\in\hs{b}$. Then, $\nabb f\in\hs{b-1}$.
\end{lemma}

\begin{proof}
We have:
\be\lab{vacances:1}
\norm{P_j\nabb f}_{\lp{2}}\les \sum_{l\geq 0}\norm{P_j\nabb P_lf}_{\lp{2}}.
\ee
If $l\leq j$, we use the finite band property of $P_j$ and $P_l$ and the Bochner inequality \eqref{eq:Bochconseqbis} for scalars to obtain:
\bea\lab{vacances1:1}
2^{j(b-1)}\norm{P_j\nabb P_lf}_{\lp{2}} &\les& 2^{j(b-2)}\norm{\nabb^2P_lf}_{\lp{2}}\\
\nn &\les& 2^{j(b-2)}2^{2l}\norm{P_lf}_{\lp{2}}\\
\nn &\les& 2^{-|j-l|(2-b)}2^{bl}\norm{P_lf}_{\lp{2}},
\eea
where we used in the last inequality the fact that $l\leq j$ and $b<2$. 

If $l>j$, we use the finite band property for $P_j$ to obtain:
\bea\lab{vacances2:1}
2^{j(b-1)}\norm{P_j\nabb P_lf}_{\lp{2}} &\les& 2^{jb}\norm{P_lf}_{\lp{2}}\\
\nn &\les& 2^{-|j-l|b}2^{bl}\norm{P_lf}_{\lp{2}},
\eea
where we used in the last inequality the fact that $l> j$ and $b>0$. Finally, \eqref{vacances:1}, \eqref{vacances1:1} and 
\eqref{vacances2:1} imply:
\bee
\sum_{j\geq 0}2^{2(b-1)j}\norm{P_j\nabb f}^2_{\lp{2}}&\les& \sum_{j\geq 0}\left(\sum_{l\geq 0}2^{-\min(b,2-b)|j-l|}2^{lb}\norm{P_lf}_{\lp{2}}\right)^2\\
&\les& \sum_{l\geq 0}2^{2lb}\norm{P_lf}^2_{\lp{2}}\\
&\les& \norm{f}_{\hs{b}}^2,
\eee
where we used the fact that $\min(b, 2-b)>0$. This concludes the proof of the lemma.
\end{proof}

Finally, the bound \eqref{ad65} allows us to prove the following Hodge inequality.
\begin{lemma}\label{lemma:bieber}
Let $F$ a symmetric 2-tensor such that tr$F=0$. Then:
\begin{equation}\label{bieber}
\norm{\nabb F}_{\lp{2}}\lesssim \norm{\divb F}_{\lp{2}}+\ep\norm{F}_{\lp{2}}.
\end{equation}
\end{lemma}

\begin{proof}
Recall the identity \eqref{hodge} for Hodge systems:
\begin{equation}\label{bieber1}
\int_{\p}(|\nabb F|^2+2K|F|^2)=2\int_{\p}|\divb F|^2.
\end{equation}
We have:
\bee
\left|\int_{\p}K|F|^2\right|&\les & \norm{K}_{\lhs{\infty}{-\frac{1}{2}}}\norm{|F|^2}_{\hs{\frac{1}{2}}}\\
&\les & \ep\norm{|F|^2}_{\hs{\frac{1}{2}}},
\eee
where we used the bound \eqref{ad65} in the last inequality. Together with \eqref{bieber1}, this implies:
\be\lab{bieber2}
\norm{\nabb F}_{\lp{2}}\les \norm{\divb F}_{\lp{2}}+\ep^{\frac{1}{2}}\norm{|F|^2}_{\hs{\frac{1}{2}}}^{\frac{1}{2}}.
\ee

Next, we estimate the last term in the right-hand side of \eqref{bieber2}. We have:
$$P_j(|F|^2)=2^{-2j}P_j\lap(|F|^2)=2^{-2j}P_j\divb(\nabb(|F|^2)).$$
Together with \eqref{lbz14bis}, we obtain:
\bee
2^{\frac{j}{2}}\norm{P_j(|F|^2)}_{\lp{2}}&\les &2^{\frac{j}{2}}2^{-2j}2^{\frac{4j}{3}}\norm{F\c \nabb F}_{\lp{\frac{3}{2}}}\\
&\les& 2^{-\frac{j}{6}}\norm{F}_{\lp{6}}\norm{\nabb F}_{\lp{2}}\\
&\les& 2^{-\frac{j}{6}}\norm{F}^{\frac{1}{3}}_{\lp{2}}\norm{\nabb F}_{\lp{2}}^{\frac{5}{3}},
\eee
where we used in the last inequality the Gagliardo-Nirenberg estimate \eqref{eq:GNirenberg}. This yields:
$$\norm{|F|^2}_{\hs{\frac{1}{2}}}\les \norm{F}^{\frac{1}{3}}_{\lp{2}}\norm{\nabb F}_{\lp{2}}^{\frac{5}{3}}.$$
Together with \eqref{bieber2}, we obtain \eqref{bieber}. This concludes the proof of the lemma.
\end{proof}

\subsection{Estimates for the commutator $[\nabla_{aN},P_j]$}\lab{sec:statepropprodcomm}

In this section, we state several estimates for the commutator $[\nabla_{aN},P_j]$.  To simplify the exposition, the proof are postponed to Appendix \ref{sec:proofpropcomm}. The reason we prefer to consider 
$[\nabla_{aN},P_j]$ instead of $[\nabla_{aN},P_j]$ is because the former does not contain any $N$ derivative in view of the commutator estimates \eqref{commutnabna2} and \eqref{dj3}.
We start with a first commutator estimate.
\begin{proposition}\lab{greveinf1}
Let $f$ a scalar function on $S$. Then, for any $j\geq 0$ and for any $\delta>0$, we have the following commutator estimate:
\be\lab{commLP1}
\norm{[\nabna, P_j]f}_{\ll{2}}\les \ep\norm{\La^{\frac{1}{2}+\delta}f}_{L^2(S)}+\ep\norm{\La^\delta f}_{\l{\infty}{2}}.
\ee
\end{proposition}

We state a second commutator estimate.
\begin{proposition}\lab{prop:commLP2}\lab{greveinf2}
Let $F$ a tensor on $S$. Then, for any $j\geq 0$ and for any $\delta>0$, we have the following commutator estimate:
\be\lab{commLP2}
\norm{[\nabna, P_j]F}_{\l{1}{2}}\les 2^{-j(1-\delta)}\ep\left(\norm{\nabb F}_{L^2(S)}+\norm{F}_{\l{\infty}{2}}\right).
\ee
\end{proposition}

Proposition \ref{prop:commLP2} yields the following corollary.
\begin{corollary}\lab{cor:commLP1}\lab{greveinf3}
For any $\p$-tangent tensor $F$ on $S$ such that $F\equiv 0$ on $u=-2$, and for all $j\geq 0$, we have:
\be\lab{clp12}
\norm{P_jF}_{\lhs{\infty}{\frac{1}{2}}}\les \norm{F}_{H^1(S)}.
\ee
\end{corollary}

We state a third commutator estimate.
\begin{proposition}\lab{prop:commLP3}\lab{greveinf4}
Let $f$ a scalar function on $S$. Then, for any $j\geq 0$ and for any $0<\delta<\alpha<1$, we have the following commutator estimate:
\be\lab{commLP3}
\norm{[\nabna, P_j]f}_{\l{1}{2}}\les 2^{j\alpha}\ep\norm{\La^{-\delta}f}_{L^2(S)}.
\ee
\end{proposition}

We state a fourth commutator estimate.
\begin{proposition}\lab{prop:commLP4}\lab{greveinf5}
Let $f$ a scalar function on $S$. Then, for any $j\geq 0$ and for any $\delta>0$, we have the following commutator estimate:
\be\lab{commLP4}
\norm{[\nabna, P_j]f}_{\ll{2}}\les 2^{-(1-\delta)j}\ep(\norm{\lap f}_{L^2(S)}+\norm{\nabb f}_{\l{\infty}{2}}).
\ee
\end{proposition}

Proposition \ref{prop:commLP4} yields the following corollary.
\begin{corollary}\lab{cor:commLP4}\lab{greveinf6}
Let a tensor $F$ on $S$ such that $F\equiv 0$ on $u=-2$, $\nabb^2F\in\ll{2}$ and $\nabn F\in\lhs{2}{b}$ for $b>0$. Then, $F \in \ll{\infty}$.
\end{corollary}

We state a fifth commutator estimate.
\begin{proposition}\lab{prop:commLP5}\lab{greveinf7}
Let $f$ a scalar function on $S$. Then, for any $j\geq 0$ and for any $\delta>0$, we have the following commutator estimate:
\be\lab{commLP5}
\norm{[\nabna, P_j]f}_{\ll{2}}\les 2^j\ep\norm{\La^{-(1-\delta)}f}_{\l{\infty}{2}}.
\ee
\end{proposition}

\subsection{Product estimates}\lab{sec:statepropprodprod}

In this section, we derive several product estimates. To simplify the exposition, the proof are postponed to Appendix \ref{sec:proofpropprod}. Note that all product estimates in this section are sharp except the first one. 

\begin{proposition}\lab{prop:fichtre}\lab{greveinff1}
Let $0<b<\frac{1}{2}$. For any tensors $F$, $G$ and $H$ on $\p$ such that $F\c G\c H$ is a scalar, we have:
\be\lab{fichtre}
\norm{F\c G\c H}_{\hs{b}}\les \norm{F}_{\hs{1}}\norm{G}_{\hs{1}}\norm{H}_{\hs{\frac{1}{2}}}.
\ee
\end{proposition}

\begin{proposition}\lab{prop:kei}\lab{greveinff2}
For any $\p$-tangent tensor $G$ and $H$ on $S$ such that $G\c H$ is a scalar, we have:
\be\lab{kei}
\norm{G\c H}_{\lp{2}}\les \norm{G}_{\hs{\frac{1}{2}}}\norm{H}_{\hs{\frac{1}{2}}}.
\ee
\end{proposition}

\begin{proposition}\lab{prop:bale}\lab{greveinff3}
For any scalars $f$ and $h$ on $\p$, we have:
\be\lab{bale}
\norm{fh}_{\hs{-\frac{1}{2}}}\les (\norm{f}_{\lp{\infty}}+\norm{\nabb f}_{\lp{2}})\norm{h}_{\hs{-\frac{1}{2}}}.
\ee
\end{proposition}

\begin{proposition}\lab{cor:commLP2}\lab{greveinff4}
For any $\p$-tangent tensor $G$ and $H$ on $S$ such that $G\c H$ is a scalar, and for all $j\geq 0$, we have:
\be\lab{clp17}
\sum_{j\geq 0}2^{-j}\norm{P_j(G\c H)}^2_{\ll{2}}\les \norm{G}^2_{H^1(S)}\norm{H}^2_{\ll{2}}.
\ee
\end{proposition}

\begin{lemma}\lab{lemma:prod1}\lab{greveinff5}
Let $F$ and $G$ two tensors on $\p$ such that the contraction $F\c G$ is a scalar. Then, we have:
\be\lab{prod1}
\sup_{j\geq 0}2^{-j}\norm{P_j(F\c G)}_{\lp{2}}\les \norm{F}_{\hs{\frac{1}{2}}}\norm{G}_{\hs{-\frac{1}{2}}}.
\ee
\end{lemma}

\begin{lemma}\lab{greveinff6}
Let $-1<b<1$. Let $f$ a scalar function on $\p$, and $G$ a 1-form on $\p$. Then, we have:
\be\lab{prod6}
\norm{\divb(fG)}_{\hs{b-1}}\les \norm{f}_{\hs{b}}(\norm{G}_{\lp{\infty}}+\norm{\nabb G}_{\lp{2}}).
\ee
\end{lemma}

\begin{lemma}\lab{greveinff7}
Let $1\leq b<2$. Let $f$ a scalar function on $S$, and $G$ a 1-form on $S$. Then, we have:
\be\lab{prod11}
\norm{\divb(fG)}_{\lhs{2}{b-1}}\les (\norm{f}_{\lhs{2}{b}}+\norm{f}_{\lhs{\infty}{b-1}})(\norm{G}_{\ll{\infty}}+\norm{\nabb G}_{\l{\infty}{2}}).
\ee
\end{lemma}

\begin{lemma}\lab{greveinff8}
Let $0<b<1$. Let $F$ a tensor on $\p$ and $h$ a scalar function on $\p$. Then, we have:
\be\lab{prod15}
\norm{Fh}_{\hs{b}}\les (\norm{F}_{\lp{\infty}}+\norm{\nabb F}_{\lp{2}})\norm{h}_{\hs{b}}.
\ee
\end{lemma}

\subsection{Estimates for parabolic equations on $S$}

Consider the following parabolic equation:
\be\lab{parab0}
(\nabn-a^{-1}\lap)f=h\textrm{ on }S,
\ee
where $f$ and $h$ are scalar functions on $S$. In Proposition \ref{p7} and Proposition \ref{p8}, we obtained estimates for such equations. In this section, we derive additional estimates involving the Littlewood Paley projections of section \ref{sec:LP}. We start with the following commutation lemma. 
\begin{lemma}\lab{lemma:parabPj}
Let $f$ satisfying equation \eqref{parab0}. Then, $P_jf$ satisfies the following parabolic equation:
\be\lab{parab1}
(\nabn-a^{-1}\lap)(P_jf)=a^{-1}P_j(ah)+a^{-1}[\nabla_{aN},P_j]f\textrm{ on }S.
\ee
\end{lemma}

\begin{proof}
We multiply equation \eqref{parab0} with $a$. We obtain:
$$(\nabla_{aN}-\lap)f=ah.$$
Next, we commute with $P_j$, using the fact that $P_j$ commutes with $\lap$. We obtain:
$$(\nabla_{aN}-\lap)(P_jf)=P_j(ah)+[\nabla_{aN},P_j]f.$$
Finally, multiplying with $a^{-1}$, we get \eqref{parab1}. This concludes the proof of Lemma \ref{lemma:parabPj}.
\end{proof}

\begin{proposition}\label{prop:parab1}
Let $f$ be a scalar function on $S$ satisfying \eqref{parab0} and such that $f\equiv 0$ on $u=-2$. Assume that there exists two tensors $G$ and $H$ on $S$ on $S$ on $S$ tangent to $\p$ such that:
\begin{equation}\label{parab2}
h=G\c H\textrm{ with }\norm{H}_{\ll{2}}\les\ep\textrm{ and }\norm{G}_{H^1(S)}\les\ep.
\end{equation}
Then, we have:
\begin{equation}\label{parab3}
\norm{f}_{\l{\infty}{4}}\lesssim\ep,
\end{equation}
and:
\be\lab{parab3bis}
\sum_{j\geq 0}\left(2^{3j}\norm{P_jf}_{\ll{2}}^2+2^j\norm{P_jf}_{\l{\infty}{2}}^2+2^{-j}\norm{P_j(\nabn f)}_{\ll{2}}^2\right)\les\ep^2.
\ee
\end{proposition}

\begin{proof}
We multiply \eqref{parab0} by $f^3$ and integrate on $-2<u'<u$ where $u\leq 2$. Using integration by parts together with \eqref{coarea} and \eqref{du}, we obtain:
\begin{equation}\label{parab4}
\begin{array}{ll}
& \ds\frac{1}{4}\norm{f(u,.)}_{L^4(P_{u})}^4+\norm{a^{-1/2}f\nabb f}_{\ll{2}}^2\\[2mm]
\ds = & \ds\frac{1}{4}\norm{f(-2,.)}_{L^4(P_{-2})}^4+\frac{1}{2}\int_{-2}^u\int_{P_{u'}}a^{-1}\trt f^4d\mu_{u'}du'+\int_{-2}^u\int_{P_{u'}} hf^3 d\mu_{u'}du'\\[1mm]
\ds\lesssim & \ds\norm{\trt}_{\l{\infty}{4}}\norm{f}^4_{\l{4}{\frac{16}{3}}}\\
\ds&+\norm{h}_{\l{2}{\frac{4}{3}}}\norm{f}^3_{\l{6}{12}},
\end{array}
\end{equation}
where we used in the last inequality the fact that $f\equiv 0$ on $u=-2$. In view of the assumptions \eqref{parab2} on $h$, we have $h=G\c H$, and thus:
\bee
\norm{h}_{\l{2}{\frac{4}{3}}}&\les& \norm{G}_{\l{\infty}{4}}\norm{H}_{\ll{2}}\\
&\les&  \norm{G}_{H^1(S)}\norm{H}_{\ll{2}}\\
&\les&\ep,
\eee
where we used Proposition \ref{p1}, and the estimates for $G$ and $H$ provided by \eqref{parab2}. Together 
with \eqref{parab4}, and the estimate \eqref{thregx1} for $\trt$, we obtain:
$$\norm{f(u,.)}_{L^4(P_{u})}\les \ep+\ep(\norm{f}_{\l{\infty}{4}}+\norm{f}_{\l{6}{12}}).$$
Taking the supremum in $u$ on the left-hand side, we get:
\be\lab{parab5}
\norm{f}_{\l{\infty}{4}}\les \ep+\ep\norm{f}_{\l{6}{12}}.
\ee

Next, we derive an estimate for $P_jf$. In view of Lemma \ref{lemma:parabPj} and since $f$ satisfies \eqref{parab0}, $P_jf$ satisfies the following parabolic equation:
\be\lab{parab6bis}
(\nabn-a^{-1}\lap)(P_jf)=a^{-1}P_j(ah)+a^{-1}[\nabla_{aN},P_j]f\textrm{ on }S.
\ee
Together with the estimate \eqref{p7e2}, we obtain:
\bee
&&\norm{P_jf}_{\l{\infty}{2}}+\norm{\nabb(P_jf)}_{\l{\infty}{2}}+\norm{\nabn(P_jf)}_{\ll{2}}+\norm{\nabb^2(P_jf)}_{\ll{2}}\\
\nn&\lesssim&\norm{a^{-1}P_j(ah)}_{\ll{2}}+\norm{a^{-1}[\nabla_{aN},P_j]f}_{\ll{2}}+\norm{P_jf(-2,.)}_{L^2(P_{-2})}\\
\nn&&+\norm{\nabb(P_jf)(-2,.)}_{L^2(P_{-2})}\\
\nn&\lesssim&\norm{a^{-1}P_j(ah)}_{\ll{2}}+\norm{a^{-1}[\nabla_{aN},P_j]f}_{\ll{2}},
\eee
where we used in the last inequality the fact that $P_f\equiv 0$ in $u=-2$. Using the finite band property for $P_j$ and the estimate \eqref{thregx1} for $a$, we obtain:
\be\lab{parab6}
2^j\norm{P_jf}_{\l{\infty}{2}}+\norm{\nabn(P_jf)}_{\ll{2}}+2^{2j}\norm{P_jf}_{\ll{2}}\lesssim\norm{P_j(ah)}_{\ll{2}}+\norm{[\nabla_{aN},P_j]f}_{\ll{2}}.
\ee

Next, we estimate the two terms in the right-hand side in \eqref{parab6} starting with the first one. Since $ah=G\c (aH)$, and in view of Proposition \ref{cor:commLP2}, we have:
\bea\lab{parab7}
\sum_{j\geq 0}2^{-j}\norm{P_j(ah)}_{\ll{2}}^2&\les& \norm{G}^2_{H^1(S)}\norm{aH}^2_{L^2(S)}\\ 
\nn&\les& \ep^2,
\eea
where we used in the last inequality the assumption \eqref{parab2} for $G$ and $H$, and the estimate \eqref{thregx1} for $a$. For the second term in the right-hand side in \eqref{parab6}, we used the commutator estimate \eqref{commLP1}, which yields for any $\delta>0$:
$$\norm{[\nabna, P_j]f}_{\ll{2}}\les \ep\norm{\La^{\frac{1}{2}+\delta}f}_{L^2(S)}+\ep\norm{\La^\delta f}_{\l{\infty}{2}}.$$
Together with \eqref{parab6} and \eqref{parab7}, we obtain:
\bea\lab{parab8}
&&\sum_{j\geq 0}\left(2^{3j}\norm{P_jf}_{\ll{2}}^2+2^j\norm{P_jf}_{\l{\infty}{2}}^2+2^{-j}\norm{P_j(\nabn f)}_{\ll{2}}^2\right)\\
\nn&\les&\ep^2(1+\norm{\La^{\frac{1}{2}+\delta}f}^2_{L^2(S)}+\norm{\La^\delta f}^2_{\l{\infty}{2}}).
\eea
Now, since $\delta>0$, we have:
\bea\lab{parab9}
&&\norm{\La^{\frac{1}{2}+\delta}f}^2_{L^2(S)}+\norm{\La^\delta f}^2_{\l{\infty}{2}}\\
\nn&\les& \left(\sum_{j\geq 0}(2^{j(\frac{1}{2}+\delta)}\norm{P_jf}_{L^2(S)}+2^{j\delta}\norm{P_jf}_{\l{\infty}{2}})\right)^2\\
\nn&\les& \sum_{j\geq 0}2^{j(1+3\delta)}\norm{P_jf}^2_{L^2(S)}+\sum_{j\geq 0}2^{j3\delta}\norm{P_jf}^2_{\l{infty}{2}}\\
\nn&\les& \sum_{j\geq 0}\left(2^{3j}\norm{P_jf}_{\ll{2}}^2+2^j\norm{P_jf}_{\l{\infty}{2}}^2\right)
\eea
where we chose in the last estimate $0<\delta\leq\frac{1}{3}$. Finally, \eqref{parab8} and \eqref{parab9} imply:
\be\lab{parab10}
\sum_{j\geq 0}\left(2^{3j}\norm{P_jf}_{\ll{2}}^2+2^j\norm{P_jf}_{\l{\infty}{2}}^2+2^{-j}\norm{P_j(\nabn f)}_{\ll{2}}^2\right)\les\ep^2.
\ee

Next, we estimate the $L^4_{[-2,2]}H^1(p)$ norm of $f$. Using the Bessel inequality for $P_j$, we have:
\bee
\norm{f}^4_{L^4_{[-2,2]}H^1(\p)}&= & \normm{\norm{f}^2_{H^1(\p)}}^2_{L^2_{[-2,2]}}\\
&\les& \normm{\sum_{j\geq 0}2^{2j}\norm{P_jf}^2_{\ll{2}}}^2_{L^2_{[-2,2]}}§\\
&\les& \left(\sum_{j\geq 0}2^{2j}\norm{P_jf}^2_{\l{4}{2}}\right)^2.
\eee
Thus, in view of \eqref{parab10}, we have:
\be\lab{parab11}
\norm{f}^4_{L^4_{[-2,2]}H^1(\p)}\les \left(\sum_{j\geq 0}2^{3j}\norm{P_jf}_{\ll{2}}^2\right)\left(\sum_{j\geq 0}2^j\norm{P_jf}_{\l{\infty}{2}}^2\right)\les \ep^4.
\ee
In view of \eqref{parab5}, we need to estimate $\norm{f}_{\l{6}{12}}$. Now, note that applying the Sobolev embedding  \eqref{eq:isoperimetric} with $f^q$ for some integer $q\geq 2$ yields:
$$\norm{f}_{\lp{2q}}^q=\norm{f^q}_{\lp{2}}\les \norm{\nabb(f^{q-1})}_{\lp{1}}+\norm{f^q}_{\lp{1}}\les \norm{f}^{q-1}_{\lp{2(q-1)}}\norm{\nabb f}_{\lp{2}}+\norm{f}_{\lp{q}}^q.$$
Using the previous inequality successively with $q=3, 4, 5, 6$ implies the following variant of the Gagliardo-Nirenberg inequality \eqref{eq:GNirenberg}:
$$\norm{f}_{\lp{12}}\les \norm{f}^{\frac{1}{3}}_{\lp{4}}\norm{\nabb f}^{\frac{2}{3}}_{\lp{2}}+\norm{f}_{\lp{2}}.$$
In particular, we obtain:
\be\lab{parab12}
\norm{f}_{\l{6}{12}}\les \norm{f}_{\l{\infty}{4}}^{\frac{1}{3}}\norm{\nabb f}^{\frac{2}{3}}_{\l{4}{2}}\les \ep^{\frac{2}{3}}\norm{f}_{\l{\infty}{4}}^{\frac{1}{3}},
\ee
where we used \eqref{parab11} in the last inequality. Finally, \eqref{parab12} and \eqref{parab5} yield:
$$\norm{f}_{\l{\infty}{4}}\les \ep,$$
which together with \eqref{parab10} implies \eqref{parab3} and \eqref{parab3bis}. This concludes the proof of the proposition.
\end{proof}

We have the following extension of Proposition \ref{p7}. 
\begin{proposition}\label{p7bis}
Let $f$ be a scalar function on $S$ such that $P_jf$ satisfies:
\be\lab{parab0bis}
(\nabn-a^{-1}\lap)(P_jf)=h\textrm{ on }S.
\ee
Assume also that $f\equiv 0$ on $u=-2$, and that we have a decomposition for $h$:
$$h=h_1+h_2.$$ 
Then, we have:
\begin{equation}\label{p7e2bis}
2^j\norm{P_jf}_{\l{\infty}{2}}+2^{2j}\norm{P_jf}_{\ll{2}}\lesssim\norm{h_1}_{\ll{2}}+2^j\norm{h_2}_{\l{1}{2}}.
\end{equation}
\end{proposition}

\begin{proof}
We multiply \eqref{parab0bis} by $\lap P_jf$ and integrate on $-2<u'<u$ where $u\leq 2$. We proceed as in \eqref{p7e5} \eqref{p7e6} \eqref{p7e7} \eqref{p7e8}, except that we estimate the integral in of \eqref{p7e3} involving $h$ as:
\bee
&&\left|\int_{-2}^{u}\int_{P_{u'}}h\lap P_jfad\mu_{u'}du'\right|\\
&\les & \left|\int_{-2}^{u}\int_{P_{u'}}h_1\lap P_jfad\mu_{u'}du'\right|+\left|\int_{-2}^{u}\int_{P_{u'}}h_2\lap P_jfad\mu_{u'}du'\right|\\
&\les & \norm{h_1}_{\ll{2}}\norm{\lap P_jf}_{\ll{2}}+\norm{h_2}_{\l{1}{2}}\norm{\lap(P_jf)}_{\l{\infty}{2}}\\
&\les & \norm{h_1}_{\ll{2}}\norm{\lap P_jf}_{\ll{2}}+2^j\norm{h_2}_{\l{1}{2}}\norm{\nabb(P_jf)}_{\l{\infty}{2}},
\eee
where we used the estimate \eqref{thregx1} for $a$ and the finite band property for $P_j$. We obtain the analog of \eqref{p7e8}:
\begin{equation}\label{p7e8bis}
\begin{array}{ll}
& \ds\norm{\nabb P_jf}_{\l{\infty}{2}}^2+\norm{\lap P_jf}_{\ll{2}}^2\\
\ds\lesssim & \ds\ep(\norm{\nabb P_jf}_{\ll{2}}^2+\norm{\nabb^2(P_jf)}_{\ll{2}}^2)+\norm{h_1}^2_{\ll{2}}+2^{2j}\norm{h_2}_{\l{1}{2}}^2.
\end{array}
\end{equation}
Finally, \eqref{p7e8bis} together with the Bochner inequality \eqref{eq:Bochconseqbis} and the finite band property for $P_j$ yields \eqref{p7e2bis}. This concludes the proof of the proposition.
\end{proof}

\begin{proposition}\label{prop:parab2}
Let $f$ be a scalar function on $S$ satisfying \eqref{parab0} and such that $f\equiv 0$ on $u=-2$. Assume that $h$ satisfies:
\begin{equation}\label{parab13}
h=h_1+h_2\textrm{ with }\sup_{j\geq 0}\norm{P_j(ah_1)}_{\ll{2}}\les 2^{2j}\ep\textrm{ and }\sup_{j\geq 0}\norm{P_j(ah_2)}_{\l{1}{2}}\les\ep 2^j.
\end{equation}
Then, we have:
\be\lab{parab14}
\sup_{j\geq 0}\norm{P_jf}_{\ll{2}}+\sup_{j\geq 0}2^{-j}\norm{P_jf}_{\l{\infty}{2}}\les\ep.
\ee
\end{proposition}

\begin{proof}
Recall from \eqref{parab6bis} that $P_jf$ satisfies the following parabolic equation:
$$(\nabn-a^{-1}\lap)(P_jf)=a^{-1}P_j(ah)+a^{-1}[\nabla_{aN},P_j]f\textrm{ on }S.$$
Together with the estimate \eqref{p7e2bis}, we obtain:
\bee
&&2^j\norm{P_jf}_{\l{\infty}{2}}+2^{2j}\norm{P_jf}_{\ll{2}}\\
\nn&\lesssim&\norm{a^{-1}P_j(ah_1)}_{\ll{2}}+2^j\norm{a^{-1}P_j(ah_2)}_{\l{1}{2}}+2^j\norm{a^{-1}[\nabla_{aN},P_j]f}_{\l{1}{2}}\\
\nn&\lesssim&\ep 2^{2j}+2^j\norm{[\nabla_{aN},P_j]f}_{\l{1}{2}},
\eee
where we used the estimate \eqref{thregx1} for $a$ and the assumption \eqref{parab13} on $h$. This yields:
\be\lab{parab15}
2^{-j}\norm{P_jf}_{\l{\infty}{2}}+\norm{P_jf}_{\ll{2}}\lesssim\ep +2^{-j}\norm{[\nabla_{aN},P_j]f}_{\l{1}{2}}.
\ee

Next, we use the commutator estimate \eqref{commLP3}. We have:
\be\lab{parab16}
\norm{[\nabna, P_j]f}_{\l{1}{2}}\les 2^{j\alpha}\ep\norm{\La^{-\delta}F}_{L^2(S)},
\ee
for any $0<\delta<\alpha<1$. Now, for any $\delta>0$, we have:
\bea\lab{parab17}
\norm{\La^{-\delta}F}_{L^2(S)}&\les& \sum_{j\geq 0}\norm{\La^{-\delta}P_jF}_{L^2(S)}\\
\nn&\les& \sum_{j\geq 0}2^{-j\delta}\norm{P_jF}_{L^2(S)}\\
\nn&\les&\sup_{j\geq 0}\norm{P_jf}_{\ll{2}}.
\eea
Finally, \eqref{parab15}, \eqref{parab16} and \eqref{parab17} imply for any $j\geq 0$:
$$2^{-j}\norm{P_jf}_{\l{\infty}{2}}+\norm{P_jf}_{\ll{2}}\lesssim\ep +\ep\sup_{j\geq 0}\norm{P_jf}_{\ll{2}},$$
which yields \eqref{parab14}. This concludes the proof of the proposition.
\end{proof}

\begin{proposition}\label{prop:parab3}
Let $f$ be a scalar function on $S$ satisfying \eqref{parab0} and such that $f\equiv 0$ on $u=-2$. Assume that $h$ satisfies:
\be\lab{parab18}
\sum_{j\geq 0}2^{-3j}\norm{P_j(ah)}^2_{\ll{2}}\les\ep^2.
\ee
Then, we have:
\be\lab{parab19}
\sum_{j\geq 0}\left(2^j\norm{P_jf}_{\ll{2}}^2+2^{-j}\norm{P_jf}_{\l{\infty}{2}}^2+2^{-3j}\norm{P_j(\nabn f)}_{\ll{2}}^2\right)\les\ep^2.
\ee
\end{proposition}

\begin{proof}
Recall the estimate \eqref{parab6}:
$$2^j\norm{P_jf}_{\l{\infty}{2}}+\norm{\nabn(P_jf)}_{\ll{2}}+2^{2j}\norm{P_jf}_{\ll{2}}\lesssim\norm{P_j(ah)}_{\ll{2}}+\norm{[\nabla_{aN},P_j]f}_{\ll{2}}.$$
This yields:
\bea\lab{parab20}
&&\sum_{j\geq 0}\left(2^j\norm{P_jf}_{\ll{2}}^2+2^{-j}\norm{P_jf}_{\l{\infty}{2}}^2+2^{-3j}\norm{P_j(\nabn f)}_{\ll{2}}^2\right)\\
\nn&\les& \sum_{j\geq 0}2^{-3j}\norm{P_j(ah)}^2_{\ll{2}}+\sum_{j\geq 0}2^{-3j}\norm{[\nabla_{aN},P_j]f}_{\ll{2}}^2\\
\nn&\les& \ep^2+\sum_{j\geq 0}2^{-3j}\norm{[\nabla_{aN},P_j]f}_{\ll{2}}^2,
\eea
where we used the assumption \eqref{parab18} in the last inequality.

Next, we use the commutator estimate \eqref{commLP5}, which yields for any $\delta>0$:
$$\norm{[\nabna, P_j]f}_{\ll{2}}\les 2^j\ep\norm{\La^{-(1-\delta)}f}_{\l{\infty}{2}}.$$
Together with \eqref{parab20}, we obtain:
\bea\lab{parab21}
&&\sum_{j\geq 0}\left(2^j\norm{P_jf}_{\ll{2}}^2+2^{-j}\norm{P_jf}_{\l{\infty}{2}}^2+2^{-3j}\norm{P_j(\nabn f)}_{\ll{2}}^2\right)\\
\nn&\les& \ep^2+\left(\sum_{j\geq 0}2^{-j}\right)\ep^2\norm{\La^{-(1-\delta)}f}_{\l{\infty}{2}}^2\\
\nn&\les& \ep^2(1+\norm{\La^{-(1-\delta)}f}_{\l{\infty}{2}}^2).
\eea
Now, since $\delta>0$, we have:
\bea\lab{parab22}
\norm{\La^{-(1-\delta)}f}_{\l{\infty}{2}}^2&\les& \left(\sum_{j\geq 0}2^{-j(1-\delta)}\norm{P_jf}_{\l{\infty}{2}})\right)^2\\
\nn&\les& \sum_{j\geq 0}2^{-(2-3\delta)j}\norm{P_jf}^2_{\l{infty}{2}}\\
\nn&\les& \sum_{j\geq 0}2^{-j}\norm{P_jf}_{\l{\infty}{2}}^2,
\eea
where we chose in the last estimate $0<\delta\leq\frac{1}{3}$. Finally, \eqref{parab21} and \eqref{parab22} imply 
\eqref{parab19}. This concludes the proof of the proposition.
\end{proof}

\begin{proposition}\label{prop:parab4}
Let $f$ be a scalar function on $S$ satisfying \eqref{parab0} and such that $f\equiv 0$ on $u=-2$. Assume that $h$ satisfies:
\be\lab{parab23}
\sum_{j\geq 0}2^{-j}\norm{P_j(ah)}^2_{\ll{2}}\les\ep^2.
\ee
Then, we have:
\be\lab{parab24}
\sum_{j\geq 0}\left(2^{3j}\norm{P_jf}_{\ll{2}}^2+2^j\norm{P_jf}_{\l{\infty}{2}}^2+2^{-j}\norm{P_j(\nabn f)}_{\ll{2}}^2\right)\les\ep^2.
\ee
\end{proposition}

\begin{proof}
The proof follows from \eqref{parab6bis} \eqref{parab6} \eqref{parab7} \eqref{parab8} \eqref{parab9} \eqref{parab10}. 
\end{proof}

\begin{proposition}\label{prop:parab5}
Let $0<b<1$. Let $f$ be a scalar function on $S$ satisfying \eqref{parab0} and such that $f\equiv 0$ on $u=-2$. Assume that $h$ satisfies:
\be\lab{parab25}
\sum_{j\geq 0}2^{2jb}\norm{P_j(ah)}^2_{\ll{2}}\les\ep^2.
\ee
Then, we have:
\be\lab{parab26}
\sum_{j\geq 0}\left(2^{(4+2b)j}\norm{P_jf}_{\ll{2}}^2+2^{(2+2b)j}\norm{P_jf}_{\l{\infty}{2}}^2+2^{2bj}\norm{P_j(\nabn f)}_{\ll{2}}^2\right)\les\ep^2.
\ee
\end{proposition}

\begin{proof}
Recall the estimate \eqref{parab6}:
$$2^j\norm{P_jf}_{\l{\infty}{2}}+\norm{\nabn(P_jf)}_{\ll{2}}+2^{2j}\norm{P_jf}_{\ll{2}}\lesssim\norm{P_j(ah)}_{\ll{2}}+\norm{[\nabla_{aN},P_j]f}_{\ll{2}}.$$
This yields:
\bea\lab{parab27}
\nn&&\sum_{j\geq 0}\left(2^{(4+2b)j}\norm{P_jf}_{\ll{2}}^2+2^{(2+2b)j}\norm{P_jf}_{\l{\infty}{2}}^2+2^{2bj}\norm{P_j(\nabn f)}_{\ll{2}}^2\right)\\
\nn&\les& \sum_{j\geq 0}2^{2bj}\norm{P_j(ah)}^2_{\ll{2}}+\sum_{j\geq 0}2^{2bj}\norm{[\nabla_{aN},P_j]f}_{\ll{2}}^2\\
&\les& \ep^2+\sum_{j\geq 0}2^{2bj}\norm{[\nabla_{aN},P_j]f}_{\ll{2}}^2,
\eea
where we used the assumption \eqref{parab25} in the last inequality.

Next, we use the commutator estimate \eqref{commLP4}, which yields for any $\delta>0$:
$$\norm{[\nabna, P_j]f}_{\ll{2}}\les 2^{-(1-\delta)j}\ep(\norm{\lap f}_{L^2(S)}+\norm{\nabb f}_{\l{\infty}{2}}).$$
Together with \eqref{parab27}, we obtain:
\bea\lab{parab28}
\nn&&\sum_{j\geq 0}\left(2^{(4+2b)j}\norm{P_jf}_{\ll{2}}^2+2^{(2+2b)j}\norm{P_jf}_{\l{\infty}{2}}^2+2^{2bj}\norm{P_j(\nabn f)}_{\ll{2}}^2\right)\\
\nn&\les& \ep^2+\left(\sum_{j\geq 0}2^{-2j(1-b-\delta)}\right)\ep^2(\norm{\lap f}^2_{L^2(S)}+\norm{\nabb f}^2_{\l{\infty}{2}})\\
&\les& \ep^2(1+\norm{\lap f}^2_{L^2(S)}+\norm{\nabb f}^2_{\l{\infty}{2}}),
\eea
where we chose in the last inequality $0<\delta<1-b$ which is possible since $b<1$. Now, the finite band property for $P_j$ yields:
\bee
\norm{\lap f}^2_{L^2(S)}+\norm{\nabb f}^2_{\l{\infty}{2}}&\les& \left(\sum_{j\geq 0}(2^{2j}\norm{P_jf}_{\ll{2}}+2^j\norm{P_jf}_{\l{\infty}{2}})\right)^2\\
&\les&\sum_{j\geq 0}\left(2^{(4+2b)j}\norm{P_jf}_{\ll{2}}^2+2^{(2+2b)j}\norm{P_jf}_{\l{\infty}{2}}^2\right),
\eee
where we used in the last inequality the fact that $b>0$. Together with \eqref{parab28}, this implies \eqref{parab26}. This concludes the proof of the proposition.
\end{proof}

\section{Estimates for $\nabn a$ and $\nabn^2a$ (proof of  Theorem \ref{thnabn2a})}\label{sec:nabla2a}

This section is dedicated to the proof of Theorem \ref{thnabn2a}. We recall the decomposition \eqref{r23} \eqref{r23bis} \eqref{r24}:  
\begin{equation}\label{zoc}
(\nabn - a^{-1}\lap)\nabn a=\divb(H)+h, 
\end{equation}
where the tensor $H$ is given by
\begin{equation}\label{zoc1}
H=-\nabn k_{.N}-R_{.N}, 
\end{equation}
and where the scalar $h$ is given schematically by
\begin{equation}\label{zoc2}
\begin{array}{ll}
\ds h= &  -a^{-1}\trt\lap a-2a^{-1}\hth\nabb^2a+2a^{-2}\nabb a\nabb\nabn a-2R_{N.}\ana -\nabb\trt\ana\\
& \ds +2\hth |\ana|^2+2\th\nabn\th+\ana\nabn k_{N.}+\nabla\th k+\theta\nabla k+R_{N.}k \\
&\ds+\theta\ana k_{N.}+2k\nabn k+\nabn k(\nabb a,N)+k(\nabb\nabn a,N)\\
&\ds+\nabn a k(\nabb a,N)+\th k+ k(\nabb a,\nabb a)+\th R.
\end{array}
\end{equation}
We introduce the scalar functions on $S$ $a_1$ and $a_2$ solutions of:
\begin{equation}\label{zoc3}
(\nabn - a^{-1}\lap)a_1=h\textrm{ on S},\, a_1(-2,.)=0,
\end{equation}
and:
\begin{equation}\label{zoc4}
(\nabn - a^{-1}\lap)a_2=\divb(H)\textrm{ on S},\, a_2(-2,.)=0,
\end{equation}
which yields, in view of \eqref{zoc}, the fact that $\nabn a(-2,.)=0$, the decomposition:
\be\lab{zoc5}
\nabn a=a_1+a_2.
\ee

\begin{remark}
In the right-hand side of \eqref{zoc}, the regularity of $h$ is better than the regularity of $\divb(H)$ (see \eqref{zoc11}). On the other hand, we can not make sense of $\nabn h$, while the contracted Bianchi identities on $\Sigma$ allow us to make sense of $\nabn\divb(H)$ (see in particular \eqref{zoc30}). Thus, the idea behind the decomposition \eqref{zoc5} is to take advantage of the regularity of $h$ for $a_1$, and to use the structure of $\divb(H)$ to obtain a useful equation for $\nabn a_2$ (see \eqref{zoc36}). We carry out this strategy in the rest of the section.
\end{remark}

The following two propositions state the regularity of $a_1$ and $a_2$.
\begin{proposition}\label{prop:zoc1}
Let $a_1$ be the solution of \eqref{zoc3}, where $h$ is defined in \eqref{zoc2}. Then, we have:
\begin{equation}\label{zoc6}
\norm{a_1}_{\l{\infty}{4}}\lesssim\ep,
\end{equation}
and:
\be\lab{zoc7}
\sum_{j\geq 0}\left(2^{3j}\norm{P_ja_1}_{\ll{2}}^2+2^j\norm{P_ja_1}_{\l{\infty}{2}}^2+2^{-j}\norm{P_j(\nabn a_1)}_{\ll{2}}^2\right)\les\ep^2.
\ee
\end{proposition}

\begin{proposition}\label{prop:zoc2}
et $a_2$ be the solution of \eqref{zoc4}, where $H$ is defined in \eqref{zoc1}. Then, we have:
\begin{equation}\label{zoc8}
\norm{a_2}_{\l{\infty}{4}}\lesssim\ep,
\end{equation}
\be\lab{zoc9}
\sum_{j\geq 0}2^{2j}\norm{P_ja_2}_{\l{\infty}{2}}^4\les\ep^4.
\ee
and:
\be\lab{zoc10}
\sup_{j\geq 0}\norm{P_j(\nabn a_2)}_{\ll{2}}\les\ep.
\ee
\end{proposition}

The proof of Proposition \ref{prop:zoc1} is postponed to section \ref{sec:zoc1}, while the proof of Proposition \ref{prop:zoc2} is postponed to section \ref{sec:zoc2}. In view of the decomposition \eqref{zoc5} for $\nabn a$, the estimates \eqref{zoc6} \eqref{zoc7} for $a_1$, and the estimates \eqref{zoc8} \eqref{zoc9} \eqref{zoc10} for $a_2$, we immediately obtain the estimate \eqref{nabn2a1} for $\nabn a$ and $\nabn^2a$. This concludes the proof of Theorem  \ref{thnabn2a}. 

\subsection{Proof of Proposition \ref{prop:zoc1}}\lab{sec:zoc1}

In view of the definition \eqref{zoc2}, the scalar function $h$ may be written as a linear combination of terms of the form $h=F\c G$, where $F$ is schematically given by:
$$F=   \nabb^2a+\nabb\nabn a+R +\nabla\th+|\ana|^2+\nabla k+\ana k+\nabla( a)k,$$
and $G$ is schematically given by:
$$G=   \th+\ana+k.$$
In view of the estimate \eqref{thregx1} for $a$ and $\th$, and in view of the assumption \eqref{small1} for $R$ and $k$, we deduce:
\begin{equation}\label{zoc11}
h=F\c G\textrm{ with }\norm{F}_{\ll{2}}\les\ep\textrm{ and }\norm{G}_{H^1(S)}\les\ep.
\end{equation}
Now, in view of the equation \eqref{zoc3} satisfied by $a_1$, and the decomposition \eqref{zoc11} for $h$, the estimates \eqref{zoc6} and \eqref{zoc7} are a consequence of the estimates \eqref{parab3} and \eqref{parab3bis} of proposition \ref{prop:parab1}. This concludes the proof of Proposition \ref{prop:zoc1}.

\subsection{Proof of Proposition \ref{prop:zoc2}}\lab{sec:zoc2}

In view of the decomposition \eqref{zoc5} of $\nabn a$, we have:
$$a_2=\nabn a-a_1$$
which together with the estimate \eqref{thregx1} for $a$ and the estimates \eqref{zoc6} and \eqref{zoc7} for $a_1$ implies:
\be\lab{zoc12}
\norm{\nabb a_2}_{\ll{2}}+\norm{a_2}_{\l{\infty}{2}}\les \ep.
\ee

Next, we derive en equation for $\nabna a_2$. We use the following commutation lemma. 
\begin{lemma}\lab{lemma:commnabna}
Let $f$ satisfying the following parabolic equation:
$$(\nabla_{aN}-\lap)f=ah.$$
Then, $\nabna f$ satisfies the following parabolic equation:
\be\lab{zoc13}
(\nabn-a^{-1}\lap)(\nabna f)=\nabn(ah)+a^{-1}[\nabla_{aN},\lap]f\textrm{ on }S.
\ee
\end{lemma}

\begin{proof}
We multiply equation the equation satisfied by $f$ with $a$. We obtain:
$$(\nabla_{aN}-\lap)f=ah.$$
Next, we commute with $\nabna$. We obtain:
$$(\nabla_{aN}-\lap)(\nabna f)=\nabna(ah)+[\nabla_{aN},\lap]f.$$
Finally, multiplying with $a^{-1}$, we get \eqref{zoc13}. This concludes the proof of Lemma \ref{lemma:parabPj}.
\end{proof}

In view of the equation \eqref{zoc4} satisfied by $a_2$, and in view of the commutation Lemma \ref{lemma:commnabna}, $\nabna a_2$ satisfies:
\be\lab{zoc14}
(\nabn-a^{-1}\lap)(\nabna a_2)=\nabn(a\divb(H))+a^{-1}[\nabla_{aN},\lap]a_2.
\ee
Next, we evaluate both terms in the right-hand side of \eqref{zoc14} starting with the second one. Using the commutation formula \eqref{dj3}, we have:
\be\lab{zoc15}
a^{-1}[\nabla_{aN},\lap]a_2=h_1+h_2,
\ee
where $h_1$ and $h_2$ are given schematically by:
\be\lab{zoc16}
h_1=\divb(\th\c \nabb a_2),
\ee
and:
\be\lab{zoc17}
h_2=(R+\nabb\th+\hth\nabb a)\nabb a_2.
\ee
We first estimate $h_1$. We have:
$$ah_1=\divb(a\th\c\nabb a_2)-\nabb(a)\c\th\nabb a_2.$$
In view of the estimate \eqref{lbz14bis} and the sharp Bernstein inequality \eqref{eq:strongbernscalarbis}, we obtain:
\bea\lab{zoc18}
&&\norm{P_j(ah_1)}_{\ll{2}}\\
\nn&\les& 2^{\frac{3j}{2}}\norm{a\th\c\nabb a_2}_{\l{2}{\frac{4}{3}}}+2^j\norm{\nabb(a)\c\th\nabb a_2}_{\ll{2}}\\
\nn&\les& 2^{\frac{3j}{2}}\norm{a}_{\ll{\infty}}\norm{\th}_{\l{\infty}{4}}\norm{\nabb a_2}_{\ll{2}}+2^j\norm{\nabb(a)}_{\l{\infty}{4}}\norm{\th}_{\l{\infty}{4}}\norm{\nabb a_2}_{\ll{2}}\\
\nn&\les& 2^{\frac{3j}{2}}\ep,
\eea
where we used in the last inequality the estimate \eqref{thregx1} for $\th$ and $a$, and the estimate \eqref{zoc12} for $\nabb a_2$. Next, we estimate $h_2$. In view of the estimate \eqref{thregx1} for $\th$ and $a$, the assumption \eqref{small1} for $R$, and the estimate \eqref{zoc12} for $\nabb a_2$, we have:
$$\norm{h_2}_{\ll{1}}\les (\norm{R}_{\ll{2}}+\norm{\nabb\th}_{\ll{2}}+\norm{\hth}_{\ll{4}}\norm{\nabb a}_{\ll{4}})\norm{\nabb a_2}_{\ll{2}}\les \ep^2,$$
which together with the dual of the sharp Bernstein inequality \eqref{eq:strongbernscalarbis} yields:
\be\lab{zoc19}
\norm{P_j(ah_2)}_{\l{1}{2}}\les 2^j\norm{a}_{\ll{\infty}}\norm{h_1}_{\ll{1}}\les 2^j\ep,
\ee
where we used in the last inequality the estimate \eqref{thregx1} for $a$. 

Next, we estimate the first term in the right-hand side of \eqref{zoc14}. We have:
\bee
a\nabn(a\divb(H))&=&  \nabna(a\divb(H))\\
&=& \nabna(a)\divb(H)+a[\nabna, \divb](H)+a\divb(\nabna H),
\eee
which together with the commutator formula \eqref{commutnabna1} yields:
\be\lab{zoc20}
a\nabn(a\divb(H))=a\divb(\nabna H)+h_3+h_4,
\ee
where $h_3$ and $h_4$ are given by:
\be\lab{zoc21}
h_3=\divb(\nabn(a)H-a\th H),
\ee
and 
\be\lab{zoc22}
h_4=-(\nabb(\nabna(a))+\nabb(a\th) +  a R_{N.} + a\th\c\nabb(a))\c H
\ee
Now, in view of the definition of $H$ \eqref{zoc1}, we have:
\be\lab{zoc23}
\norm{H}_{\ll{2}}\les \norm{\nabla k}_{\ll{2}}+\norm{R}_{\ll{2}}\les \ep,
\ee
where we used in the last inequality the assumption \eqref{small1} on $R$ and $k$. We first estimate $h_3$. 
In view of the estimate \eqref{lbz14bis}, we obtain:
\bea\lab{zoc24}
&&\norm{P_j(h_3)}_{\ll{2}}\\
\nn&\les& 2^{\frac{3j}{2}}(\norm{\nabn(a)H}_{\l{2}{\frac{4}{3}}}+\norm{a\th H}_{\l{2}{\frac{4}{3}}})\\
\nn&\les& 2^{\frac{3j}{2}}\norm{a}_{\ll{\infty}}(\norm{\nabn(a)}_{\l{\infty}{4}}+\norm{\th}_{\l{\infty}{4}})\norm{H}_{\ll{2}}\\
\nn&\les& 2^{\frac{3j}{2}}\ep(1+\norm{a_2}_{\l{\infty}{4}}),
\eea
where we used in the last inequality the estimate \eqref{thregx1} for $\th$ and $a$, the estimate \eqref{zoc23} for $H$, the decomposition $\nabn a=a_1+a_2$, and the estimate \eqref{zoc6} for $a_1$. Next, we estimate $h_4$. In view of the estimate \eqref{thregx1} for $\th$ and $a$, the assumption \eqref{small1} for $R$, and the estimate \eqref{zoc23} for $H$, we have:
$$\norm{h_4}_{\ll{1}}\les (\norm{\nabb\nabn a}_{\ll{2}}+\norm{\nabb(a\th)}_{\ll{2}}+\norm{aR}_{\ll{2}}+\norm{\hth}_{\ll{4}}\norm{\nabb a}_{\ll{4}})\norm{H}_{\ll{2}}\les \ep^2,$$
which together with the dual of the sharp Bernstein inequality \eqref{eq:strongbernscalarbis} yields:
\be\lab{zoc25}
\norm{P_j(h_4)}_{\l{1}{2}}\les 2^j\norm{h_4}_{\ll{1}}\les 2^j\ep.
\ee

Next, we estimate the first term in the right-hand side of \eqref{zoc20}, i.e. $a\divb(\nabna H)$. Recall the definition of $H$ \eqref{zoc1}:
$$H=-\nabn k_{.N}-R_{.N}.$$
We take the $\nabna$ derivative of each of the two terms in the definition of $H$ starting with the first one. 
Using the constraint equations \eqref{const} and the fact that we have a maximal foliation yields:
\begin{equation}\label{zoc26}
\nabn k_{NA}=-\nabla^Bk_{BA}.  
\end{equation}
Now, we have:
$$\nabla^Bk_{BA}=\divb k_A+\trt k_{NA}+\th_{AB}k_{NB}.$$
Together with \eqref{zoc26}, we obtain, schematically:
\be\lab{rabouam}
\nabn k_{N.}=-\divb k-\th\c k_{N.}.
\ee
Taking the $\nabna$ derivative, we obtain:
$$\nabna(\nabn k_{N.})=-\divb(\nabna k)-[\nabna, \divb]k-\nabna(\th)\c k_{N.}-\th\c\nabna k_{N.}-\th\c k_{. \nabna N}.$$
Using the commutation formula \eqref{commutnabna1} and the structure equation \eqref{frame1}, we obtain, schematically:
\be\lab{zoc27}
\nabna(\nabn k_{N.})=-\divb(\nabna k)+a\th\nabla k+a(R+\nabla\th+\th\nabb(a))k.
\ee
Next, we take the $\nabna$ derivative of the second term in the definition of $H$. The  twice-contracted Bianchi identity on $\s$ yields:
$$\nabn R_{NA}=-\nabla^BR_{AB}+\frac{1}{2}\nabla_AR,$$
which together with the constraint equations \eqref{const} implies:
\be\lab{zoc28}
\nabn R_{NA}=-\nabla^BR_{AB}+k\c\nabla_Ak.
\ee
Now, we have:
$$\nabla^BR_{AB}=\nabb^BR_{AB}+\trt R_{NA}+\th_{AB}R_{NB}.$$
Together with \eqref{zoc28}, we obtain schematically:
\be\lab{rabouam1}
\nabn R_{N.}=-\divb R-\th\c R+k\c\nabla k.
\ee
This yields:
\be\lab{zoc29}
\nabna (R_{N.})=-\divb (aR)+\nabb(a)\c R-a\th\c R+ak\c\nabla k.
\ee
Finally, in view of the definition \eqref{zoc1} of $H$, \eqref{zoc27} and \eqref{zoc29}, we obtain schematically:
$$\nabna H=\divb(\nabna k+aR)+a(\th+k)\nabla k+a(R+\nabla\th+\th\nabb(a))k+(\nabb(a)+a\th)\c R.$$
This yields:
\be\lab{zoc30}
\nabna H=\divb(H_1)+H_2,
\ee
where $H_1$ is a 2-tensor given by:
$$H_1=\nabna k+aR,$$
and $H_2$ is a 1-form given by:
$$H_2=a(\th+k)\nabla k+a(R+\nabla\th+\th\nabb(a))k+(\nabb(a)+a\th)\c R.$$
In particular, we have:
\be\lab{zoc31}
\norm{H_1}_{\ll{2}}\les \norm{a}_{\ll{\infty}}(\norm{\nabla k}_{\ll{2}}+\norm{R}_{\ll{2}})\les\ep,
\ee
where we used in the last inequality the assumption \eqref{small1} on $k$ and $R$. Also, we have:
\bea
\nn\norm{H_2}_{\l{2}{\frac{4}{3}}}&\les& \norm{a}_{\ll{\infty}}(\norm{\th}_{\l{\infty}{4}}+\norm{k}_{\l{\infty}{4}}+\norm{\ana}_{\l{\infty}{4}})\\
\nn&&(\norm{\nabla k}_{\ll{2}}+\norm{R}_{\ll{2}}+\norm{\nabla\th}_{\ll{2}})\\
\lab{zoc32}&\les& \ep,
\eea
where we used in the last inequality the assumption \eqref{small1} on $k$ and $R$ and the estimate \eqref{thregx1} for $a$ and $\th$. In view of \eqref{zoc30}, we obtain schematically:
\bea\lab{zoc33}
\nn a\divb(\nabna H)&=&\divb\divb(aH_1)+\divb(\nabb(a)\c H_1)+\nabb^2(a)\c H_1+\divb(aH_2)+\nabla(a)\c H_2\\
&=& h_5+h_6,
\eea 
where the scalar functions on $S$ $h_5$ and $h_6$ are given by:
$$h_5=\divb\divb(aH_1)+\divb(\nabb(a)\c H_1)+\divb(aH_2)+\nabla(a)\c H_2,$$
and:
$$h_6=\nabb^2(a)\c H_1.$$
Using the finite band property for $P_j$, the sharp Bernstein inequality \eqref{eq:strongbernscalarbis} and the estimate \eqref{lbz14bis}, we obtain:
\bee
&&\norm{P_j(h_5)}_{\ll{2}}\\
&\les& 2^{2j}\norm{\lap^{-1}\divb\divb}_{\mathcal{L}(\lp{2})}\norm{aH_1}_{\ll{2}}+2^{\frac{3j}{2}}\norm{\nabb(a)\c H_1}_{\ll{2}}+2^{\frac{3j}{2}}\norm{aH_2}_{\l{2}{\frac{4}{3}}}\\
&&+2^j\norm{\nabla(a)\c H_2}_{\l{2}{1}}\\
&\les& 2^{2j}\norm{\nabb^2\lap^{-1}}_{\mathcal{L}(\lp{2})}\norm{a}_{\ll{\infty}}\norm{H_1}_{\ll{2}}+2^{\frac{3j}{2}}\norm{\nabb(a)}_{\l{\infty}{4}}\norm{H_1}_{\ll{2}}\\
&&+2^{\frac{3j}{2}}\norm{a}_{\ll{\infty}}\norm{H_2}_{\l{2}{\frac{4}{3}}}+2^j\norm{\nabla(a)}_{\l{\infty}{4}}\norm{H_2}_{\l{2}{\frac{4}{3}}},
\eee
which together with the Bochner inequality for scalars \eqref{eq:Bochconseqbis}, the estimate \eqref{thregx1} for $a$, and the estimates \eqref{zoc31} and \eqref{zoc32} for $H_1$ and $H_2$ implies:
\be\lab{zoc34}
\norm{P_j(h_5)}_{\ll{2}}\les 2^{2j}\ep.
\ee
Next, we estimate $h_6$. In view of the estimate \eqref{thregx1} for $a$ and the estimate \eqref{zoc31} for $H_1$, we have:
$$\norm{h_6}_{\ll{1}}\les\norm{\nabb^2a}_{\ll{2}}\norm{H_1}_{\ll{2}}\les \ep.$$
Together with the sharp Bernstein inequality \eqref{eq:strongbernscalarbis}, we obtain:
\be\lab{zoc35}
\norm{P_j(h_6)}_{\l{1}{2}}\les 2^j\ep.
\ee

Finally, in view of \eqref{zoc14}, \eqref{zoc15}, \eqref{zoc20} and \eqref{zoc33}, we have:
\be\lab{zoc36}
(\nabn-a^{-1}\lap)(\nabna a_2)=h_7+h_8,
\ee
where $h_7$ and $h_8$ are defined by:
$$h_7=h_1+a^{-1}h_3+a^{-1}h_5,$$
and:
$$h_8=h_2+a^{-1}h_4+a^{-1}h_6.$$
In view of \eqref{zoc18}, \eqref{zoc24} and \eqref{zoc34}, we have:
\be\lab{zoc37}
\norm{P_j(ah_7)}_{\ll{2}}\les 2^{2j}\ep(1+\norm{a_2}_{\l{\infty}{4}}).
\ee
Also, in view of \eqref{zoc19}, \eqref{zoc25} and \eqref{zoc35}, we have:
\be\lab{zoc38}
\norm{P_j(ah_8)}_{\l{1}{2}}\les 2^j\ep.
\ee
Now, in view of \eqref{zoc36}, \eqref{zoc37} and \eqref{zoc38}, \eqref{parab14} implies:
\be\lab{zoc39}
\sup_{j\geq 0}\norm{P_j(\nabna a_2)}_{\ll{2}}\les\ep(1+\norm{a_2}_{\l{\infty}{4}}).
\ee

Next, we state three lemma.
\begin{lemma}\lab{lemma:zoc1}
For any scalar function $f$ on $S$, and for any $0\leq b<1$, we have:
\be\lab{zoc40}
\sup_{j\geq 0}2^{jb}\norm{P_j(a^{\pm 1}f)}_{\ll{2}}\les \sup_{j\geq 0}2^{jb}\norm{P_jf}_{\ll{2}}.
\ee
\end{lemma}

\begin{lemma}\lab{lemma:zoc2}
For any scalar function $f$ on $S$ such that $f\equiv 0$ on $u=-2$, we have:
\be\lab{zoc41}
\left(\sum_{j\geq 0}2^{2j}\norm{P_jf}^4_{\l{\infty}{2}}\right)^{\frac{1}{4}}\les \norm{\nabb f}_{\ll{2}}+\norm{f}_{\l{\infty}{2}}+\sup_{j\geq 0}\norm{P_j(\nabn f)}_{\ll{2}}.
\ee
\end{lemma}

\begin{lemma}\lab{lemma:zoc3}
For any scalar function $f$ on $S$ such that $f\equiv 0$ on $u=-2$, we have:
\be\lab{zoc42}
\norm{f}_{\l{\infty}{4}}\les \norm{\nabb f}_{\ll{2}}+\sup_{j\geq 0}\norm{P_j(\nabn f)}_{\ll{2}}.
\ee
\end{lemma}

The proof of Lemma \ref{lemma:zoc1} is postponed to section \ref{sec:lemmazoc1}, the proof of Lemma \ref{lemma:zoc2} is postponed to section \ref{sec:lemmazoc2} and the proof of Lemma \ref{lemma:zoc1} is postponed to section \ref{sec:lemmazoc3}. Let us now conclude the proof of Proposition \ref{prop:zoc2}. In view of \eqref{zoc12} and Lemma \ref{lemma:zoc3}, we have:
\be\lab{zoc43}
\norm{a_2}_{\l{\infty}{4}}\les \ep+\sup_{j\geq 0}\norm{P_j(\nabn a_2)}_{\ll{2}}.
\ee
Together with \eqref{zoc39}, we obtain:
$$\sup_{j\geq 0}\norm{P_j(\nabna a_2)}_{\ll{2}}\les\ep(1+\sup_{j\geq 0}\norm{P_j(\nabn a_2)}_{\ll{2}}).$$
Together with Lemma \ref{lemma:zoc1} and \eqref{zoc43}, we obtain:
\be\lab{zoc44}
\sup_{j\geq 0}\norm{P_j(\nabn a_2)}_{\ll{2}}+\norm{a_2}_{\l{\infty}{4}}\les\ep.
\ee
Finally, \eqref{zoc12}, \eqref{zoc44} and Lemma \ref{lemma:zoc2} imply:
$$\left(\sum_{j\geq 0}2^{2j}\norm{P_jf}^4_{\l{\infty}{2}}\right)^{\frac{1}{4}}\les\ep.$$
This concludes the proof of Proposition \ref{prop:zoc2}.

\subsection{Proof of Lemma \ref{lemma:zoc1}}\lab{sec:lemmazoc1}

We have:
\be\lab{zoc45}
\norm{P_j(a^{\pm 1}f)}_{\ll{2}}\les \sum_{l\geq 0}\norm{P_j(a^{\pm 1}P_lf)}_{\ll{2}}.
\ee
If $l\leq j$, we use the finite band property for $P_j$, and we obtain:
\bea\lab{zoc46}
&&2^{jb}\norm{P_j(a^{\pm 1}P_lf)}_{\ll{2}}\\
\nn&\les& 2^{-j(1-b)}\norm{\nabb(a^{\pm 1}P_lf)}_{\ll{2}}\\
\nn&\les& 2^{-j(1-b)}\norm{a^{\pm 1}\nabb P_lf}_{\ll{2}}+2^{-j(1-b)}\norm{\nabb(a^{\pm 1})P_lf}_{\ll{2}}\\
\nn&\les& 2^{-j(1-b)}\norm{a^{\pm 1}}_{\ll{\infty}}\norm{\nabb P_lf}_{\ll{2}}+2^{-j(1-b)}\norm{\nabb(a^{\pm 1})}_{\l{\infty}{4}}\norm{P_lf}_{\l{2}{4}}\\
\nn&\les& 2^{-(j-l)(1-b)}\sup_{q\geq 0}2^{qb}\norm{P_qf}_{\ll{2}},
\eea
where we used in the last inequality the finite band property and Bernstein for $P_l$, and the estimate \eqref{thregx1} for $a$. 

If $l>j$, we use the fact that:
\bee
P_j(a^{\pm 1}P_lf)&=&2^{-2l}P_j(a^{\pm 1}\lap P_lf)\\
&=&2^{-2l}P_j(\divb(a^{\pm 1}\nabb P_lf))-2^{-2l}P_j(\nabb(a^{\pm 1})\nabb P_lf),
\eee
which together with the finite band property and Bernstein for $P_j$ yields:
\bea\lab{zoc47}
\nn 2^{jb}\norm{P_j(a^{\pm 1}P_lf)}_{\ll{2}}&\les& 2^{-2l+j(1+b)}\norm{a^{\pm 1}\nabb P_lf}_{\ll{2}}+2^{-2l+\frac{j}{2}}\norm{\nabb(a^{\pm 1})\nabb P_lf}_{\l{2}{4}}\\
\nn&\les& 2^{-2l+j(1+b)}(\norm{a^{\pm 1}}_{\ll{\infty}}+\norm{\nabb(a^{\pm 1})}_{\l{\infty}{4}})\norm{\nabb P_lf}_{\ll{2}}\\
&\les& 2^{-(l-j)(1+b)}\sup_{q\geq 0}2^{qb}\norm{P_qf}_{\ll{2}},
\eea
where we used in the last inequality the finite band property for $P_l$, and the estimate \eqref{thregx1} for $a$.

Finally, \eqref{zoc45}, \eqref{zoc46} and \eqref{zoc47} yield for all $j\geq 0$:
$$2^{jb}\norm{P_j(a^{\pm 1}f)}_{\ll{2}}\les \left(\sum_{l\geq 0}(2^{-|j-l|(1-b)}+2^{-|j-l|(1+b)})\right)\sup_{q\geq 0}\norm{P_qf}_{\ll{2}}\les \sup_{q\geq 0}\norm{P_qf}_{\ll{2}},$$
where we used in the last inequality the fact that $0\leq b<1$. Taking the supremum in $j$ yields \eqref{zoc40}. This concludes the proof of Lemma \ref{lemma:zoc1}.

\subsection{Proof of Lemma \ref{lemma:zoc2}}\lab{sec:lemmazoc2}

We follow the proof of Corollary \ref{cor:commLP1}. Proceeding as in \eqref{clp14}, we obtain for all $j\geq 0$:
\begin{equation}\label{zoc48}
2^j\norm{P_jf}^2_{\l{\infty}{2}}\lesssim 2^j\left(\int_{-2}^{-2}\norm{P_jf}_{\lp{2}}\norm{\nabna P_jf}_{\lp{2}}du\right)+2^{2j}\norm{P_jf}_{\ll{2}}^2.
\end{equation}
Then, proceeding as in \eqref{clp14bis}, we obtain in view of \eqref{zoc48}:
\bee
&&2^j\norm{P_jf}^2_{\l{\infty}{2}}\\
&\les& 2^j\norm{P_jf}_{\ll{2}}\norm{P_j(\nabna f)}_{\ll{2}}+2^j\norm{P_jf}_{\l{\infty}{2}}\norm{[\nabna, P_j]F}_{\l{1}{2}}\\
&&+2^{2j}\norm{P_jf}_{\ll{2}}^2.
\eee
This implies:
\bee
2^j\norm{P_jf}^2_{\l{\infty}{2}}&\les& 2^j\norm{P_jf}_{\ll{2}}\norm{P_j(\nabna f)}_{\ll{2}}\\
\nn&&+2^j\norm{[\nabna, P_j]F}^2_{\l{1}{2}}+2^{2j}\norm{P_jf}_{\ll{2}}^2.
\eee
Taking the square on both side, summing in $j\geq 0$, and using the Bessel inequality yields:
\bee
\sum_{j\geq 0}2^{2j}\norm{P_jf}^4_{\l{\infty}{2}}&\les& \left(\sum_{j\geq 0}2^{2j}\norm{P_jf}^2_{\ll{2}}\right)\left(\sup_{j\geq 0}\norm{P_j(\nabna f)}^2_{\ll{2}}\right)\\
\nn&&+\sum_{\j\geq 0}2^{2j}\norm{[\nabna, P_j]F}^4_{\l{1}{2}}+\sum_{j\geq 0}2^{4j}\norm{P_jf}_{\ll{2}}^4.
\eee
Using the Bessel inequality for $P_j$ and Lemma \ref{lemma:zoc1}, we finally obtain:
\bea\label{zoc49}
&&\sum_{j\geq 0}2^{2j}\norm{P_jf}^4_{\l{\infty}{2}}\\
\nn&\les& \norm{\nabb f}^4_{\ll{2}}+\left(\sup_{j\geq 0}\norm{P_j(\nabn f)}_{\ll{2}}\right)^4+\sum_{\j\geq 0}2^{2j}\norm{[\nabna, P_j]F}^4_{\l{1}{2}}.
\eea

Now, we have in view of the commutator estimate \eqref{commLP2}:
$$\norm{[P_j,\nabna]f}_{\l{1}{2}}\les 2^{-j(1-\delta)}\ep(\norm{\nabb f}_{\ll{2}}+\norm{f}_{\l{\infty}{2}}),$$
for any $\delta>0$. Proceeding as in \eqref{clp16}, we obtain:
\be\lab{zoc50}
\sum_{j\geq 0}2^{2j}\norm{P_jf}^4_{\l{\infty}{2}}\les \norm{\nabb f}^4_{\ll{2}}+\norm{f}^4_{\l{\infty}{2}}.
\ee
Finally, \eqref{zoc49} and \eqref{zoc50} yield \eqref{zoc41}. This concludes the proof of Lemma \ref{lemma:zoc3}.

\subsection{Proof of Lemma \ref{lemma:zoc3}}\lab{sec:lemmazoc3}

Lemma \ref{lemma:zoc3} is an improvement of Proposition \ref{p1} where one has a slightly weaker assumption on $\nabn f$. Proceeding as in \eqref{p1e1}, we have:
\bea\label{zoc51}
&&\norm{f(u,.)}_{\lp{4}}^4\\
\nn&=&\norm{f(-2,.)}_{L^4(P_{-2})}^4+4\int_{-2}^u\int_{P_{u'}}\nabn f(u',x')f^3du'd\mu_{u'}+\int_{-2}^u\int_{P_{u'}}\trt f(u',x')^4du'd\mu_{u'}\\
\nn&\lesssim & \sum_{j\geq 0}\left|\int_{-2}^u\int_{P_{u'}}P_j(\nabn f)(u',x')P_j(f^3)du'd\mu_{u'}\right|+\norm{\trt}_{\l{\infty}{4}}\norm{f}^4_{\l{4}{\frac{16}{3}}}\\
\nn&\lesssim & \left(\sup_{j\geq 0}\norm{P_j(\nabn f)}_{\ll{2}}\right)\left(\sum_{j\geq 0}\norm{P_j(f^3)}_{\ll{2}}\right)+\ep\norm{f}^2_{\l{\infty}{4}}\norm{f}_{\l{2}{8}}^2,
\eea
where we used the fact that $f(-2,.)=0$ and the estimate \eqref{thregx1} for $\trt$. \eqref{zoc51} yields:
$$\norm{f(u,.)}_{\lp{4}}^4\lesssim  \left(\sup_{j\geq 0}\norm{P_j(\nabn f)}_{\ll{2}}\right)^4+\left(\sum_{j\geq 0}\norm{P_j(f^3)}_{\ll{2}}\right)^{\frac{4}{3}}+\norm{\nabb f}^4_{\ll{2}},$$
which after taking the supremum in $u$ on the left-hand side implies:
\be\label{zoc52}
\norm{f}_{\l{\infty}{4}}\lesssim  \sup_{j\geq 0}\norm{P_j(\nabn f)}_{\ll{2}}+\left(\sum_{j\geq 0}\norm{P_j(f^3)}_{\ll{2}}\right)^{\frac{1}{3}}+\norm{\nabb f}_{\ll{2}}.
\ee

Next, we estimate the second term in the right-hand side of \eqref{zoc52}. We have:
\be\lab{zoc52bis}
\norm{P_j(f^3)}_{\ll{2}}\les \sum_{l, m, q}\norm{P_j(P_lf P_mf P_qf)}_{\ll{2}}.
\ee
We may assume:
$$l\geq m\geq q,$$
and we consider the following three cases:
$$q>j,\, q\leq j<l\textrm{ and }l\leq j.$$
We start with the case $q>j$. Then, using the strong Bernstein inequality for scalars \eqref{eq:strongbernscalar} for  $P_j$, we have:
\bea\lab{zoc53}
\norm{P_j(P_lf P_mf P_qf)}_{\ll{2}}&\les& 2^j\norm{P_lf P_mf P_qf}_{\l{2}{1}}\\
\nn &\les& 2^j\norm{P_lf}_{\ll{2}} \norm{P_mf}_{\l{\infty}{4}} \norm{P_qf}_{\l{\infty}{4}}\\
\nn &\les& 2^{j+\frac{m}{2}+\frac{q}{2}}\norm{P_lf}_{\ll{2}} \norm{P_mf}_{\l{\infty}{2}} \norm{P_qf}_{\l{\infty}{2}},
\eea
where we used Bernstein for $P_m$ and $P_q$ in the last inequality.

Next, we consider the case $q\leq j<l$. Then, the boundedness of $P_j$ on $\lp{2}$ yields:
\bea\lab{zoc54}
\norm{P_j(P_lf P_mf P_qf)}_{\ll{2}}&\les& \norm{P_lf P_mf P_qf}_{\ll{2}}\\
\nn &\les& \norm{P_lf}_{\ll{2}} \norm{P_mf}_{\ll{\infty}} \norm{P_qf}_{\ll{\infty}}\\
\nn &\les& 2^{m+q}\norm{P_lf}_{\ll{2}}\norm{P_mf}_{\l{\infty}{2}} \norm{P_qf}_{\l{\infty}{2}},
\eea
where we used in the last inequality the strong Bernstein inequality for scalars \eqref{eq:strongbernscalar} for $P_m$ and $P_q$. 

Finally, we consider the case $l\leq j$. Using the finite band property for $P_j$, we have:
\bea\lab{zoc55}
&& P_j(P_lf P_mf P_qf)\\
\nn&=&2^{-2j}P_j(\lap(P_lf P_mf P_qf))\\
\nn&=& 2^{-2j}P_j(\lap(P_lf) P_mf P_qf)+2^{-2j}P_j(\nabb(P_lf)\nabb(P_mf) P_qf)+\textrm{ permutations of }(l, m, p).
\eea
Using the boundedness of $P_j$ on $\lp{2}$, we have:
\bea\lab{zoc56}
\norm{P_j(\lap(P_lf) P_mf P_qf)}_{\ll{2}}&\les& \norm{\lap(P_lf) P_mf P_qf}_{\ll{2}}\\
\nn&\les& \norm{\lap P_lf}_{\ll{2}}\norm{P_mf}_{\ll{\infty}}\norm{P_qf}_{\ll{\infty}}\\
\nn&\les& 2^{2l+m+q}\norm{P_lf}_{\ll{2}}\norm{P_mf}_{\l{\infty}{2}}\norm{P_qf}_{\l{\infty}{2}},
\eea
where we used in the last inequality the finite band property for $P_l$, and the strong Bernstein inequality for scalars \eqref{eq:strongbernscalar} for $P_m$ and $P_q$. Using again the boundedness of $P_j$ on $\lp{2}$, we have:
\bea\lab{zoc57}
\norm{P_j(\nabb(P_l)f \nabb(P_mf) P_qf)}_{\ll{2}}&\les& \norm{\nabb(P_lf) \nabb(P_mf) P_qf}_{\ll{2}}\\
\nn&\les& \norm{\nabb P_lf}_{\l{2}{4}}\norm{\nabb P_mf}_{\l{\infty}{4}}\norm{P_qf}_{\ll{\infty}}\\
\nn&\les& 2^{\frac{3l}{2}+\frac{3m}{2}+q}\norm{P_lf}_{\ll{2}}\norm{P_mf}_{\l{\infty}{2}}\norm{P_qf}_{\l{\infty}{2}},
\eea
where we used in the last inequality the Gagliardo-Nirenberg inequality \eqref{eq:GNirenberg}, the Bochner inequality for scalars \eqref{eq:Bochconseqbis}, the finite band property for $P_l$ and $P_m$, and the strong Bernstein inequality for scalars \eqref{eq:strongbernscalar} for $P_q$. Now, since we assumed that $l\geq m$, \eqref{zoc55}, \eqref{zoc56} and \eqref{zoc57} imply:
\be\lab{zoc58}
\norm{P_j(P_lf P_mf P_qf)}_{\ll{2}}\les 2^{-2j+2l+m+q}\norm{P_lf}_{\ll{2}}\norm{P_mf}_{\l{\infty}{2}}\norm{P_qf}_{\l{\infty}{2}}.
\ee

Finally, in view of \eqref{zoc52bis}, \eqref{zoc53}, \eqref{zoc54} and \eqref{zoc58}, and since we assumed that $l\geq m\geq q$, we obtain:
\bee
&&\norm{P_j(f^3)}_{\ll{2}}\\
\nn&\les& \sum_{l, m, q}2^{-\frac{|j-l|}{6}-\frac{|j-m|}{6}-\frac{|j-q|}{6}}(2^l\norm{P_lf}_{\ll{2}})(2^{\frac{m}{2}}\norm{P_mf}_{\l{\infty}{2}})(2^{\frac{q}{2}}\norm{P_qf}_{\l{\infty}{2}}).
\eee
This yields:
\bea\lab{zoc59}
&&\sum_{j\geq 0}\norm{P_j(f^3)}_{\ll{2}}\\
\nn&\les& \sum_{j\geq 0}\left(\sum_{l\geq 0}2^{-\frac{|j-l|}{6}}2^{2l}\norm{P_lf}_{\ll{2}}^2\right)^{\frac{1}{2}}\left(\sum_{m\geq 0}2^{-\frac{|j-m|}{6}}2^{2m}\norm{P_mf}_{\l{\infty}{2}}^4\right)^{\frac{1}{4}}\\
\nn&&\times\left(\sum_{q\geq 0}2^{-\frac{|j-q|}{6}}2^{2q}\norm{P_qf}_{\l{\infty}{2}}^4\right)^{\frac{1}{4}}\\
\nn&\les & \left(\sum_{j\geq 0}2^{2j}\norm{P_jf}_{\ll{2}}^2\right)^{\frac{3}{2}}+\left(\sum_{j\geq 0}2^{2j}\norm{P_jf}_{\l{\infty}{2}}^4\right)^{\frac{3}{4}}\\
\nn&\les & \norm{\nabb f}_{\ll{2}}\left(\sum_{j\geq 0}2^{2j}\norm{P_jf}_{\l{\infty}{2}}^4\right)^{\frac{1}{2}},
\eea
where we used in the last inequality the Bessel inequality. Now, \eqref{zoc52} and \eqref{zoc59} imply:
$$\norm{f}_{\l{\infty}{4}}\lesssim  \sup_{j\geq 0}\norm{P_j(\nabn f)}_{\ll{2}}+\norm{\nabb f}_{\ll{2}}^{\frac{1}{3}}\left(\sum_{j\geq 0}2^{2j}\norm{P_jf}_{\l{\infty}{2}}^4\right)^{\frac{1}{6}}+\norm{\nabb f}_{\ll{2}},$$
which together with Lemma \ref{lemma:zoc2} yields \eqref{zoc42}. This concludes the proof of  Lemma \ref{lemma:zoc3}. 

\section{Regularity of the foliation with respect to $\o$}\label{regomega}

Let $u(x,\o)$ the function constructed in section \ref{regx}. In this section, we prove Theorem \ref{thregomega} which deals with the control of the derivatives with respect to $\o$ of the foliation generated by $u(x,\o)$ 
on $\s$. Recall that $(\s,g,k)$ coincides with $(\R^3,\de,0)$ in $|x|\geq 2$. Also, $u(x,\o)$ coincides with $x.\o$ in $|x|\geq 2$, and so $a\equiv 1$, $N\equiv\o$ and $\th\equiv 0$ in this region. Thus, $u$ clearly satisfies the estimates of Theorem \ref{thregomega} in $|x|\geq 2$ and it is enough to control the derivatives with respect to $\o$ of the function $u(x,\o)$ solution to: 
\begin{equation}\label{choice6}
\left\{\begin{array}{l}
\trt-k_{NN}=1-a,\textrm{ on }-2<u<2,\\
u(.,\o)=-2\textrm{ on }x.\o=-2,
\end{array}\right.
\end{equation}
in the strip $S=\{x/\,-2<u<2\}$.

To $u(x,\o)$, we associate the quantities $N$, $a$, $\th$ and $K$ as in section \ref{sec:foliation}.
We will have to differentiate these quantities several times with respect to $\o$. Since $a$ and $K$ 
(resp. $N$) are scalars (resp. is a vectorfield) defined on $-2<u<2$,  the meaning 
of $\po N$, $\po a$ and $\po K$ is clear. On the other hand, $\th$ is a 2-tensor on $\p$, and we need 
to extend it to a 2-tensor on $-2<u<2$ for $\po\th$ to be properly defined. We choose the trivial extension:
\begin{equation}\label{extth}
\th(N,.)=\th(.,N)\equiv 0,
\end{equation} 
so that $\th$ is a symmetric 2-tensor on $-2<u<2$. For consistency, we extend its traceless part $\hth$ in 
the same way:
\begin{equation}\label{extth1}
\hth(N,.)=\hth(.,N)\equiv 0,
\end{equation} 
so that $\hth$ is a symmetric 2-tensor on $-2<u<2$ satisfying:
\begin{equation}\label{extth2}
\hth(X,Y)=\th(X,Y)-\frac{1}{2}\trt (X.Y-(X.N)(Y.N)),
\end{equation} 
where $X, Y$ are two vectorfields on $\s$.

\subsection{First order derivatives with respect to $\o$}\lab{sec:regomega1} 

The goal of this section is to prove \eqref{threomega1}. We first give an outline of the proof. Differentiating the second equation of \eqref{struct1} with respect to $\o$, we obtain:
\begin{equation}\label{diffom1}
(\nabn-a^{-1}\lap)\po a = 2\nabb\nabn a+2R_{N \po N}+\cdots
\end{equation}
where the first term on the right-hand side comes from the commutator $[\po,\lap]$ (see \eqref{commutom3}). 
Since $\nabb\nabn a$ and $R$ are in $\ll{2}$ respectively by \eqref{boot} and \eqref{small2}, this suggests in view of Proposition \ref{p7} that: 
\begin{equation}\label{diffom2}
\norm{\nabn\po a}_{\ll{2}}+\norm{\nabb\po a}_{\l{\infty}{2}}+\norm{\nabb^2\po a}_{\ll{2}} \lesssim \ep.
\end{equation}
Next, we differentiate \eqref{diffom1} with respect to $\nabn$. We obtain:
\begin{equation}\label{diffom3}
(\nabn-a^{-1}\lap)\nabn\po a = 2\nabb\nabn^2a+2\nabn R_{N \po N}+\cdots.
\end{equation} 
The term $\nabn R_{N \po N}$ may be treated using the contracted Bianchi identity for $R$ - as we did for $\nabn R_{NN}$ in section \ref{thregx} - and turns out to be in $\lhs{2}{-1-\delta}$. On the other hand, in view of the estimate \eqref{nabn2a1} for $\nabn^2a$, $\nabb\nabn^2a$ belongs to $\lhs{2}{-\frac{3}{2}}$. This suggest in view of Proposition \ref{prop:parab3} that:
\begin{equation}\label{diffom4bis}
\norm{\nabn\po a}_{\lhs{2}{\frac{1}{2}}}+\norm{\nabn^2\po a}_{\lhs{2}{-\frac{3}{2}}} \lesssim \ep.
\end{equation}
By interpolation between \eqref{diffom2} and \eqref{diffom4bis}, we should obtain $\po a$ in $\lhs{\infty}{\frac{5}{4}}$ which embeds in $\ll{\infty}$ since $\p$ has dimension 2.
 
We now turn to the estimates for $\po\th$. Since $\trt=a-1+k_{NN}$, we differentiate in $\o$, and we easily obtain from the assumption on $k$ \eqref{small2}  and the estimate \eqref{diffom2} that $\nabla\po\trt\in\ll{2}$. To obtain estimates for $\po\th$, we differentiate the last two equations of 
\eqref{struct1} with respect to $\o$:
\be\lab{diffom8}
\left\{\begin{array}{l}
\nabb^B\po\hth_{AB}=\frac{1}{2}\nabb_A\po\trt+\cdots,\\[1mm]
\nabn\po\th_{AB}=-\nabb\nabn a-\nabb_A\nabb_B\po a+\cdots,
\end{array}\right.
\ee
where the first term on the right-hand side of the second equation comes from the commutator $[\po,\nabb^2]$ (see \eqref{commutom2}). Using the fact that $\nabla\po\trt\in\ll{2}$, $\nabb\nabn a\in\ll{2}$ 
and $\nabb^2\po a\in\ll{2}$, we then obtain $\nabla\po\th\in\ll{2}$.

Finally, we turn to the estimates for $\po N$. Differentiating \eqref{frame1} with respect to $\o$, we obtain:
\be\lab{diffom9}
\left\{\begin{array}{l}
\nabb\po N=\po\th +\cdots,\\
\nabn\po N=-\nabb\po a+\cdots.
\end{array}\right.
\ee
Together with the fact that $\nabb\po\th$ and $\nabb^2\po a$ belong to $\ll{2}$, this implies that 
$\nabb^2\po N$ and $\nabb\nabn\po N$ belong to $\ll{2}$. Using Proposition \ref{p3}, we obtain that 
$\po N$ belongs to $\ll{\infty}$.

The rest of this section is as follows. We start by deriving commutator formulas for $[\po,\nabb]$, 
$[\po,\lap]$ and $[\po,\nabb^2]$. Then, we prove the estimates for $\po a$. We continue with 
the estimates for $\po\th$. And we conclude with the estimates for $\po N$.

\subsubsection{Commutator formulas}

We have the following commutator formulas:
\begin{lemma}\label{commutom}
Let $f$ a scalar on $\s$. We have:
\begin{equation}\label{commutom1}
[\po,\nabb]f=-\nabb_{\po N}f N-\nabn f \po N,
\end{equation}
\begin{equation}\label{commutom2}
\begin{array}{ll}
[\po,\nabb^2]f(e_A,e_B)= & -(\po N)_A\nabla^2f(N,e_B)-(\po N)_B\nabla^2f(N,e_A)\\
& -\po\th_{AB}\nabn f-\th_{AB}\nabb_{\po N}f. 
\end{array}
\end{equation}
and
\begin{equation}\label{commutom3}
[\po,\lap]f=-2\nabla^2f(N,\po N)-\po\trt\nabn f-\trt\nabb_{\po N}f. 
\end{equation}
\end{lemma}

\begin{proof}
Differentiating with respect to $\o$ the equality 
\begin{equation}\label{comom1}
\nabb f=\nabla f-\nabn f N,
\end{equation}
and using the fact that $\po$ commutes with $\nabla$ since $g$ is independent of $\o$, we obtain:
\begin{equation}\label{comom1bis}
[\po,\nabb]f=-\nabla_{\po N}f N-\nabn f \po N.
\end{equation}
Now, we have:
\begin{equation}\label{comom1ter}
g(\po N,N)=0,
\end{equation}
which follows from the differentiation of $g(N,N)=1$ with respect to $\o$. Thus, $\po N$ is tangent to $\p$ which implies that $\nabla_{\po N}f=\nabb_{\po N}f$. Together with \eqref{comom1bis}, this yields 
\eqref{commutom1}.

We now turn to the proof of \eqref{commutom2}. Differentiating \eqref{comom1} by $\nabb$, we obtain:
\begin{equation}\label{comom2}
\nabb^2f(e_A,e_B)=\nabla^2f(e_A,e_B)-\nabn f\th_{AB}.
\end{equation}
Let $\Pi$ denote the projection of vectorfields of $\s$ on vectorfields tangent to $\p$:
\begin{equation}\label{comom3}
\Pi X = X- (X.N) N.
\end{equation}
The commutator $[\po,\Pi]$ satisfies:
\begin{equation}\label{comom4}
[\po,\Pi]X = - (X.\po N) N-(X.N)\po N.
\end{equation}
For $X, Y$ two vectorfields on $\s$ independent of $\o$, we differentiate $\nabb^2f(\Pi X,\Pi Y)$ with respect to $\o$ using \eqref{comom2}, \eqref{comom4} and the fact that $\po$ commutes with $\nabla$:
\begin{equation}\label{comom5}
\begin{array}{l}
\ds\po(\nabb^2f(\Pi X,\Pi Y))=\nabla^2\po f(\Pi X,\Pi Y)- (X.\po N) \nabla^2f(N,\Pi Y)\\
\ds -(X.N)\nabla^2f(\po N,\Pi Y)- (Y.\po N) \nabla^2f(N,\Pi X)-(Y.N)\nabla^2f(\po N,\Pi X)\\
-\nabn\po f\th_{XY}-\nabb_{\po N}f\th_{XY}-\nabn f\po\th_{XY},
\end{array}
\end{equation}
where we have used the fact that $\th_{\Pi X\Pi Y}=\th_{XY}$ from \eqref{extth}. evaluating \eqref{comom5} at $X=e_A, Y=e_B$ yields \eqref{commutom2}. Finally, taking the trace of \eqref{commutom2} 
yields \eqref{commutom3}.
\end{proof}

\begin{lemma}
Let $\rho$ a symmetric 2-tensor on $\s$ such that $\rho(N,.)\equiv 0$. Then, we have:
\begin{equation}\label{commutom4}
\begin{array}{ll}
\ds ([\po,\divb\,]\rho)_A= & \ds -\trt\rho_{\po N A}-\th_{AB}\rho_{B\po N}-
\nabla_N\rho_{\po N B}+\th_{\po N C}\rho_{CA}\\
& \ds +(\po N)_A\th_{BC}\rho_{CB}.
\end{array}
\end{equation}
\end{lemma}

\begin{proof}
For any symmetric 2-tensor $\nu$ on $\s$, we have:
\begin{equation}\label{comom7}
\begin{array}{l}
\nabb_C\nu_{AB}=e_C(\nu_{AB})-\nu(\nabb_Ce_A,e_B)-\nu(e_A,\nabb_C e_B)\\
=e_C(\nu_{AB})-\nu(\nabla_Ce_A-g(\nabla_Ce_A,N)N,e_B)-\nu(e_A,\nabla e_B-g(\nabla_Ce_B,N)N)\\
=e_C(\nu_{AB})-\nu(\nabla_Ce_A,e_B)-\nu(e_A,\nabla_Ce_B)-\th_{AC}\nu_{NB}-\th_{BC}\nu_{NA}\\
=\nabla_C\nu_{AB}-\th_{AC}\nu_{NB}-\th_{BC}\nu_{NA}.
\end{array}
\end{equation}
Applying \eqref{comom7} to $\rho$ and using the fact that $\rho(N,.)=\rho(.,N)\equiv 0$, we obtain:
\begin{equation}\label{comom8}
\nabb_C\rho_{AB}=\nabla_C\rho_{AB}.
\end{equation}
Let $X$ a vectorfield on $\s$ independent of $\o$. Using \eqref{comom4} and \eqref{comom8}, we have:
\begin{equation}\label{comom9}
\begin{array}{lll}
\ds\po((\divb\rho)_{\Pi X}) & = & \ds\po(\nabb_A\rho_{A\Pi X})=\po(\nabla_A\rho_{A \Pi X})\\
& \ds = & \ds\nabla_{\po e_A}\rho_{A \Pi X}+\nabla_A\rho_{\po e_A \Pi X}+\nabla_A\po\rho_{A \Pi X}
-(X.N)\nabla_A\rho_{A\po N}\\
& & \ds -(X.\po N)\nabla_A\rho_{AN}.
\end{array}
\end{equation}
Now, differentiating $g(e_A,e_B)=\delta_{AB}$ and $g(e_A,N)=0$ with respect to $\o$, we obtain:
\begin{equation}\label{om20}
\left\{\begin{array}{l}
\po e_1=g(\po e_1,e_2)e_2-g(\po N,e_1)N,\\
\po e_2=-g(\po e_1,e_2)e_1-g(\po N,e_2)N.
\end{array}\right.
\end{equation}
This yields
\begin{equation}\label{comom10}
\nabla_{\po e_A}\rho_{A \Pi X}+\nabla_A\rho_{\po e_A \Pi X}=-\nabla_N\rho_{\po N \Pi X}-\nabla_{\po N}\rho_{N \Pi X}.
\end{equation}
Since $\rho(N,.)=\rho(.,N)\equiv 0$, we have:
\begin{equation}\label{comom11}
\nabla_A\rho_{N B}=e_A(\rho_{NB})-\rho_{\nabla_AN B}-\rho_{N \nabla_Ae_B}=-\th_{AC}\rho_{CB}.
\end{equation}
Using again $\rho(N,.)=\rho(.,N)\equiv 0$, we have:
\begin{equation}\label{comom11bis}
\po\rho(N,.)=-\rho(\po N,.),
\end{equation}
which together with \eqref{comom7} applied to $\po\rho$ yields:
\begin{equation}\label{comom12}
\begin{array}{ll}
\ds \nabla_A\po\rho_{AB} & =\nabb_A\po\rho_{AB}+\th_{AA}\po\rho_{NB}+\th_{AB}\po\rho_{AN}\\
& \ds =(\divb\po\rho)_{B}-\trt\rho_{\po N B}-\th_{AB}\rho_{\po N A}.
\end{array}
\end{equation}
Finally, \eqref{comom9}, \eqref{comom10}, \eqref{comom11} and \eqref{comom12} yield:
\begin{equation}\label{comom13}
\begin{array}{ll}
\ds\po((\divb\rho)_{\Pi X})= & \ds (\divb\po\rho)_{\Pi X}-\trt\rho_{\po N \Pi X}-\th_{B \Pi X}\rho_{B \po N}-
\nabla_N\rho_{\po N \Pi X}\\
& \ds +\th_{\po N B}\rho_{B \Pi X}-(X.N)\nabla_B\rho_{B\po N}+(X.\po N)\th_{BC}\rho_{CB}.
\end{array}
\end{equation}
Taking $X=e_A$ in \eqref{comom13} yields \eqref{commutom4}.
\end{proof}

We conclude this section by recalling the link between $\nabb_A\nabn f$ and $\nabla^2f(e_A,N)$ for a 
scalar function $f$:
\begin{equation}\label{comom6}
\nabb_A\nabn f=\nabla^2f(e_A,N)+\th(e_A,\nabb f).
\end{equation}
  
\subsubsection{Estimates for $\po a$}

Note that the first equation of \eqref{struct1}, \eqref{init} and the fact that $(g,k,\s)$ coincides with $(\delta,0,\R^3)$ for $|x|\geq 2$ yields:
\begin{equation}\label{initom}
\nabla^p(\po a)=0,\,\nabla^p\po\theta=0,\,\nabla^p(\po N-\po\o)=0\textrm{ for all }p\in\N\textrm{ on }u=-2,
\end{equation} 
so that integrations by parts will not create boundary terms at $u=-2$. 

Differentiating the second equation of \eqref{struct1} with respect to $\o$, and using the commutator formula \eqref{commutom1}, \eqref{commutom3}, the fact that $\po N$ is tangent to $\p$ by \eqref{comom1ter}, and 
\eqref{comom6}, we obtain:
\begin{equation}\label{om1}
\nabn\po a -a^{-1}\lap\po a = h,
\end{equation}
where $h$ is given by:
\begin{equation}\label{om2}
\begin{array}{ll}
h= &-\nabb_{\po N}a-a^{-2}\po a\lap a -2\nabb_{\po N}\nabn a+2\th(\po N,\nabb a)\\
& -\po\trt\nabn a-\trt\nabb_{\po N} a +2\th\po\th+\po(\nabn(k_{NN}))+2R_{N\po N}.
\end{array}
\end{equation}
Using \eqref{frame1}, we have:
\begin{equation}\label{om3}
\nabn (k_{NN})=\nabn k_{NN}- 2k(\nabb a,N),
\end{equation}
which together with \eqref{commutom1} yields:
\begin{equation}\label{om4}
\begin{array}{ll}
\ds\po(\nabn (k_{NN}))= & \nabla_{\po N}k_{NN}+2\nabn k_{N \po N}- 2k(\nabb a,\po N)\\
& -\ds 2k(\nabb\po a,N)+ 2\nabb_{\po N} ak_{NN}+ 2\nabn ak_{N\po N}.
\end{array}
\end{equation} 
Using \eqref{om2} and \eqref{om4}, we estimate the norm of $h$ in $\ll{2}$:
\begin{equation}\label{om5}
\begin{array}{l}
\ds\norm{h}_{\ll{2}}\lesssim\norm{\nabb a}_{\ll{2}}\norm{\po N}_{\ll{\infty}}+\norm{a^{-2}}_{L^\infty}\norm{\lap a}_{\ll{2}}\norm{\po a}_{L^\infty} \\
\ds +\norm{\nabb\nabn a}_{\ll{2}}\norm{\po N}_{\ll{\infty}}+\norm{\th}_{\ll{4}}\norm{\nabb a}_{\ll{4}}\norm{\po N}_{\ll{\infty}} \\
\ds +\norm{\po\trt}_{\ll{4}}\norm{\nabn a}_{\ll{4}}+\norm{\trt}_{\ll{4}}\norm{\nabb a}_{\ll{4}}\norm{\po N}_{\ll{\infty}}\\
\ds \times\norm{\po N}_{\ll{\infty}}+\norm{\th}_{\ll{4}}\norm{\po\th}_{\ll{4}}+\norm{\nabla k}_{\ll{2}}\norm{\po N}_{\ll{\infty}}\\
\ds +\norm{k}_{\ll{4}}\norm{\nabla a}_{\ll{2}}\norm{\po N}_{\ll{\infty}}+\norm{k}_{\ll{4}}\norm{\nabb\po a}_{\ll{4}}\\
\ds +\norm{R}_{\ll{2}}\norm{\po N}_{\ll{\infty}}.
\end{array}
\end{equation}
Together with \eqref{small2}, \eqref{boot}, \eqref{appboot}, \eqref{appboot1} and \eqref{appsmall2}, this yields:
\begin{equation}\label{om6}
\ds\norm{h}_{\ll{2}}\lesssim\ep\norm{\po N}_{\ll{\infty}}+\ep\norm{\po a}_{\ll{\infty}}+\ep\norm{\nabb\po a}_{\ll{4}}+\ep\norm{\po\th}_{\ll{4}}.
\end{equation}
Proposition \ref{p7}, \eqref{initom}, \eqref{om1} and \eqref{om6} yield:
\begin{equation}\label{om15}
\begin{array}{r}
\ds\norm{\po a}_{\l{\infty}{2}}+\norm{\nabla\po a}_{\ll{2}}+\norm{\nabb\po a}_{\l{\infty}{2}}\\
\ds +\norm{\nabb^2\po a}_{\ll{2}}\lesssim \ep(\norm{\po a}_{L^\infty}+\norm{\po N}_{\ll{\infty}}+\norm{\po\th}_{\ll{4}}).
\end{array}
\end{equation}

Next, we differentiate equation \eqref{om1} by $\nabn$. We obtain:
\begin{equation}\label{zoo}
\nabn(\nabn\po a) -a^{-1}\lap(\nabn\po a) = [\nabn, a^{-1}\lap](\po a)+\nabn(h),
\end{equation}
where $h$ is given by \eqref{om2}. Next, we estimate each term in the right-hand side of \eqref{zoo} starting with the first one. In view of the commutator formula \eqref{commut1}, we have:
\be\label{zoo1}
a[\nabn,a^{-1}\lap]\po a=h_1+2a^{-1}\nabb a\nabb\nabn \po a+a^{-1}\lap a\nabn\po a,
\ee
where the scalar function $h_1$ is given by:
$$h_1=-(\trt+a^{-1}\nabn a)\lap \po a-2\hth\c\nabb^2\po a-2R_{N.}\c\nabb \po a-\nabb\trt\c\nabb \po a+2\hth\c a^{-1}\nabb a\c\nabb \po a.$$
$h_1$ satisfies the following estimate:
\bee
\norm{h_1}_{\l{2}{1}}&\les& (\norm{\th}_{\l{\infty}{2}}+\norm{a^{-1}\nabn a}_{\l{\infty}{2}})\norm{\nabb^2\po a}_{\ll{2}}\\
\nn&&+(\norm{R}_{\ll{2}}+\norm{\nabb\trt}_{\ll{2}}+\norm{\hth a^{-1}\nabb a}_{\ll{2}})\norm{\nabb\po a}_{\l{\infty}{2}}
\eee
Together with the estimates \eqref{om15} for $\po a$, \eqref{small2} for $R$, \eqref{appboot} for $a$, and \eqref{appboot1} for $\th$, we obtain:
\be\lab{zoo2}
\norm{h_1}_{\l{2}{1}}\les \ep(\norm{\po a}_{\ll{\infty}}+\norm{\po N}_{\ll{\infty}}+\norm{\po\th}_{\ll{4}}).
\ee
Also, the second term in \eqref{zoo1} satisfies in view of the product estimate \eqref{prod1}:
$$\sup_{j\geq 0}2^{-j}\norm{P_j(a^{-1}\nabb a\nabb\nabn \po a)}_{\ll{2}}\les  \norm{\nabb\nabn \po a}_{\lhs{2}{-\frac{1}{2}}}\norm{a^{-1}\nabb a}_{\lhs{\infty}{\frac{1}{2}}},$$
which together with the Lemma \ref{lemma:vacances}, the embedding \eqref{clp12}, and the estimate \eqref{boot} for $a$ yield:
\be\lab{zoo2bis}
\sup_{j\geq 0}2^{-j}\norm{P_j(a^{-1}\nabb a\nabb\nabn \po a)}_{\ll{2}}\les \ep\norm{\nabn \po a}_{\lhs{2}{\frac{1}{2}}}.
\ee
The third term in the right-hand side of \eqref{zoo1} satisfies in view of the product estimate \eqref{prod1}:
\bee
&&\sup_{j\geq 0}2^{-j}\norm{P_j(a^{-1}\lap a\nabn \po a)}_{\ll{2}}\\
&\les&  \norm{\nabn \po a}_{\lhs{2}{\frac{1}{2}}}\norm{a^{-1}\lap a}_{\lhs{\infty}{-\frac{1}{2}}}\\
&\les&  \norm{\nabn \po a}_{\lhs{2}{\frac{1}{2}}}(\norm{\divb(a^{-1}\nabb a)}_{\lhs{\infty}{-\frac{1}{2}}}+\norm{a^{-2}|\nabb a|^2}_{\lhs{\infty}{-\frac{1}{2}}})
\eee
which together with the Lemma \ref{lemma:vacances}, the embedding \eqref{clp12}, and the estimate \eqref{boot} for $a$ yield:
\be\lab{zoo2ter}
\sup_{j\geq 0}2^{-j}\norm{P_j(a^{-1}\lap a\nabn \po a)}_{\ll{2}}\les \ep\norm{\nabn \po a}_{\lhs{2}{\frac{1}{2}}}.
\ee

Next, we estimate the second term in \eqref{zoo}, i.e. $\nabn(h)$. In view of \eqref{om2}, we have:
\be\lab{zoo3}
\nabn(h)= h_2+h_3-2a^{-1}\divb(a\po N\nabn^2 a)+\nabn(\po(\nabn (k_{NN})))+2\nabn R_{N\po N},
\ee
where $h_2$ and $h_3$ are given respectively by:
\bee
h_2&=& a^{-2}\nabb\po a\c\nabb\nabn a+a^{-3}\po a\nabb a\c\nabb\nabn a+a^{-2}\po a[\nabn, \lap]a-2a^{-3}\nabn a\po a\lap a \\
&&+2a^{-3}|\nabb a|^2\nabn\po a+2\nabn\th(\po N, \nabb a)+2\th(\nabn\po N, \nabb a)+\th(\po N, \nabn\nabb a)\\
&&-\nabn(\po\trt)\nabn a-\nabn\trt\nabb_{\po N}( a)-\trt\nabn\nabb_{\po N} a+2\nabn\nabb a\c \nabb\po a\\
&&-2\nabn\nabb_{\po N} a\nabn a+2\nabn\th\c\po\th+2\th\nabn\po\th-2R_{\ana \po N}+2R_{N \nabn\po N},
\eee
and:
\bee
h_3&=& a^{-1}\divb(a^{-1}\po a\nabb\nabn a)+a^{-2}\nabb a\c\nabb\nabn\po a+a^{-2}\nabn(\po a)\lap a+2\divb(\po N)\nabn^2 a\\
&&-6a^{-1}\nabb_{\po N}a\nabn^2 a-\po\trt\nabn^2 a+a\nabb a\c\nabn\nabb(\po a).
\eee
Let us first estimate $h_2$ and $h_3$. In view of the definition of $h_2$, we have:
\bee
&&\norm{h_2}_{\l{2}{1}}\\
&\les & \norm{a^{-2}}_{L^\infty}(\norm{\nabb\po a}_{\l{\infty}{2}}+\norm{\po a}_{\l{\infty}{4}}\norm{\nabb a}_{\l{\infty}{4}})\norm{\nabb\nabla a}_{\ll{2}}\\
&&+\norm{a^{-3}\po a}_{\ll{\infty}}(\norm{[\nabn, \lap]a}_{\l{2}{1}}+\norm{\nabn a}_{\ll{\infty}{2}}\norm{\lap a}_{\ll{2}})\\
&&+\norm{a^{-3}}_{L^\infty}\norm{\nabb a}^2_{\l{\infty}{4}}\norm{\nabn\po a}_{\ll{2}}\\
&&+\norm{\nabn\th}_{\ll{2}}\norm{\po N\nabb a}_{\l{\infty}{2}}+\norm{\th}_{\l{\infty}{2}}\norm{\nabn\po N}_{\ll{4}}\norm{\nabb a)}_{\ll{4}}\\
&&+\norm{\th\po N}_{\l{\infty}{2}}\norm{\nabn\nabb a}_{\ll{2}}+\norm{\nabn(\po\trt)}_{\ll{2}}\norm{\nabn a}_{\l{\infty}{2}}\\
&&+\norm{\nabn\trt}_{\ll{2}}\norm{\nabb_{\po N}( a)}_{\l{\infty}{2}}+\norm{\trt}_{\l{\infty}{2}}\norm{\nabn\nabb_{\po N} a}_{\ll{2}}\\
&&+\norm{\nabn\nabb a}_{\ll{2}}\norm{\nabb\po a}_{\l{\infty}{2}}+\norm{\nabn\nabb_{\po N} a}_{\ll{2}}\norm{\nabn a}_{\l{\infty}{2}}\\
&&+\norm{\nabn\th}_{\ll{2}}\norm{\po\th}_{\l{\infty}{2}}+\norm{\th}_{\l{\infty}{2}}\norm{\nabn\po\th}_{\ll{2}}\\
&&+\norm{R}_{\ll{2}}\norm{\ana}_{\l{\infty}{2}}\norm{\po N}_{\ll{\infty}}+\norm{R}_{\ll{2}}\norm{\nabn\po N}_{\l{\infty}{2}},
\eee
which together with the commutator formula \eqref{ad12} for $[\nabn, \lap]$, the estimates \eqref{boot} \eqref{appboot} for $a$, \eqref{boot1} \eqref{appboot1} for $\th$, the estimate \eqref{small2} for $R$, and the estimate \eqref{om15} for $\po a$ yields:
\bea\lab{zoo4}
\nn\norm{h_2}_{\l{2}{1}}&\les& \ep(\norm{\po a}_{\ll{\infty}}+\norm{\nabn\po N}_{\ll{4}}+\norm{\po N}_{\ll{\infty}}+\norm{\po\th}_{\ll{4}}\\
&&+\norm{\nabn\po\th}_{\ll{2}}+\norm{\po\th}_{\l{\infty}{2}}).
\eea
Also, in view of the definition of $h_3$ and the product estimate \eqref{prod1}, the finite band property for $P_j$ and the estimates \eqref{zoo2bis} and \eqref{zoo2ter}, we have:
\bee
&&\sup_{j\geq 0}2^{-j}\norm{P_j(ah_3)}_{\lp{2}}\\
&\les & \norm{a^{-1}\po a}_{\ll{\infty}}\norm{\nabb\nabn a}_{\ll{2}}+\ep\norm{\nabn\po a}_{\lhs{2}{\frac{1}{2}}}+(\norm{\divb(\po N)}_{\lhs{\infty}{\frac{1}{2}}}\\
&&+\norm{\nabb_{\po N}a}_{\lhs{\infty}{\frac{1}{2}}}+\norm{a\po\trt}_{\lhs{\infty}{\frac{1}{2}}})\norm{\nabn^2 a}_{\lhs{2}{-\frac{1}{2}}},
\eee
which together with the estimate \eqref{nabn2a1} for $\nabn a$ and $\nabn^2a$, the estimates  \eqref{boot} \eqref{appboot} for $a$, the estimate \eqref{om15} for $\po a$, the embedding \eqref{clp12} and the product estimate \eqref{prod15} yields:
\bea\lab{zoo5}
\sup_{j\geq 0}2^{-j}\norm{P_j(ah_3)}_{\lp{2}}&\les& \ep(\norm{\po a}_{\ll{\infty}}+\norm{\nabb\po N}_{\l{\infty}{2}}+\norm{\po N}_{\ll{\infty}}\\
\nn&&+\norm{\divb(\po N)}_{\lhs{\infty}{\frac{1}{2}}}+\norm{\po\th}_{\ll{4}}+\norm{\nabla\po\trt}_{\ll{2}}).
\eea

Next, we estimate the third term in \eqref{zoo3}. In view of the product estimate \eqref{prod6} with $b=\frac{1}{2}$, $f=\nabn^2 a$ and $G=a\po N$, we have:
\bee
&&\norm{\divb(a\po N\nabn^2 a)}_{\lhs{2}{-\frac{3}{2}}}\\
&\les& \norm{\nabn^2 a}_{\lhs{2}{-\frac{1}{2}}}(\norm{a\po N}_{\ll{\infty}}+\norm{\nabb (a\po N)}_{\l{\infty}{2}}).
\eee
Together with the estimate \eqref{nabn2a1} for $\nabn^2a$, the estimate \eqref{nabn2a1} for $\nabn a$, and the estimates \eqref{boot} \eqref{appboot} for $a$, this yields:
\be\lab{zoo6}
\norm{\divb(a\po N\nabn^2 a)}_{\lhs{2}{-\frac{3}{2}}}\les\ep(\norm{\po N}_{\ll{\infty}}+\norm{\nabb\po N}_{\l{\infty}{2}}).
\ee

Next, we estimate the third term in \eqref{zoo3}. We have:
\bee
&&\nabn(\po(\nabn(k_{NN})))\\
&=&\nabla_{\po N}(\nabn k_{NN})+[\nabn, \nabla_{\po N}](k_{NN})+2\nabn(\nabn(k_{N \po N}))\\
&=&\divb(\po N \nabn (k_{NN}))-\divb(\po N)\nabn(k_{NN})+\nabla_{\nabn\po N}(k_{NN})+a^{-1}\nabb_{\po N}a\nabn (k_{NN})\\
&&-\th(\po N,A)\nabb_A(k_{NN})+2\nabn(\nabn k_{N \po N}-k(\ana, \po N)+k(N, \nabn\po N))\\
&=&\divb(\po N (\nabn k_{NN}-2k(\ana, N)))-\divb(\po N)(\nabn k_{NN}-2k(\ana, N))\\
&&+\nabla_{\nabn\po N}(k_{NN})+a^{-1}\nabb_{\po N}a(\nabn k_{NN}-2k(\ana, N))\\
&&-\th(\po N,A)(\nabla_Ak_{NN}+2\th_{AB}k_{BN})+2\nabn(\nabn k_{N \po N})-2\nabn k(\ana, \po N)\\
&& -2k(\nabn(\ana), \po N) -4k(\ana, \nabn\po N)+2\nabn k(N, \nabn\po N)\\
&&+2k(N, \nabn\nabn\po N),
\eee
where we used the structure equations \eqref{frame1} for $N$ and the commutator formula \eqref{scommut}. This yields:
\be\lab{zoo7}
\nabn(\po(\nabn (k_{NN})))= a^{-1}\divb(a\po N \nabn k_{NN})+2\nabn(\nabn k_{N \po N})+2k(N, \nabn\nabn\po N)+h_4,
\ee
with $h_4$ satisfying:
\bea\lab{zoo8}
&&\norm{h_4}_{\l{2}{1}}\\
\nn&\les& \Big(\norm{\nabla k}_{\ll{2}}(\norm{\ana}_{\l{\infty}{2}}+\norm{\th}_{\l{\infty}{2}})\\
\nn&&+\norm{k}_{\l{\infty}{2}}(\norm{\nabb\nabla a}_{\ll{2}}+\norm{\ana}^2_{\ll{4}}+\norm{\th}^2_{\ll{4}})\Big)\norm{\po N}_{\ll{\infty}}\\
\nn&&+\Big(\norm{\nabla k}_{\ll{2}}+\norm{k}_{\l{\infty}{4}}\norm{\ana}_{\l{\infty}{4}}\Big)\norm{\nabla\po N}_{\l{\infty}{2}})\\
\nn&\les& \ep(\norm{\po N}_{\ll{\infty}}+\norm{\nabla\po N}_{\l{\infty}{2}}),
\eea
where we used in the last inequality the estimate \eqref{small2} for $k$, the estimate \eqref{boot} for $a$ and the estimate \eqref{appboot1} for $\th$. Now, in view of the constraint equations \eqref{const1}, we have
$$\nabn k_{NA}=-\nabla_Bk_{BA}.$$
This yields:
\bee
\nabn k_{N \po N}&=&-\nabla_Bk_{B \po N}\\
&=&-\divb(k_{\po N .})+k_{AB}(\nabla_A\po N)_B-\trt k_{N \po N}.
\eee
Differentiating with respect to $\nabn$, we obtain:
\bee
&&\nabn(\nabn k_{N \po N})\\
&=&-\divb(\nabn (k_{\po N .}))-[\nabn, \divb](k_{\po N .})+\nabn k_{AB}(\nabla_A\po N)_B+k_{AB}(\nabn\nabla_A\po N)_B\\
&&-\nabn\trt k_{N \po N}-\trt \nabn k_{N \po N}+\trt k(\ana, \po N)-\trt k_{N \nabn\po N}\\
&=&-\divb(\nabn k_{\po N .})-(\ana\nabn+\th\nabb+R+\ana\c\th)\c k_{\po N .}\\
&&+\nabn k_{AB}(\nabla_A\po N)_B+k_{AB}([\nabn,\nabb]_A\po N)_B-\nabn\trt k_{N \po N}-\trt \nabn k_{N \po N}\\
&&+\trt k(\ana, \po N)-\trt k_{N \nabn\po N}\\
&=&-\divb(\nabn k_{\po N .})-(\ana\nabn+\th\nabb+R+\ana\c\th)\c k_{\po N .}\\
&&+\nabn k_{AB}(\nabla_A\po N)_B+k\c (\ana\nabn+\th\nabb+R+\ana\c\th)\c\po N\\
&&-\nabn\trt k_{N \po N}-\trt \nabn k_{N \po N}+\trt k(\ana, \po N)-\trt k_{N \nabn\po N},
\eee
where we used the commutator estimate \eqref{commut}. This yields:
\be\lab{zoo9}
\nabn(\nabn k_{N \po N})= -a^{-1}\divb(a\nabn k_{\po N .})+h_5,
\ee
with $h_5$ satisfying:
\bea\lab{zoo10}
&&\norm{h_5}_{\l{2}{1}}\\
\nn&\les& \norm{\nabla k}_{\ll{2}}((\norm{\ana}_{\l{\infty}{2}}+\norm{\th}_{\l{\infty}{2}})\norm{\po N}_{\ll{\infty}}+\norm{\nabla\po N}_{\l{\infty}{2}})\\
\nn&&+\norm{\nabla\po N}_{\l{\infty}{2}}\norm{k}_{\ll{4}}(\norm{\th}_{\ll{4}}+\norm{\ana}_{\ll{4}})\\
\nn&&+\norm{\nabn\th}_{\ll{2}}\norm{k}_{\l{\infty}{2}}\norm{\po N}_{\ll{\infty}}\\
\nn&\les& \ep(\norm{\po N}_{\ll{\infty}}+\norm{\nabla\po N}_{\l{\infty}{2}}),
\eea
where we used in the last inequality the estimate \eqref{small2} for $k$, the estimates \eqref{boot} \eqref{appboot} for $a$, and the estimates \eqref{boot1} \eqref{appboot1} for $\th$. 

Next, we estimate the fourth term in \eqref{zoo3}. Using the twice contracted Bianchi identities \eqref{bianchi}, we have
$$\nabn R_{AN}=-\nabla_BR_{AB}+k\c\nabla_Ak.$$
This yields:
$$\nabn R_{N \po N}=-\divb(R_{\po N .})+\trt R_{N \po N}-R_{AB}(\nabla_A\po N)_B+k \c\nabla_{\po N}k.$$
We obtain:
\be\lab{zoo11}
\nabn R_{N \po N}=-a^{-1}\divb(aR_{\po N .})+h_6,
\ee
with $h_6$ satisfying:
\bea\lab{zoo11bis}
\norm{h_6}_{\l{2}{1}}&\les& \norm{R}_{\ll{2}}(\norm{\nabb\po N}_{\l{\infty}{2}}+(\norm{\ana}_{\l{\infty}{2}}\\
\nn&&+\norm{\th}_{\l{\infty}{2}})\norm{\po N}_{\ll{\infty}})+\norm{\nabla k}_{\ll{2}}\norm{k}_{\l{\infty}{2}}\norm{\po N}_{\ll{\infty}}\\
\nn&\les& \ep(\norm{\po N}_{\ll{\infty}}+\norm{\nabb\po N}_{\l{\infty}{2}}),
\eea
where we used in the last inequality the estimate \eqref{small2} for $R$ and $k$, the estimate \eqref{appsmall2} for $k$, and the estimate \eqref{appboot1} for $\th$. 

Finally, \eqref{zoo3}, \eqref{zoo7}, \eqref{zoo9} and \eqref{zoo11} imply:
\bee
\nabn(h)&=& h_2+h_3+h_4+2h_5+2h_6-2a^{-1}\divb(a\po N\nabn^2 a)+2k(N, \nabn\nabn\po N)\\
&&+a^{-1}\divb(a\po N \nabn k_{NN})-2a^{-1}\divb(a\nabn k_{\po N .})-2a^{-1}\divb(aR_{\po N .}).
\eee
Together with \eqref{zoo} and \eqref{zoo1}, this implies:
\begin{equation}\label{zoo12}
\nabn(\nabn\po a) -a^{-1}\lap(\nabn\po a) = h_7,
\end{equation}
where $h_7$ is given by:
\bee
h_7&=& a^{-1}(h_1+2a^{-1}\nabb a\nabb\nabn \po a+a^{-1}\lap a\nabn\po a)+h_2+h_3+h_4+2h_5+2h_6\\
&&-2a^{-1}\divb(a\po N\nabn^2 a)+2k(N, \nabn\nabn\po N)+a^{-1}\divb(a\po N \nabn k_{NN})\\
&&-2a^{-1}\divb(a\nabn k_{\po N .})-2a^{-1}\divb(aR_{\po N .}).
\eee
Together with the strong Bernstein inequality for scalars \eqref{eq:strongbernscalarbis}, the finite band property for $P_j$, and the estimate \eqref{boot} for $a$, this yields:
\bee
&&\norm{P_j(ah_7)}_{\ll{2}}\\
&\les& 2^j\Big(\norm{h_1}_{\l{2}{1}}+\norm{h_2}_{\l{2}{1}}+\norm{h_4}_{\l{2}{1}}+\norm{h_5}_{\l{2}{1}}\\
&&+\norm{h_6}_{\l{2}{1}}\Big)+\norm{P_j(ah_3)}_{\ll{2}}+\norm{P_j(a^{-1}\nabb a\c\nabb\nabn \po a)}_{\ll{2}}\\
&&+\norm{P_j(a^{-1}\lap a\nabn \po a)}_{\ll{2}}+\norm{P_j(\divb(a\po N\nabn^2 a))}_{\ll{2}}\\
&&+\norm{P_j(\divb(ak(N, \nabn\nabn\po N)))}_{\ll{2}}+2^j(\norm{\nabn k}_{\ll{2}}+\norm{R}_{\ll{2}})\norm{a}_{\ll{\infty}}\norm{\po N}_{\ll{\infty}}.
\eee
Together with the estimates \eqref{om6}, \eqref{zoo2}, \eqref{zoo2bis}, \eqref{zoo2ter}, \eqref{zoo4}, \eqref{zoo8}, \eqref{zoo10} and \eqref{zoo11bis}, the estimate \eqref{boot} for $a$, the estimate \eqref{om15} for $\po a$, and the estimate \eqref{small2} for $R$ and $k$, we obtain:
\bee
&&\norm{P_j(ah_7)}_{\ll{2}}\\
&\les& \ep 2^j\Big(\norm{\po a}_{\ll{\infty}}+\norm{\po N}_{\ll{\infty}}+\norm{\nabn\po N}_{\ll{4}}+\norm{\nabla(\po N)}_{\l{\infty}{2}}\\
&&+\norm{\divb(\po N)}_{\lhs{\infty}{\frac{1}{2}}}+\norm{\po\th}_{\ll{4}}+\norm{\nabla\po\th}_{\ll{2}}+\norm{\po\th}_{\l{\infty}{2}}\\
&&+\norm{\nabn \po a}_{\lhs{2}{\frac{1}{2}}}\Big)+\norm{P_j(\divb(a\po N\nabn^2 a))}_{\ll{2}}+\norm{P_j(ak(N, \nabn\nabn\po N))}_{\ll{2}}.
\eee
Together with the estimate \eqref{zoo6}, we finally obtain:
\bea\lab{zoo13}
&&\norm{ah_7}_{\lhs{2}{-\frac{3}{2}}}\\
\nn&\les& \ep\Big(\norm{\po a}_{\ll{\infty}}+\norm{\po N}_{\ll{\infty}}+\norm{\nabn\po N}_{\ll{4}}+\norm{\nabla\po N}_{\l{\infty}{2}}\\
\nn&&+\norm{\divb(\po N)}_{\lhs{\infty}{\frac{1}{2}}}+\norm{\nabla\po\th}_{\ll{2}}+\norm{\nabn \po a}_{\lhs{2}{\frac{1}{2}}}\Big)\\
\nn&&+\norm{ak(N, \nabn\nabn\po N)}_{\lhs{2}{-\frac{3}{2}}}.
\eea

Now, in view of \eqref{zoo12} and Proposition \ref{prop:parab3}, we have:
$$\norm{\nabn\po a}_{\lhs{2}{\frac{1}{2}}}+\norm{\nabn\po a}_{\lhs{\infty}{-\frac{1}{2}}}+\norm{\nabn^2\po a}_{\lhs{2}{-\frac{3}{2}}}\les \norm{ah_7}_{\lhs{2}{-\frac{3}{2}}}.$$
Together with \eqref{zoo13}, this yields:
\bea\lab{xxxooo}
&&\norm{\nabn\po a}_{\lhs{2}{\frac{1}{2}}}+\norm{\nabn\po a}_{\lhs{\infty}{-\frac{1}{2}}}+\norm{\nabn^2\po a}_{\lhs{2}{-\frac{3}{2}}}\\
\nn&\les& \ep\Big(\norm{\po a}_{\ll{\infty}}+\norm{\po N}_{\ll{\infty}}+\norm{\nabn\po N}_{\ll{4}}+\norm{\nabla\po N}_{\l{\infty}{2}}\\
\nn&&+\norm{\divb(\po N)}_{\lhs{\infty}{\frac{1}{2}}}+\norm{\nabla\po\th}_{\ll{2}}+\norm{\nabn \po a}_{\lhs{2}{\frac{1}{2}}}\Big)\\
\nn&&+\norm{ak(N, \nabn\nabn\po N)}_{\lhs{2}{-\frac{3}{2}}}.
\eea
In view of Corollary \ref{cor:commLP4}, we have
\be\lab{xoox}
\norm{\po a}_{\ll{\infty}}\les \norm{\nabb^2\po a}_{\ll{2}}+\norm{\nabn\po a}_{\lhs{2}{\frac{1}{2}}}.
\ee
Thus, we finally obtain, in view of \eqref{xxxooo} and \eqref{xoox}:
\bea\lab{zoo14}
&&\norm{\nabn\po a}_{\lhs{2}{\frac{1}{2}}}+\norm{\nabn\po a}_{\lhs{\infty}{-\frac{1}{2}}}+\norm{\nabn^2\po a}_{\lhs{2}{-\frac{3}{2}}}\\
\nn&\les& \ep\Big(\norm{\po N}_{\ll{\infty}}+\norm{\nabn\po N}_{\ll{4}}+\norm{\nabla\po N}_{\l{\infty}{2}}\\
\nn&&+\norm{\divb(\po N)}_{\lhs{\infty}{\frac{1}{2}}}+\norm{\nabla\po\th}_{\ll{2}}\Big)+\norm{ak(N, \nabn\nabn\po N)}_{\lhs{2}{-\frac{3}{2}}}.
\eea

\subsubsection{Estimates for $\po\th$}

Let us start by showing that $\po\hth$ is traceless when seen as a tensor on $\p$. Differentiating 
\eqref{extth2} with respect to $\o$, we obtain:
\begin{equation}\label{om16}
\begin{array}{ll}
\po\hth(X,Y)= & \ds\po\th(X,Y)-\frac{1}{2}\po\trt (X.Y-(X.N)(Y.N))\\
& \ds+\frac{1}{2}\trt ((X.\po N)(Y.N)+(X.N)(Y.\po N)).
\end{array}
\end{equation}
which yields:
\begin{equation}\label{om17}
\po\hth_{AB}=\po\th_{AB}-\frac{1}{2}\po\trt\delta_{AB},
\end{equation}
so that:
\begin{equation}\label{om18}
\textrm{tr}(\po\hth)=\textrm{tr}(\po\th)-\po\trt.
\end{equation}
We compute $\po\trt$:
\begin{equation}\label{om19}
\po\trt=\po(\th_{AA})=\textrm{tr}(\po\th)+2\th(e_A,\po e_A).
\end{equation}
Together with \eqref{extth}, \eqref{om19} and \eqref{om20}, this yields
\begin{equation}\label{om21}
\po\trt=\textrm{tr}(\po\th).
\end{equation}
Finally, \eqref{om18} and \eqref{om21} imply that $\po\hth$ is traceless:
\begin{equation}\label{om22}
\textrm{tr}(\po\hth)=0.
\end{equation}

We now turn to the estimates for $\po\trt$. Differentiating the first equation of \eqref{struct1} with 
respect to $\o$, we obtain:
\begin{equation}\label{om23}
\po\trt=-\po a+2k_{N\po N},
\end{equation}
so that:
\begin{equation}\label{om24}
\begin{array}{ll}
\ds\nabla\po\trt= & \ds -\nabla\po a+2\nabla k(N,\po N)+2k(\nabla N,\po N)\\
& \ds +2k(N,\nabla\po N),
\end{array}
\end{equation}
which in turn yields:
\begin{equation}\label{om25}
\begin{array}{ll}
\ds\norm{\nabla\po\trt}_{\ll{2}}\lesssim &\norm{\nabla\po a}_{\ll{2}}+\norm{\nabla k}_{\ll{2}}\norm{\po N}_{\ll{\infty}}\\
&\ds+\norm{k}_{\ll{4}}\norm{\nabla N}_{\ll{4}}\norm{\po N}_{\ll{\infty}}+\norm{k}_{\ll{4}}\norm{\nabla\po N}_{\ll{4}},
\end{array}
\end{equation}
Together with Proposition \ref{p2}, \eqref{small2}, \eqref{appboot} and \eqref{appsmall2}, we obtain:
\begin{equation}\label{om26}
\ds\norm{\nabla\po\trt}_{\ll{2}}\lesssim \ds\norm{\nabla\po a}_{\ll{2}}+\ep\norm{\po N}_{\ll{\infty}}+\ep\norm{\nabla\po N}_{\ll{4}}.
\end{equation}

We now turn to the estimates for $\nabb\po\hth$. We differentiate the third equation of \eqref{struct1} with respect to $\o$. Using \eqref{commutom1}, and 
\eqref{commutom4}, we obtain:
\begin{equation}\label{om28}
(\divb\po\hth)_A = h,
\end{equation}
where $h$ is given by:
\begin{equation}\label{om29}
\begin{array}{ll}
\ds h = & \ds \trt\hth_{\po N A}+\th_{AB}\hth_{B\po N}+
\nabla_N\hth_{\po N A}-\th_{\po N C}\hth_{CA}\\
& \ds -(\po N)_A\th_{BC}\hth_{CB}+\frac{1}{2}\nabb_A\po\trt-\frac{1}{2}\nabn\trt(\po N)_A+R_{A \po N}\\
& \ds -(\po N)_AR_{NN}.
\end{array}
\end{equation}
Differentiating \eqref{extth2} with respect to $\nabn$, we obtain:
\begin{equation}\label{om27}
\nabn\hth_{AB}=\nabn\th_{AB}-\frac{1}{2}\nabn\trt\delta_{AB}
\end{equation}
Also, the definition of $\trt$ and $\hth$ implies:
\begin{equation}\label{om30}
\th_{AC}\hth_{CB}-\hth_{AC}\th_{CB}=0,
\end{equation}
which together with \eqref{om29} and \eqref{om27} yields:
\begin{equation}\label{om31}
\begin{array}{ll}
\ds h = & \ds \frac{1}{2}\nabb_A\po\trt+\nabn\th_{\po N B}-\frac{1}{2}\nabn\trt(\po N)_A+\trt\hth_{\po N A}-(\po N)_A\th_{BC}\hth_{CB}\\
& \ds +R_{A \po N}-(\po N)_AR_{NN}.
\end{array}
\end{equation}
We estimate the norm of $h$ in $\ll{2}$:
\begin{equation}\label{om32}
\begin{array}{ll}
\ds \norm{h}_{\ll{2}} \lesssim & \ds \norm{\nabb\po\trt}_{\ll{2}}+\norm{\nabn\th}_{\ll{2}}\norm{\po N}_{\ll{\infty}}+\norm{\th}^2_{\ll{4}}\norm{\po N}_{\ll{\infty}}\\
& \ds +\norm{R}_{\ll{2}}\norm{\po N}_{\ll{\infty}}.
\end{array}
\end{equation}
Together with \eqref{small2}, \eqref{boot1} and \eqref{appboot1}, this yields:
\begin{equation}\label{om33}
\norm{h}_{\ll{2}} \lesssim\norm{\nabb\po\trt}_{\ll{2}}+\ep\norm{\po N}_{\ll{\infty}}.
\end{equation}
Proposition \ref{p6}, \eqref{om22}, \eqref{om28} and \eqref{om33} imply:
\begin{equation}\label{om34}
\norm{\nabb\po\hth}_{\ll{2}}\lesssim\norm{\nabb\po\trt}_{\ll{2}}+\ep\norm{\po N}_{\ll{\infty}}+\norm{K}^{\frac{1}{2}}_{\ll{2}}\norm{\po\hth}_{\ll{4}}.
\end{equation}
Finally, Corollary \ref{c0}, \eqref{boot}, \eqref{initom}, \eqref{om26} and \eqref{om34} yield:
\begin{equation}\label{om35}
\begin{array}{ll}
\ds\norm{\nabb\po\hth}_{\ll{2}}\lesssim & \ds\ep\norm{\po a}_{\l{\infty}{2}}+\norm{\nabla\po a}_{\ll{2}}\\
& \ds +\ep\norm{\po N}_{\ll{\infty}}+\ep\norm{\nabla\po N}_{\ll{4}}+\ep^{\frac{1}{2}}\norm{\nabla\po\th}_{\ll{2}}.
\end{array}
\end{equation}

We now turn to the estimates for $\nabn\po\th$. Let $X, Y$ two vectorfields on $\s$ independent of $\o$. \eqref{extth} and the last equation of \eqref{struct1} imply:
\begin{equation}\label{om37}
a^{-1}\nabb_{\Pi X}\nabb_{\Pi Y} a+\nabn\th_{\Pi X \Pi Y}+\th_X^j\th_{j Y}+K g(\Pi X,\Pi Y)=R(\Pi X,\Pi Y).
\end{equation}
We differentiate \eqref{om37} with respect to $\o$. Using \eqref{commutom2},  
\eqref{comom4} and , and evaluating at $X=e_A, Y=e_B$, we obtain:
\begin{equation}\label{om38}
\nabn\po\th_{AB}=h,
\end{equation}
where $h$ is given by:
\begin{equation}\label{om39}
\begin{array}{ll}
\ds h= & \ds -a^{-1}\nabb_A\nabb_B\po a+(\po N)_Aa^{-1}\nabla^2 a(e_B,N)+(\po N)_Ba^{-1}\nabla^2 a(e_A,N)\\
& \ds +\po\th_{AB}a^{-1}\nabn a
+\th_{AB}a^{-1}\nabb_{\po N} a+a^{-2}\po a\nabb_A\nabb_B\po a-\nabla_{\po N}\th_{AB}\\
&\ds -(\po N)_A\nabn\th_{NB} -(\po N)_B\nabn\th_{NA}-\po\th_A^C\th_{CB}-\th_A^C\po\th_{CB}
-\po K\gamma_{AB}\\
&\ds -(\po N)_AR_{NB}-(\po N)_BR_{NA}.
\end{array}
\end{equation}
Using \eqref{frame1} and \eqref{extth}, we have:
\begin{equation}\label{om40}
\nabn\th_{NA}=\th(\nabb a,e_A).
\end{equation}
Using \eqref{comom7}, \eqref{comom6} and \eqref{om40}, we rewrite \eqref{om39} as: 
\begin{equation}\label{om41}
\begin{array}{ll}
\ds h= & \ds -a^{-1}\nabb_A\nabb_B\po a+(\po N)_Aa^{-1}\nabb_B\nabn a+(\po N)_Ba^{-1}\nabb_A\nabn a\\
& \ds +\po\th_{AB}a^{-1}\nabn a+\th_{AB}a^{-1}\nabb_{\po N} a+a^{-2}\po a\nabb_A\nabb_B\po a-\nabb_{\po N}\th_{AB}\\
&\ds -2(\po N)_A\th(\nabb a,e_B)-2(\po N)_B\th(\nabb a,e_A)-\po\th_A^C\th_{CB}-\th_A^C\po\th_{CB}\\
&\ds -\po K\gamma_{AB}-(\po N)_AR_{NB}-(\po N)_BR_{NA}.
\end{array}
\end{equation}
Thus, we estimate the norm of $h$ in $\ll{2}$ by:
\bea\label{om42}
\norm{h}_{\ll{2}}&\lesssim&\norm{a^{-1}}_{\ll{\infty}}\Big(\norm{\nabb^2\po a}_{\ll{2}}+\norm{\nabb\nabn a}_{\ll{2}}\norm{\po N}_{\ll{\infty}}\\
\nn&& +\norm{\po\th}_{\ll{4}}\norm{\nabn a}_{\ll{4}}+\norm{\th}_{\ll{4}}\norm{\nabb a}_{\ll{4}}\norm{\po N}_{\ll{\infty}}\Big)\\
\nn&& +\norm{a^{-2}}_{\ll{\infty}}\norm{\po a}_{\ll{\infty}}\norm{\nabb a}_{\ll{2}}+\norm{\nabb\th}_{\ll{2}}\norm{\po N}_{\ll{\infty}}\\
\nn&& +\norm{\th}_{\ll{4}}\norm{\nabb a}_{\ll{4}}\norm{\po N}_{\ll{\infty}}+\norm{\po\th}_{\ll{4}}\norm{\th}_{\ll{4}}+\norm{\po K}_{\ll{2}}\\
\nn&& +\norm{R}_{\ll{2}}\norm{\po N}_{\ll{\infty}}.
\eea
Differentiating \eqref{gauss1} with respect to $\o$ and using Corollary \ref{c0}, \eqref{small2}, \eqref{appboot1} and \eqref{initom}, we obtain:
\begin{equation}\label{om43}
\begin{array}{ll}
\ds\norm{\po K}_{\ll{2}} & \ds\lesssim\norm{R}_{\ll{2}}\norm{\po N}_{\ll{\infty}}+\norm{\th}_{\ll{4}}\norm{\po\th}_{\ll{4}}\\
& \ds\lesssim\ep\norm{\po N}_{\ll{\infty}}+\ep\norm{\po\th}_{\ll{4}}\\
& \ds\lesssim\ep\norm{\po N}_{\ll{\infty}}+\ep\norm{\nabla\po\th}_{\ll{2}}.
\end{array}
\end{equation}
\eqref{small2}, \eqref{boot}, \eqref{appboot}, \eqref{appboot1}, \eqref{om42} and \eqref{om43} yield:
\begin{equation}\label{om44}
\norm{h}_{\ll{2}}\lesssim  \norm{\nabb^2\po a}_{\ll{2}}+\ep\Big(\norm{\po a}_{\ll{\infty}}+\norm{\po N}_{\ll{\infty}}+\norm{\po\th}_{\ll{4}}+\norm{\nabla\po\th}_{\ll{2}}\Big).
\end{equation}
Corollary \ref{c0}, Proposition \ref{p2}, \eqref{initom}, \eqref{om38} and \eqref{om44} yield:
\begin{equation}\label{om45}
\norm{\nabn\po\th}_{\ll{2}}\lesssim  \norm{\nabb^2\po a}_{\ll{2}}+\ep\Big(\norm{\po a}_{\ll{\infty}}+\norm{\po N}_{\ll{\infty}}+\norm{\nabla\po\th}_{\ll{2}}\Big).
\end{equation}
Finally, \eqref{om26}, \eqref{om35} and \eqref{om45} yield:
\begin{equation}\label{om46}
\norm{\nabla\po\th}_{\ll{2}}\lesssim  \norm{\nabb^2\po a}_{\ll{2}}+\ep\Big(\norm{\po a}_{\ll{\infty}}+\norm{\po N}_{\ll{\infty}}+\norm{\nabla\po N}_{\ll{4}}\Big).
\end{equation}

\subsubsection{Estimates for $\po N$}

We start by estimating the norm of $\nabla\po N$ in $\ll{4}$. Let $X, Y$ two vectorfields on $\s$ independent of $\o$. We rewrite the first equation of \eqref{frame1} 
as:
\begin{equation}\label{om47}
g(\nabla_{\Pi X}N,\Pi Y)=\th(\Pi X,\Pi Y).
\end{equation}
We differentiate \eqref{om47} with respect to $\o$. Using \eqref{comom4} and evaluating at $X=e_A, Y=e_B$, we obtain:
\begin{equation}\label{om48}
g(\nabla_A\po N,e_B)=\po\th_{AB}-(\po N)_Aa^{-1}\nabb_B a.
\end{equation}
Also, using \eqref{comom1ter}, we have:
\begin{equation}\label{om49}
g(\nabla_A\po N,N)=-g(\po N,\nabla_AN)=-\th(\po N,e_A).
\end{equation}
Differentiating the second equation of \eqref{frame1} and using \eqref{commutom1}, we obtain:
\begin{equation}\label{om50}
\nabn\po N=-\th(\po N,e_A)e_A-a^{-1}\nabb\po a+a^{-1}\nabn a \po N+a^{-1}\nabb_{\po N} a N+a^{-2}\po a\nabb a.
\end{equation}
\eqref{om48}, \eqref{om49} and \eqref{om50} yield:
\begin{equation}\label{om51}
\begin{array}{ll}
\ds\norm{\nabla\po N}_{\ll{4}}\lesssim & \ds\norm{\po\th}_{\ll{4}}+\norm{\nabb\po a}_{\ll{4}}+(\norm{\nabb a}_{\ll{4}}+\norm{\th}_{\ll{4}}\\
&\ds+\norm{\nabn a}_{\ll{4}})(\norm{\po N}_{\ll{\infty}}+\norm{\po a}_{\ll{\infty}}).
\end{array}
\end{equation}
Together with \eqref{appboot} and \eqref{appboot1}, this yields:
\begin{equation}\label{om51bis}
\norm{\nabla\po N}_{\ll{4}}\lesssim\norm{\po\th}_{\ll{4}}+\norm{\nabb\po a}_{\ll{4}}+\ep\norm{\po N}_{\ll{\infty}}+\ep\norm{\po a}_{\ll{\infty}}.
\end{equation}
Finally, using Corollary \ref{c0} and Proposition \ref{p2}, we obtain:
\begin{equation}\label{om51ter}
\begin{array}{ll}
\ds\norm{\nabla\po N}_{\ll{4}}\lesssim & \norm{\nabla\po\th}_{\ll{2}}+\norm{\nabb\po a}_{\l{\infty}{2}}+\norm{\nabb^2\po a}_{\ll{2}}\\
&\ds +\ep\norm{\po N}_{\ll{\infty}}+\ep\norm{\po a}_{\ll{\infty}}.
\end{array}
\end{equation}

Next, we estimate the norm of $\nabla\po N$ in $\l{\infty}{2}$. In view of \eqref{om48}, \eqref{om49} and \eqref{om50}, we have:
\begin{displaymath}
\begin{array}{ll}
&\ds\norm{\nabla\po N}_{\l{\infty}{2}}\\
\lesssim & \ds\norm{\po\th}_{\l{\infty}{2}}+\norm{a^{-1}\nabb\po a}_{\l{\infty}{2}}+(\norm{a^{-1}\nabb a}_{\l{\infty}{2}}\\
&\ds +\norm{\th}_{\l{\infty}{2}}+\norm{a^{-1}\nabn a}_{\l{\infty}{2}})(\norm{\po N}_{\ll{\infty}}+\norm{\po a}_{\ll{\infty}}).
\end{array}
\end{displaymath}
Together with \eqref{appboot} for $a$ and \eqref{appboot1} for $\th$, this yields:
$$\norm{\nabla\po N}_{\l{\infty}{2}}\lesssim\norm{\po\th}_{\l{\infty}{2}}+\norm{\nabb\po a}_{\l{\infty}{2}}+\ep(\norm{\po N}_{\ll{\infty}}+\norm{\po a}_{\ll{\infty}}).$$
Finally, we obtain:
\begin{equation}\label{om51quatre}
\ds\norm{\nabla\po N}_{\l{\infty}{2}}\lesssim  \norm{\nabla\po\th}_{\ll{2}}+\norm{\nabb\po a}_{\l{\infty}{2}}+\ep(\norm{\po N}_{\ll{\infty}}+\norm{\po a}_{\ll{\infty}}).
\end{equation}

Next, we estimate the norm of $\divb(\po N)$ in $\lhs{\infty}{\frac{1}{2}}$. In view of \eqref{om48}, we have:
$$\divb(\po N)=\textrm{tr}(\po\th)-a^{-1}\nabb_{\po N} a.$$
This yields:
$$\norm{\divb(\po N)}_{\lhs{\infty}{\frac{1}{2}}}\les \norm{\po\th}_{\lhs{\infty}{\frac{1}{2}}}+\norm{a^{-1}\nabb_{\po N} a}_{\lhs{\infty}{\frac{1}{2}}}.$$
In view of Corollary \ref{cor:commLP1}, we finally obtain:
\bea\lab{om51cinq}
&&\norm{\divb(\po N)}_{\lhs{\infty}{\frac{1}{2}}}\\
\nn&\les& \norm{\nabla\po\th}_{\ll{2}}+\norm{\nabla(a^{-1}\nabb_{\po N} a)}_{\ll{2}}\\
\nn&\les& \norm{\nabla\po\th}_{\ll{2}}+\norm{\nabla(a^{-1}\nabb a)}_{\ll{2}}\norm{\po N}_{\ll{\infty}}+\norm{a^{-1}\nabb a}_{\ll{4}}\norm{\nabla\po N}_{\ll{4}}\\
\nn&\les& \norm{\nabla\po\th}_{\ll{2}}+\ep(\norm{\po N}_{\ll{\infty}}+\norm{\nabla\po N}_{\ll{4}}),
\eea
where we used in the last inequality the estimates \eqref{boot} and \eqref{appboot} for $a$.

In view of the right-hand side of \eqref{zoo14}, we need to control $\norm{ak(N, \nabn\nabn\po N)}_{\lhs{2}{-\frac{3}{2}}}$. In view of \eqref{om50}, we have:
\bee
&&\nabn\nabn\po N\\
&=&-\nabn\th(\po N,e_A)e_A-\th(\nabn\po N,e_A)e_A-a^{-1}\nabb\nabn\po a-a^{-1}[\nabn,\nabb]\po a\\
&&+a^{-2}\nabn a\nabb\po a+a^{-1}\nabn^2a \po N+a^{-1}\nabn a \nabn\po N-a^{-2}(\nabn a)^2 \po N\\
&&+\nabn(a^{-1}\nabb a)_{\po N} N+a^{-1}\nabb_{\nabn\po N} a N-a^{-1}\nabb_{\po N} a \ana+\nabn(a^{-2}\nabb a)\po a\\
&&+a^{-2}\nabn\po a\nabb a.
\eee
This yields
\be\lab{mezul}
\nabn\nabn\po N=-a^{-1}\nabb\nabn\po a-a^{-2}\nabb a\nabn\po a+a^{-1}\nabn^2a \po N+H,
\ee
where, in view of the commutator formula \eqref{scommut}, the vectorfield $H$ is given by
\bee
H &=& -\nabn\th(\po N,e_A)e_A-\th(\nabn\po N,e_A)e_A+a^{-1}\th\c\nabb\po a+a^{-2}\nabn a\nabb\po a\\
&&+a^{-1}\nabn a \nabn\po N-a^{-2}(\nabn a)^2 \po N+\nabn(a^{-1}\nabb a)_{\po N} N+a^{-1}\nabb_{\nabn\po N} a N\\
&&-a^{-1}\nabb_{\po N} a \ana+\nabn(a^{-2}\nabb a)\po a.
\eee
We have
\bee
\norm{H}_{\ll{2}}&\les& \Big(\norm{\nabla\th}_{\ll{2}}+\norm{\nabn(a^{-2}\nabb a)}_{\ll{2}}+(\norm{\th}_{\ll{4}}+\norm{\nabla\po N}_{\ll{4}}\\
&&+\norm{a^{-1}\nabb\po a}_{\ll{4}}+\norm{a^{-1}\nabla a}_{\ll{4}})^2\Big)(1+\norm{\po N}_{\ll{\infty}}+\norm{\po a}_{\ll{\infty}})
\eee
which together with \eqref{boot}, \eqref{boot1}, \eqref{appboot} and \eqref{appboot1} yields:
\be\lab{mezul1}
\norm{H}_{\ll{2}}\les \ep(\norm{\nabb\po a}_{\ll{4}}+\norm{\nabla\po N}_{\ll{4}}+\norm{\po N}_{\ll{\infty}}+\norm{\po a}_{\ll{\infty}}).
\ee
Using the finite band property for $P_j$, we obtain
\bea\lab{mezul2}
&&\norm{P_j(ak(N, H))}_{\ll{2}}\\
\nn&\les& 2^j\norm{ak(N, H)}_{\l{2}{1}}\\
\nn&\les& \norm{a}_{\ll{\infty}}\norm{k}_{\l{\infty}{2}}\norm{H}_{\ll{2}}\\
\nn&\les& \ep(\norm{\nabb\po a}_{\ll{4}}+\norm{\nabla\po N}_{\ll{4}}+\norm{\po N}_{\ll{\infty}}+\norm{\po a}_{\ll{\infty}}),
\eea
where we used in the last inequality the estimate \eqref{mezul1} for $H$, the estimate \eqref{boot} for $a$ and the estimate \eqref{appsmall2} for $k$. Next, we estimate the other terms generated by the right-hand side of \eqref{mezul}. In view of the product estimate \eqref{prod1}, the embeddings \eqref{clp12} and Lemma \ref{lemma:vacances}, we have
\bea\lab{mezul3}
\norm{P_j(ak(N, a^{-1}\nabb\nabn\po a))}_{\ll{2}}&\les& 2^j\norm{kN}_{\lhs{\infty}{\frac{1}{2}}}\norm{\nabb\nabn\po a}_{\lhs{2}{-\frac{1}{2}}}\\
\nn&\les & 2^j\norm{kN}_{\h}\norm{\nabn\po a}_{\lhs{2}{\frac{1}{2}}}\\
\nn&\les & 2^j\ep\norm{\nabn\po a}_{\lhs{2}{\frac{1}{2}}},
\eea
where we used in the last inequality the estimates \eqref{small2} and \eqref{appsmall2} for $k$, \eqref{appboot} and \eqref{frame1}. Using the finite band property for $P_j$, we have
\bea\lab{mezul4}
\norm{P_j(ak(N, a^{-2}\nabb a\nabn\po a))}_{\ll{2}}&\les& 2^j\norm{ak(N, a^{-2}\nabb a\nabn\po a)}_{\l{2}{1}}\\
\nn&\les& 2^j\norm{k a^{-1}\nabb a}_{\l{\infty}{2}}\norm{\nabn\po a}_{\ll{2}}\\
\nn&\les& 2^j\norm{k}_{\l{\infty}{4}}\norm{a^{-1}\nabb a}_{\l{\infty}{4}}\norm{\nabn\po a}_{\ll{2}}\\
\nn&\les& 2^j\ep\norm{\nabn\po a}_{\ll{2}},
\eea
where we used in the last inequality the estimate \eqref{appsmall2} for $k$ and the estimate \eqref{appboot} for $a$. In view of the product estimate \eqref{prod1} and the embeddings \eqref{clp12}, we have
\bea\lab{mezul5}
\norm{P_j(ak(N, a^{-1}\nabn^2a \po N))}_{\ll{2}}&\les& 2^j\norm{kN\po N}_{\lhs{\infty}{\frac{1}{2}}}\norm{\nabn^2a}_{\lhs{2}{-\frac{1}{2}}}\\
\nn&\les& 2^j\norm{kN\po N}_{\h}\norm{\nabn^2a}_{\lhs{2}{-\frac{1}{2}}}\\
\nn&\les& 2^j\ep(\norm{\po N}_{\ll{\infty}}+\norm{\nabla\po N}_{\ll{4}}),
\eea
where we used in the last inequality the estimates \eqref{small2} and \eqref{appsmall2} for $k$, \eqref{appboot}, \eqref{frame1}, and the estimate \eqref{nabn2a1} for $\nabn^2a$. Finally, \eqref{mezul}-\eqref{mezul5} imply
\bea\lab{mezul6}
&&\norm{ak(N, \nabn\nabn\po N)}_{\lhs{2}{-\frac{3}{2}}}\\
\nn&\les&  \ep\Big(\norm{\nabb\po a}_{\ll{4}}+\norm{\nabla\po N}_{\ll{4}}+\norm{\po N}_{\ll{\infty}}+\norm{\po a}_{\ll{\infty}}+\norm{\nabn\po a}_{\lhs{2}{\frac{1}{2}}}\Big).
\eea

We now estimate the norm of $\po N$ in $\ll{\infty}$. Using \eqref{initom} and the fact that $(\s,g)$  coincides with $(\R^3,\delta)$ for $|x|\geq 1$ by section \ref{reducsmall}, we have:
\begin{equation}\label{om52}
g(\po N,\po N)=I\textrm{ on }x.\o=-2,
\end{equation}
where $I$ is the $2\times 2$ identity matrix. We will estimate the $\ll{\infty}$ norm of $g(\po N,\po N)-I$ using Proposition \ref{p3}. To this end, 
we need to estimate the norm of:
\begin{equation}\label{om53}
\nabb^2(g(\po N,\po N)-I)=2\nabb(g(\nabb\po N,\po N)),
\end{equation}
and 
\begin{equation}\label{om54}
\nabb\nabn(g(\po N,\po N)-I)=2\nabb(g(\nabn\po N,\po N)),
\end{equation}
in $\ll{2}$. First we estimate the norm of \eqref{om53} in $\ll{2}$. Using \eqref{om48}, we have:
\begin{equation}\label{om55}
g(\nabb_A\po N,\po N)=\po\th(\po N,e_A)-(\po N)_Aa^{-1}\nabb_{\po N} a,
\end{equation}
which together with \eqref{om53} yields:
\begin{equation}\label{om56}
\begin{array}{l}
\ds\nabb^2_{AB}(g(\po N,\po N)-I)=2\nabb_A(\po\th)(e_B,\po N)+2\po\th(e_B,\nabb_A\po N)\\
\ds -2g(\nabb_A\po N,e_B)\nabb_{\po N} a-2(\po N)_B\nabb(\ana)(e_A,\po N)\\
\ds -2(\po N)_Bg(\nabb_A\po N,\ana).
\end{array}
\end{equation}
We estimate the norm of $\nabb^2(g(\po N,\po N)-I)$:
\begin{equation}\label{om57}
\begin{array}{r}
\ds\norm{\nabb^2(g(\po N,\po N)-I)}_{\ll{2}}\lesssim \norm{\nabb\po\th}_{\ll{2}}\norm{\po N}_{\ll{\infty}}\hspace{4cm}\\
\ds +(\norm{\po\th}_{\ll{4}}+\norm{\po N}_{\ll{\infty}}\norm{\ana}_{\ll{4}})\norm{\nabb\po N}_{\ll{4}}\\
\ds +\norm{\po N}^2_{\ll{\infty}}\norm{\nabb(\ana)}_{\ll{2}},
\end{array}
\end{equation}
which together with \eqref{boot} and \eqref{appboot} yields:
\begin{equation}\label{om58}
\begin{array}{r}
\ds\norm{\nabb^2(g(\po N,\po N)-I)}_{\ll{2}}\lesssim \norm{\nabb\po\th}_{\ll{2}}\norm{\po N}_{\ll{\infty}}\hspace{4cm}\\
\ds +(\norm{\po\th}_{\ll{4}}+\ep\norm{\po N}_{\ll{\infty}})\norm{\nabb\po N}_{\ll{4}}+\ep\norm{\po N}^2_{\ll{\infty}}.
\end{array}
\end{equation}

We turn to the estimate of the norm of \eqref{om54} in $\ll{2}$. Using \eqref{om50}, we have:
\begin{equation}\label{om59}
g(\nabn\po N,\po N)=-\th(\po N,\po N)-a^{-1}\nabb_{\po N}(\po a)+a^{-1}\nabn a|\po N|^2+a^{-2}\po a\nabb_{\po N}a,
\end{equation}
which together with \eqref{om54} yields:
\bea\label{om60}
&&\nabb_A\nabn(g(\po N,\po N)-I)\\
\nn&=&-2\nabb_A\th(\po N,\po N)-4\th(\nabb_A\po N,\po N)-2\nabb(a^{-1}\nabb(\po a))(\po N,e_A)\\
\nn&&-2g(\nabb_A\po N,a^{-1}\nabb\po a)+2\nabb(a^{-1}\nabn a)|\po N|^2+4a^{-1}\nabn ag(\nabb\po N,\po N)\\
\nn&&+\nabb(a^{-2}\nabb a)\po a\c \po N+a^{-2}\po a\nabb_{\nabb\po N}a+a^{-2}\nabb(\po a)\nabb_{\po N}a.
\eea
We estimate the norm of $\nabb\nabn(g(\po N,\po N)-I)$:
\bea\label{om61}
&&\norm{\nabb\nabn(g(\po N,\po N)-I)}_{\ll{2}}\\
\nn&\lesssim& \norm{\nabb^2\po a}_{\ll{2}}\norm{\po N}_{\ll{\infty}}+(\norm{\nabb\th}_{\ll{2}}+\norm{\nabb\nabn a}_{\ll{2}})\norm{\po N}^2_{\ll{\infty}}\\
\nn&&+(\norm{\th}_{\ll{4}}+\norm{\nabn a}_{\ll{4}})\norm{\nabb\po N}_{\ll{4}}\norm{\po N}_{\ll{\infty}}+\norm{\nabb\po N}_{\ll{4}}\norm{\nabb\po a}_{\ll{4}}
\eea
which together with \eqref{boot}, \eqref{boot1}, \eqref{appboot} and \eqref{appboot1} yields:
\bea\label{om62}
&&\norm{\nabb\nabn(g(\po N,\po N)-I)}_{\ll{2}}\\
\nn&\lesssim& \norm{\nabb^2\po a}_{\ll{2}}\norm{\po N}_{\ll{\infty}}+\ep\norm{\po N}^2_{\ll{\infty}}+\ep\norm{\nabb\po N}_{\ll{4}}\norm{\po N}_{\ll{\infty}}\\
\nn&&+\norm{\nabb\po N}_{\ll{4}}\norm{\nabb\po a}_{\ll{4}}.
\eea
Proposition \ref{p3}, \eqref{om58} and \eqref{om62} yield:
\bea\label{om63}
&&\norm{g(\po N,\po N)-I}_{\ll{\infty}}\\
\nn&\lesssim&\norm{\nabb\po N}_{\ll{4}}(\norm{\nabb\po a}_{\ll{4}}+\norm{\po\th}_{\ll{4}})+(\norm{\nabb^2\po a}_{\ll{2}}+\ep\norm{\po N}_{\ll{\infty}}\\
\nn&&+\ep\norm{\nabb\po N}_{\ll{4}}+\norm{\nabb\po\th}_{\ll{2}})\norm{\po N}_{\ll{\infty}}.
\eea
\eqref{om63} implies:
\begin{equation}\label{om64}
\norm{\po N}_{\ll{\infty}}\lesssim  1+\norm{\nabb\po N}_{\ll{4}}+\norm{\nabb\po a}_{\ll{4}}+\norm{\po\th}_{\ll{4}}+\norm{\nabb^2\po a}_{\ll{2}}+\norm{\nabb\po\th}_{\ll{2}}.
\end{equation}
Together with Corollary \ref{c0}, Proposition \ref{p2} and \eqref{om51ter}, we obtain:
\begin{equation}\label{om65}
\norm{\po N}_{\ll{\infty}}\lesssim  1+\norm{\nabb\po a}_{\l{\infty}{2}}+\norm{\nabb^2\po a}_{\ll{2}} +\norm{\nabla\po\th}_{\ll{2}}.
\end{equation}

Finally, \eqref{om15}, \eqref{zoo14}, \eqref{om46}, \eqref{xoox}, \eqref{om51ter}, \eqref{om51quatre}, \eqref{om51cinq}, \eqref{mezul6} and \eqref{om65} yield:
\bea\label{om66}
\nn&&\norm{\po a}_{\l{\infty}{2}}+\norm{\nabb\po a}_{\l{\infty}{2}}+\norm{\nabb^2\po a}_{\ll{2}} +\norm{\nabn\po a}_{\lhs{2}{\frac{1}{2}}}\\
\nn&&+\norm{\nabn\po a}_{\lhs{\infty}{-\frac{1}{2}}}+\norm{\nabn^2\po a}_{\lhs{2}{-\frac{3}{2}}}+\norm{\po a}_{\ll{\infty}}+\norm{\nabla\po\th}_{\ll{2}}\\
\nn&&+\norm{\nabla\po N}_{\ll{4}}\\
&\lesssim&\ep,
\eea
and
\begin{equation}\label{om67}
\norm{\po N}_{\ll{\infty}}\lesssim 1,
\end{equation}
which concludes the proof of \eqref{threomega1}.

\subsection{Second order derivatives with respect to $\o$}\lab{sec:regomega2}

The goal of this section is to prove \eqref{threomega2}. We first give an outline of the proof. Differentiating the equation \eqref{om1} for $\po a$ with respect to $\o$, we obtain:
\begin{equation}\label{diffomm1}
(\nabn-a^{-1}\lap)\po^2a = 2\nabn^2a+\nabb\nabn a+2R_{\po N \po N}+\cdots
\end{equation}
where the first two terms on the right-hand side come respectively from the commutators $[\po,\nabb]$ and $[\po,\lap]$ (see \eqref{commutom1} and \eqref{commutom3}). 
Since $R$ is in $\ll{2}$ by \eqref{small2}, $\nabn^2a$ is in $\lhs{2}{-\frac{1}{2}}$ by \eqref{nabn2a1}, and $\nabn\po a$ is in $\lhs{2}{\frac{1}{2}}$ by \eqref{threomega1}, this suggests in view of Proposition \ref{prop:parab4} that: 
\begin{equation}\label{diffomm2}
\norm{\po^2 a}_{\lhs{2}{\frac{3}{2}}}+\norm{\po^2 a}_{\lhs{\infty}{\frac{1}{2}}}+\norm{\nabn\po^2a}_{\lhs{2}{-\frac{1}{2}}} \lesssim \ep.
\end{equation}

\begin{remark}\lab{theend}
Note that we may not differentiate the equation \eqref{diffomm1} for $\po^2 a$ with respect to $\nabn$. Indeed, 
the term $\nabn R_{\po N \po N}$ has no structure: unlike $R_{NN}$ and $R_{N \po N}$ which were involved in the equation for $a$ and $\po a$, $R_{\po N \po N}$ does not contain any contraction with $N$ since $\po N$ is tangent to $\p$. Thus, unlike $\nabn R_{NN}$ and $\nabn R_{N \po N}$, we can not write $\nabn R_{\po N\po N}$ as a tangential derivative using the contracted Bianchi identities for $R$. In turn, we can not obtain any estimate for $\nabn^2\po^2 a$. 
\end{remark}

Next, we turn to the estimates for $\po^2\th$. Differentiating the equation \eqref{om23} for $\po\trt$ and the equation \eqref{om28} for $\po\hth$ with respect to $\o$, we obtain:
\be\lab{diffomm8}
\left\{\begin{array}{l}
\nabb\po^2\trt=\nabla k_{N \po^2N}+\cdots,\\[1mm]
\nabb^B\po^2\hth_{AB}=R_{N \po^2N}+\cdots,
\end{array}\right.
\ee
which together with the estimate \eqref{small2} for $R$ and $k$ yields $\nabb\po^2\th\in\ll{2}$ provided $\po^2N$ belongs to $\ll{\infty}$.

Finally, we turn to the estimates for $\po^2N$. Differentiating the equations \eqref{om48}, \eqref{om49} and \eqref{om50} for $\po N$ with respect to $\o$, we obtain:
\be\lab{diffomm9}
\left\{\begin{array}{l}
\nabb\po^2N=\po^2\th +\cdots,\\
\nabn\po^2N=-a^{-1}\nabb\po^2 a+\cdots.
\end{array}\right.
\ee
Together with the fact that $\nabb\po\th$ belong to $\ll{2}$ and $\po a$ belongs to $\lhs{1}{\frac{3}{2}}$, this suggests using interpolation that $\po^2N$ belongs to $\lhs{\infty}{\frac{5}{4}}$. Since $\frac{5}{4}>1$, and since $\p$ is 2 dimensional, we obtain that $\po^2N$ belongs indeed to $\ll{\infty}$.

The rest of this section is as follows. We first prove the estimates for $\po^2 a$. Then, we prove the estimates for $\po^2\th$. Finally, we conclude with the estimates for $\po^2N$. 

\subsubsection{Estimates for $\po^2 a$}

Recall \eqref{om1} and \eqref{om2}. $\po a$ satisfies:
\begin{equation}\label{zaa1}
\nabn\po a -a^{-1}\lap\po a = a^{-1}h,
\end{equation}
where $h$ is given by:
\begin{equation}\label{zaa2}
\begin{array}{ll}
h= &-\nabb_{\po N}a-a^{-2}\po a\lap a -2\nabb_{\po N}\nabn a+2\th(\po N,\nabb a)\\
& -\po\trt\nabn a-\trt\nabb_{\po N} a +2\th\po\th+\po(\nabn(k_{NN}))+2R_{N\po N}.
\end{array}
\end{equation}
Now, differentiating \eqref{zaa1} with respect to $\o$ and using the commutator formula \eqref{commutom3}, we obtain:
\begin{equation}\label{zaa4}
\nabn\po^2 a -a^{-1}\lap\po^2 a = -2a^{-1}\divb(\nabn(\po a)\po N)+\po h+a^{-1}h_1,
\end{equation}
where $h_1$ is given by:
\bee
h_1&=& -a\nabb_{\po N}(\po a)-a^{-1}\po a\lap(\po a)+2\divb(\po N)\nabn(\po a)\\
&&+2\th(\po N,\nabb(\po a))-\po\trt\nabn(\po a)-\trt\nabb_{\po N}(\po a).
\eee
Together with the product estimate \eqref{kei}, this yields:
\bea\lab{zaa5}
&&\norm{h_1}_{\ll{2}}\\
\nn&\les& \norm{a}_{\ll{\infty}}\norm{\po N}_{\ll{\infty}}\norm{\nabb(\po a)}_{\ll{2}}+\norm{\po a}_{\ll{\infty}}\norm{\lap(\po a)}_{\ll{2}}\\
\nn&&+\norm{\divb(\po N)}_{\lhs{\infty}{\frac{1}{2}}}\norm{\nabn(\po a)}_{\lhs{2}{\frac{1}{2}}}\\
\nn&&+\norm{\th}_{\ll{4}}\norm{\po N}_{\ll{\infty}}\norm{\nabb(\po a)}_{\ll{4}}+\norm{\po\trt}_{\lhs{\infty}{\frac{1}{2}}}\norm{\nabn(\po a)}_{\lhs{2}{\frac{1}{2}}}\\
\nn&&+\norm{\trt}_{\ll{4}}\norm{\po N}_{\ll{\infty}}\norm{\nabb(\po a)}_{\ll{4}}+\norm{\po a}_{\ll{\infty}}\norm{h}_{\ll{2}}\\
\nn&\les& \ep(1+\norm{\divb(\po N)}_{\lhs{\infty}{\frac{1}{2}}}+\norm{\po\trt}_{\lhs{\infty}{\frac{1}{2}}}),
\eea
where we used in the last inequality the estimate \eqref{boot} for $a$, the estimate \eqref{appboot1} for $\th$, and the estimate \eqref{threomega1} for $\po a$ and $\po N$. In view of Corollary \ref{cor:commLP1} and the estimate \eqref{om51cinq}, we have:
\bee
&&\norm{\divb(\po N)}_{\lhs{\infty}{\frac{1}{2}}}+\norm{\po\trt}_{\lhs{\infty}{\frac{1}{2}}}\\
\nn&\les& \norm{\nabla\po\th}_{\ll{2}}+\norm{\nabla(\nabb_{\po N} a)}_{\ll{2}}\\
\nn&\les& \norm{\nabla\po\th}_{\ll{2}}+\norm{\nabla\nabb a}_{\ll{2}}\norm{\po N}_{\ll{\infty}}+\norm{\nabb a}_{\ll{4}}\norm{\nabla\po N}_{\ll{4}}\\
\nn&\les& \ep,
\eee
where we used in the last inequality the estimates \eqref{boot} and \eqref{appboot} for $a$, and the estimate \eqref{threomega1} for $\po\th$ and $\po N$. Together with \eqref{zaa5}, this finally yields:
\be\lab{zaa6}
\norm{h_1}_{\ll{2}}\les\ep.
\ee

Next, we estimate the first term in the right-hand side of \eqref{zaa4}. In view of the product estimate \eqref{prod6}, we have:
\bea\lab{zaa7}
&&\norm{\divb(\nabn(\po a)\po N)}_{\lhs{2}{-\frac{1}{2}}}\\
\nn&\les& \norm{\nabn(\po a)}_{\lhs{2}{\frac{1}{2}}}(\norm{\po N}_{\ll{\infty}}+\norm{\nabb\po N}_{\l{\infty}{2}})\\
\nn&\les&\ep,
\eea
where we used in the last inequality the estimate \eqref{threomega1} for $\po a$ and $\po N$.

Finally, we estimate the second term in the right-hand side of \eqref{zaa4}. We first provide a decomposition of $\po^2N$. Differentiating \eqref{comom1ter} with respect to $\o$, we obtain:
\be\lab{zaa8}
g(\po^2N,N)=-g(\po N, \po N),
\ee
which yields:
\begin{equation}\label{zaa9}
\po^2 N=\Pi(\po^2N)-|\po N|^2N.
\end{equation}
Next, we compute $\po h$. Differentiating \eqref{zaa2} with respect to $\o$ and using \eqref{zaa9} and the commutator formula \eqref{commutom1}, we obtain:
\be\lab{zaa10}
\po h=2|\po N|^2\nabn^2 a-2a^{-1}\divb(\nabn(\po a)\po N)+h_2,
\ee
where $h_2$ is given by:
\bee
h_2&=&-\nabb_{\Pi(\po^2N)}a-\nabb_{\po N}(\po a)-a^{-2}\po a\lap(\po a)+2a^{-3}(\po a)^2\lap a-a^{-2}\po a[\po,\lap]a\\
&&-2\nabb_{\Pi(\po^2N)}\nabn a-2\nabb_{\po N}\nabb_{\po N} a+2\th(\po^2N,\nabb a)+2\po\th(\po N,\nabb a)+2\th(\po N,\nabb\po a\\
&&-\nabn( a)\po N)-\po^2\trt\nabn a-\po\trt\nabn(\po a)-2\po\trt\nabb_{\po N} a-\trt\nabla_{\po^2N} a\\
&&-\trt\nabb_{\po N}(\po a)+2\nabb a\nabb\po^2 a+2|\nabb\po a|^2-2\nabn a\nabb_{\po N}\po a-2\nabb_{\po N}( a)\nabn\po a\\
&& -2\nabb_{\po^2N} a\nabn a-2|\nabb_{\po N} a|^2-2\nabb_{\po N}(\po a)\nabn a-2\nabb_{\po N} a\nabn(\po a)+2\th\po^2\th\\
&& +2|\po\th|^2+\nabla_{\po^2N}k_{NN}+2\nabn k_{N \po^2N}+4\nabla_{\po N}k_{N \po N}+2\nabn k_{\po N \po N}+2R_{N \po^2N}\\
&&+2R_{\po N \po N}.
\eee
Together with the product estimate \eqref{kei}, this yields:
\bee
&&\norm{h_2}_{\ll{2}}\\
&\les&  \norm{a^{-2}\po a}_{\ll{\infty}}\norm{\lap\po a}_{\ll{2}}+\norm{a^{-3}(\po a)^2}_{\ll{\infty}}\norm{\lap a}_{\ll{2}}\\
&&+\norm{a^{-2}\po a}_{\ll{\infty}}\norm{[\po,\lap]a}_{\ll{2}}+\norm{\nabb a}_{\ll{2}}\norm{\po^2N}_{\ll{\infty}}\\
&&+\norm{\nabb\po a}_{\ll{2}}\norm{\po N}_{\ll{\infty}}+\norm{\nabb\nabn a}_{\ll{2}}\norm{\po^2N}_{\ll{\infty}}+\norm{\nabb^2a}_{\ll{2}}\norm{\po N}^2_{\ll{\infty}}\\
&&+\norm{\nabb a}_{\ll{4}}\norm{\nabb\po N}_{\ll{4}}\norm{\po N}_{\ll{\infty}}+\norm{\th}_{\ll{4}}\norm{\po^2N}_{\ll{\infty}}\norm{\nabb a}_{\ll{4}}\\
&&+\norm{\po\th}_{\ll{4}}\norm{\po N}_{\ll{\infty}}\norm{\nabb a}_{\ll{4}}+\norm{\th}_{\ll{4}}\norm{\po N}_{\ll{\infty}}(\norm{\nabb\po a}_{\ll{4}}\\
&&+\norm{\nabn a}_{\ll{4}}\norm{\po N}_{\ll{\infty}}) +\norm{\po^2\trt}_{\l{2}{4}}\norm{\nabn a}_{\l{\infty}{4}}\\
&&+\norm{\po\trt}_{\lhs{\infty}{\frac{1}{2}}}\norm{\nabn\po a}_{\lhs{2}{\frac{1}{2}}}+\norm{\po\trt}_{\ll{4}}\norm{\nabb a}_{\ll{4}}\norm{\po N}_{\ll{\infty}}\\
&&+\norm{\trt}_{\ll{4}}\norm{\nabla a}_{\ll{4}}\norm{\po^2N}_{\ll{\infty}}+\norm{\trt}_{\ll{4}}\norm{\nabb\po a}_{\ll{4}}\norm{\po N}_{\ll{\infty}}\\
&&+\norm{\nabb a}_{\lhs{\infty}{\frac{1}{2}}}\norm{\nabb\po^2 a}_{\lhs{2}{\frac{1}{2}}}+\norm{\nabb\po a}_{\ll{4}}^2\\
&&+\norm{\nabn a}_{\ll{4}}\norm{\nabb\po a}_{\ll{4}}\norm{\po N}_{\ll{\infty}}+\norm{\nabb_{\po N}( a)}_{\lhs{\infty}{\frac{1}{2}}}\norm{\nabn\po a}_{\lhs{2}{\frac{1}{2}}}\\
&&+\norm{\nabb a}_{\ll{4}}\norm{\nabn a}_{\ll{4}}\norm{\po^2N}_{\ll{\infty}} +\norm{\nabb a}_{\ll{4}}^2\norm{\po N}^2_{\ll{\infty}}\\
&&+\norm{\nabb\po a}_{\ll{4}}\norm{\nabn a}_{\ll{4}}\norm{\po N}_{\ll{\infty}}+\norm{\nabb_{\po N}( a)}_{\lhs{\infty}{\frac{1}{2}}}\norm{\nabn\po a}_{\lhs{2}{\frac{1}{2}}}\\
&&+\norm{\th}_{\l{\infty}{4}}\norm{\po^2\th}_{\l{2}{4}}+\norm{\po\th}_{\ll{4}}^2\\
&&+(\norm{\nabla k}_{\ll{2}}+\norm{R}_{\ll{2}})(\norm{\po^2N}_{\ll{\infty}}+\norm{\po N}^2_{\ll{\infty}})\\
&\les& \ep(1+\norm{\po^2 a}_{\l{2}{4}}+\norm{\nabb\po^2 a}_{\lhs{2}{\frac{1}{2}}}+\norm{\po^2N}_{\ll{\infty}}+\norm{\po^2\th}_{\l{2}{4}}),
\eee
where we used in the last inequality the commutator formula \eqref{commutom3} and the identity \eqref{comom6}, the estimates \eqref{boot} \eqref{appboot} for $a$, the estimate \eqref{nabn2a1} for $\nabn a$, the estimate \eqref{appboot1} for $\po\th$, the estimate \eqref{small2} for $k$ and $R$, the estimate \eqref{threomega1} for $\po a$, $\po\th$ and $\po N$, and Corollary \ref{cor:commLP1}. Together with the Gagliardo-Nirenberg inequality \eqref{eq:GNirenberg} and the Lemma \ref{lemma:vacances:1}, this yields:
\be\lab{zaa11}
\norm{h_2}_{\ll{2}}\les \ep(1+\norm{\po^2a}_{\lhs{2}{\frac{3}{2}}}+\norm{\po^2N}_{\ll{\infty}}+\norm{\nabb\po^2\th}_{\ll{2}}).
\ee

Next, we evaluate the first term in the right-hand side of \eqref{zaa10}. In view of Proposition \ref{prop:bale}, we have:
\bea\lab{zaa12}
\nn\norm{a|\po N|^2\nabn^2 a}_{\lhs{2}{-\frac{1}{2}}}&\les& (\norm{a|\po N|^2}_{\ll{\infty}}+\norm{\nabb(a|\po N|^2)}_{\l{\infty}{2}})\norm{\nabn^2 a}_{\hs{-\frac{1}{2}}}\\
\nn&\les& \ep(\norm{a}_{\ll{\infty}}\norm{\po N}^2_{\ll{\infty}}+\norm{\nabb a}_{\l{\infty}{2}}\norm{\po N}^2_{\ll{\infty}}\\
\nn&&+\norm{\nabb\po N}_{\l{\infty}{2}}\norm{a}_{\ll{\infty}}\norm{\po N}_{\ll{\infty}})\\
&\les&\ep,
\eea
where we used the estimate \eqref{nabn2a1} for $\nabn^2a$, the estimates \eqref{boot} \eqref{appboot} for $a$, and the estimate \eqref{threomega1} for $\po N$.

Finally, in view of \eqref{zaa4} and \eqref{zaa10}, we have:
\be\label{zaa13}
\nabn\po^2 a -a^{-1}\lap\po^2 a = h_3,
\ee
where $h_3$ is given by:
$$h_3= -4a^{-1}\divb(\nabn(\po a)\po N)+2|\po N|^2\nabn^2 a+a^{-1}h_1+h_2.$$
Together with the estimates \eqref{zaa6}, \eqref{zaa7}, \eqref{zaa11}, \eqref{zaa12} and the estimate \eqref{boot} for $a$, this yields:
\bea\lab{zaa14}
&&\norm{ah_3}_{\lhs{2}{-\frac{1}{2}}}\\
\nn&\les & \norm{\divb(\nabn(\po a)\po N)}_{\lhs{2}{-\frac{1}{2}}}+\norm{a|\po N|^2\nabn^2 a}_{\lhs{2}{-\frac{1}{2}}}+\norm{h_1}_{\ll{2}}+\norm{ah_2}_{\ll{2}}\\
\nn&\les & \ep(1+\norm{\po^2 a}_{\lhs{2}{\frac{3}{2}}}+\norm{\po^2N}_{\ll{\infty}}+\norm{\nabb\po^2\th}_{\ll{2}}).
\eea

Now, in view of \eqref{zaa13} and Proposition \ref{prop:parab4}, we have:
$$\norm{\po^2 a}_{\lhs{2}{\frac{3}{2}}}+\norm{\po^2 a}_{\lhs{\infty}{\frac{1}{2}}}+\norm{\nabn\po^2 a}_{\lhs{2}{-\frac{1}{2}}}\les \norm{h_3}_{\lhs{2}{-\frac{1}{2}}}.$$
Together with \eqref{zaa14}, this yields:
\bee
&&\norm{\po^2 a}_{\lhs{2}{\frac{3}{2}}}+\norm{\po^2 a}_{\lhs{\infty}{\frac{1}{2}}}+\norm{\nabn\po^2 a}_{\lhs{2}{-\frac{1}{2}}}\\
\nn&\les & \ep(1+\norm{\po^2 a}_{\lhs{2}{\frac{3}{2}}}+\norm{\po^2N}_{\ll{\infty}}+\norm{\nabb\po^2\th}_{\ll{2}}).
\eee
Thus, we finally obtain:
\bea\lab{zaa15}
&&\norm{\po^2 a}_{\lhs{2}{\frac{3}{2}}}+\norm{\po^2 a}_{\lhs{\infty}{\frac{1}{2}}}+\norm{\nabn\po^2 a}_{\lhs{2}{-\frac{1}{2}}}\\
\nn&\les & \ep(1+\norm{\po^2N}_{\ll{\infty}}+\norm{\nabb\po^2\th}_{\ll{2}}).
\eea

\subsubsection{Estimates for $\po^2\th$}

Let us start by computing the trace of $\po\hth$ when seen as a tensor on $\p$. Differentiating 
\eqref{om16} with respect to $\o$, we obtain:
\bee
\po^2\hth(X,Y)&= & \ds\po^2\th(X,Y)-\frac{1}{2}\po^2\trt (X.Y-(X.N)(Y.N))\\
\nn&&+\po\trt ((X.\po N)(Y.N)+(X.N)(Y.\po N))\\
\nn&&+\frac{1}{2}\trt ((X.\po^2N)(Y.N)+(X.N)(Y.\po^2N)+2(\po N\c X)(\po N\c Y)),
\eee
which yields:
\begin{equation}\label{toc}
\po^2\hth_{AB}=\po^2\th_{AB}-\frac{1}{2}\po^2\trt\delta_{AB}+\trt(\po N)_A(\po N)_B,
\end{equation}
so that:
\begin{equation}\label{toc1}
\textrm{tr}(\po^2\hth)=\textrm{tr}(\po^2\th)-\po^2\trt+\trt|\po N|^2.
\end{equation}
We compute $\po^2\trt$. In view of \eqref{om21}, we have:
\begin{equation}\label{toc2}
\po^2\trt=\po(\po\th_{AA})=\textrm{tr}(\po^2\th)+2\po\th(e_A,\po e_A).
\end{equation}
Now, in view of \eqref{extth}, we have:
\be\lab{toc3}
\po\th(N,.)=-\th(\po N,.).
\ee
\eqref{toc2}, \eqref{toc3} and \eqref{om20} yield
\begin{equation}\label{toc4}
\po^2\trt=\textrm{tr}(\po^2\th)+2\th(\po N, \po N).
\end{equation}
Finally, \eqref{toc1} and \eqref{toc4} imply:
\begin{equation}\label{toc5}
\textrm{tr}(\po^2\hth)=\trt|\po N|^2-2\th(\po N, \po N).
\end{equation}

We now turn to the estimates for $\po^2\trt$. Differentiating \eqref{om23} with 
respect to $\o$, we obtain:
$$\po^2\trt=-\po^2 a+2k_{N \po^2N}+2k_{\po N \po N},$$
so that:
$$\nabb\po^2\trt = -\nabb\po^2 a+2k_{N \nabb\po^2N}+2k_{\nabb N \po^2N}+2\nabb k_{N \po^2N}+4k_{\po N \nabb\po N} +2\nabb k_{\po N \po N}$$
which in turn yields:
\bee
&&\norm{\nabb\po^2\trt}_{\ll{2}} \\
&\les& \norm{\nabb\po^2 a}_{\ll{2}}+\norm{k}_{\l{\infty}{4}}(\norm{\nabb\po^2N}_{\l{2}{4}}+\norm{\nabb N}_{\l{2}{4}}\norm{\po^2N}_{\ll{\infty}}\\
&&+\norm{\nabb\po N}_{\l{2}{4}}\norm{\po N}_{\ll{\infty}})+\norm{\nabb k}_{\ll{2}}(\norm{\po^2N}_{\ll{\infty}}+\norm{\po N}^2_{\ll{\infty}}).
\eee
Together with the Gagliardo-Nirenberg inequality \eqref{eq:GNirenberg}, Lemma \ref{lemma:vacances:1}, the estimate \eqref{threomega1} for $\po a$ and $\po N$, and the estimate \eqref{small2} for $k$, we obtain:
\be\lab{toc6}
\norm{\nabb\po^2\trt}_{\ll{2}}\les \norm{\nabb\po^2 a}_{\ll{2}}+\ep(1+\norm{\nabb^2\po^2N}_{\ll{2}}+\norm{\po^2N}_{\ll{\infty}}).
\ee

We now turn to the estimates for $\nabb\po^2\hth$. We differentiate the third equation of \eqref{struct1} with respect to $\o$. We introduce the symmetric tensor $\sigma$ on $S$ defined by:
\bea\lab{toc7}
\sigma(X, Y)&=&\po\hth(X,Y)+\left(\th(\po N,Y)-\frac{1}{2}\trt \po N\c Y\right)N\c X\\
\nn&&+\left(\th(\po N,X)-\frac{1}{2}\trt \po N\c X\right)N\c Y,
\eea
which in view of \eqref{om16} and \eqref{toc3} satisfies:
$$\sigma(N,.)=\sigma(.,N)=0.$$
We may thus apply the commutator formula \eqref{commutom4} to $\sigma$. We obtain:
\bea\lab{toc8}
([\po,\divb\,]\sigma)_A &= &  -\trt\sigma_{\po N A}-\th_{AB}\sigma_{B\po N}-
\nabla_N\sigma_{\po N B}+\th_{\po N C}\sigma_{CA}\\
\nn&& \ds +(\po N)_A\th_{BC}\sigma_{CB},
\eea
Now, in view of the definition \eqref{toc7} of $\sigma$, and the structure equation for $N$ \eqref{frame1}, we have:
\bea\lab{toc9}
\sigma_{AB}&=&\po\hth_{AB},\\
\nn\nabn\sigma_{AB}&=&\nabn\po\hth_{AB}-\hth_{\po N B}\nabb_A a-\hth_{\po N A}\nabb_B a\\
\nn\nabb_C\sigma_{AB}&=&\nabb_C\po\hth_{AB}+\hth_{\po N B}\th_{AC}+\hth_{\po N A}\th_{BC},
\eea
which together with \eqref{toc8} implies:
\bea
\nn ([\po,\divb\,]\sigma)_A &= &  -\trt\po\hth_{\po N A}-\th_{AB}\po\hth_{B\po N}
-\nabla_N\po\hth_{\po N B}+\hth_{\po N B}\nabb_{\po N} a\\
\lab{toc10}&&+\hth_{\po N \po N}\nabb_A a+\th_{\po N C}\po\hth_{CA}+(\po N)_A\th_{BC}\po\hth_{CB}.
\eea

Now, we have in view of \eqref{toc9}, \eqref{om28} and  \eqref{om31}:
\begin{equation}\label{toc11}
(\divb\sigma)_A = h,
\end{equation}
where $h$ is given by:
\begin{equation}\label{toc12}
\begin{array}{ll}
\ds h = & \ds \frac{1}{2}\nabb_A\po\trt+\nabn\th_{\po N B}-\frac{1}{2}\nabn\trt(\po N)_A+2\trt\hth_{\po N A}+\th_{AB}\hth_{B\po N}\\
&\ds -(\po N)_A\th_{BC}\hth_{CB}+R_{A \po N}-(\po N)_AR_{NN}.
\end{array}
\end{equation}
Differentiating \eqref{toc11} and \eqref{toc12} with respect to $\o$, and using \eqref{toc10} and the commutator formula \eqref{commutom1}, we obtain:
\begin{equation}\label{toc13}
(\divb\po\sigma)_A = h_1,
\end{equation}
where $h_1$ is given by:
\bee
h_1 & = & \frac{1}{2}\nabb_A\po^2\trt -\frac{1}{2}\nabn\po\trt(\po N)_A+\nabn\th_{\po^2N A}+2\nabn\po\th_{\po N A}+\nabb_{\po N}\th_{\po N A}\\
\nn&&-\frac{1}{2}\nabn\trt(\po^2N)_A-\frac{1}{2}\nabn\po\trt(\po N)_A-\frac{1}{2}\nabb_{\po N}\trt(\po N)_A+2\trt\hth_{\po^2N A}\\
\nn&&+3\trt\po\hth_{\po N A}+2\po\trt\hth_{\po N A}+\th_{AB}\hth_{B\po^2N}+\po\th_{AB}\hth_{B\po N}+2\th_{AB}\po\hth_{B\po N}\\
\nn&&-(\po^2N)_A\th_{BC}\hth_{CB}-(\po N)_A\po\th_{BC}\hth_{CB}-2(\po N)_A\th_{BC}\po\hth_{CB}-\th_{\po N C}\po\hth_{CA}\\
\nn&& -\hth_{\po N B}\nabb_{\po N} a-\hth_{\po N \po N}\nabb_A a+R_{A \po^2N}-(\po^2N)_AR_{NN}-2(\po N)_AR_{N \po N}.
\eee
$h_1$ satisfies:
\bee
&&\norm{h_1}_{\ll{2}}\\
&\les& \norm{\nabb\po^2\trt}_{\ll{2}}+\norm{\nabla\th}_{\ll{2}}(\norm{\po^2N}_{\ll{\infty}}+\norm{\po N}^2_{\ll{\infty}})
+\norm{\nabla\po\th}_{\ll{2}}\norm{\po N}_{\ll{\infty}}\\
&& +\norm{\th}_{\ll{4}}\norm{\hth}_{\ll{4}}(\norm{\po^2N}_{\ll{\infty}}+\norm{\po N}^2_{\ll{\infty}})+\norm{\po\hth}_{\ll{4}}\norm{\th}_{\ll{4}}\norm{\po N}_{\ll{\infty}}\\
&&+\norm{\po\th}_{\ll{4}}\norm{\hth}_{\ll{4}}\norm{\po N}_{\ll{\infty}}+\norm{\hth}_{\ll{4}}\norm{\nabb a}_{\ll{4}}\norm{\po N}^2_{\ll{\infty}}\\
&&+\norm{R}_{\ll{2}}(\norm{\po^2N}_{\ll{\infty}}+\norm{\po N}^2_{\ll{\infty}}).
\eee
Together with the estimate \eqref{toc6} for $\po^2\trt$, the estimates \eqref{boot1} \eqref{appboot1} for $\th$, the estimate \eqref{threomega1} for $\po\th$ and $\po N$, the estimate \eqref{appboot} for $a$,  and the estimate \eqref{small2} for $R$, we obtain:
\be\lab{toc14}
\norm{h_1}_{\ll{2}}\les \norm{\nabb\po^2 a}_{\ll{2}}+\ep(1+\norm{\nabb^2\po^2N}_{\ll{2}}+\norm{\po^2N}_{\ll{\infty}}).
\ee

Next, we compare $\divb\po\sigma$ to $\divb\po^2\th$. Differentiating the definition \eqref{toc7} of $\sigma$ with respect to $\o$ first, and then $\divb$, we obtain:
\be\lab{toc15}
\divb\po\sigma=\divb\po^2\hth+h_2,
\ee
where $h_2$ is given schematically by:
$$h_2= \nabb\th(\po^2N+(\po N)^2)+\nabb\po\th \po N+\po\th \nabb\po N+\th(\nabb\po^2N+\po N\nabb\po N).$$
$h_2$ satisfies:
\bee
&&\norm{h_2}_{\ll{2}}\\
&\les& \norm{\nabb\th}_{\ll{2}}(\norm{\po^2N}_{\ll{\infty}}+\norm{\po N}^2_{\ll{\infty}})+\norm{\nabb\po\th}_{\ll{2}}\norm{\po N}_{\ll{\infty}}\\
&&+\norm{\po\th}_{\l{\infty}{4}}\norm{\nabb\po N}_{\l{2}{4}}\\
&&+\norm{\th}_{\l{\infty}{4}}(\norm{\nabb\po^2N}_{\l{2}{4}}+\norm{\po N}_{\ll{\infty}}\norm{\nabb\po N}_{\l{2}{4}}).
\eee
Together with the Gagliardo-Nirenberg inequality \eqref{eq:GNirenberg}, the estimates \eqref{boot1} \eqref{appboot1} for $\th$ and the estimate \eqref{threomega1} for $\po\th$ and $\po N$, we obtain
\be\lab{toc16}
\norm{h_2}_{\ll{2}}\les \ep(1+\norm{\po^2N}_{\ll{\infty}}+\norm{\nabb^2\po^2N}_{\ll{2}}).
\ee

Finally, in view of \eqref{toc13}, \eqref{toc14}, \eqref{toc15} and \eqref{toc16}, we have:
\be\lab{toc17}
\norm{\divb(\po^2\hth)}_{\ll{2}}\les  \norm{\nabb\po^2 a}_{\ll{2}}+\ep(1+\norm{\po^2N}_{\ll{\infty}}+\norm{\nabb^2\po^2N}_{\ll{2}}).
\ee
Next, we estimate $\nabb\textrm{tr}(\po^2\hth)$. In view of \eqref{toc5}, we have:
\bee
\nabb\textrm{tr}(\po^2\hth)=|\po N|^2\nabb\trt+2\trt\po N\nabb\po N-2\nabb\th(\po N, \po N)-4\th(\po N, \nabb\po N).
\eee
This yields:
$$\norm{\nabb\textrm{tr}(\po^2\hth)}_{\ll{2}}\les \norm{\nabb\th}_{\ll{2}}\norm{\po N}^2_{\ll{\infty}}+\norm{\th}_{\ll{4}}\norm{\nabb\po N}_{\ll{4}}\norm{\po N}_{\ll{\infty}},$$
which together with the estimates \eqref{boot1} \eqref{appboot1} for $\th$, and the estimate \eqref{threomega1} for $\po N$ and $\po\th$ implies:
\be\lab{toc18}
\norm{\nabb\textrm{tr}(\po^2\hth)}_{\ll{2}}\les \ep.
\ee
Together with \eqref{toc17}, we obtain:
$$\norm{\divb(\po^2\hth-\textrm{tr}(\po^2\hth))}_{\ll{2}}\les  \norm{\nabb\po^2 a}_{\ll{2}}+\ep(1+\norm{\po^2N}_{\ll{\infty}}+\norm{\nabb^2\po^2N}_{\ll{2}}).$$
In view of the Hodge estimate \eqref{bieber}, this implies:
$$\norm{\nabb(\po^2\hth-\textrm{tr}(\po^2\hth))}_{\ll{2}}\les  \norm{\nabb\po^2 a}_{\ll{2}}+\ep(1+\norm{\po^2N}_{\ll{\infty}}+\norm{\nabb^2\po^2N}_{\ll{2}}),$$
which together with \eqref{toc18} yields:
\be\lab{toc19}
\norm{\nabb\po^2\hth}_{\ll{2}}\les  \norm{\nabb\po^2 a}_{\ll{2}}+\ep(1+\norm{\po^2N}_{\ll{\infty}}+\norm{\nabb^2\po^2N}_{\ll{2}}).
\ee
Now, in view of \eqref{toc}, we have:
\bee
\norm{\nabb\po^2\th}_{\ll{2}}&\les& \norm{\nabb\po^2\hth}_{\ll{2}}+\norm{\nabb\po^2\trt}_{\ll{2}}+\norm{\nabb\trt}_{\ll{2}}\norm{\po N}^2_{\ll{\infty}}\\
&&+\norm{\trt}_{\ll{4}}\norm{\nabb\po N}_{\ll{4}}\norm{\po N}_{\ll{\infty}}
\eee
Together with the estimate \eqref{toc6} for $\po^2\trt$, the estimate \eqref{toc19} for $\po^2\hth$, the estimates \eqref{boot1} \eqref{appboot1} for $\trt$ and the estimate \eqref{threomega1} for $\po N$, we finally obtain:
\be\lab{toc20}
\norm{\nabb\po^2\th}_{\ll{2}}\les  \norm{\nabb\po^2 a}_{\ll{2}}+\ep(1+\norm{\po^2N}_{\ll{\infty}}+\norm{\nabb^2\po^2N}_{\ll{2}}).
\ee

\subsubsection{Estimates for $\po^2N$}

Let $X, Y$ two vectorfields on $\s$ independent of $\o$. We rewrite \eqref{om48}:
\begin{equation}\label{zia}
g(\nabla_{\Pi X}\po N,\Pi Y)=\po\th_{\Pi X\Pi Y}-(\po N)_{\Pi X}\nabla_{\Pi Y} a.
\end{equation}
We differentiate \eqref{zia} with respect to $\o$. Using \eqref{comom4} and evaluating at $X=e_A, Y=e_B$, we obtain:
\bee
&& g(\nabla_A\po^2N,e_B)-g(\nabla_N\po N,e_B)(\po N)_A-g(\nabla_A\po N,N)(\po N)_B\\
&=& \po^2\th_{AB}-\po\th_{NB}(\po N)_A-\po\th_{AN}(\po N)_B-(\po^2N)_A\nabb_B a\\
&&-(\po N)_A\nabb_B(\po a)+(\po N)_A(\po N)_B\nabn a.
\eee
Together with the identities \eqref{om49}, \eqref{om50}, \eqref{toc3}, we obtain:
\bea\lab{zia1}
&& g(\nabla_A\po^2N,e_B)\\
\nn&=& \po^2\th_{AB}-(\po^2N)_A\nabb_B a-2(\po N)_A\nabb_B(\po a)+2(\po N)_A(\po N)_B\nabn a.
\eea
Next, we differentiate the identity \eqref{zaa8}. We obtain:
$$g(\nabla_A\po^2N, N)+g(\po^2N, \nabla_AN)=-2g(\nabla_A\po N, \po N).$$
Together with \eqref{frame1} and \eqref{om48}, we obtain:
\be\lab{zia2}
g(\nabla_A\po^2N, N) = -\th_{A \po^2N}-2\po\th_{A \po N}+(\po N)_A\nabb_{\po N} a.
\ee
Finally, differentiating \eqref{om50}, and using the commutator formula \eqref{commutom1}, and the identities \eqref{extth} and \eqref{om20}, we obtain:
\bea\lab{zia3}
&&\nabn\po^2N\\
\nn&=&-\th(\po^2N, e_A)e_A-\nabb(\po^2 a)+\nabla_{\po^2N} a+\nabn a\po^2N-\nabla_{\po N}\po N\\
\nn&&-\po\th(\po N, e_A)e_A+2\nabb_{\po N}(\po a)N+2\nabn(\po a)\po N+2\nabb_{\po N}( a)\po N.
\eea

Next, we estimate $\nabb^2\po^2N$. Differentiating \eqref{zia1} and \eqref{zia2}, we obtain:
\bee
&&\norm{\nabb^2\po^2N}_{\ll{2}}\\
&\les& \norm{\nabb\po^2\th}_{\ll{2}}+(\norm{\po^2N}_{\ll{\infty}}+\norm{\nabb\po^2N}_{\l{2}{4}})(\norm{\nabb a}_{\l{\infty}{4}}\\
&&+\norm{\nabb^2 a}_{\ll{2}}+\norm{\nabb\th}_{\ll{2}}+\norm{\th}_{\l{\infty}{4}})+(\norm{\po N}_{\ll{\infty}}+\norm{\nabb\po N}_{\l{\infty}{4}})\\
&&\times (\norm{\nabb\po a}_{\l{2}{4}}+\norm{\nabb^2\po a}_{\ll{2}}+\norm{\nabb\po\th}_{\ll{2}}+\norm{\th}_{\l{\infty}{4}})\\
&&+(\norm{\po N}_{\ll{\infty}}+\norm{\nabb\po N}_{\l{\infty}{4}})^2(\norm{\nabb\nabla  a}_{\ll{2}}+\norm{\nabla a}_{\l{2}{4}}).
\eee
Together with the Gagliardo-Nirenberg inequality \eqref{eq:GNirenberg}, the estimates \eqref{boot} \eqref{appboot} for $a$, the estimates \eqref{boot1} \eqref{appboot1} for $\th$, and the estimate \eqref{threomega1} for $\po a$, $\po\th$ and $\po N$, we obtain:
$$\norm{\nabb^2\po^2N}_{\ll{2}}\les \norm{\nabb\po^2\th}_{\ll{2}}+\ep(1+\norm{\po^2N}_{\ll{\infty}}+\norm{\nabb^2\po^2N}_{\ll{2}}),$$
and thus:
\be\lab{zia4}
\norm{\nabb^2\po^2N}_{\ll{2}}\les \norm{\nabb\po^2\th}_{\ll{2}}+\ep(1+\norm{\po^2N}_{\ll{\infty}}).
\ee

Next, we estimate $\nabn\po^2N$. In view of \eqref{zia3}, we have:
\be\lab{zia5}
\nabn\po^2N=-\nabb(\po^2 a)+2\nabn(\po a)\po N+H,
\ee
where $H$ is given by:
\bee
H &=& -\th(\po^2N, e_A)e_A+\nabla_{\po^2N} a+\nabn a\po^2N-\nabla_{\po N}\po N\\
\nn&&-\po\th(\po N, e_A)e_A+2\nabb_{\po N}(\po a)N+2\nabb_{\po N}( a)\po N.
\eee
We have:
\bee
&&\norm{\nabb H}_{\ll{2}}\\
&\les& (\norm{\po^2N}_{\ll{\infty}}+\norm{\nabb\po^2N}_{\l{2}{4}})(\norm{\th}_{\l{\infty}{4}}\\
&&+\norm{\nabb\th}_{\ll{2}}+\norm{\nabla a}_{\l{\infty}{4}}+\norm{\nabb\nabla a}_{\ll{2}})+\norm{\nabb\nabla_{\po N}\po N}_{\ll{2}}\\
&&+ (\norm{\po N}_{\ll{\infty}}+\norm{\nabb\po N}_{\l{\infty}{4}})(\norm{\po\th}_{\l{\infty}{4}}+\norm{\nabb\po\th}_{\ll{2}}\\
&&+\norm{\nabb\po a}_{\l{\infty}{4}}+\norm{\nabb^2\po a}_{\ll{2}}) +(\norm{\po N}_{\ll{\infty}}\\
&&+\norm{\nabb\po N}_{\l{\infty}{4}})^2(\norm{\nabb a}_{\l{\infty}{4}}+\norm{\nabb^2 a}_{\ll{2}}). 
\eee
Together with the Gagliardo-Nirenberg inequality \eqref{eq:GNirenberg}, the estimates \eqref{boot} \eqref{appboot} for $a$, the estimates \eqref{boot1} \eqref{appboot1} for $\th$, and the estimate \eqref{threomega1} for $\po a$, $\po\th$ and $\po N$, we obtain:
\be\lab{zia6}
\norm{\nabb H}_{\ll{2}}\les \ep(1+\norm{\po^2N}_{\ll{\infty}}+\norm{\nabb^2\po^2N}_{\ll{2}}).
\ee
Also, Lemma \ref{lemma:vacances:1} yields:
\be\lab{zia7}
\norm{\nabb(\po^2 a)}_{\lhs{2}{\frac{1}{2}}}\les \norm{\po^2 a}_{\lhs{2}{\frac{3}{2}}}.
\ee
The product estimate \eqref{prod15} implies:
$$\norm{\nabn(\po a)\po N}_{\lhs{2}{\frac{1}{2}}}\les \norm{\nabn(\po a)}_{\lhs{2}{\frac{1}{2}}}(\norm{\po N}_{\ll{\infty}}+\norm{\nabb\po N}_{\l{\infty}{2}}),$$
which together with the estimate \eqref{threomega1} for $\po a$ and $\po N$ yields:
\be\lab{zia8}
\norm{\nabn(\po a)\po N}_{\lhs{2}{\frac{1}{2}}}\les\ep.
\ee
Finally, \eqref{zia5}, \eqref{zia6}, \eqref{zia7} and \eqref{zia8} imply:
\be\lab{zia9}
\norm{\nabn\po^2N}_{\lhs{2}{\frac{1}{2}}}\les \norm{\po^2 a}_{\lhs{2}{\frac{3}{2}}}+\ep(1+\norm{\po^2N}_{\ll{\infty}}+\norm{\nabb^2\po^2N}_{\ll{2}}).
\ee

Next, we estimate the $\ll{\infty}$ norm of $\po^2N$. In view of Corollary \ref{cor:commLP4}, we have:
$$\norm{\po^2N}_{\ll{\infty}}\les \norm{\nabb^2\po^2N}_{\ll{2}}+\norm{\nabn\po^2N}_{\lhs{2}{\frac{1}{2}}}.$$
Together with \eqref{zia4} and \eqref{zia9}, we finally obtain:
\bea\lab{zia10}
&&\norm{\nabb^2\po^2N}_{\ll{2}}+\norm{\nabn\po^2N}_{\lhs{2}{\frac{1}{2}}}+\norm{\po^2N}_{\ll{\infty}}\\
\nn&\les& \norm{\nabb\po^2\th}_{\ll{2}}+\norm{\po^2 a}_{\lhs{2}{\frac{3}{2}}}+\ep.
\eea

Finally, \eqref{zaa15}, \eqref{toc20} and \eqref{zia10} yield:
\bea
\norm{\po^2 a}_{\lhs{2}{\frac{3}{2}}}+\norm{\po^2 a}_{\lhs{\infty}{\frac{1}{2}}}+\norm{\nabn\po^2 a}_{\lhs{2}{-\frac{1}{2}}}&&\\
\nn+\norm{\nabb\po^2\th}_{\ll{2}}+\norm{\nabb^2\po^2N}_{\ll{2}}+\norm{\nabn\po^2N}_{\lhs{2}{\frac{1}{2}}}&\les&\ep,
\eea
and: 
\be
\norm{\po^2N}_{\ll{\infty}}\les 1,
\ee
which concludes the proof of \eqref{threomega2}.

\subsection{Third order derivatives with respect to $\o$}\label{thirdorom}

The goal of this section is to prove \eqref{threomega3}. We first give an outline of the proof. We start 
with the derivation of an equation for $\po^3u$. Recall that  
div$(N)=\trt$, $N=\nabla u/|\nabla u|$, $a=1/|\nabla u|$ and $\trt=1-a+k_{NN}$, so that:
\begin{equation}\label{diffommm1}
\textrm{div}\left(\frac{\nabla u}{|\nabla u|}\right)=1-\frac{1}{|\nabla u|}+k_{NN}.
\end{equation}
Differentiating \eqref{diffommm1} three times with respect to $\o$ yields:
\be\lab{diffommm2}
(\nabn-a^{-1}\lap)\po^3u=\nabb\po^2 a+\cdots.
\ee
In view of the estimate \eqref{threomega2} for $\po^2 a$ and the parabolic estimate \eqref{parab26}, this 
suggests that $\po^3u$ satisfies the following estimate:
\be\lab{diffommm3}
\norm{\po^3u}_{\lhs{2}{\frac{5}{2}}}+\norm{\po^3u}_{\lhs{\infty}{\frac{3}{2}}}+\norm{\nabn\po^3u}_{\lhs{2}{\frac{1}{2}}}\les 1.
\ee
Now, since $\po^3u\in \lhs{\infty}{\frac{3}{2}}$ and $\p$ is 2-dimensional, we obtain that $\po^3u$ belongs to $\ll{\infty}$.

The rest of this section is as follows. We start by deriving the equations for $\po^3u$ and $\po^3N$.  Then, we prove the estimates for $\po^3u$. 

\begin{remark}\label{decay}
Note that $\po^3u=\po^3(x.\o)$ on $x.\o=-2$, which yields: 
\begin{equation}
|\po^3u|\sim |x|\textrm{ when }|x|\rightarrow +\infty\textrm{ on }x.\o=-2.
\end{equation}
This lack of decay is a problem when one tries to solve \eqref{diffommm2}. However, recall 
from section \ref{regx} that the final solution will be equal to $x.\o$ in the region $|x|\geq 2$ so that the estimate \eqref{threomega3} is clearly satisfied there. Thus, we may estimate $\varphi\po^3u$ instead of $\po^3u$, where $\varphi$ is a smooth function on $\s$ equal to 1 on $|x|\leq 2$, $\varphi>0$ on $\s$, and $\varphi\sim |x|^{-3}$ when $|x|$ goes to infinity. Then, $\varphi\po^3u$ is $L^2$ on $x.\o=-2$. Also, the lower order terms generated by commuting \eqref{diffomm2} with the multiplication by $\varphi$ are all under control since they are localized in a compact region of $|x|\geq 2$ where $u$ is explicitly given by $u=x\cdot\o$. In the rest of the section, we omit this 
detail and we assume that the decay of $\po^3u$ is sufficient at $x.\o=-2$.
\end{remark}

\begin{remark}\lab{bestpossible}
One may ask whether it is possible to obtain estimates for higher order derivatives of $u$ and $a$ with respect to $\o$. Consider first $\po^4 a$. Differentiating the equation \eqref{zaa4} for $\po^2 a$ twice would yield:
$$(\nabn -a^{-1}\lap)\po^4 a=\nabn^2\po^2 a+\cdots.$$
Now, we notice in Remark \ref{theend} that one can not obtain an estimate for $\nabn^2 a$, so that the above equation for $\po^4 a$ is useless. On the other hand, differentiating the equation \eqref{diffommm2} with respect to $\o$, we obtain:
\be\lab{diffommm4}
(\nabn-a^{-1}\lap)\po^4u=\nabb\po^3 a+\cdots.
\ee
Now, differentiating $\nabla u=a^{-1}N$ three times with respect to $\o$, we obtain:
$$\po^3 a=-a\nabn\po^3u+\cdots$$
which together with \eqref{diffommm3} suggests that $\po^3 a$ belongs to $\lhs{2}{\frac{1}{2}}$. Thus, in view of \eqref{diffommm4} and the parabolic estimate \eqref{parab24}, we see that $\po^4u$ is at best in $\lhs{\infty}{\frac{1}{2}}$ which does not embed in $\ll{\infty}$. Interpolating with \eqref{diffommm3}, we see that the best estimate we might hope for is:
\be\lab{diffommm5}
\po^{3+\delta}u\in\ll{\infty}\textrm{ for all }\delta<\frac{1}{2}.
\ee
\end{remark}

\begin{remark}\lab{loeb2011}
Note in conjunction with Remark \ref{paslinfty} that the estimate \eqref{diffommm5} would still be at least half a derivative away from allowing to apply the $TT^*$ method in step {\bf C2}. 
\end{remark}

\subsubsection{Derivation of the equation for $\po^3N$ and $\po^3u$}

We first establish the link between $\po^2\log(a)$ and $\po^2u$:
\begin{lemma}\label{omm1b}
$\po^2\log(a)$ and $\po^2u$ are linked by the following equality:
\begin{equation}\label{omm2b}
\po^2\log(a)=-a\nabn(\po^2u)-|\po N|^2+(\po\log(a))^2.
\end{equation}
\end{lemma}

\begin{proof}
We start with the equality $\nabla u=a^{-1}N$. Differentiating it with respect to $\o$, we obtain:
\begin{equation}\label{omm3}
\nabla\po u=a^{-1}\po N-a^{-1}\po\log(a) N,
\end{equation}
which together with \eqref{comom1ter} yields:
\begin{equation}\label{omm4}
\left\{\begin{array}{l}
\nabb\po u=a^{-1}\po N,\\
\nabn\po u=-a^{-1}\po\log(a).
\end{array}\right.
\end{equation}

Differentiating the second equation of \eqref{omm4} with respect to $\o$ yields:
\begin{equation}\label{omm10b}
\begin{array}{l}
\nabn\po^2u+\nabb_{\po N}\po u=-a^{-1}\po^2\log(a)+a^{-1}(\po\log(a))^2.
\end{array}
\end{equation}
Together with \eqref{omm4}, this yields \eqref{omm2b}.
\end{proof}

Next, we establish the link between $\po^3N$ and $\po^3u$:
\begin{lemma}\label{ommm6bis} 
$\po^3N$ and $\po^3u$ are linked by the following equality:
\begin{equation}\label{ommm6ter}
\begin{array}{ll}
\ds\po^3N= & a\nabb(\po^3u)+(3\po^2\log(a)-3(\po\log(a))^2)\po N+3\po\log(a)\po^2N\\
& \ds +(-3g(\po N,\po^2N)+3\po\log(a)|\po N|^2)N.
\end{array}
\end{equation}
\end{lemma}

\begin{proof}
Differentiating the first equation of \eqref{omm4} with respect to $\o$ and using \eqref{commutom1} yields:
\begin{equation}\label{omm10}
\begin{array}{l}
\nabb\po^2u-\nabb_{\po N}(\po u) N-\nabn(\po u)\po N=a^{-1}\po^2N-a^{-1}\po\log(a)\po N.
\end{array}
\end{equation}
Together with \eqref{omm4}, this yields:
\begin{equation}\label{omm2}
\po^2N=a\nabb(\po^2u)+2\po\log(a)\po N -|\po N|^2N.
\end{equation}

Differentiating \eqref{omm2} with respect to $\o$, we obtain:
\begin{equation}\label{ommm7}
\begin{array}{ll}
\ds\po^3N= & a\po(\nabb(\po^2u))+a\po\log(a)\nabb(\po^2u)+2\po^2\log(a)\po N+2\po\log(a)\po^2N\\
& \ds -2g(\po^2N,\po N)N-|\po N|^2\po N.
\end{array}
\end{equation}
\eqref{commutom1}, \eqref{omm2}, \eqref{omm2b} and \eqref{ommm7} yield:
\begin{equation}\label{ommm8}
\begin{array}{lll}
\ds\po^3N & = & \ds a(\nabb(\po^3u)-\nabb_{\po N}(\po^2u)N-\nabn(\po^2u)\po N)+a\po\log(a)\nabb(\po^2u)\\
& & \ds +2\po^2\log(a)\po N+2\po\log(a)\po^2N -2g(\po^2N,\po N)N-|\po N|^2\po N\\
& = & \ds a\nabb(\po^3u)+(3\po^2\log(a)-3(\po\log(a))^2)\po N+3\po\log(a)\po^2N\\
& & \ds +(-3g(\po N,\po^2N)+3\po\log(a)|\po N|^2)N,
\end{array}
\end{equation}
which implies \eqref{ommm6ter}. 
\end{proof}

We finally derive an equation for $\po^3u$: 
\begin{lemma}\label{ommm9}
$\po^3u$ satisfies the following equation:
\bea\label{ommm10}
&&(\nabn-a^{-1}\lap)\po^3u\\
\nn&=&2a^{-2}\nabla_{\po^3N}(\log(a))-2a^{-2}k(N,\po^3N)-a^{-1}\po\log(a)\lap\po^2u+2a^{-1}\nabb_{\po N}\nabn\po^2u\\
\nn&& +\po^2\log(a)(3a^{-2}\po\trt-2a^{-1}\po\log(a) -8a^{-2}\nabb_{\po N}\log(a)+4a^{-2}k(N,\po N))\\
\nn&& +4a^{-2}\nabb_{\po N}\po^2\log(a)+6a^{-2}\nabla_{\po^2N}\po\log(a)-12a^{-2}\po\log(a)\nabla_{\po^2N}\log(a)\\
\nn&& +12a^{-2}\po\log(a)k(N,\po^2N)+6a^{-2}\th(\po N,\po^2N)-3a^{-2}\trt g(\po N,\po^2N)\\
\nn&& -6a^{-2}k(\po N,\po^2N)-3a^{-1}g(\po N,\po^2N)-5a^{-2}(\po\log(a))^2\po\trt \\
\nn&& +2a^{-2}\po\th(\po N,\po N)-16a^{-2}\po\log(a)\nabb_{\po N}(\po\log(a))-2a^{-1}(\po\log(a))^3\\
\nn&& +a^{-2}(\po\log(a))^2(16\nabb_{\po N}(\log(a))-8k(N,\po N))+\po\log(a)(4a^{-2}\trt |\po N|^2\\
\nn&& -8a^{-2}\th(\po N,\po N)+12a^{-2}k(\po N,\po N)+3a^{-1}|\po N|^2).
\eea
\end{lemma}

\begin{proof}
We start by obtaining an equation for $\po u$. We differentiate the first equation of 
\eqref{struct1} by $\o$:
\begin{equation}\label{omm5}
\po\trt-2k(N,\po N)=-a\po\log(a).
\end{equation}
By \eqref{frame1}, we have $\trt=$div$(N)$, and differentiating with respect to $\o$, we obtain:
\begin{equation}\label{omm6}
\po\trt=\textrm{div}(\po N).
\end{equation}
Now, for any vectorfield $X$ tangent to $\p$, we have:
\begin{equation}\label{omm7}
\textrm{div}(X)=\divb(X)+\nabb_X\log(a),
\end{equation}
which together with \eqref{comom1ter} and \eqref{omm6} yields:
\begin{equation}\label{omm8}
\po\trt=\divb(\po N)+\nabb_{\po N}\log(a).
\end{equation}
\eqref{omm4}, \eqref{omm5} and \eqref{omm8} imply:
\begin{equation}\label{omm9}
(\nabn -a^{-1}\lap)\po u=a^{-1}\nabb\log(a)\nabb\po u+a^{-2}\nabb_{\po N}\log(a)-2a^{-2}k(N,\po N).
\end{equation}
which together with the first equation of \eqref{omm4} yields:
\begin{equation}\label{omm11}
(\nabn -a^{-1}\lap)\po u=2a^{-2}\nabb_{\po N}\log(a)-2a^{-2}k(N,\po N).
\end{equation}

We differentiate \eqref{omm11} with respect to $\o$ to obtain an equation for $\po^2$:
\bea\label{omm12}
\nn&&\nabn(\po^2u)+\nabb_{\po N}(\po u)-a^{-1}\lap(\po^2u)-a^{-1}[\po,\lap](\po u)+a^{-1}\po(\log(a))\lap(\po u)\\
\nn&=&2a^{-2}\nabb_{\po N}(\po\log(a))+2a^{-2}\nabla_{\po^2N}(\log(a))-4a^{-2}\po\log(a)\nabb_{\po N}\log(a)\\
&& -2a^{-2}k(N,\po^2N)-2a^{-2}k(\po N,\po N)+4a^{-2}\po\log(a)k(N,\po N).
\eea
The first equation of \eqref{omm4} and \eqref{omm8} yield:
\begin{equation}\label{omm13}
\lap(\po u)=a^{-1}\po\trt-2a^{-1}\nabb_{\po N}\log(a).
\end{equation}
\eqref{commutom3}, \eqref{omm4}, \eqref{omm12} and \eqref{omm13} imply:
\begin{equation}\label{omm14}
\begin{array}{r}
\ds (\nabn-a^{-1}\lap)\po^2u= -2a^{-1}\nabla^2(\po u)(N,\po N)+2a^{-2}\nabla_{\po^2N}(\log(a))\\
\ds -2a^{-2}k(N,\po^2N)+2a^{-2}\po(\log(a))\po\trt+2a^{-2}\nabb_{\po N}(\po\log(a))\\
\ds -6a^{-2}\po\log(a)\nabb_{\po N}\log(a)+4a^{-2}\po\log(a)k(N,\po N)\\
\ds -a^{-2}\trt|\po N|^2-2a^{-2}k(\po N,\po N)-a^{-1}|\po N|^2.
\end{array}
\end{equation}
Using \eqref{omm4}, we rewrite $\nabla^2(\po u)(N,\po N)$ as:
\begin{equation}\label{omm15}
\begin{array}{lll}
\ds \nabla^2(\po u)(N,\po N) & = & \ds\nabla_{\po N}(\nabla_N(\po u))-\nabla_{\nabla_{\po N}N}(\po u)\\
& = & \ds\nabb_{\po N}(-a^{-1}\po\log(a))-\th(\po N,e_A)\nabb_A(\po u)\\
& = & \ds -a^{-1}\nabb_{\po N}(\po\log(a))+a^{-1}\po\log(a)\nabb_{\po N}(\log(a))\\
& & \ds -a^{-1}\th(\po N,\po N).
\end{array}
\end{equation}
\eqref{omm14} and \eqref{omm15} yield:
\begin{equation}\label{omm2bis}
\begin{array}{r}
\ds (\nabn-a^{-1}\lap)\po^2u= 2a^{-2}\nabla_{\po^2N}(\log(a))-2a^{-2}k(N,\po^2N)+2a^{-2}\po(\log(a))\po\trt\\
\ds +4a^{-2}\nabb_{\po N}(\po\log(a))-8a^{-2}\po\log(a)\nabb_{\po N}\log(a)+4a^{-2}\po\log(a)k(N,\po N)\\
\ds -a^{-2}\trt|\po N|^2+2a^{-2}\th(\po N,\po N)-2a^{-2}k(\po N,\po N)-a^{-1}|\po N|^2.
\end{array}
\end{equation}

Differentiating \eqref{omm2bis} with respect to $\o$, we obtain:
\begin{equation}\label{ommm11}
\begin{array}{ll}
& \ds (\nabn-a^{-1}\lap)\po^3u+\nabb_{\po N}\po^2u+a^{-1}\po\log(a)\lap\po^2u-a^{-1}[\po,\lap]\po^2u\\
\ds = & \ds 2a^{-2}\nabla_{\po^3N}(\log(a))-2a^{-2}k(N,\po^3N)+2a^{-2}\po\log(a)\po^2\trt\\
& \ds +\po^2\log(a)(2a^{-2}\po\trt-8a^{-2}\nabb_{\po N}\log(a)+4a^{-2}k(N,\po N))\\
& \ds +4a^{-2}\nabb_{\po N}\po^2\log(a)+6a^{-2}\nabla_{\po^2N}\po\log(a)-12a^{-2}\po\log(a)\nabla_{\po^2N}\log(a)\\
& \ds +8a^{-2}\po\log(a)k(N,\po^2N)+4a^{-2}\th(\po N,\po^2N)-2a^{-2}\trt g(\po N,\po^2N)\\
& \ds -6a^{-2}k(\po N,\po^2N)-2a^{-1}g(\po N,\po^2N)-4a^{-2}(\po\log(a))^2\po\trt \\
& \ds -a^{-2}\po\trt|\po N|^2+2a^{-2}\po\th(\po N,\po N)-16a^{-2}\po\log(a)\nabb_{\po N}(\po\log(a))\\
& \ds +a^{-2}(\po\log(a))^2(16\nabb_{\po N}(\log(a))-8k(N,\po N))+\po\log(a)(2a^{-2}\trt |\po N|^2\\
& \ds -4a^{-2}\th(\po N,\po N)+8a^{-2}k(\po N,\po N)+a^{-1}|\po N|^2).
\end{array}
\end{equation}
Using \eqref{omm2}, we have:
\begin{equation}\label{ommm12}
\nabb_{\po N}\po^2u=a^{-1}g(\po N,\po^2N)-2a^{-1}\po\log(a)|\po N|^2.
\end{equation}
\eqref{commutom3} and \eqref{comom6} yield:
\begin{equation}\label{ommm13}
\begin{array}{ll}
\ds [\po,\lap]\po^2u & =-2\nabla^2\po^2u(N,\po N)-\po\trt\nabn\po^2u-\trt\nabb_{\po N}\po^2u\\
\ds & =-2\nabb_{\po N}\nabn\po^2u+2\th(\po N,\nabb\po^2u)-\po\trt\nabn\po^2u-\trt\nabb_{\po N}\po^2u,
\end{array}
\end{equation}
which together with \eqref{omm2} and \eqref{omm2b} implies:
\begin{equation}\label{ommm14}
\begin{array}{ll}
\ds [\po,\lap]\po^2u= & -2\nabb_{\po N}\nabn\po^2u+a^{-1}\po\trt\po^2\log(a)+2a^{-1}\th(\po N,\po^2N)\\
& \ds -a^{-1}\trt g(\po N,\po^2N)-a^{-1}\po\trt(\po\log(a))^2+a^{-1}\po\trt |\po N|^2\\
& \ds -4a^{-1}\po\log(a)\th(\po N,\po N)+2\trt\po\log(a)|\po N|^2.
\end{array}
\end{equation}
Differentiating the first equation of \eqref{struct1} twice with respect to $\o$, we obtain:
\begin{equation}\label{ommm15}
\po^2\trt=-a\po^2\log(a)-a(\po\log(a))^2+2k(N,\po^2N)+2k(\po N,\po N).
\end{equation}
Finally, \eqref{ommm11}, \eqref{ommm12}, \eqref{ommm14} and \eqref{ommm15} imply \eqref{ommm10}.
\end{proof}

\subsubsection{Estimates for $\po^3u$}

The equation \eqref{ommm10} takes the form:
\begin{equation}\label{ommm16}
\ds (\nabn-a^{-1}\lap)\po^3u=h,
\end{equation}
where $h$ is given by:
\begin{displaymath}
\begin{array}{l}
\ds h= 2a^{-2}\nabla_{\po^3N}(\log(a))-2a^{-2}k(N,\po^3N)\\
\ds -a^{-1}\po\log(a)\lap\po^2u+2a^{-1}\nabb_{\po N}\nabn\po^2u+\po^2\log(a)(3a^{-2}\po\trt-2a^{-1}\po\log(a)\\
\ds -8a^{-2}\nabb_{\po N}\log(a)+4a^{-2}k(N,\po N))+4a^{-2}\nabb_{\po N}\po^2\log(a)\\
\ds +6a^{-2}\nabla_{\po^2N}\po\log(a)-12a^{-2}\po\log(a)\nabla_{\po^2N}\log(a)\\
\ds +12a^{-2}\po\log(a)k(N,\po^2N)+6a^{-2}\th(\po N,\po^2N)-3a^{-2}\trt g(\po N,\po^2N)\\
\ds -6a^{-2}k(\po N,\po^2N)-3a^{-1}g(\po N,\po^2N)-5a^{-2}(\po\log(a))^2\po\trt \\
\ds +2a^{-2}\po\th(\po N,\po N)-16a^{-2}\po\log(a)\nabb_{\po N}(\po\log(a))-2a^{-1}(\po\log(a))^3\\
\ds +a^{-2}(\po\log(a))^2(16\nabb_{\po N}(\log(a))-8k(N,\po N))+\po\log(a)(4a^{-2}\trt |\po N|^2\\
\ds -8a^{-2}\th(\po N,\po N)+12a^{-2}k(\po N,\po N)+3a^{-1}|\po N|^2).
\end{array}
\end{displaymath}
Let $0<b<\frac{1}{2}$. We estimate the norm of $h$ in $\lhs{2}{b}$. Using the product estimate \eqref{fichtre}, we have:
\bee
&&\norm{h}_{\lhs{2}{b}}\\
&\les& \norm{a^{-2}\nabla_{\po^3N}(\log(a))}_{\lhs{2}{b}}+\norm{\po^3N}_{\lhs{2}{1}}\norm{a^{-2}k}_{\lhs{\infty}{\frac{1}{2}}}\\
&&+\norm{\po a}_{\lhs{\infty}{\frac{1}{2}}}\norm{\lap\po^2u}_{\lhs{2}{1}}+\norm{a^{-1}\nabb_{\po N}\nabn\po^2u}_{\lhs{2}{b}}\\
&&+\norm{\po^2\log(a)}_{\lhs{2}{1}}(\norm{a^{-2}\po\trt}_{\lhs{\infty}{\frac{1}{2}}}+\norm{a^{-1}\po\log(a)}_{\lhs{\infty}{\frac{1}{2}}}\\
&&+\norm{a^{-2}\nabb_{\po N}\log(a)}_{\lhs{\infty}{\frac{1}{2}}}+\norm{a^{-2}k(N,\po N)}_{\lhs{\infty}{\frac{1}{2}}})\\
&&+\norm{a^{-2}\nabb_{\po N}\po^2\log(a)}_{\lhs{2}{b}}+\norm{a^{-2}\nabla_{\po^2N}\po\log(a)}_{\lhs{2}{b}}\\
&&+\norm{a^{-2}\po\log(a)}_{\lhs{\infty}{1}}\norm{\po^2N}_{\lhs{\infty}{\frac{1}{2}}}\norm{\nabla\log(a)}_{\lhs{2}{1}}\\
&&+\norm{\po^2N}_{\lhs{2}{1}}(\norm{a^{-2}\po\log(a)}_{\lhs{\infty}{1}}\norm{k}_{\lhs{\infty}{\frac{1}{2}}}+\norm{a^{-2}\th\po N}_{\lhs{\infty}{\frac{1}{2}}}\\
&&+\norm{a^{-2}\trt\po N}_{\lhs{\infty}{\frac{1}{2}}}+\norm{a^{-2}k\po N}_{\lhs{\infty}{\frac{1}{2}}})\\
&&+\norm{a^{-1}\po\log(a)}_{\lhs{\infty}{1}}^2\norm{\po\trt}_{\lhs{2}{\frac{1}{2}}}+\norm{\po\th}_{\lhs{2}{\frac{1}{2}}}\norm{a^{-1}\po N}^2_{\lhs{\infty}{1}}\\
&&+\norm{\po\log(a)}_{\lhs{\infty}{\frac{1}{2}}}\norm{a^{-2}\po N}_{\l{\infty}{1}}\norm{\nabb(\po\log(a))}_{\lhs{2}{1}}\\
&&+\norm{a^{-1}\po\log(a)}_{\lhs{2}{\frac{1}{2}}}\norm{\po\log(a)}_{\lhs{\infty}{1}}^2+\norm{a^{-1}\po\log(a)}^2_{\lhs{\infty}{1}}\\
&&\times(\norm{\nabb_{\po N}(\log(a)}_{\lhs{2}{\frac{1}{2}}}+\norm{k\po N}_{\lhs{2}{\frac{1}{2}}})\\
&&+\norm{\po\log(a)}_{\lhs{\infty}{1}}(\norm{\th}_{\lhs{\infty}{\frac{1}{2}}}+\norm{k}_{\lhs{\infty}{\frac{1}{2}}})\norm{a^{-2}(\po N)^2}_{\lhs{2}{1}}.
\eee
Together with the embedding \eqref{clp12}, the estimates \eqref{boot} and \eqref{appboot} for $a$, the estimate \eqref{boot1} and \eqref{appboot1} for $\th$, the estimates \eqref{small2} \eqref{appsmall2} for $k$, the estimate \eqref{threomega1} for $\po a$, $\po N$ and $\po\th$, and the estimate \eqref{threomega2} for $\po^2\log(a)$ and $\po^2N$, we obtain:
\bea\lab{ommm17}
&&\norm{h}_{\lhs{2}{b}}\\
\nn&\les& \norm{a^{-2}\nabla_{\po^3N}(\log(a))}_{\lhs{2}{b}}+\ep\norm{\lap\po^2u}_{\lhs{2}{1}}+\norm{a^{-1}\nabb_{\po N}\nabn\po^2u}_{\lhs{2}{b}}\\
\nn&&+\norm{a^{-2}\nabb_{\po N}\po^2\log(a)}_{\lhs{2}{b}}+\norm{a^{-2}\nabla_{\po^2N}\po\log(a)}_{\lhs{2}{b}}+\ep+\ep\norm{\po^3N}_{\lhs{2}{1}}.
\eea

Next, we estimate the various terms in the right-hand side of \eqref{ommm17} starting with the fifth one. Using the decomposition \eqref{zaa9} of $\po^2N$, we have:
$$a^{-2}\nabla_{\po^2N}\po\log(a)=-a^{-2}|\po N|^2\nabn\po\log(a)+a^{-2}\nabb_{\Pi(\po^2N)}\po\log(a).$$
Together with the product estimate \eqref{fichtre}, this yields:
\bea\lab{ommm18}
&&\norm{a^{-2}\nabla_{\po^2N}\po\log(a)}_{\lhs{2}{b}}\\
\nn&\les& \norm{a^{-1}\po N}_{\lhs{\infty}{1}}^2\norm{\nabn\po\log(a)}_{\lhs{2}{\frac{1}{2}}}+\norm{a^{-2}\po^2N}_{\lhs{\infty}{\frac{1}{2}}}\norm{\nabb\po\log(a)}_{\lhs{2}{1}}\\
\nn&\les& \ep,
\eea
where we used in the last inequality the embedding \eqref{clp12}, the estimate \eqref{boot} for $a$, the estimate \eqref{threomega1} for $\po a$, $\po N$, and the estimate \eqref{threomega2} for $\po^2N$.

Next, we estimate the first term in the right-hand side of \eqref{ommm17}. We first provide a decomposition of $\po^3N$. Differentiating \eqref{zaa8} with respect to $\o$, we obtain:
$$g(\po^3N,N)=-3g(\po^2N, \po N),$$
which yields:
\begin{equation}\label{ommm19}
\po^3 N=\Pi(\po^3N)-3g(\po^2N, \po N)N.
\end{equation}
We obtain:
$$a^{-2}\nabla_{\po^3N}(\log(a))=-3a^{-2}g(\po^2N, \po N)\nabn(\log(a))+a^{-2}\nabb_{\Pi(\po^3N)}(\log(a)).$$
Together with the product estimate \eqref{fichtre}, this yields:
\bea
\nn\norm{a^{-2}\nabla_{\po^3N}(\log(a))}_{\lhs{2}{b}}&\les& \norm{\po^2N}_{\lhs{\infty}{\frac{1}{2}}}\norm{a^{-2}\po N}_{\lhs{\infty}{1}}\norm{\nabn(\log(a))}_{\lhs{2}{1}}\\
\nn&&+\norm{\po^3N}_{\lhs{2}{1}}\norm{a^{-2}\nabb\log(a)}_{\lhs{\infty}{\frac{1}{2}}}\\
\lab{ommm20}&\les& \ep(1+\norm{\po^3N}_{\lhs{2}{1}}),
\eea
where we used in the last inequality the embedding \eqref{clp12}, the estimates \eqref{boot} \eqref{appboot} for $a$, the estimate \eqref{threomega1} for $\po N$, and the estimate \eqref{threomega2} for $\po^2N$.

Next, we estimate the fourth term in the right-hand side of \eqref{ommm17}. We have:
$$a^{-2}\nabb_{\po N}\po^2\log(a)=\divb(a^{-2}\po^2\log(a)\po N)-\divb(a^{-2}\po N)\po^2\log(a).$$
Together with the product estimates \eqref{prod11} and \eqref{fichtre}, this yields:
\bea\lab{ommm21}
&&\norm{a^{-2}\nabb_{\po N}\po^2\log(a)}_{\lhs{2}{b}}\\
\nn&\les& \norm{\divb(a^{-2}\po^2\log(a)\po N)}_{\lhs{2}{b}}+\norm{\divb(a^{-2}\po N)\po^2\log(a)}_{\lhs{2}{b}}\\
\nn&\les& (\norm{\po^2\log(a)}_{\lhs{2}{\frac{3}{2}}}+\norm{\po^2\log(a)}_{\lhs{\infty}{\frac{1}{2}}})(\norm{a^{-2}\po N}_{\ll{\infty}}\\
\nn&&+\norm{\nabb(a^{-2}\po N)}_{\l{\infty}{2}}+\norm{\divb(a^{-2}\po N)}_{\lhs{2}{1}})\\
\nn&\les& \ep,
\eea
where we used in the last inequality the estimate \eqref{boot} for $a$, the estimate \eqref{threomega1} for $\po N$, and the estimate \eqref{threomega2} for $\po^2\log(a)$.

Next, we estimate the second term in the right-hand side of \eqref{ommm17}. In view of \eqref{omm2}, we have:
$$\Pi(\po^2N)=a\nabb(\po^2u)+2\po\log(a)\po N.$$
Differentiating, we obtain:
$$\divb(\Pi(\po^2N))=a\lap(\po^2u)+\nabb(a)\c\nabb(\po^2u)+2\nabb_{\po N}(\po\log(a))+2\po\log(a)\divb(\po N),$$
which together with \eqref{omm2} implies:
\bee
\lap(\po^2u)&=&a^{-1}\divb(\Pi(\po^2N))-a^{-2}\nabb_{\Pi(\po^2N)}a+2a^{-1}\nabb_{\po N}(a)\po\log(a)\\
&&-2a^{-1}\nabb_{\po N}(\po\log(a))-2a^{-1}\po\log(a)\divb(\po N).
\eee
This yields:
\bea\lab{ommm22}
&&\norm{\lap(\po^2u)}_{\lhs{2}{1}}\\
\nn&\les& (\norm{a^{-1}}_{\ll{\infty}}+\norm{\nabb(a^{-1})}_{\l{\infty}{4}})(\norm{\divb(\Pi(\po^2N))}_{\lhs{2}{1}}\\
\nn&&+\norm{\nabb^2\po\log(a)}_{\ll{2}}\norm{\po N}_{\ll{\infty}})+\norm{\nabb^2a}_{\ll{2}}(\norm{\po^2N}_{\ll{\infty}}\\
\nn&&+\norm{\po\log(a)}_{\ll{\infty}}\norm{\po N}_{\ll{\infty}})+\norm{\nabb^2\po\log(a)}_{\ll{2}}\norm{a^{-1}\po N}_{\ll{\infty}}\\
\nn&&+ \norm{\po\log(a)}_{\ll{\infty}}\norm{\nabb\po N}_{\lhs{2}{1}}+\norm{\nabb^2\po\log(a)}_{\ll{2}}\norm{\nabb\po N}_{\l{\infty}{4}}\\
\nn&\les&\ep,
\eea
where we used in the last inequality the estimates \eqref{boot} \eqref{appboot1} for $a$, the estimate \eqref{threomega1} for $\po a$ and $\po N$, and the estimate \eqref{threomega2} for $\po^2N$.

Next, we estimate the third term in the right-hand side of \eqref{ommm17}. Differentiating the identity \eqref{omm2b}, we have:
\bee
&& a^{-1}\nabb_{\po N}\nabn\po^2u\\
&=& a^{-1}\nabb_{\po N}\left(-a^{-1}\po^2\log(a)-a^{-1}|\po N|^2+a^{-1}(\po\log(a))^2\right)\\
&=& -a^{-2}\nabb_{\po N}(\po^2\log(a))+a^{-3}\nabb_{\po N}\log(a)(\po^2\log(a)+|\po N|^2-(\po\log(a))^2)\\
&&-2a^{-2}\po N\c\nabb_{\po N}\po N+2a^{-2}\po\log(a)\nabb_{\po N}(\po\log(a)).
\eee
Together with the product estimate \eqref{fichtre}, we obtain:
\bea\lab{ommm23}
&&\norm{a^{-1}\nabb_{\po N}\nabn\po^2u}_{\lhs{2}{b}}\\
\nn&\les& \norm{a^{-2}\nabb_{\po N}(\po^2\log(a))}_{\lhs{2}{b}}+\norm{a^{-3}\nabb_{\po N}\log(a)}_{\lhs{\infty}{\frac{1}{2}}}(\norm{\po^2\log(a)}_{\lhs{2}{1}}\\
\nn&&+\norm{\po N}_{\lhs{\infty}{1}}\norm{\po N}_{\lhs{2}{1}}+\norm{\po\log(a)}_{\lhs{\infty}{1}}\norm{\po\log(a)}_{\lhs{2}{1}})\\
\nn&&+\norm{\po N}_{\lhs{\infty}{1}}\norm{\nabb\po N}_{\lhs{2}{1}}\\
\nn&&+\norm{\po\log(a)}_{\lhs{\infty}{\frac{1}{2}}}\norm{\nabb\po\log(a)}_{\lhs{2}{1}}\norm{\po N}_{\lhs{\infty}{1}}\\
\nn&\les&\ep,
\eea
where we used in the last inequality the estimate \eqref{ommm21}, the embedding \eqref{clp12}, the estimates \eqref{boot} \eqref{appboot} for $a$, the estimate \eqref{threomega1} for $\po a$ and $\po N$, and the estimate \eqref{threomega2} for $\po^2\log(a)$.

Finally, in view of \eqref{ommm17}, \eqref{ommm18}, \eqref{ommm20}, \eqref{ommm21}, \eqref{ommm22} and \eqref{ommm23}, we obtain:
\be\lab{ommm24}
\norm{h}_{\lhs{2}{b}}\les \ep+\ep\norm{\po^3N}_{\lhs{2}{1}}.
\ee
In view of the equation \eqref{ommm16} for $\po^3u$, the estimate \eqref{ommm24}, and the estimate \eqref{parab26} for parabolic equations, we obtain:
\be\lab{ommm25}
\norm{\po^3u}_{\lhs{2}{2+b}}+\norm{\po^3u}_{\lhs{\infty}{1+b}}+\norm{\nabn\po^3u}_{\lhs{2}{b}}\les  \ep+\ep\norm{\po^3N}_{\lhs{2}{1}},
\ee
 for any $0<b<\frac{1}{2}$.

Next, we estimate $\po^3N$. Recall \eqref{ommm6ter}:
\bee
\po^3N &= & a\nabb(\po^3u)+(3\po^2\log(a)-3(\po\log(a))^2)\po N+3\po\log(a)\po^2N\\
&&+(-3g(\po N,\po^2N)+3\po\log(a)|\po N|^2)N.
\eee
Together with the Gagliargdo-Nirenberg inequality \eqref{eq:GNirenberg}, this yields:
\bea\lab{ommm26}
&&\norm{\po^3N}_{\lhs{2}{1}}\\
\nn&\les& (\norm{a}_{\ll{\infty}}+\norm{\nabb a}_{\l{\infty}{4}})\norm{\po^3u}_{\lhs{2}{2}}\\
\nn&&+(\norm{\po^2\log(a)}_{\lhs{2}{1}}+(\norm{\po\log(a)}_{\lhs{\infty}{1}}+\norm{\po\log(a)}_{\ll{\infty}})^2)\\
\nn&&\times(\norm{\po N}_{\lhs{\infty}{1}}+\norm{\po N}_{\ll{\infty}})+(\norm{\po\log(a)}_{\lhs{\infty}{1}}+\norm{\po\log(a)}_{\ll{\infty}}\\
\nn&&+\norm{\po N}_{\lhs{\infty}{1}}+\norm{\po N}_{\ll{\infty}})(\norm{\po^2N}_{\lhs{2}{1}}+\norm{\po^2N}_{\ll{\infty}})\\
\nn&&+(\norm{\po\log(a)}_{\lhs{\infty}{1}}+\norm{\po\log(a)}_{\ll{\infty}})(\norm{\po N}_{\lhs{\infty}{1}}+\norm{\po N}_{\ll{\infty}})^2\\
\nn&\les& \ep(1+\norm{\po^3u}_{\lhs{2}{2}}),
\eea
where we used in the last inequality the estimates \eqref{boot} \eqref{appboot1} for $a$, the estimate \eqref{threomega1} for $\po a$ and $\po N$, and the estimate \eqref{threomega2} for $\po^2\log(a)$ and $\po^2N$. 
\eqref{ommm25} and \eqref{ommm26} imply:
$$\norm{\po^3u}_{\lhs{2}{2+b}}+\norm{\po^3u}_{\lhs{\infty}{1+b}}+\norm{\nabn\po^3u}_{\lhs{2}{b}}\les  \ep(1+\norm{\po^3u}_{\lhs{2}{2}}),$$
 for any $0<b<\frac{1}{2}$. This yields:
\be\lab{ommm27}
\norm{\po^3u}_{\lhs{2}{2+b}}+\norm{\po^3u}_{\lhs{\infty}{1+b}}+\norm{\nabn\po^3u}_{\lhs{2}{b}}\les  \ep,
\ee
 for any $0<b<\frac{1}{2}$. Now, the strong Bernstein inequality for scalars \eqref{eq:strongbernscalarbis} yields:
\bee
\norm{\po^3u}_{\ll{\infty}}&\les& \sum_{j\geq 0}\norm{P_j\po^3u}_{\ll{\infty}}\\
&\les& \sum_{j\geq 0}2^j\norm{P_j\po^3u}_{\l{\infty}{2}}\\
&\les& \left(\sum_{j\geq 0}2^{-jb}\right)\norm{\po^3u}_{\lhs{\infty}{1+b}}\\
&\les& \norm{\po^3u}_{\lhs{\infty}{1+b}},
\eee
where the last inequality hods for any $b>0$. Together with \eqref{ommm27}, we finally obtain:
$$\norm{\po^3u}_{\ll{\infty}}\les 1.$$
This concludes the proof of \eqref{threomega3}.

\section{A global coordinate system on $\p$ and $\Sigma$}\label{sec:globalcoord}

The inequalities in section \ref{sec:ineq} and \ref{sec:addition} have been derived under the assumption that $\p$ can be covered by a finite number of charts satisfying the conditions \eqref{eq:coordchart} 
and \eqref{eq:gammaL2} such that the constant $c>0$ in \eqref{eq:coordchart} and \eqref{eq:gammaL2} and the number of charts is independent of $u$. In this section, we prove that a covering of $\p$ by such  
coordinate systems exists. We first prove the existence of a global coordinate system on $\p$, which corresponds 
to the proof of Proposition \ref{gl0}. We then show that \eqref{eq:coordchart} and \eqref{eq:gammaL2} hold for this global coordinate system on $\p$ with a constant $c>0$ independent of $u$. Finally, we also introduce a global coordinate system on $\s$ for which we control the determinant of the corresponding Jacobian, which corresponds to the proof of Proposition \ref{gl20}.

\subsection{Proof of Proposition \ref{gl0}}

Recall the definition \eqref{gl1} of $\Phi_u:\p\rightarrow T_\o\S$:
$$\Phi_u(x):=\po u(x,\o).$$
where $T_\o\S$ is the tangent space to $\S$ at $\o$.\\

{\bf step1} $\Phi_u$ is a local $C^1$ diffeomorphism\\ 

We first prove that $\Phi_u$ is a local $C^1$ diffeomorphism. Using \eqref{omm4} we obtain a formula for $d\Phi_u$:
\begin{equation}\label{gl1bis}
d\Phi_u=\nabb\po u=a^{-1}\po N.
\end{equation}
In particular, if $e_1, e_2$ is an orthonormal frame on $T\p$ and $(\varphi,\psi)$ are the usual spherical coordinates on $\S$, we have:
\begin{equation}\label{gl2}
\textrm{Jac}\Phi_u=a^{-1}
\left(\begin{array}{cc}
g(\partial_\varphi N,e_1) & g(\partial_\psi N,e_1)\\
g(\partial_\varphi N,e_2) & g(\partial_\psi N,e_2)
\end{array}\right).
\end{equation}
Our estimates for $a$ and $\po N$ together with \eqref{gl2} imply that we 
control $\Phi_u$ in $C^1$. We deduce a formula for (Jac$\Phi_u)^*$Jac$\Phi_u$ from \eqref{gl2}:
\begin{equation}\label{gl3}
(\textrm{Jac}\Phi_u)^*\textrm{Jac}\Phi_u=a^{-2}
\left(\begin{array}{cc}
g(\partial_\varphi N,\partial_\varphi N) & g(\partial_\psi N,\partial_\varphi N)\\
g(\partial_\psi N,\partial_\varphi N) & g(\partial_\psi N,\partial_\psi N)
\end{array}\right),
\end{equation}
which we denote for simplicity by:
\begin{equation}\label{gl4}
(\textrm{Jac}\Phi_u)^*\textrm{Jac}\Phi_u=a^{-2}g(\po N,\po N).
\end{equation}
Recall that $u$ coincides with $x.\o$ in $|x|\geq 2$, so that $(\textrm{Jac}\Phi_u)^*\textrm{Jac}\Phi_u$ is equal to the $2\times 2$ identity matrix I in this region. According to \eqref{thregx1} and \eqref{om63}, we have:
\begin{equation}\label{gl5}
\norm{(\textrm{Jac}\Phi_u)^*\textrm{Jac}\Phi_u-I}_{\lli{\infty}}\lesssim\ep,
\end{equation}
so that $|\det((\textrm{Jac}\Phi_u)^*\textrm{Jac}\Phi_u)-1|\lesssim\ep$. In turn, this yields:
\begin{equation}\label{gl6}
\norm{|\det(\textrm{Jac}\Phi_u)|-1}_{\lli{\infty}}\lesssim\ep.
\end{equation}
From the fact that $\Phi_u$ is $C^1$ and \eqref{gl6}, we deduce that $\Phi_u$ is a $C^1$ local diffeomorphism.\\

{\bf step2} $\Phi_u$ is onto\\ 

We continue by showing that $\Phi_u$ is onto. The image of $\Phi_u$ is a nonempty subset of $T_\o\S$ which is open since it is a local diffeomorphism at each point in $\p$. Let us show that the image of $\Phi_u$ is also closed in $T_\o\S$. Indeed, consider a subsequence $\Phi_u(x_n)=y_n$ that converges to some $\underline{y}$ in $T_\o\S$. In particular, $y_n$ is a bounded sequence. Since $u$ coincides with $x.\o$ in the region $|x|\geq 2$, it is easy to check that
$$\lim_{x\in\p,\,|x|\rightarrow +\infty}|\Phi_u(x)|=+\infty,$$
so that $x_n$ must be a bounded sequence too. Thus, we may extract a subsequence from $x_n$ that converges towards some $\underline{x}\in\p$. Finally, we have $\Phi_u(\underline{x})=\underline{y}$ by the continuity of $\Phi_u$, so that the image of $\Phi_u$ is closed. Thus, the image of $\Phi_u$ is a nonempty open and closed subset of $T_\o\S$. Since $T_\o\S$ is connex, the image of $\Phi_u$ coincides with $T_\o\S$, and $\Phi_u$ is onto.\\

{\bf step3} $\Phi_u$ is one-to-one\\ 

We conclude the proof of Proposition \ref{gl0} by showing that $\Phi_u$ is one-to-one. Let us assume the contrary. Then, there exists $x_1$ and $x_2$ in $\p$ such that $x_1\neq x_2$ and $\Phi_u(x_1)=
\Phi_u(x_2)$. In particular, using the definition \eqref{gl1} of $\Phi_u$ and the usual spherical coordinates $(\varphi,\psi)$ on $\S$, we have:
\begin{equation}\label{gl7}
\partial_\varphi u(x_1,\o)=\partial_\varphi u(x_2,\o)\textrm{ and }\partial_\psi u(x_1,\o)=\partial_\psi u(x_2,\o).
\end{equation}
We define $\a:=\partial_\varphi u(x_1,\o)$ and $\b:=\partial_\psi u(x_1,\o)$. \eqref{gl7} implies that:
\begin{equation}\label{gl8}
\{\partial_\varphi u(.,\o)=\a\}\textrm{ and }\{\partial_\psi u(.,\o)=\b\}\textrm{ intersect at two distinct points in }\p.
\end{equation}

Our goal from now on is to prove that the situation described in \eqref{gl8} can not happen. Let us first show that the level curve $\{\partial_\varphi u(.,\o)=\a\}$ is connex in $\p$. Note that $\{\partial_\varphi u(.,\o)=\a\}$ coincides with the union of two half straight lines in the region $|x|\geq 2$ since $u$ coincides with $x.\o$ there. Let us call $C_-$ and $C_+$ the connex component containing each of these half straight lines. Let $x_0$ a point on $\{\partial_\varphi u(.,\o)=\a\}$. We consider the following curve:
\begin{equation}\label{gl9}
\frac{d\mu}{d\tau}=\partial_\varphi N(\mu(\tau)),\,\mu(0)=x_0.
\end{equation}
Since $\partial_\varphi N$ is tangent to $\{\partial_\varphi u(.,\o)=\a\}$, we see that the curve $\mu$ is contained inside $\{\partial_\varphi u(.,\o)=\a\}$. Note also that according to \eqref{gl3} and \eqref{gl5}, we have 
$|\partial_\varphi N|\simeq 1$ everywhere, so that $\mu$ exists for all $\tau\in\R$ and does not have a limit in $\p$ when $\tau\rightarrow \pm\infty$. Let us prove that:
\begin{equation}\label{gl10}
\lim_{\tau\rightarrow\pm\infty}|\mu(\tau)|=+\infty.
\end{equation}

Indeed, if \eqref{gl10} does not hold, then we can construct a sequence $(\tau_n)_{n\in\N}$ such that $\tau_n\rightarrow\pm\infty$ and $\mu(\tau_n)\rightarrow\underline{x}$ for some $\underline{x}$ in 
$\{\partial_\varphi u(.,\o)=\a\}$. Now, since $\nabb\partial_\varphi u=a^{-1}\partial_\varphi N(\underline{x})\neq 0$, the implicit function Theorem implies the existence of a neighborhood $V$ of $\underline{x}$ in $\p$ such that 
$\{\partial_\varphi u(.,\o)=\a\}$ coincides with a single arc of curve in $V$. Let $n_0\in\N$ large enough such that $\mu(\tau_n)\in V$ for all $n\geq n_0$. Then, for each $n\geq n_0$ and for $\tau$ sufficiently close to $\tau_n$, 
$\mu(\tau)$ lies inside $V$ and is therefore on this arc of curve. Since $\mu$ does not have a limit in $\p$ when $\tau\rightarrow \pm\infty$, this implies that $\mu(\tau)$ covers the whole arc of curve inside $V$ for each $n\geq n_0$ and for $\tau$ sufficiently close to $\tau_n$. Thus, $\mu(\tau)$ must be periodic. 

Let us now consider the connex components of $\p\setminus\{\mu(\tau),\,\tau\in\R\}$. If there is only one such component, then there is a neighborhood $W$ in $\p$ of $\{\mu(\tau),\,\tau\in\R\}$ where $\partial_\varphi u\neq \alpha$ on $W\setminus\{\mu(\tau),\,\tau\in\R\}$ and $W\setminus\{\mu(\tau),\,\tau\in\R\}$ is connex. Thus, either $\partial_\varphi u>\a$ everywhere on $W\setminus\{\mu(\tau),\,\tau\in\R\}$, or $\partial_\varphi u<\a$ everywhere on $W\setminus\{\mu(\tau),\,\tau\in\R\}$. In both cases, $\partial_\varphi u$ reaches a local extrema on $\{\mu(\tau),\,\tau\in\R\}$, and its gradient vanishes. This is impossible since $\nabb\partial_\varphi u=a^{-1}\partial_\varphi N(\underline{x})\neq 0$ everywhere.

Assume now that $\p\setminus\{\mu(\tau),\,\tau\in\R\}$ has at least two connex components. Since $\{\mu(\tau),\,\tau\in\R\}$ is periodic, it is compact, and at least one connex component must be precompact. The boundary of this connex component is $\{\mu(\tau),\,\tau\in\R\}$ where $\partial_\varphi u=\a$. So $\partial_\varphi u$ reaches a local extrema inside this precompact connex component, and its gradient  vanishes there. This is impossible since $a^{-1}\partial_\varphi N(\underline{x})\neq 0$ everywhere. This concludes the proof of \eqref{gl10}.

Since \eqref{gl10} holds, this means that any point $x_0$ in $\{\partial_\varphi u(.,\o)=\a\}$ belongs  either to $C_-$ or to $C_+$. We now prove that $C_-=C+$. Assume the contrary. Consider $x_0$ for example on $C_+$. Then since $C_+$ coincides with a half straight line in the region $|x|\geq 2$, 
\eqref{gl10} implies that $C_+\cap\{|x|\geq 2\}$ is covered at least twice by $\mu(\tau)$ (when $\tau\rightarrow -\infty$ and when $\tau\rightarrow +\infty$). Thus, $\mu(\tau)$ takes at least one value twice and must be periodic, which is in contradiction with \eqref{gl10}. Thus $C_-=C_+$ and the level curve $\{\partial_\varphi u(.,\o)=\a\}$ is connex in $\p$.

We now prove that the situation described in \eqref{gl8} can not happen. Let $x_1$ and $x_2$  the two distinct points of \eqref{gl8} where $\{\partial_\varphi u(.,\o)=\a\}$ and $\{\partial_\psi u(.,\o)=\b\}$ intersect. Since the level curve $\{\partial_\varphi u(.,\o)=\a\}$ is connex in $\p$, $\{\partial_\varphi u(.,\o)=\a\}\setminus (\{x_1\}\cup\{x_2\})$ has three connex components in $\p$. Also, since $\{\partial_\varphi u(.,\o)=\a\}$ coincides with the union of two half straight lines in the region $|x|\geq 2$, one of these three connex components is precompact. Let us call $\underline{C}$ this precompact connex component of 
$\{\partial_\varphi u(.,\o)=\a\}\setminus (\{x_1\}\cup\{x_2\})$. Note that its boundary $\partial\underline{C}$ consists of $\{x_1\}\cup\{x_2\}$. Then, since $\partial_\psi u(x_1)=\partial_\psi u(x_2)=\b$ by \eqref{gl7}, $\partial_\psi u$ reaches a local extrema at a point $\underline{x}$ inside $\underline{C}$. 
Thus, the tangent vector to $\{\partial_\varphi u(.,\o)=\a\}$ and $\{\partial_\psi u(.,\o)=\b\}$ at 
$\underline{x}$ must be collinear. This implies that $\partial_\varphi N(\underline{x})$ and 
$\partial_\psi N(\underline{x})$ must be collinear. It is impossible since \eqref{gl3} and \eqref{gl5} yield 
$|\partial_\varphi N|\simeq 1$, $|\partial_\psi N|\simeq 1$ and $|g(\partial_\varphi N,\partial_\psi N)|\lesssim\ep$.

Finally, we have proved that the situation in \eqref{gl8} can not happen so that $\Phi_u$ is one-to-one. This concludes the proof of Proposition \ref{gl0}. 

\subsection{The control of the Christoffel symbols}

We now show that the global coordinate system induced by $\Phi_u$ on $\p$ satisfies \eqref{eq:coordchart} and \eqref{eq:gammaL2} such that the constant $c>0$ in \eqref{eq:coordchart} and \eqref{eq:gammaL2} is independent of $u$. 
\begin{proposition}\label{gl11}
Let $\o\in\S$. Let $\Phi_u:\p\rightarrow T_\o\S$ defined by \eqref{gl1}. Then, it induces a global coordinate system on $\p$ which satisfies:
\be\lab{gl12}
|\ga_{AB}(p)\xi^A\xi^B-|\xi|^2|\lesssim\ep |\xi|^2, \qquad \mbox{uniformly for  all }
\,\, p\in \R^2.
\end{equation}
Moreover, the Christoffel symbols $\Gamma^A_{BC}$ verify,
\be\lab{gl13}
\sum_{A,B,C}\int_{\R^2}|\Gamma^A_{BC}|^2 dx^1dx^2\lesssim \ep^2.
\end{equation}
\end{proposition}

\begin{proof}
The coordinates functions on $\p$ induced by the global $C^1$ diffeomorphism $\Phi_u$ defined 
in \eqref{gl1} are given by:
\be\lab{gl14}
x_1=\partial_\varphi u(.,\o),\,x_2=\partial_\psi u(.,\o),
\end{equation}
which using \eqref{gl1bis} implies:
\be\lab{gl15}
\frac{\partial}{\partial x_1}=a^{-1}\partial_\varphi N,\,\frac{\partial}{\partial x_2}=a^{-1}\partial_\psi N.
\end{equation}
Since $\ga_{AB}=g(\frac{\partial}{\partial x_A},\frac{\partial}{\partial x_B})$, \eqref{gl3}, \eqref{gl5} and \eqref{gl15} imply \eqref{gl12}.

We now turn to the proof of \eqref{gl13}. By definition of the Christoffel symbols $\Gamma^A_{BC}$, we have:
\be\lab{gl16}
\Gamma^A_{BC}=g\left(\nabla_B\frac{\partial}{\partial x_C},\frac{\partial}{\partial x_A}\right).
\end{equation}
In view of \eqref{gl15} and \eqref{gl16}, the Christoffel symbols are of the form:
\be\lab{gl17}
\Gamma=a^{-3}g(\nabla_{\po N}\po N,\po N)-a^{-3}\nabb_{\po N} a g(\po N,\po N),
\end{equation}
which together with \eqref{om55} implies:
\be\lab{gl18}
\Gamma=a^{-3}\po\th(\po N,\po N)-2a^{-3}\nabb_{\po N} a |\po N|^2.
\end{equation}
\eqref{thregx1}, \eqref{threomega1} and \eqref{gl18} imply:
\be\lab{gl19}
\ds\norm{\Gamma}_{\li{\infty}{2}}\lesssim\norm{\po\th}_{\li{\infty}{2}}\norm{\po N}_{\lli{\infty}}^2
\ds +\norm{\nabb a}_{\li{\infty}{2}}\norm{\po N}_{\lli{\infty}}^3\lesssim\ep,
\end{equation}
which is \eqref{gl13}. This concludes the proof of Proposition \ref{gl11}.
\end{proof}

\subsection{Proof of Proposition \ref{gl20}}

Let $\o\in\S$. Recall the definition \eqref{gl21} of $\Phi:\s\rightarrow\R^3$:
$$\Phi(x):=u(x,\o)\o+\po u(x,\o)=u(x,\o)\o+\Phi_u(x),$$
where $\Phi_u$ has been defined in Proposition \ref{gl0}. 

We start by showing that $\Phi$ is one-to-one. Assume that $\Phi(x_1)=\Phi(x_2)$ for $x_1$ and $x_2$ in $\s$. Then, since the image of $\Phi_u$ is contained in $T_\o\S$, $\o.\Phi_u(x)=0$ for all $x\in\s$, so we have from \eqref{gl21}:
\begin{equation}\label{gl23}
u(x_1,\o)=u(x_2,\o)\textrm{ and }\Phi_{u(x_1,\o)}(x_1)=\Phi_{u(x_2,\o)}(x_2).
\end{equation}
Since $\Phi_u$ is one-to-one by Proposition \ref{gl0}, \eqref{gl23} implies $x_1=x_2$. Thus, $\Phi$ is one-to-one.

We now prove that $\Phi$ is onto. Let $y\in\R^3$. Then $y=(y.\o)\o+y'$ where $y'$ belongs to $T_\o\S$. Let $u=y.\o$. Since $\Phi_u$ is onto by Proposition \ref{gl0}, there exists $x\in\p$ such that $\Phi_u(x)=y'$. Thus, $u(x,\o)=y.\o$ and $\Phi_u(x)=y'$ so that $\Phi(x)=y$ by \eqref{gl21}. Therefore, $\Phi$ is onto.

We now turn to the proof of \eqref{gl22}. Using the fact that $\nabla u=a^{-1}N$ together with \eqref{omm4} we obtain a formula for $d\Phi$:
\begin{equation}\label{gl24}
d\Phi=(\nabla u)\o+\nabla\po u=a^{-1}N\o+ a^{-1}\po N-a^{-1}\po a N.
\end{equation}
In particular, if $e_1, e_2$ is an orthonormal frame on $T\p$ and $(\varphi,\psi)$ are the usual spherical coordinates on $\S$, we have:
\begin{equation}\label{gl25}
\textrm{Jac}\Phi=a^{-1}
\left(\begin{array}{ccc}
1 & -\partial_\varphi a & -\partial_\psi a\\
0 & g(\partial_\varphi N,e_1) & g(\partial_\psi N,e_1)\\
0 & g(\partial_\varphi N,e_2) & g(\partial_\psi N,e_2)
\end{array}\right).
\end{equation}
We deduce from \eqref{gl25} a formula for (Jac$\Phi_u)^*$Jac$\Phi_u$:
\begin{displaymath}
\begin{array}{l}
(\textrm{Jac}\Phi_u)^*\textrm{Jac}\Phi_u=a^{-2}\\
\times\left(\begin{array}{ccc}
1 & -\partial_\varphi a & -\partial_\psi a\\
-\partial_\varphi a & (\partial_\varphi a)^2+g(\partial_\varphi N,\partial_\varphi N) & \partial_\varphi a\partial_\psi a g(\partial_\psi N,\partial_\varphi N)\\
-\partial_\psi a & \partial_\varphi a\partial_\psi a+ g(\partial_\psi N,\partial_\varphi N) & (\partial_\psi a)^2+g(\partial_\psi N,\partial_\psi N)
\end{array}\right).
\end{array}
\end{displaymath}
Taking the determinant yields:
\begin{equation}\label{gl27}
\det((\textrm{Jac}\Phi)^*\textrm{Jac}\Phi)=a^{-2}\det((\textrm{Jac}\Phi_u)^*\textrm{Jac}\Phi_u),
\end{equation}
which together with \eqref{gl5} implies:
\begin{equation}\label{gl28}
\norm{\det((\textrm{Jac}\Phi)^*\textrm{Jac}\Phi)-1}_{\lli{\infty}}\lesssim\ep.
\end{equation}
\eqref{gl28} yields \eqref{gl22}. This concludes the proof of Proposition \ref{gl20}.

\section{Additional estimates}\label{sec:addition} 

This section is dedicated to the proof of Proposition \ref{prop:estimatesadded}, Proposition \ref{propbesov} and Proposition \ref{cordecfr}.

\subsection{Proof of Proposition \ref{prop:estimatesadded}}

We start with the proof of the estimate \eqref{gogol1}. We first derive an estimate for $\nabla N$ and $\nabla^2N$. In view of the structure equation \eqref{frame1}, we have:
\bea
\nn\norm{\nabla N}_{\ll{2}}+\norm{\nabla^2N}_{\ll{2}}&\les& \norm{\th}_{\ll{2}}+\norm{\ana}_{\ll{2}}+\norm{\nabla\th}_{\ll{2}}+\norm{\nabla(\ana)}_{\ll{2}}\\
\lab{thu}&\les& \ep+\norm{\nabla(\ana)}_{\ll{2}},
\eea
where we used in the last inequality the estimate \eqref{boot} for $a$ and the estimate \eqref{boot1} for $\th$. Now, we have:
\bee
\norm{\nabla(\ana)}_{\ll{2}}&\les& \norm{\nabb(\ana}_{\ll{2}}+\norm{\nabn\ana}_{\ll{2}}\\
\nn&\les& \norm{a^{-1}}_{\ll{\infty}}(\norm{\nabb\nabla a}_{\ll{2}}+\norm{[\nabn, \nabb]a}_{\ll{2}})+\norm{a^{-2}}_{\ll{\infty}}\norm{\nabla a}_{\ll{4}}^2\\
\nn&\les& \ep+\norm{[\nabn, \nabb]a}_{\ll{2}},
\eee
where we used in the last inequality the estimates \eqref{boot} \eqref{appboot} for $a$. Together with the commutator estimate \eqref{scommut}, we deduce:
\bee
\norm{\nabla(\ana)}_{\ll{2}}&\les& \ep+\norm{[\nabn, \nabb]a}_{\ll{2}}\\
&\les& \ep+(\norm{\th}_{\ll{4}}+\norm{\ana}_{\ll{4}})\norm{\nabla a}_{\ll{4}}\\
&\les& \ep,
\eee
where we used in the last inequality the estimates \eqref{boot} \eqref{appboot} for $a$ and the estimate \eqref{boot1} for $\th$. In view of \eqref{thu}, we finally obtain:
\be\lab{thu1}
\norm{\nabla N}_{\ll{2}}+\norm{\nabla^2N}_{\ll{2}}\les\ep.
\ee

Next, recall from Proposition \ref{p3} the following bound on the $\ll{\infty}$ norm of a tensor $F$ on $S$. We have:
\be\lab{thu2}
\norm{F}_{\ll{\infty}}\les \norm{F(-2,.)}_{L^4(P_{-2})}+\norm{\nabla F}_{\ll{2}}+\norm{\nabb\nabla F}_{\ll{2}}.
\ee
Now, recall that $u=x\cdot\o$ in $|x|\geq 2$, and thus $P_{u=-2}=\{x\cdot\o=-2\}$. Therefore, $P_{-2}$ is included in the region $|x|\geq 2$. In particular, if $F\equiv 0$ in $|x|\geq 2$, we may use \eqref{thu2} and obtain:
\be\lab{thu2bis}
\norm{F}_{\lli{\infty}}\les \norm{\nabla F}_{\lli{2}}+\norm{\nabb\nabla F}_{\lli{2}}\textrm{ for all }F\textrm{ such that }F=0\textrm{ in }|x|\geq 2.
\ee
Also, working in the global coordinate system on $\p$ given by Proposition \ref{gl0}, we easily derive
$$\norm{f}_{\lp{2}}\les\norm{\nabb f}_{\lp{2}}\textrm{ for any  scalar }f\textrm{ such that }f=0\textrm{ in }\p\cap\{ |x|\geq 2\}.$$
Integrating in $u$, and in view of coarea formula \eqref{coarea}, we deduce
$$\norm{f}_{\lli{2}}\les\norm{\nabb f}_{\lli{2}}\textrm{ for any scalar }f\textrm{ such that }f=0\textrm{ in }|x|\geq 2.$$
With the choice $f=|F|$, this yields
$$\norm{F}_{\lli{2}}\les\norm{\nabb F}_{\lli{2}}\textrm{ for any tensor }F\textrm{ such that }F=0\textrm{ in }|x|\geq 2.$$
Together with \eqref{thu2bis}, we finally obtain
\be\lab{thu3}
\norm{F}_{\lli{\infty}}\les \norm{\nabb\nabla F}_{\lli{2}}\textrm{ for all }F\textrm{ such that }F=0\textrm{ in }|x|\geq 2.
\ee  
Since  $u=x\cdot\o$ in $|x|\geq 2$, we have in particular $N=\o$ in $|x|\geq 2$. This yields:
$$N(x,\o)+N(x,-\o)=\o-\o=0\textrm{ in }|x|\geq 2.$$
Thus, using the estimate \eqref{thu3} with $F=N(.,\o)+N(.,-\o)$ implies:
\bee
&&\norm{N(.,\o)+N(.,-\o)}_{\lli{\infty}}\\
&\les& \norm{\nabla^2N(.,\o)}_{\lli{2}}+\norm{\nabla^2N(.,-\o)}_{\lli{2}}\\
&\les& \ep,
\eee
where we used the fact that $\nabla N\equiv 0$ in $|x|\geq 2$ and the estimate \eqref{thu1}. This concludes the proof of the estimate \eqref{gogol1}.

Next, we prove the estimate \eqref{threomega1ter}. We have:
$$N(x,\o')=N(x,\o)+\po N(x,\o)(\o-\o')+\int_{[\o,\o']}\po^2N(.,\o'')d\o''(\o-\o')^2.$$
This yields:
\bea\lab{thu3bis}
||N(x,\o)-N(x,\o')|-|\po N(x,\o)(\o-\o')||&\les& \norm{\po^2N}_{\ll{\infty}}|\o-\o'|^2\\
\nn&\les& |\o-\o'|^2
\eea
where we used in the last inequality the estimate \eqref{threomega2} for $\po^2N$. Now, the estimate \eqref{om63}  implies:
$$\norm{g(\po N,\po N)-I}_{\ll{\infty}}\lesssim\ep.$$
This yields:
$$||\po N(x,\o)(\o-\o')|-|\o-\o'||\les \ep|\o-\o'|.$$
Together with \eqref{thu3bis}, we obtain:
$$||N(x,\o)-N(x,\o')|-|\o-\o'||\les |\o-\o'|(\ep+|\o-\o'|).$$
This concludes the proof of the estimate \eqref{threomega1ter}. 

Finally, we prove the estimate \eqref{gogol2}. We first estimate $u$, $\po u$ and $\po^2u$. Differentiating the equality $\nabla u=a^{-1}N$, and using the structure equation \eqref{frame1}, we obtain:
\be\lab{thu4}
\norm{\nabla^2u}_{\ll{2}}\les \norm{a^{-2}\nabla a}_{\ll{2}}+\norm{a^{-1}\th}_{\ll{2}}\les\ep,
\ee
where we used in the last inequality the estimate \eqref{boot} for $a$ and the estimate \eqref{boot1} for $\th$. 
Also, differentiating the identity \eqref{omm3} for $\nabla\po u$, and using the structure equation \eqref{frame1}, we obtain:
\bea\lab{thu5}
&&\norm{\nabla^2\po u}_{\ll{2}}\\
\nn&\les& \norm{a^{-1}\nabla\po N}_{\ll{2}}+\norm{a^{-1}\nabla\po a}_{\ll{2}}+(\norm{a^{-2}\nabla a}_{\ll{4}}+\norm{a^{-1}\th}_{\ll{4}})\\
\nn&&\times(\norm{\po a}_{\ll{\infty}}+\norm{\po N}_{\ll{\infty}})\\
\nn&\les& \ep,
\eea
where we used in the last inequality the estimate \eqref{boot} for $a$, the estimate \eqref{boot1} for $\th$ and the estimate \eqref{threomega1} for $\po N$ and $\po a$. Finally, differentiating \eqref{omm3} with respect to $\o$, we obtain:
$$\nabla(\po^2u)=a^{-1}\po^2N-a^{-1}\po^2 aN-2a^{-1}\po N\po a+a^{-2}(\po a)^2N.$$
Differentiating with respect to $\nabb$, we obtain:
\bea\lab{thu6}
&&\norm{\nabb\nabla\po^2u}_{\ll{2}}\\
\nn&\les& \norm{a^{-1}}_{\ll{\infty}}(\norm{\nabla\po^2N}_{\ll{2}}+\norm{\nabb\po^2 a}_{\ll{2}})+(\norm{a^{-2}\nabb a}_{\l{\infty}{4}}\\
\nn&&+\norm{a^{-1}\th}_{\l{\infty}{4}})(\norm{\po^2N}_{\l{2}{4}}+\norm{\po^2 a}_{\l{2}{4}}\\
\nn&&+\norm{\po a}_{\l{2}{4}}\norm{\po N}_{\ll{\infty}}+\norm{\po a}_{\l{4}{8}}^2)\\
\nn&&+\norm{\nabb\po a}_{\ll{2}}\norm{\po N}_{\ll{\infty}}\\
\nn&&+\norm{\po a}_{\ll{\infty}}(\norm{a^{-2}\nabb a}_{\ll{2}}+\norm{a^{-1}\nabb\po N}_{\ll{2}})\\
\nn&\les&\ep,
\eea
where we used in the last inequality the estimates \eqref{boot} \eqref{appboot} for $a$, the estimates \eqref{boot1} \eqref{appboot1} for $\th$, the estimate \eqref{threomega1} for $\po N$ and $\po a$, and the estimate \eqref{threomega2} for $\po^2N$ and $\po^2 a$.

Recall that for $\o\in\S$, the map $\Phi_\o:\s\rightarrow\R^3$ is defined by:
$$\Phi_\o(x):=u(x,\o)\o+\po u(x,\o).$$
Since  $u=x\cdot\o$ in $|x|\geq 2$, we have:
$$\Phi_\o(x)=x\textrm{ for }|x|\geq 2,$$
which yields:
\bea\lab{thu7}
&&u(x,\o)-\Phi_\nu(x)\c\o=0,\, \po u(x,\o)-\po(\Phi_\nu(x)\c\o)=0,\\
\nn&& \textrm{and }\po^2 u(x,\o)-\po^2(\Phi_\nu(x)\c\o)=0\textrm{ in }|x|\geq 2.
 \eea 

Now, let $\nu\in\S$. We first estimate $\po^2 u(x,\o)-\po^2(\Phi_\nu(x)\c\o)$. In view of \eqref{thu7} and \eqref{thu3}, we have:
\bea\lab{thu9}
&&\norm{\po^2 u(.,\o)-\po^2(\Phi_\nu(.)\c\o)}_{\lli{\infty}}\\
\nn&\les& \norm{\nabb\nabla\po^2u(.,\o)}_{\lli{2}}+\norm{\nabla^2\po u(.,\nu)}_{\lli{2}}+\norm{\nabla^2u(.,\nu)}_{\lli{2}}\\
\nn&\les& \ep,
\eea
where we used in the last inequality the estimate \eqref{thu4} for $u$, the estimate \eqref{thu5} for $\po u$,  and the estimate \eqref{thu6} for $\po^2u$. 

Next, we estimate $\po u(x,\o)-\po(\Phi_\nu(x)\c\o)$. We have:
\bea\lab{thu10}
\Phi_\nu(x)\cdot\o &=& u(x,\nu)\nu\c\o+\po u(x,\nu)\o\\
\nn&=& u(x,\nu)+\po u(x,\nu)(\o-\nu)-\frac{|\o-\nu|^2}{2}u(x,\nu),
\eea
where we used in the last equality the fact that $\po u(x,\nu)\nu=0$. Thus, we obtain:
\bee
\po u(x,\o)-\po(\Phi_\nu(x)\c\o)&=&\po u(x,\o)-\po u(x, \nu)-(\o-\nu)u(x,\nu)\\
&=&\int \po^2u(x,\o')(\o-\nu)-(\o-\nu)u(x,\nu),
\eee
where $\o'$ is on the arc $[\o,\nu]$ of $\S$. Together with \eqref{thu7} and \eqref{thu3}, this implies:
\bea\lab{thu11}
&&\norm{\po u(.,\o)-\po(\Phi_\nu(.)\c\o)}_{\lli{\infty}}\\
\nn&\les& |\o-\nu|(\norm{\nabb'\nabla\po^2u(.,\o')}_{\lli{2}}+\norm{\nabla^2 u(.,\nu)}_{\lli{2}})\\
\nn&\les& \ep|\o-\nu|,
\eea
where we used in the last inequality the estimate \eqref{thu4} for $u$ and the estimate \eqref{thu6} for $\po^2u$.

Finally, we estimate $u(x,\o)-\Phi_\nu(x)\c\o$. In view of \eqref{thu10}, we have:
\bee
u(x,\o)-\Phi_\nu(x)\c\o&=& u(x,\o)-u(x,\nu)-\po u(x,\nu)(\o-\nu)+\frac{|\o-\nu|^2}{2}u(x,\nu)\\
&=&\int \po^2u(x,\o')(\o-\nu)^2+\frac{|\o-\nu|^2}{2}u(x,\nu),
\eee
where $\o'$ is on the arc $[\o,\nu]$ of $\S$. Together with \eqref{thu7} and \eqref{thu3}, this implies:
\bea\lab{thu12}
&&\norm{u(.,\o)-\Phi_\nu(.)\c\o}_{\lli{\infty}}\\
\nn&\les& |\o-\nu|^2(\norm{\nabb'\nabla\po^2u(.,\o')}_{\lli{2}}+\norm{\nabla^2 u(.,\nu)}_{\lli{2}})\\
\nn&\les& \ep|\o-\nu|^2,
\eea
where we used in the last inequality the estimate \eqref{thu4} for $u$ and the estimate \eqref{thu6} for $\po^2u$.

Finally, \eqref{thu9}, \eqref{thu11} and \eqref{thu12} imply \eqref{gogol2}. This concludes the proof of Proposition \ref{prop:estimatesadded}.

\subsection{Proof of Proposition \ref{propbesov}}

Recall from the first equation of \eqref{struct1} that $\trt-k_{NN}=1-a$. Now, since $a$ satisfies \eqref{thregx1}, $\trt-k_{NN}$ satisfies:
$$\norm{\nabn(\trt-k_{NN})}_{\ll{2}}+\norm{\nabb(\trt-k_{NN})}_{\l{\infty}{2}}+\norm{\nabb\nabla(\trt-k_{NN})}_{\ll{2}}\lesssim\ep.$$
Thus, Proposition \ref{propbesov} is a direct consequence of the following proposition: 
\begin{proposition}\label{ad0}
Let a scalar function $f$ on $\s$ such that $f\equiv 0$ on $u=-2$ and:
\begin{equation}\label{ad1}
\norm{\nabb f}_{\l{\infty}{2}}+\norm{\nabn f}_{\ll{2}}+\norm{\nabb\nabla f}_{\ll{2}}\lesssim\ep.
\end{equation}
Then, we have:
\begin{equation}\label{ad2}
\norm{f}_{\mathcal{B}}\lesssim\ep.
\end{equation}
\end{proposition}

The rest of this section is dedicated to the proof of Proposition \ref{ad0}.

\begin{proof}
Using the definition \eqref{defbesov}, \eqref{c0e1}, property (iii) of Theorem \ref{thm:LP} and \eqref{ad1}, we have:
\begin{equation}\label{ad3}
\begin{array}{lll}
\ds\norm{f}_{\mathcal{B}} & = & \ds\sum_{j\ge 0} 2^{j}\norm{P_jf}_{\l{\infty}{2}} +
\norm{P_{<0}f}_{\l{\infty}{2}}\\
& \lesssim & \ds\sum_{j\ge 0} 2^{j}\norm{P_jf}^{\frac{1}{2}}_{\ll{2}}(\norm{\nabn P_jf}_{\ll{2}}+\norm{\nabb P_jf}_{\ll{2}})^{\frac{1}{2}}\\
& & \ds +\norm{P_{<0}f}^{\frac{1}{2}}_{\ll{2}}(\norm{\nabn P_{<0}f}_{\ll{2}}+\norm{\nabb P_{<0}f}_{\ll{2}})^{\frac{1}{2}}\\
& \lesssim & \ds\sum_{j\ge 0}\norm{\lap f}^{\frac{1}{2}}_{\ll{2}}(\norm{\nabn P_jf}_{\ll{2}}+2^{-j}\norm{\lap f}_{\ll{2}})^{\frac{1}{2}}\\
& & \ds +\norm{f}^{\frac{1}{2}}_{\ll{2}}(\norm{\nabn P_{<0}f}_{\ll{2}}+\norm{f}_{\ll{2}})^{\frac{1}{2}}\\
& \lesssim & \ds\ep^{\frac{1}{2}}\left(\sum_{j\ge 0}\norm{\nabna P_jf}_{\ll{2}}^{\frac{1}{2}}
+\norm{\nabna P_{<0}f}_{\ll{2}}^{\frac{1}{2}}\right)+\ep,
\end{array}
\end{equation}
where we used the estimate \eqref{threomega1} for $a$ in the last estimate. The term $\norm{\nabna P_{<0}f}_{\ll{2}}$ is easier to bound, so we concentrate on estimating the sum $\sum_{j\ge 0}\norm{\nabna P_jf}_{\ll{2}}^{\frac{1}{2}}$. 

Let $0<\delta<1$. In view of the finite band property for $P_j$, and the commutator estimate \eqref{commLP4}, we have:
\bea\lab{cosco}
&&\norm{\nabna P_jf}_{\ll{2}}\\
\nn&\les& \norm{P_j(\nabna f)}_{\ll{2}}+\norm{[\nabna, P_j]f}_{\ll{2}}\\
\nn&\les& 2^{-j}\norm{\nabb(\nabna f)}_{\ll{2}}+2^{-(1-\delta)j}\ep(\norm{\lap f}_{L^2(S)}+\norm{\nabb f}_{\l{\infty}{2}})\\
\nn&\les& 2^{-(1-\delta)j}\ep(\norm{a}_{\ll{\infty}}\norm{\nabb\nabn f}_{\ll{2}}+\norm{\nabb a}_{\l{\infty}{4}}\norm{\nabn f}_{\l{2}{4}}\\
\nn&&+\norm{\lap f}_{L^2(S)}+\norm{\nabb f}_{\l{\infty}{2}})\\
\nn&\les& 2^{-(1-\delta)j}\ep,
\eea
where we used in the last inequality the Gagliardo-Nirenberg inequality \eqref{eq:GNirenberg}, the estimate \eqref{thregx1} for $a$ and the estimate \eqref{ad1} for $f$. Since $\delta<1$, \eqref{ad3} and \eqref{cosco} imply \eqref{ad2}. This concludes the proof of the proposition.
\end{proof}

\subsection{Proof of Proposition \ref{cordecfr}}

We decompose $\nabn a$ in the following way:
\begin{equation}\label{ad66}
\nabn a=a_1^j+a_2^j,\textrm{ where }a_1^j=P_{>j/2}(\nabn a)\textrm{ and }a_2^j=P_{\leq j/2}(\nabn a).
\end{equation}
Using the estimate \eqref{boot} for $a$ and the finite band property for $P_j$, we obtain:
\begin{equation}\label{ad67}
\norm{a_1^j}_{\ll{2}}\leq\sum_{l>j/2}\norm{P_l\nabn a}_{\ll{2}}\lesssim\sum_{l>j/2}2^{-l}\norm{\nabb\nabn a}_{\ll{2}}\lesssim 2^{-j/2}\ep.
\end{equation}
We also have:
\begin{equation}\label{ad68}
\norm{\nabn a_2^j}_{\ll{2}}\leq\sum_{l\leq j/2}\norm{\nabn P_l\nabn a}_{\ll{2}}\les \sum_{l\leq j/2}\norm{\nabna P_l\nabn a}_{\ll{2}},
\end{equation}
where we used in the last inequality the estimate \eqref{boot} for $a$. 

Next, we estimate $\nabna P_l\nabn a$. Let $\delta>0$. In view of the finite band property for $P_l$ and the commutator estimate \eqref{commLP1}, we have:
\bee
\norm{\nabna P_l\nabn a}_{\ll{2}}&\les&  \norm{P_l(a\nabn^2a)}_{\ll{2}}+\norm{[\nabna, P_l]\nabn a}_{\ll{2}}\\
\nn&\les&  2^{\frac{l}{2}}\norm{a\nabn^2a}_{\lhs{2}{-\frac{1}{2}}}+ \ep\norm{\La^{\frac{1}{2}+\delta}\nabn a}_{L^2(S)}+\ep\norm{\La^\delta \nabn a}_{\l{\infty}{2}}.
\eee
Together with the product estimate \eqref{bale}, we obtain:
\bea\lab{cocococo}
&&\norm{\nabna P_l\nabn a}_{\ll{2}}\\
\nn&\les&  2^{\frac{l}{2}}(\norm{a}_{\ll{\infty}}+\norm{\nabb a}_{\l{\infty}{2}})\norm{\nabn^2a}_{\lhs{2}{-\frac{1}{2}}}+ \ep\norm{\nabb\nabn a}_{L^2(S)}+\ep\norm{\nabn a}_{\lhs{\infty}{\delta}}\\
\nn&\les&2^{\frac{l}{2}}\ep+\ep\norm{\nabn a}_{\lhs{\infty}{\delta}},
\eea
where we used in the last inequality the estimate \eqref{thregx1} for $a$ and the estimate \eqref{nabn2a1} for $\nabn^2a$. Now, in view of the decomposition \eqref{zoc5} of $\nabn a$, and the estimates \eqref{zoc7} and \eqref{zoc9}, we have for all $j\geq 0$:
$$\norm{P_j\nabn a}_{\l{\infty}{2}}\les 2^{-\frac{j}{2}}\ep,$$
which yields:
\be\lab{crocro}
\norm{\nabn a}_{\lhs{\infty}{(\frac{1}{2})_-}}\les \ep.
\ee
Choosing $0<\delta<\frac{1}{2}$ in \eqref{cocococo} and using \eqref{crocro} finally yields:
$$\norm{\nabna P_l\nabn a}_{\ll{2}}\les 2^{\frac{l}{2}}\ep.$$
Together with \eqref{ad68}, we obtain:
\begin{equation}\label{xxxxxxx}
\norm{\nabn a_2^j}_{\ll{2}}\les \sum_{l\leq j/2} 2^{\frac{l}{2}}\ep\les 2^{\frac{j}{4}}\ep.
\end{equation}
Finally, in view of \eqref{ad66}, \eqref{ad67} and \eqref{xxxxxxx}, we obtain the conclusion of the proposition.

\appendix

\section{Proof of Proposition \ref{h2}}\lab{sec:proofhigherregapp}

Remark first that \eqref{h3} for $j=1$ has already been obtained in the section \ref{sec:apriorilow}. We prove \eqref{h3} by iteration on $j$. Let us first start with the case $j=2$. 

\subsection{Proof of \eqref{h3} for $j=2$} 

We start by estimating $\norm{\nabn^2a}_{\ll{2}}$ and $\norm{\nabb^2\nabn a}_{\ll{2}}$. By \eqref{r20} and \eqref{commut1}, we have:
\begin{equation}\label{h4}
(\nabn - a^{-1}\lap)\nabn a=h, 
\end{equation}
where $h$ is defined by:
\begin{equation}\label{h5}
\begin{array}{ll}
\ds h= &  -a^{-1}\trt\lap a-2a^{-1}\hth\nabb^2a+2a^{-2}\nabb a\nabb\nabn a-2R_{N.}\ana -\nabb\trt\ana\\
& \ds +2\hth |\ana|^2+2\th\nabn\th+\nabn^2k_{NN}+\nabn R_{NN}.
\end{array}
\end{equation}
We estimate $\norm{h}_{\ll{2}}$:
\begin{equation}\label{h9}
\begin{array}{ll}
\ds \norm{h}_{\ll{2}}\lesssim & \norm{\th}_{\ll{\infty}}\norm{\nabb^2a}_{\ll{2}}+\norm{\nabb a}_{\l{\infty}{4}}\norm{\nabb\nabn a}_{\l{2}{4}}\\
& \ds+\norm{R}_{\ll{\infty}}\norm{\nabb a}_{\ll{2}}+ \norm{\nabb\trt}_{\l{2}{4}}\norm{\nabb a}_{\l{\infty}{4}}\\
& \ds +\norm{\th}_{\ll{6}}\norm{\nabb a}_{\ll{6}}^2+\norm{\th}_{\l{\infty}{4}}\norm{\nabn\th}_{\l{2}{4}}\\
& \ds+\norm{\nabn^2k_{NN}}_{\ll{2}}+\norm{\nabn R_{NN}}_{\ll{2}},
\end{array}
\end{equation}
which together with \eqref{boot}, \eqref{appboot}, \eqref{appboot1}, \eqref{h1} and \eqref{h9} yields:
\begin{equation}\label{h11}
\begin{array}{r}
\ds \norm{h}_{\ll{2}}\lesssim \ep(\norm{\th}_{\ll{\infty}}+\norm{\nabb\nabn a}_{\l{2}{4}}+\norm{\nabb\trt}_{\l{2}{4}}\\
+\norm{\nabn\th}_{\l{2}{4}}+\ep^2)+M.
\end{array}
\end{equation}
Together with \eqref{eq:GNirenberg}, Proposition \ref{p3}, \eqref{boot} and \eqref{boot1}, this implies:
\begin{equation}\label{h12}
\begin{array}{r}
\ds \norm{h}_{\ll{2}}\lesssim \ep(\norm{\nabla^2\th}_{\ll{2}}+\norm{\nabb^2\nabn a}_{\ll{2}}+\ep)+M.
\end{array}
\end{equation}
Proposition \ref{p7}, \eqref{init1}, \eqref{boot}, \eqref{appboot}, \eqref{appboot1}, \eqref{h4} and \eqref{h12} yield: 
\begin{equation}\label{h14}
\norm{\nabb\nabn a}_{\l{\infty}{2}}+\norm{\nabn^2a}_{\ll{2}}+\norm{\nabb^2\nabn a}_{\ll{2}}\lesssim 
\ep\norm{\nabla^2\th}_{\ll{2}}+M.
\end{equation}

Let us now estimate $\norm{\nabb^3a}_{\ll{2}}$. We differentiate the second equation of \eqref{struct1} with respect to $\nabb$ and we obtain, in view of the commutator formula \eqref{commut2}:
\begin{equation}\label{hh1}
a^{-1}\lap\nabb a=h+\nabb\nabn a, 
\end{equation}
where $h$ is defined by:
\begin{equation}\label{hh2}
\begin{array}{ll}
\ds h= &  -a^{-2}\nabb a\lap a+K\ana+2\th\nabb\th+\nabb\nabn k_{NN}+\nabb R_{NN}.
\end{array}
\end{equation}
\eqref{hh1} yields:
\begin{equation}\label{hh3}
\norm{a^{-1}\lap\nabb a}_{\ll{2}}\leq \norm{h}_{\ll{2}}+\norm{\nabb\nabn a}_{\ll{2}}.
\end{equation}
We estimate $\norm{h}_{\ll{2}}$:
\begin{equation}\label{hh6}
\begin{array}{ll}
\ds \norm{h}_{\ll{2}}\lesssim & \norm{\nabb a}_{\l{\infty}{4}}\norm{\nabb^2 a}_{\l{2}{4}}+\norm{K}_{\ll{3}}\norm{\nabb a}_{\ll{6}}\\
& \ds +\norm{\th}_{\l{\infty}{4}}\norm{\nabb\th}_{\l{2}{4}}+\norm{\nabb\nabn k_{NN}}_{\ll{2}}\\
& \ds +\norm{\nabb R_{NN}}_{\ll{2}}.
\end{array}
\end{equation}
Together with \eqref{gauss1}, \eqref{appboot}, \eqref{appboot1} and \eqref{h1}, this yields:
\begin{equation}\label{hh7}
\norm{h}_{\ll{2}}\lesssim\ep(\norm{\nabb^2 a}_{\l{2}{4}}+\ep^2+M)+\ep\norm{\nabb\th}_{\l{2}{4}}+M.
\end{equation}
Together with \eqref{eq:GNirenberg}, \eqref{boot} and \eqref{boot1}, this implies:
\begin{equation}\label{hh8}
\norm{h}_{\ll{2}}\lesssim\ep(\norm{\nabb^3 a}_{\ll{2}}+\norm{\nabb^2\th}_{\ll{2}}+M)+M.
\end{equation}
Now, \eqref{gauss1}, \eqref{appboot1} and \eqref{h1} imply:
\begin{equation}\label{h21}
\norm{K}_{\ll{3}}+\norm{K}_{\l{\infty}{2}} \lesssim \norm{\th}^2_{\ll{6}}+\norm{\th}^2_{\l{\infty}{4}}+M\lesssim M+\ep^2.
\end{equation}
\eqref{vboch} and \eqref{h21} yield:
\begin{equation}\label{hh9}
\begin{array}{lll}
\ds\norm{\nabb^3a}_{\ll{2}} & \lesssim & \ds\norm{\lap\nabb a}_{\ll{2}}+\norm{K}_{\l{\infty}{2}}^{1/2}\norm{\nabb^2a}_{\l{2}{4}}\\
& & \ds +\norm{K}_{\ll{3}}\norm{\nabb a}_{\ll{6}}\\
& \ds\lesssim & \ds \norm{\lap\nabb a}_{\ll{2}}+(M+\ep^2)\norm{\nabb^2a}_{\l{2}{4}}+(M+\ep^2)\ep.
\end{array}
\end{equation}
Together with \eqref{eq:GNirenberg} and \eqref{boot}, this implies:
\begin{equation}\label{hh10}
\norm{\nabb^3a}_{\ll{2}}\lesssim \norm{\lap\nabb a}_{\ll{2}}+(M^2+\ep^2)\ep.
\end{equation}
\eqref{boot}, \eqref{hh3}, \eqref{hh8} and  \eqref{hh10} yield:
\begin{equation}\label{hh11}
\norm{\nabb^3a}_{\ll{2}}\lesssim \ep\norm{\nabb^2\th}_{\ll{2}}+M.
\end{equation}

Let us now estimate $\norm{\nabb^2\th}_{\ll{2}}$. We differentiate the first equation of \eqref{choice1}, which yields together with \eqref{boot} and \eqref{h1}:
\begin{equation}\label{h17}
\norm{\nabb^2\trt}_{\ll{2}}\leq \norm{\nabb^2a}_{\ll{2}}+\norm{\nabb^2k_{NN}}\lesssim\ep +M.
\end{equation}
Let us now estimate $\norm{\nabb^2\hth}_{\ll{2}}$. We consider the Hodge operator $\mathcal{D}_2$ 
which takes any symmetric traceless 2-tensor $F$ on $\p$ into the 1-form $\divb F$. Let 
$^*\mathcal{D}_2$ its adjoint which takes 1-forms on $\p$ into the 2-covariant symmetric traceless 
tensor $(^*\mathcal{D}_2F)_{AB}=\nabb_BF_A+\nabb_AF_B-(\divb F)\gamma_{AB}$. We have the 
following identity:
\begin{equation}\label{h18}
^*\mathcal{D}_2\mathcal{D}_2=-\frac{1}{2}\lap +K.
\end{equation}
Thus, applying $^*\mathcal{D}_2$ to the third equation of \eqref{struct1}, we obtain:
\begin{equation}\label{h19}
\lap\hth = 2K\hth-^*\mathcal{D}_2(\nabb\trt)-2^*\mathcal{D}_2(R_{N.}).
\end{equation}
\eqref{h19} together with \eqref{h1} and \eqref{h17} yields:
\begin{equation}\label{h20}
\norm{\lap\hth}_{\ll{2}} \lesssim \norm{K\hth}_{\ll{2}}+M+\ep.
\end{equation}
The analog of \eqref{vboch} for 2-tensors, \eqref{appboot1}, \eqref{h21} and \eqref{h20} yield:
\begin{equation}\label{h22}
\begin{array}{ll}
\ds\norm{\nabb^2\hth}_{\ll{2}} & \ds\lesssim \norm{K}_{\l{\infty}{2}}^{1/2}\norm{\nabb\hth}_{\l{2}{4}}+\norm{K}_{\ll{3}}\norm{\hth}_{\ll{6}}+M+\ep\\
& \ds\lesssim (M+\ep^2)\norm{\nabb\hth}_{\l{2}{4}}+(M+\ep^2)\ep+M+\ep.
\end{array}
\end{equation}
Together with \eqref{eq:GNirenberg} and \eqref{boot1}, this implies:
\begin{equation}\label{h23}
\norm{\nabb^2\hth}_{\ll{2}} \lesssim M.
\end{equation}
Finally, \eqref{h17} and \eqref{h23} yield:
\begin{equation}\label{h24}
\norm{\nabb^2\th}_{\ll{2}} \lesssim M.
\end{equation}

Let us now estimate $\norm{\nabla\nabn\th}_{\ll{2}}$. Differentiating the last equation of 
\eqref{struct1} by $\nabla$, taking the norm in $\ll{2}$, using \eqref{commut}, and estimating 
the various quantities in the same fashion as previously, we obtain:
\begin{equation}\label{h25}
\norm{\nabla\nabn\th}_{\ll{2}}\lesssim \norm{\nabb^2\nabn a}_{\ll{2}}+\ep\norm{\nabla^2\th}_{\ll{2}}+M.
\end{equation}

Finally, \eqref{h14}, \eqref{hh11}, \eqref{h24} and \eqref{h25} yield the proof of \eqref{h3} 
for $j=2$. 

\subsection{Proof of \eqref{h3} for $j+1$ assuming \eqref{h3} for $j$ with $j\geq 2$}

We state two lemmas which will be used in the course of the proof. 
\begin{lemma}\label{h26}
Let $F$ a tensor on $S$ and $l\in\N$. Assume that \eqref{h3} holds with $j=2$. Assume also that 
$\norm{\nabb\nabn^2a}_{\ll{2}}\leq C(M)$. Then, we have the following inequality: 
\begin{equation}\label{h27}
\norm{\nabla^lF}_{\ll{2}}\leq C(M)\left(\norm{\nabn^lF}_{\ll{2}}+\norm{\nabb^lF}_{\ll{2}}+\sum_{m=0}^{l-1}\norm{\nabla^mF}_{\ll{2}}\right). 
\end{equation}
\end{lemma}

\begin{lemma}\label{h28}
Let $f$ a scalar function on $S$. We have the following commutator formula:
\begin{equation}\label{h29}
\begin{array}{l}
\ds[\nabn^j,a^{-1}\lap]f=j(2\ana\nabb+a^{-1}\lap a+(j-1)|\ana|^2)\nabn^jf\\
\ds +\left(\prod_{l=1}^p\nabb^{t^1_l}\nabn^{t^2_l}a\right)\left(\prod_{m=1}^q\nabb^{v^1_m}\nabn^{v^2_m}\th\right)
\left(\prod_{n=1}^r\nabb^{w^1_n}\nabn^{w^2_n}R\right)\nabb^{s_1}\nabn^{s_2}f
\end{array}
\end{equation}
where $t^1_l, t^2_l, v^1_m, v^2_m, w^1_n$ and $w^2_n$ satisfy:
\begin{equation}\label{h30}
\begin{array}{l}
t^1_1+\cdots+t^1_p+v^1_1+\cdots+v^1_q+w^1_1+\cdots+w^1_r+s_1=2,\\
t^2_1+\cdots+t^2_p+v^2_1+\cdots+v^2_q+w^2_1+\cdots+w^2_r+s_2=j-q-r,\\
t^2_l\leq j-1, 0\leq l\leq p,\,s_2\leq j-1.
\end{array}
\end{equation}
\end{lemma}
We postpone the proof of Lemma \ref{h26} to section \ref{sec:applemma1}, and the proof of Lemma \ref{h28} to section \ref{sec:applemma2}. We now continue the proof of Proposition \ref{h2}. We differentiate the second equation of 
\eqref{struct1} by $\nabn^j$:
\begin{equation}\label{h31}
(\nabn-a^{-1}\lap)\nabn^ja=h,
\end{equation}
where $h$ is defined by:
\begin{equation}\label{h32}
h=[\nabn^j,a^{-1}\lap]a+\nabn^j(|\th|^2)+\nabn^{j+1}k_{NN}+\nabn^jR_{NN}.
\end{equation}
We estimate $\norm{h}_{\ll{2}}$. Using \eqref{h1} and \eqref{h32}, we obtain:
\begin{equation}\label{h33}
\norm{h}_{\ll{2}}\lesssim\norm{[\nabn^j,a^{-1}\lap]a}_{\ll{2}}+\norm{\nabn^j(|\th|^2)}_{\ll{2}}+M.
\end{equation}
If $j=2$, we have:
\begin{equation}\label{h34}
\begin{array}{l}
\norm{\nabn^2(|\th|^2)}_{\ll{2}}\lesssim \norm{\th}_{\ll{\infty}}\norm{\nabla^2\th}_{\ll{2}}
+\norm{\nabla\th}^2_{\ll{4}}\lesssim M^2,
\end{array}
\end{equation}
where we have used Proposition \ref{p3} to bound the $\ll{\infty}$ norm. If $j\geq 3$, using \eqref{h3} for $j$ and Leibnitz formula yields:
\begin{equation}\label{h35}
\norm{\nabn^j(|\th|^2)}_{\ll{2}}\lesssim\sum_{0\leq p\leq j/2}\norm{\nabla^p\th}_{\ll{\infty}}\norm{\nabla^{j-p}\th}_{\ll{2}}\leq C(M).
\end{equation}
We now estimate $\norm{[\nabn^j,a^{-1}\lap]a}_{\ll{2}}$ with the help of \eqref{h29}. We have:
\begin{equation}\label{h36}
\begin{array}{ll}
\ds\norm{\ana\nabb\nabn^ja}_{\ll{2}} & \ds\lesssim \norm{\ana}_{\l{\infty}{4}}\norm{\ana\nabb\nabn^ja}_{\l{2}{4}}\\
& \ds\lesssim \ep\norm{\nabb^2\nabn^ja}^{1/2}_{\ll{2}}\norm{\nabb\nabn^ja}^{1/2}_{\ll{2}},
\end{array}
\end{equation}
where we have used \eqref{eq:GNirenberg} and \eqref{appboot}. Using the estimate \eqref{boot} for $a$, and the Gagliardo-Nirenberg inequality \eqref{eq:GNirenberg}, we have:
\bea\label{h39bis}
\norm{a^{-1}\lap a\nabn^ja}_{\ll{2}}&\lesssim& \norm{a^{-1}}_{\ll{\infty}}\norm{\lap a}_{\l{\infty}{4}}\norm{\nabn^ja}_{\l{2}{4}}\\
\nn&\les& \norm{\lap a}_{\ll{2}}^{\frac{1}{2}}\norm{\nabb\lap a}_{\ll{2}}^{\frac{1}{2}}\norm{\nabn^ja}_{\ll{2}}^\frac{1}{2}\norm{\nabb\nabn^ja}_{\ll{2}}^\frac{1}{2}\\
\nn&\les& C(M)\norm{\nabb\nabn^ja}_{\ll{2}}^\frac{1}{2},
\eea
where we used in the last inequality  \eqref{h3} for $j$ and for 2. Using \eqref{h3} for $j$ and for 2 yields:
\begin{equation}\label{h39}
\norm{|\ana|^2\nabn^ja}_{\ll{2}}\lesssim \norm{\ana}_{\ll{\infty}}^2\norm{\nabn^ja}_{\ll{2}}\leq C(M),
\end{equation}
where we have used Proposition \ref{p3} to bound $\norm{\ana}_{\ll{\infty}}$. Using \eqref{h3} for $j$ and for 2 together with \eqref{h30} yields:
\begin{equation}\label{h40}
\normm{\left(\prod_{l=1}^p\nabb^{t^1_l}\nabn^{t^2_l}a\right)\left(\prod_{m=1}^q\nabb^{v^1_l}\nabn^{v^2_l}\th\right)
\left(\prod_{n=1}^r\nabb^{w^1_l}\nabn^{w^2_l}R\right)\nabb^{s_1}\nabn^{s_2}a}_{\ll{2}}\leq C(M).
\end{equation}
\eqref{h29}, \eqref{h36}, \eqref{h39bis}, \eqref{h39} and \eqref{h40} yield:
\begin{equation}\label{h41}
\norm{[\nabn^j,a^{-1}\lap]a}_{\ll{2}}\lesssim C(M)(1+\norm{\nabb\nabn^ja}_{\ll{2}}^\frac{1}{2})+\ep(\norm{\nabb^2\nabn^ja}_{\ll{2}}+\norm{\nabb\nabn^{j}a}_{\ll{2}}).
\end{equation}
Finally, \eqref{h33}, \eqref{h34}, \eqref{h35} and \eqref{h41} yield:
\begin{equation}\label{h42}
\norm{h}_{\ll{2}}\lesssim C(M)(1+C(M)\norm{\nabb\nabn^ja}_{\ll{2}}^\frac{1}{2})+\ep(\norm{\nabb^2\nabn^ja}_{\ll{2}}+\norm{\nabb\nabn^{j}a}_{\ll{2}}).
\end{equation}
Proposition \ref{p7}, \eqref{h31} and \eqref{h42} yield:
\begin{equation}\label{h46}
\norm{\nabb\nabn^j a}_{\l{\infty}{2}}+\norm{\nabb^2\nabn^ja}_{\ll{2}}+\norm{\nabn^{j+1}a}_{\ll{2}}\leq C(M). 
\end{equation}
Now, \eqref{h3} for $j$, \eqref{h27} and \eqref{h46} yield:
\begin{equation}\label{hh46}
\norm{\nabla^{j+1}a}_{\ll{2}}\leq C(M). 
\end{equation}

Let us now estimate $\norm{\nabb^{j+1}\th}_{\ll{2}}$. We differentiate the first equation of \eqref{struct1} by $\nabb^{j+1}$, which yields together with \eqref{h1} and \eqref{h3} for $j$:
\begin{equation}\label{h47}
\norm{\nabb^{j+1}\trt}_{\ll{2}}\leq \norm{\nabb^{j+1}a}_{\ll{2}}+\norm{\nabb^{j+1}k_{NN}}_{\ll{2}}\leq C(M). 
\end{equation}
Differentiating \eqref{h19} by $\nabb^{j-1}$, we obtain:
\begin{equation}\label{h48}
\nabb^{j-1}\lap\hth = 2\nabb^{j-1}(K\hth)-\nabb^{j-1}(^*\mathcal{D}_2(\nabb\trt))-2\nabb^{j-1}(^*\mathcal{D}_2(R_{N.})).
\end{equation}
\eqref{h1} and \eqref{h47} yield:
\begin{equation}\label{h49}
\norm{\nabb^{j-1}(^*\mathcal{D}_2(\nabb\trt))}_{\ll{2}}+\norm{\nabb^{j-1}(^*\mathcal{D}_2(R_{N.}))}_{\ll{2}}\leq C(M).
\end{equation}
Using Leibnitz formula together with \eqref{gauss1}, \eqref{h1} and \eqref{h3} for $j$, we obtain:
\begin{equation}\label{h50}
\norm{\nabb^{j-1}(K\hth)}_{\ll{2}}\leq C(M).
\end{equation}
\eqref{h48}, \eqref{h49} and \eqref{h50} yield:
\begin{equation}\label{h51}
\norm{\nabb^{j-1}\lap\hth}_{\ll{2}}\leq C(M).
\end{equation}
Now, \eqref{commut2} yields:
\begin{equation}\label{h52}
\lap\nabb^{j-1}\hth = \nabb^{j-1}\lap\hth+\sum_{p=1}^{j-1}\nabb^{j-1-p}K\nabb^{p}\hth,
\end{equation}
which together with \eqref{gauss1}, \eqref{h1}, \eqref{h3} and \eqref{h51} implies:
\begin{equation}\label{h53}
\norm{\lap\nabb^{j-1}\hth}_{\ll{2}} \leq C(M).
\end{equation}
The analog of \eqref{vboch} for 2-tensors, \eqref{h21} and \eqref{h53} yield:
\begin{equation}\label{h54}
\begin{array}{ll}
& \ds\norm{\nabb^{j+1}\hth}_{\ll{2}}\\
\ds\lesssim & \ds\norm{\lap\nabb^{j-1}\hth}_{\ll{2}}+\norm{K}_{\l{\infty}{2}}^{1/2}\norm{\nabb^j\hth}_{\l{2}{4}}+\norm{K}_{\ll{3}}\norm{\nabb^{j-1}\hth}_{\ll{6}}\\
\ds\lesssim & \ds (M+\ep^2)(\norm{\nabb^j\hth}_{\l{2}{4}}+\norm{\nabb^{j-1}\hth}_{\ll{6}})+C(M).
\end{array}
\end{equation}
Together with \eqref{eq:GNirenberg} and \eqref{h3} for $j$, this implies:
\begin{equation}\label{h55}
\norm{\nabb^{j+1}\hth}_{\ll{2}} \leq C(M).
\end{equation}
Finally, \eqref{h47} and \eqref{h55} yield:
\begin{equation}\label{h56}
\norm{\nabb^{j+1}\th}_{\ll{2}} \leq C(M).
\end{equation}

Let us now estimate $\norm{\nabn^{j+1}\th}_{\ll{2}}$. Differentiating the last equation of \eqref{struct1} 
by $\nabn^j$, taking the norm in $\ll{2}$, using the computation \eqref{ap11} of $[\nabn^j,\nabb^2]$ proved in the  Appendix, \eqref{h3} for $j$, \eqref{h46}, and estimating the various quantity in the same fashion as previously, we obtain:
\begin{equation}\label{h57}
\norm{\nabn^{j+1}\th}_{\ll{2}} \leq C(M).
\end{equation}

Now, \eqref{h3} for $j$, \eqref{h27}, \eqref{h56} and \eqref{h57} yield:
\begin{equation}\label{h58}
\norm{\nabla^{j+1}\th}_{\ll{2}} \leq C(M).
\end{equation}
Also, differentiating the last equation of \eqref{struct1} by $\nabla^j$, taking the norm in $\ll{2}$, 
\eqref{h3} for $j$, \eqref{h58}, and estimating the various quantity in the same fashion as previously, we obtain:
\begin{equation}\label{h59}
\norm{\nabb^2\nabla^{j-1}a}_{\ll{2}} \leq C(M).
\end{equation}
Finally, \eqref{hh46}, \eqref{h58} and \eqref{h59} yield \eqref{h3} for $j+1$ so that \eqref{h3} is true 
for all $j$. This concludes the proof of Proposition \ref{h2}. 

\subsection{Proof of Lemma \ref{h26}}\lab{sec:applemma1}

Let us first recall the following result (see for instance \cite{CoMe}). If 
the symbol $a(x,\xi)$ satisfies:
\begin{equation}\label{ap1}
\sup_{\xi}\norm{a(.,\xi)}_{H^{3/2+\delta}(\R^3)}<+\infty
\end{equation}
for some $\delta>0$, then the pseudodifferential operator $a(x,D)$ acting on $\R^3$ is bounded on $L^2(\R^3)$. Now, assume that the symbol $a(x,\xi)$ satisfies:
\begin{equation}\label{ap2}
\sup_{\xi}\norm{a(.,\xi)}_{H^{5/2+\delta}(\R^3)}<+\infty
\end{equation}
for some $\delta>0$ and:
\begin{equation}\label{ap3}
a(x,\xi)\geq 1\textrm{ for all }(x,\xi).
\end{equation}
Then, using the previous result and the symbolic calculus for the adjoint and the composition of pseudodifferential operators, one can show that:
\begin{equation}\label{ap4}
a(x,D)-\sqrt{a}(x,D)^*\sqrt{a}(x,D)\textrm{ is bounded from }H^{-1}(\R^3)\textrm{ to }L^2(\R^3).
\end{equation}
Thus, under the assumptions \eqref{ap2} \eqref{ap3}, the Garding inequality holds:
\begin{equation}\label{ap5}
(a(x,D)v,v)\geq -C\norm{v}_{H^{-1}(\R^3)},
\end{equation}
where $v$ is in $L^2(\R^3)$ and $C>0$ is a constant depending in the quantity in \eqref{ap2}. 

Now, consider 
\begin{equation}\label{ap6}
a(x,\xi)=2^{l+1}\left(\left(N.\frac{\xi}{|\xi|}\right)^{2l}+\left(e.\frac{\xi}{|\xi|}\right)^{2l}\right)-1.
\end{equation}
Then, we clearly have \eqref{ap3}. We also have \eqref{ap2}:
\begin{equation}\label{ap7}
\sup_{\xi}\norm{a(.,\xi)}_{H^{5/2+\delta}}\leq C(\norm{N}_{H^{5/2+\delta}})
\leq C(\norm{\nabla^3N}_{\ll{2}})\leq C(M),
\end{equation}
where we have used \eqref{frame1}, \eqref{h3} with $j=2$ and $\norm{\nabb\nabn^2a}_{\ll{2}}\leq C(M)$. Thus, $a$ defined by \eqref{ap6} satisfies \eqref{ap5}, which together with the choice $v=|D|^lF$ concludes the proof of Lemma \ref{h26}. 

\subsection{Proof of Lemma \ref{h28}}\lab{sec:applemma2} 

We start by deriving a formula for the 
commutator $[\nabn^j,\nabb]$. Let $F$ a tensor on $S$. Using \eqref{commut}, 
one proves the following commutator formula by iteration:
\begin{equation}\label{ap8}
\begin{array}{ll}
\ds[\nabn^j,\nabb]F= & \ds j\nabb a\nabn^jF\\
& \ds +\nabb\nabn^{t_1} a\left(\prod_{l=2}^p\nabn^{t_l} a\right)\left(\prod_{m=1}^q\nabn^{v_m}\th\right)
\left(\prod_{n=1}^r\nabn^{w_n}R\right)\nabn^{s}F\\
& \ds +\left(\prod_{l=1}^p\nabn^{t_l} a\right)\left(\prod_{m=1}^q\nabn^{v_m}\th\right)
\left(\prod_{n=1}^r\nabn^{w_n}R\right)\nabb\nabn^{s}F,
\end{array}
\end{equation}
where $t_l, v_m$ and $w_n$ satisfy:
\begin{equation}\label{ap9}
\begin{array}{l}
t_1+\cdots+t_p+v_1+\cdots+v_q+w_1+\cdots+w_r+s=j-q-r,\\
t_l\leq j-1, 1\leq l\leq p, s\leq j-1.
\end{array}
\end{equation}
Then, using the fact that:
\begin{equation}\label{ap10}
[\nabn^j,\nabb^2]=[\nabn^j,\nabb]\nabb+\nabb[\nabn^j,\nabb],
\end{equation}
we deduce from \eqref{ap8} and \eqref{ap9} the following commutator formula:
\begin{equation}\label{ap11}
\begin{array}{l}
\ds[\nabn^j,\nabb^2]F=\left(\prod_{l=1}^p\nabb^{t^1_l}\nabn^{t^2_l} a\right)\left(\prod_{m=1}^q\nabb^{v^1_m}\nabn^{v^2_m}\th\right)
\left(\prod_{n=1}^r\nabb^{w^1_n}\nabn^{w^2_n}R\right)\nabb^{s_1}\nabn^{s_2}F,
\end{array}
\end{equation}
where $t^1_l, t^2_l, v^1_m, v^2_m, w^1_n$ and $w^2_n$ satisfy:
\begin{equation}\label{ap12}
\begin{array}{l}
t^1_1+\cdots+t^1_p+v^1_1+\cdots+v^1_q+w^1_1+\cdots+w^1_r+s_1=2,\\
t^2_1+\cdots+t^2_p+v^2_1+\cdots+v^2_q+w^2_1+\cdots+w^2_r+s_2=j-q-r,\\
t^2_l\leq j-1, 0\leq l\leq p,\,s_1+s_2\leq j+1.
\end{array}
\end{equation}
Now, using \eqref{commut1}, we have for any scalar $f$ on $S$:
\begin{equation}\label{ap13}
\begin{array}{lll}
\ds [\nabn^j,a^{-1}\lap]f & = & \sum_{l=1}^j\nabn^{l-1}[\nabn,a^{-1}\lap]\nabn^{j-l}f\\
\ds & = & \ds\sum_{l=1}^j\nabn^{l-1}(-(\trt+a^{-1}\nabn a)\lap -2\hth\c\nabb^2+2a^{-1}\nabb a\c\nabb\nabn\\
\ds & & \ds +a^{-1}\lap\nabn-2R_{N.}\c\nabb-\nabb\trt\c\nabb+2\hth\c a^{-1}\nabb a\c\nabb)\nabn^{j-l}f\\
\ds & = & \ds 2\sum_{l=1}^j\nabn^{l-1}(a^{-1}\nabb a\nabb\nabn^{j+1-l}f+a^{-1}\lap a\nabn^{j+1-l}f)\\
& & \ds+\sum_{l=1}^j\nabn^{l-1}(-(\trt+a^{-1}\nabn a)\lap -2\hth\c\nabb^2-2R_{N.}\c\nabb\\
&&\ds -\nabb\trt\c\nabb+2\hth\c a^{-1}\nabb a\c\nabb)\nabn^{j-l}f.
\end{array}
\end{equation}
We rewrite the first term in the right-hand side of \eqref{ap13}:
\begin{equation}\label{ap14}
\begin{array}{ll}
&\ds \sum_{l=1}^j\nabn^{l-1}(a^{-1}\nabb a\nabb\nabn^{j+1-l}f+a^{-1}\lap a\nabn^{j+1-l}f)\\
= & ja^{-1}\nabb a\nabb\nabn^jf+ja^{-1}\lap a\nabn^jf\\
&\ds +\sum_{l=1}^ja^{-1}\nabb a[\nabn^{l-1},\nabb]\nabn^{j+1-l}f+\sum_{l=1}^j\sum_{m=1}^{l-1}\nabn^m(a^{-1}\nabb a)\nabn^{l-1-m}\nabb\nabn^{j+1-l}f\\
\ds &\ds +\sum_{l=1}^j\sum_{m=1}^{l-1}\nabn^m(a^{-1}\lap a)\nabn^{j-m}f\\
\ds = & \ds ja^{-1}\nabb a\nabb\nabn^jf+ja^{-1}\lap a\nabn^jf+\frac{j(j-1)}{2}|\nabb a|^2\nabn^jf\\
&\ds+\sum_{l=1}^ja^{-1}\nabb a([\nabn^{l-1},\nabb]-(l-1)\nabb a\nabn^{l-1})\nabn^{j+1-l}f\\
&\ds +\sum_{l=1}^j\sum_{m=1}^{l-1}\nabn^m(a^{-1}\nabb a)\nabn^{l-1-m}\nabb\nabn^{j+1-l}f+\sum_{l=1}^j\sum_{m=1}^{l-1}\nabn^m(a^{-1}\lap a)\nabn^{j-m}f.
\end{array}
\end{equation}
Finally, \eqref{ap8}, \eqref{ap9}, \eqref{ap11}, \eqref{ap12}, \eqref{ap13} and \eqref{ap14} yield \eqref{h29} and \eqref{h30}. 

\section{Proof of the estimates for the commutator $[\nabla_{aN},P_j]$}\lab{sec:proofpropcomm}

In this section, we prove the commutator estimates stated in section \ref{sec:statepropprodcomm}.

\subsection{Proof of Proposition \ref{greveinf1}}

Proceeding as in \eqref{ad28} \eqref{ad29} \eqref{ad30}, we have:
\begin{equation}\label{clp1}
[\nabna, P_j]f=\int_0^\infty m_j(\tau)V(\tau) d\tau.
\end{equation}
where $V$ is given by:
\begin{equation}\label{clp2}
(\partial_{\tau}-\lap)V(\tau)=[\nabna,\lap]U(\tau)f,\,V(0)=0.
\end{equation}

In view of \eqref{clp1}, we have:
$$\norm{[\nabna, P_j]f}_{\ll{2}}\les \int_0^\infty m_j(\tau)\norm{V(\tau)}_{\ll{2}} d\tau.$$
Thus, to obtain \eqref{commLP1}, it suffices to show:
\be\lab{clp3}
\sup_\tau\norm{V(\tau)}_{\ll{2}}\les  \ep\norm{\La^{\frac{1}{2}+\delta}f}_{L^2(S)}+\ep\norm{\La^\delta f}_{\l{\infty}{2}}.
\ee
From now on, we focus on proving \eqref{clp3}. 

In view of \eqref{clp2} and the heat flow estimate \eqref{heatF2}, we have:
$$\norm{V(\tau)}^2_{L^2(\p)}+\int_0^\tau\norm{\nabb V(\tau')}^2_{L^2(\p)}d\tau'\lesssim \int_0^\tau\int_{\p}V(\tau')[\nabna,\lap]U(\tau')d\mu_ud\tau'.$$
Using the commutator formula \eqref{dj3}, and integrating the second order derivative by parts, we obtain the following estimate:
\bee
&&\norm{V(\tau)}^2_{L^2(\p)}+\int_0^\tau\norm{\nabb V(\tau')}^2_{L^2(\p)}d\tau'\\
\nn&\lesssim& (\norm{a\nabb(\th)}_{\lp{2}}+\norm{\nabb(a)\th}_{\lp{2}}+\norm{aR}_{\lp{2}})\int_0^\tau\norm{\nabb U(\tau')}_{\lp{4}}\norm{V(\tau')}_{\lp{4}}d\tau'\\
\nn&& \norm{a\th}_{\li{\infty}{4}}\int_0^\tau\norm{\nabb U(\tau')}_{\lp{4}}\norm{\nabb V(\tau')}_{\lp{2}}d\tau.
\eee
Together with the Gagliardo-Nirenberg inequality \eqref{eq:GNirenberg}, Proposition \ref{p1}, and the estimate \eqref{thregx1} for $a$ and $\th$, we obtain:
\bee
&&\norm{V(\tau)}^2_{L^2(\p)}+\int_0^\tau\norm{\nabb V(\tau')}^2_{L^2(\p)}d\tau'\\
\nn&\lesssim& (\norm{a\nabb(\th)}^2_{\lp{2}}+\norm{\nabb(a)\th}^2_{\lp{2}}+\norm{aR}^2_{\lp{2}})\\
\nn&&\times\int_0^\tau {\tau'}^{\frac{1}{2}-\delta}\norm{\nabb^2U(\tau')}_{\lp{2}}\norm{\nabb U(\tau')}_{\lp{2}}d\tau\\
\nn&&+\ep^2\int_0^\tau\norm{\nabb^2U(\tau')}_{\lp{2}}\norm{\nabb U(\tau')}_{\lp{2}}d\tau+\frac{1}{2}\int_0^\tau \norm{\nabb V(\tau')}^2_{\lp{2}}d\tau'\\
\nn&&+\frac{1}{2}\int_0^\tau {\tau'}^{-1+\delta}\norm{V(\tau')}^2_{\lp{2}}d\tau,
\eee
for any $\delta>0$. This yields:
\bee
\norm{V(\tau)}^2_{L^2(\p)}+\int_0^\tau\norm{\nabb V(\tau')}^2_{L^2(\p)}d\tau'&\lesssim& (\norm{a\nabb(\th)}^2_{\lp{2}}+\norm{\nabb(a)\th}^2_{\lp{2}}+\norm{aR}^2_{\lp{2}})\\
\nn&&\times\int_0^\tau {\tau'}^{\frac{1}{2}-\delta}\norm{\nabb^2U(\tau')}_{\lp{2}}\norm{\nabb U(\tau')}_{\lp{2}}d\tau\\
\nn&&+\ep^2\int_0^\tau\norm{\nabb^2U(\tau')}_{\lp{2}}\norm{\nabb U(\tau')}_{\lp{2}}d\tau
\eee
and integrating in $u$, we obtain:
\bea\lab{clp4}
&&\norm{V(\tau)}^2_{L^2(S)}+\int_0^\tau\norm{\nabb V(\tau')}^2_{L^2(S)}d\tau'\\
\nn&\lesssim& \ep^2\sup_u\left(\int_0^\tau {\tau'}^{\frac{1}{2}-\delta}\norm{\nabb^2U(\tau')}_{\lp{2}}\norm{\nabb U(\tau')}_{\lp{2}}d\tau\right)\\
\nn&&+\ep^2\int_0^\tau\norm{\nabb^2U(\tau')}_{\ll{2}}\norm{\nabb U(\tau')}_{\ll{2}}d\tau,
\eea
where we used the estimate \eqref{thregx1} for $a$ and $\th$, and the smallness assumption \eqref{small1} for $R$. 
Now, we have:
\bee
&&\int_0^\tau {\tau'}^{\frac{1}{2}-\delta}\norm{\nabb^2U(\tau')}_{\lp{2}}\norm{\nabb U(\tau')}_{\lp{2}}d\tau'\\
&\les& \int_0^\tau {\tau'}^{1-2\delta}\norm{\lap U(\tau')}^2_{\lp{2}}d\tau'+\int_0^\tau\norm{\nabb U(\tau')}^2_{\lp{2}}d\tau',
\eee
where we used the Bochner inequality for scalars \eqref{eq:Bochconseqbis}. Together with the heat flow estimate \eqref{eq:heat3}, we obtain:
\be\lab{clp5}
\sup_u\left(\int_0^\tau {\tau'}^{\frac{1}{2}-\delta}\norm{\nabb^2U(\tau')}_{\lp{2}}\norm{\nabb U(\tau')}_{\lp{2}}d\tau'\right)\les \norm{\La^{3\delta}f}^2_{\l{\infty}{2}}.
\ee
Also, we have:
\bea\lab{clp6}
&&\int_0^\tau\norm{\nabb^2U(\tau')}_{\ll{2}}\norm{\nabb U(\tau')}_{\ll{2}}d\tau'\\
\nn&\les&\int_0^\tau{\tau'}^{\frac{1}{2}-\delta}\norm{\lap U(\tau')}^2_{\ll{2}}d\tau'+\left(\sup_\tau{\tau'}^{\frac{1}{2}}\norm{\nabb U(\tau)}^2_{\ll{2}}\right)\left(\int_0^\tau{\tau'}^{-1+\delta}\right)\\
\nn&\les& \norm{\La^{\frac{1}{2}+2\delta}f}^2_{\ll{2}},
\eea
where we used in the last inequality the Bochner inequality for scalars \eqref{eq:Bochconseqbis} and a heat flow estimate. Finally, \eqref{clp4}, \eqref{clp5} and \eqref{clp6} imply:
$$\sup_\tau\norm{V(\tau)}_{L^2(S)}\les  \ep\norm{\La^{3\delta}f}_{\l{\infty}{2}}+\ep\norm{\La^{\frac{1}{2}+2\delta}f}_{\ll{2}}.$$ 
Since $\delta>0$ is arbitrary, this yields \eqref{clp3}, which concludes the proof of the proposition.

\subsection{Proof of Proposition \ref{greveinf2}}

Proceeding as in \eqref{clp1} \eqref{clp2}, we have:
\begin{equation}\label{clp7}
[\nabna, P_j]F=\int_0^\infty m_j(\tau)V(\tau) d\tau.
\end{equation}
where $V$ is given by:
\begin{equation}\label{clp8}
(\partial_{\tau}-\lap)V(\tau)=[\nabna,\lap]U(\tau)F,\,V(0)=0.
\end{equation}

In view of \eqref{clp7}, we have:
$$\norm{[\nabna, P_j]F}_{\l{1}{2}}\les \int_0^\infty m_j(\tau)\norm{V(\tau)}_{\l{1}{2}} d\tau.$$
Thus, to obtain \eqref{commLP2}, it suffices to show:
\be\lab{clp9}
\norm{V(\tau)}_{\l{1}{2}}\les  \tau^{\frac{1}{2}-\frac{\delta}{2}}\ep\left(\norm{\nabb F}_{L^2(S)}+\norm{F}_{\l{\infty}{2}}\right).
\ee
From now on, we focus on proving \eqref{clp9}. 

In view of \eqref{clp8} and the heat flow estimate \eqref{heatF2}, we have:
\be\lab{clp10}
\norm{V(\tau)}^2_{L^2(\p)}+\int_0^\tau\norm{\nabb V(\tau')}^2_{L^2(\p)}d\tau'\lesssim \int_0^\tau\int_{\p}V(\tau')[\nabna,\lap]U(\tau')d\mu_ud\tau'.
\ee
Injecting the commutator formula \eqref{commutnabna2} in \eqref{clp10}, integrating by parts, and using the $\lp{\infty}$ estimate \eqref{Linfty}, we obtain the following estimate:
\bee
&&\norm{V(\tau)}^2_{L^2(\p)}+\int_0^\tau\norm{\nabb V(\tau')}^2_{L^2(\p)}d\tau'\\
\nn&\lesssim& (\norm{a\nabb(\th)}_{\lp{2}}+\norm{\nabb(a)\th}_{\lp{2}}+\norm{aR}_{\lp{2}})\int_0^\tau\norm{\nabb U(\tau')}_{\lp{p}}\norm{\nabb V(\tau')}_{\lp{2}}d\tau',
\eee
where $2<p<4$ will be chosen later. Together with the Gagliardo-Nirenberg inequality \eqref{eq:GNirenberg}, we obtain:
\bee
&&\norm{V(\tau)}^2_{L^2(\p)}+\int_0^\tau\norm{\nabb V(\tau')}^2_{L^2(\p)}d\tau'\\
\nn&\lesssim& (\norm{a\nabb(\th)}^2_{\lp{2}}+\norm{\nabb(a)\th}^2_{\lp{2}}+\norm{aR}^2_{\lp{2}})\int_0^\tau\norm{\nabb^2U(\tau')}^{2(1-\frac{2}{p})}_{\lp{2}}\norm{\nabb U(\tau')}^{\frac{4}{p}}_{\lp{2}}d\tau'.
\eee
Taking the square root, and integrating in $u$, this yields:
\bea\lab{clp11}
\norm{V(\tau)}_{\l{1}{2}}&\lesssim& (\norm{a\nabb(\th)}_{\ll{2}}+\norm{\nabb(a)\th}_{\ll{2}}+\norm{aR}_{\ll{2}})\\
\nn&&\times\normm{\left(\int_0^\tau\norm{\nabb^2U(\tau')}^{2(1-\frac{2}{p})}_{\lp{2}}\norm{\nabb U(\tau')}^{\frac{4}{p}}_{\lp{2}}d\tau'\right)^{\frac{1}{2}}}_{L^2_u}\\
\nn&\les& \ep\normm{\left(\int_0^\tau\norm{\nabb^2U(\tau')}^{2(1-\frac{2}{p})}_{\lp{2}}\norm{\nabb U(\tau')}^{\frac{4}{p}}_{\lp{2}}d\tau'\right)^{\frac{1}{2}}}_{L^2_u},
\eea
where we used in the last inequality the estimate \eqref{thregx1} for $a$ and $\th$, and the smallness assumption \eqref{small1} for $R$. Now, in view of the Bochner identity for tensors \eqref{Bochtensorineq}, we have:
\bee
&&\int_0^\tau\norm{\nabb^2U(\tau')}^{2(1-\frac{2}{p})}_{\lp{2}}\norm{\nabb U(\tau')}^{\frac{4}{p}}_{\lp{2}}d\tau'\\
&\les& \int_0^\tau\left(\norm{\lap U(\tau')}_{\lp{2}}+\norm{K}_{\lp{2}}\norm{\nabb U(\tau')}_{\lp{2}}+\norm{K}^2_{\lp{2}}\norm{U(\tau')}_{\lp{2}}\right)^{2(1-\frac{2}{p})}\\
\nn&&\times\norm{\nabb U(\tau')}^{\frac{4}{p}}_{\lp{2}}d\tau'\\
&\les& \left( \int_0^\tau\norm{\lap U(\tau')}^2_{\lp{2}}d\tau'\right)^{1-\frac{2}{p}}\left(\int_0^\tau\norm{\nabb U(\tau')}^2_{\lp{2}}d\tau'\right)^{\frac{2}{p}}\\
\nn&&+\norm{K}_{\lp{2}}^{4(1-\frac{2}{p})}\int_0^\tau(\norm{\nabb U(\tau')}^2_{\lp{2}}+\norm{U(\tau')}^2_{\lp{2}})d\tau'.
\eee
Integrating in $u$, and using the fact that $2<p<4$, this yields:
\bee
&&\int_{-2}^2\int_0^\tau\norm{\nabb^2U(\tau')}^{2(1-\frac{2}{p})}_{\lp{2}}\norm{\nabb U(\tau')}^{\frac{4}{p}}_{\lp{2}}d\tau'du\\
\nn&\les& \tau^{\frac{2}{p}}\left(\sup_\tau\norm{\nabb U(\tau')}^2_{\ll{2}}+\int_0^\tau\norm{\lap U(\tau')}^2_{\ll{2}}d\tau'\right)\\
\nn&&+\norm{K}_{\ll{2}}^{4(1-\frac{2}{p})}\tau^{\frac{4}{p}-1}\bigg(\sup_\tau(\norm{\nabb U(\tau')}^2_{\ll{2}}+\norm{U(\tau')}^2_{\ll{2}})\\
\nn&&+\sup_u\left(\int_0^\tau(\norm{\nabb U(\tau')}^2_{\lp{2}}+\norm{U(\tau')}^2_{\lp{2}})d\tau'\right)\bigg).
\eee
Together with the estimate \eqref{thregx1} for $K$, and the heat flow estimates \eqref{eq:l2heat1} and \eqref{eq:l2heatnab}, we obtain:
\bee
\int_{-2}^2\int_0^\tau\norm{\nabb^2U(\tau')}^{2(1-\frac{2}{p})}_{\lp{2}}\norm{\nabb U(\tau')}^{\frac{4}{p}}_{\lp{2}}d\tau'du\les \left(\tau^{\frac{2}{p}}+\tau^{\frac{4}{p}-1}\right)(\norm{\nabb F}^2_{\ll{2}}+\norm{F}^2_{\l{\infty}{2}}).
\eee
Together with \eqref{clp11}, we finally obtain:
$$\norm{V(\tau)}_{\l{1}{2}}\lesssim\ep \left(\tau^{\frac{1}{p}}+\tau^{\frac{2}{p}-\frac{1}{2}}\right)(\norm{\nabb F}_{\ll{2}}+\norm{F}_{\l{\infty}{2}}).$$
Since $\delta>0$, we may choose $p$ such that:
$$2<p<\min\left(4, \frac{4}{2-\delta}\right),$$
which yields \eqref{clp9}. This concludes the proof of the proposition.

\subsection{Proof of Corollary \ref{greveinf3}}

Using the inequality \eqref{ad27bis}, the fact that $P_jF\equiv 0$ on $u=-2$, and properties (ii) and (iii) of Theorem \ref{thm:LP}, we have:
\begin{equation}\label{clp14}
\begin{array}{ll}
&\ds\sum_{j\geq 0}2^j\norm{P_jF}^2_{\l{\infty}{2}}\\
\ds\lesssim &\ds\sum_{j\ge 0} 2^j\left(\int_{-2}^{2}\norm{P_jF}_{\lp{2}}\norm{\nabn P_jF}_{\lp{2}}du+\norm{P_jF}_{\ll{2}}\norm{\nabb P_jF}_{\ll{2}}\right)\\
\ds\lesssim &\ds\sum_{j\ge 0} 2^j\left(\int_{-2}^{-2}\norm{P_jF}_{\lp{2}}\norm{\nabn P_jF}_{\lp{2}}du\right)+\sum_{j\ge 0}2^{2j}\norm{P_jF}^2_{\ll{2}}\\
\ds\lesssim &\ds\sum_{j\ge 0} 2^j\left(\int_{-2}^{-2}\norm{P_jF}_{\lp{2}}\norm{\nabna P_jF}_{\lp{2}}du\right)+\norm{\nabb F}_{\ll{2}}^2,
\end{array}
\end{equation}
where we used the estimate \eqref{thregx1} for $a$ in the last inequality. Now, we have:
$$\norm{\nabna P_jF}_{\lp{2}}\les \norm{P_j(\nabna F)}_{\lp{2}}+\norm{[\nabna, P_j]F}_{\lp{2}}$$
which together with \eqref{clp14}, and the properties (ii) and (iii) of Theorem \ref{thm:LP} implies:
\bea\lab{clp14bis}
&&\sum_{j\geq 0}2^j\norm{P_jF}^2_{\l{\infty}{2}}\\
\nn&\lesssim &\sum_{j\ge 0} 2^j\norm{P_jF}_{\ll{2}}\norm{P_j(\nabna F)}_{\ll{2}}+\sum_{j\ge 0} 2^j\norm{P_jF}_{\l{\infty}{2}}\norm{[\nabna, P_j]F}_{\l{1}{2}}\\
\nn&&+\norm{\nabb F}_{\ll{2}}^2\\
\nn&\lesssim &\sum_{j\ge 0} 2^{2j}\norm{P_jF}^2_{\ll{2}}+\sum_{j\ge 0}\norm{P_j(\nabna F)}^2_{\ll{2}}\\
\nn&&+\left(\sum_{j\ge 0}2^j\norm{P_jF}^2_{\l{\infty}{2}}\right)^{\frac{1}{2}}\left(\sum_{j\ge 0}2^j\norm{[P_j,\nabna]F}^2_{\l{1}{2}}\right)^{\frac{1}{2}}+\norm{\nabb F}_{\ll{2}}^2\\
\nn&\lesssim & \left(\sum_{j\ge 0}2^j\norm{P_jF}^2_{\l{\infty}{2}}\right)^{\frac{1}{2}}\left(\sum_{j\ge 0}2^j\norm{[P_j,\nabna]F}^2_{\l{1}{2}}\right)^{\frac{1}{2}}+\norm{\nabla F}_{\ll{2}}^2.
\eea
This yields:
\be\label{clp15}
\sum_{j\geq 0}2^j\norm{P_jF}^2_{\l{\infty}{2}}\lesssim  \sum_{j\ge 0}2^j\norm{[P_j,\nabna]F}^2_{\l{1}{2}}+\norm{\nabla F}_{\ll{2}}^2.
\ee

Now, we have in view of the commutator estimate \eqref{commLP2}:
$$\norm{[P_j,\nabna]F}^2_{\l{1}{2}}\les 2^{-j(1-\delta)}\ep(\norm{\nabb F}_{\ll{2}}+\norm{F}_{\l{\infty}{2}}),$$
for any $\delta>0$. In view of Corollary \ref{c0} and the fact that $F\equiv 0$ on $u=-2$, we obtain:
$$\norm{[P_j,\nabna]F}^2_{\l{1}{2}}\les 2^{-j(1-\delta)}\ep\norm{F}_{H^1(S)}.$$
Together with \eqref{clp15}, this yields:
$$\sum_{j\geq 0}2^j\norm{P_jF}^2_{\l{\infty}{2}}\lesssim  \left(1+\sum_{j\ge 0}2^{-j(1-2\delta)}\right)\norm{F}_{H^1(S)}^2.$$
Choosing $0<\delta<1/2$, we obtain:
\be\lab{clp16}
\sum_{j\geq 0}2^j\norm{P_jF}^2_{\l{\infty}{2}}\les \norm{F}^2_{H^1(S)},
\ee
which is the wanted estimate. This concludes the proof of the corollary. 

\subsection{Proof of Proposition \ref{greveinf4}}

In view of \eqref{clp1}, we have:
$$\norm{[\nabna, P_j]f}_{\l{1}{2}}\les \int_0^\infty m_j(\tau)\norm{V(\tau)}_{\l{1}{2}} d\tau,$$
where $V$ is given by:
\begin{equation}\label{clp22}
(\partial_{\tau}-\lap)V(\tau)=[\nabna,\lap]U(\tau)f,\,V(0)=0.
\end{equation}
Thus, to obtain \eqref{commLP3}, it suffices to show:
\be\lab{clp23}
\norm{\La^{-\a}V(\tau)}_{\l{1}{2}}+\int_{-2}^{2}\left(\int_0^\tau\norm{\nabb\La^{-\a}V(\tau')}^2_{L^2(\p)}d\tau'\right)^{\frac{1}{2}}du\lesssim \ep\norm{\La^{-\delta}F}_{L^2(S)}.
\ee
Indeed, once \eqref{clp23} is obtained, one proceeds as in \eqref{ad61} \eqref{ad32} to deduce \eqref{commLP3}. From now on, we focus on proving \eqref{clp23}. 

In view of \eqref{clp22} and the heat flow estimate \eqref{eq:l2heat1bis}, we have:
$$\norm{\La^{-\alpha}V(\tau)}^2_{L^2(\p)}+\int_0^\tau\norm{\nabb\La^{-\alpha} V(\tau')}^2_{L^2(\p)}d\tau'\lesssim \int_0^\tau\int_{\p}\La^{-2\a}V(\tau')[\nabna,\lap]U(\tau')d\mu_ud\tau'.$$
Injecting the commutator formula \eqref{dj3}, integrating by parts, we obtain the following estimate:
\bea\lab{clp24}
&&\norm{\La^{-\a}V(\tau)}^2_{L^2(\p)}+\int_0^\tau\norm{\nabb\La^{-\a} V(\tau')}^2_{L^2(\p)}d\tau'\\
\nn&\lesssim& (\norm{a\nabb(\th)}_{\lp{2}}+\norm{\nabb(a)\th}_{\lp{2}}+\norm{aR}_{\lp{2}})\\
\nn&&\times\int_0^\tau\norm{\nabb U(\tau')}_{\lp{p}}\norm{\nabb\La^{-2\a} V(\tau')}_{\lp{2}}d\tau',
\eea
where 
$$2<p<\frac{2}{1-\a}$$ 
will be chosen later. Now, we have in view of \eqref{La1} and \eqref{interpolLa}:
$$\norm{\nabb\La^{-2\a} V(\tau')}_{\lp{2}}\les \norm{\La^{-\a} V(\tau')}^\a_{\lp{2}}\norm{\nabb\La^{-2\a} V(\tau')}^{1-\a}_{\lp{2}}$$
which together with \eqref{clp24} implies:
\bea\lab{clp25}
&&\norm{\La^{-\a}V(\tau)}^2_{L^2(\p)}+\int_0^\tau\norm{\nabb\La^{-\a} V(\tau')}^2_{L^2(\p)}d\tau'\\
\nn&\lesssim& (\norm{a\nabb(\th)}^2_{\lp{2}}+\norm{\nabb(a)\th}^2_{\lp{2}}+\norm{aR}^2_{\lp{2}})\int_0^\tau{\tau'}^{\a_-}\norm{\nabb U(\tau')}^2_{\lp{p}}d\tau'.
\eea
The Gagliardo-Nirenberg inequality \eqref{eq:GNirenberg} and the Bochner inequality \eqref{eq:Bochconseqbis} imply: 
\bee 
&&\int_0^\tau{\tau'}^{\a_-}\norm{\nabb U(\tau')}^2_{\lp{p}}d\tau'\\
&\les& \int_0^\tau{\tau'}^{\a_-}\norm{\nabb U(\tau')}^{\frac{4}{p}}_{\lp{2}}\norm{\lap U(\tau')}^{2(1-\frac{2}{p})}_{\lp{2}}d\tau'\\
&\les& \int_0^\tau{\tau'}^b\norm{\nabb U(\tau')}^2_{\lp{2}}d\tau'+\int_0^\tau{\tau'}^{1+b}\norm{\lap U(\tau')}^2_{\lp{2}}d\tau'
\eee
where $b$ is given by: 
\be\lab{clp26}
b=\a_--1+\frac{2}{p}.
\ee
We have $0<b<1$ from the choice of $\a$ and $p$. Thus, we obtain in view of the heat flow estimates \eqref{eq:heat1} and \eqref{eq:heat3}:
$$\int_0^\tau{\tau'}^{\a_-}\norm{\nabb U(\tau')}^2_{\lp{p}}d\tau'\les \norm{\La^{-b_-}f}^2_{\lp{2}}.$$
Together with \eqref{clp25}, this yields:
\bee
&&\norm{\La^{-\a}V(\tau)}^2_{L^2(\p)}+\int_0^\tau\norm{\nabb\La^{-\a} V(\tau')}^2_{L^2(\p)}d\tau'\\
\nn&\lesssim& (\norm{a\nabb(\th)}^2_{\lp{2}}+\norm{\nabb(a)\th}^2_{\lp{2}}+\norm{aR}^2_{\lp{2}}) \norm{\La^{-b_-}f}^2_{\lp{2}}.
\eee
Integrating in $u$, this yields:
\bea\lab{clp27}
&&\norm{\La^{-\a}V(\tau)}_{\l{1}{2}}+\int_{-2}^{2}\left(\int_0^\tau\norm{\nabb\La^{-\a}V(\tau')}^2_{L^2(\p)}d\tau'\right)^{\frac{1}{2}}du\\
\nn&\lesssim& (\norm{a\nabb(\th)}_{\ll{2}}+\norm{\nabb(a)\th}_{\ll{2}}+\norm{aR}_{\ll{2}}) \norm{\La^{-b_-}f}_{\ll{2}}\\
\nn&\les&\ep\norm{\La^{-b_-}F}_{L^2(S)},
\eea
where we used in the last inequality the estimate \eqref{thregx1} for $a$ and $\th$, and the smallness assumption \eqref{small1} for $R$. Now, in view of the definition \eqref{clp26} of $b$, and since $\delta<\a$, we may choose $p>2$ close enough to 2 such that $b_->\delta$, which together with \eqref{clp27} implies \eqref{clp23}. This concludes the proof of the proposition. 

\subsection{Proof of Proposition \ref{greveinf5}}

Proceeding as in \eqref{clp1} \eqref{clp2}, we have:
\begin{equation}\label{clp28}
[\nabna, P_j]f=\int_0^\infty m_j(\tau)V(\tau) d\tau.
\end{equation}
where $V$ is given by:
\begin{equation}\label{clp29}
(\partial_{\tau}-\lap)V(\tau)=[\nabna,\lap]U(\tau)f,\,V(0)=0.
\end{equation}

In view of \eqref{clp28}, we have:
$$\norm{[\nabna, P_j]f}_{\ll{2}}\les \int_0^\infty m_j(\tau)\norm{V(\tau)}_{\ll{2}} d\tau.$$
Thus, to obtain \eqref{commLP3}, it suffices to show:
\be\lab{clp30}
\norm{V(\tau)}_{\ll{2}}\les  \tau^{\frac{1}{2}-\frac{\delta}{2}}\ep\left(\norm{\lap f}_{L^2(S)}+\norm{\nabb f}_{\l{\infty}{2}}\right).
\ee
From now on, we focus on proving \eqref{clp30}. 

In view of \eqref{clp29} and the heat flow estimate \eqref{heatF2}, we have:
$$\norm{V(\tau)}^2_{L^2(\p)}+\int_0^\tau\norm{\nabb V(\tau')}^2_{L^2(\p)}d\tau'\lesssim \int_0^\tau\int_{\p}V(\tau')[\nabna,\lap]U(\tau')d\mu_ud\tau'.$$
Using the commutator formula \eqref{dj3}, we obtain the following estimate:
\bee
&&\norm{V(\tau)}^2_{L^2(\p)}+\int_0^\tau\norm{\nabb V(\tau')}^2_{L^2(\p)}d\tau'\\
\nn&\lesssim& \norm{a\th}_{\l{\infty}{4}}\int_0^\tau\norm{\nabb^2U(\tau')}_{\lp{2}}\norm{V(\tau')}_{\lp{4}}d\tau'\\
\nn&&+(\norm{a\nabb\th}_{\lp{2}}+\norm{\nabb(a)\th}_{\lp{2}}+\norm{aR}_{\lp{2}})\int_0^\tau\norm{\nabb U(\tau')}_{\lp{4}}\norm{V(\tau')}_{\lp{4}}d\tau'.
\eee
Together with the Gagliardo-Nirenberg inequality \eqref{eq:GNirenberg}, Proposition \ref{p1}, and the estimate \eqref{thregx1} for $a$ and $\th$, we obtain:
\bee
&&\norm{V(\tau)}^2_{L^2(\p)}+\int_0^\tau\norm{\nabb V(\tau')}^2_{L^2(\p)}d\tau'\\
\nn&\lesssim& (\norm{a\nabb(\th)}^2_{\lp{2}}+\norm{\nabb(a)\th}^2_{\lp{2}}+\norm{aR}^2_{\lp{2}})\\
\nn&&\times\int_0^\tau {\tau'}^{\frac{1}{2}-\delta}\norm{\nabb^2U(\tau')}_{\lp{2}}\norm{\nabb U(\tau')}_{\lp{2}}d\tau+\ep^2\int_0^\tau{\tau'}^{1-2\delta}\norm{\nabb^2U(\tau')}_{\lp{2}}^2 d\tau\\
\nn&&+\frac{1}{2}\int_0^\tau \norm{\nabb V(\tau')}^2_{\lp{2}}d\tau'+\frac{1}{2}\int_0^\tau {\tau'}^{-1+\delta}\norm{V(\tau')}^2_{\lp{2}}d\tau,
\eee
for any $\delta>0$. This yields:
\bee
&&\norm{V(\tau)}^2_{L^2(\p)}+\int_0^\tau\norm{\nabb V(\tau')}^2_{L^2(\p)}d\tau'\\
\nn&\lesssim& (\norm{a\nabb(\th)}^2_{\lp{2}}+\norm{\nabb(a)\th}^2_{\lp{2}}+\norm{aR}^2_{\lp{2}})\\
\nn&&\times\int_0^\tau {\tau'}^{\frac{1}{2}-\delta}\norm{\nabb^2U(\tau')}_{\lp{2}}\norm{\nabb U(\tau')}_{\lp{2}}d\tau+\ep^2\int_0^\tau{\tau'}^{\frac{1}{2}-\delta}\norm{\nabb^2U(\tau')}^2_{\lp{2}}d\tau
\eee
and integrating in $u$, we obtain:
\bea\lab{clp31}
&&\norm{V(\tau)}^2_{L^2(S)}+\int_0^\tau\norm{\nabb V(\tau')}^2_{L^2(S)}d\tau'\\
\nn&\lesssim& \ep^2\sup_u\left(\int_0^\tau {\tau'}^{\frac{1}{2}-\delta}\norm{\nabb^2U(\tau')}_{\lp{2}}\norm{\nabb U(\tau')}_{\lp{2}}d\tau\right)+\ep^2\int_0^\tau{\tau'}^{\frac{1}{2}-\delta}\norm{\nabb^2U(\tau')}^2_{\ll{2}}d\tau,
\eea
where we used the estimate \eqref{thregx1} for $a$ and $\th$, and the smallness assumption \eqref{small1} for $R$. 
Now, we have:
\bea\lab{clp32}
&&\int_0^\tau {\tau'}^{\frac{1}{2}-\delta}\norm{\nabb^2U(\tau')}_{\lp{2}}\norm{\nabb U(\tau')}_{\lp{2}}d\tau'\\
\nn&\les& \tau^{1-\delta}\sup_\tau\norm{\nabb U(\tau)}_{\lp{2}}\left(\int_0^\tau \norm{\lap U(\tau')}^2_{\lp{2}}d\tau'\right)^{\frac{1}{2}}\\
\nn&\les& \tau^{1-\delta}\norm{\nabb f}^2_{\lp{2}},
\eea
where we used the Bochner inequality for scalars \eqref{eq:Bochconseqbis} and the heat flow estimate \eqref{eq:l2heatnab}.
Also, we have:
\bea\lab{clp33}
\int_0^\tau{\tau'}^{\frac{1}{2}-\delta}\norm{\nabb^2U(\tau')}^2_{\ll{2}}d\tau'&\les& \tau^{\frac{3}{2}-\delta}\sup_\tau\norm{\lap U(\tau)}^2_{\ll{2}}\\
\nn&\les& \tau^{\frac{3}{2}-\delta}\norm{\lap f}^2_{\ll{2}},
\eea
where we used the Bochner inequality for scalars \eqref{eq:Bochconseqbis} and a heat flow estimate. Finally, \eqref{clp31}, \eqref{clp32} and \eqref{clp33} yield \eqref{clp30}. This concludes the proof of the proposition.

\subsection{Proof of Proposition \ref{greveinf6}}

Let us start by proving the corollary in the case where $f$ is a scalar function on $S$ satisfying the same assumptions that $F$. We estimate $\norm{P_jf}^2_{\l{\infty}{2}}$. Using the inequality \eqref{ad27bis} and the fact that $P_jf\equiv 0$ on $u=-2$, we have:
\begin{equation}\label{clp34}
\begin{array}{ll}
&\norm{P_jf}^2_{\l{\infty}{2}}\\
\ds\lesssim &\ds\norm{P_jf}_{\ll{2}}\norm{\nabn P_jf}_{\ll{2}}+\norm{P_jf}_{\ll{2}}\norm{\nabb P_jf}_{\ll{2}}\\
\ds\lesssim &\ds\norm{P_jf}_{\ll{2}}\norm{\nabna P_jf}_{\ll{2}}+2^j\norm{P_jf}_{\ll{2}}^2,
\end{array}
\end{equation}
where we used  in the last inequality the estimate \eqref{thregx1} for $a$, and the finite band property for $P_j$. 
Now, we have:
$$\norm{\nabna P_jf}_{\lp{2}}\les \norm{P_j(\nabna f)}_{\lp{2}}+\norm{[\nabna, P_j]f}_{\lp{2}}$$
which together with \eqref{clp34} implies:
\bee
&&\norm{P_jf}^2_{\l{\infty}{2}}\\
\nn&\lesssim & \norm{P_jf}_{\ll{2}}\norm{P_j(\nabna f)}_{\ll{2}}+\norm{P_jf}_{\ll{2}}\norm{[\nabna, P_j]f}_{\ll{2}}+2^j\norm{P_jf}_{\ll{2}}^2\\
\nn&\lesssim & \Big(2^{-(2+b)j}\norm{\nabna f}_{\lhs{2}{b}}+2^{-2j}\norm{[\nabna, P_j]f}_{\ll{2}}+2^{-3j}\norm{\lap f}_{\ll{2}}\Big)\norm{\lap f}_{\ll{2}},
\eee
where we used in the last inequality the finite band property for $P_j$, and the definition of $\hs{b}$. Together with \eqref{zoc40} and the commutator estimate \eqref{commLP4}, we obtain:
\bea\lab{clp34bis}
&&\norm{P_jf}^2_{\l{\infty}{2}}\\
\nn&\lesssim & \Big(2^{-(2+b)j}\norm{\nabn f}_{\lhs{2}{b}}+2^{-(3-\delta)j}(\norm{\nabb f}_{\l{\infty}{2}}+\norm{\lap f}_{\ll{2}})\Big)\norm{\lap f}_{\ll{2}},
\eea
for any $\delta>0$. Now, in view of Proposition \ref{p5}, we have:
\be\label{clp35}
\norm{\nabb f}_{\l{\infty}{2}}\les \norm{\lap f}_{\ll{2}}+\norm{\nabn f}_{\ll{2}}.
\ee
Since $b\geq 0$, \eqref{clp34bis} and \eqref{clp35} imply:
\be\lab{clp36}
\norm{P_jf}_{\l{\infty}{2}}\les 2^{-(1+\frac{b}{2})j}(\norm{\nabn f}_{\lhs{2}{b}}+\norm{\nabb^2f}_{\ll{2}}).
\ee
Now, we have:
$$\norm{f}_{\ll{\infty}}\les \sum_{j\geq 0}\norm{P_jf}_{\ll{\infty}}\les \sum_{j\geq 0}2^j\norm{P_jf}_{\l{\infty}{2}},$$
where we used in the last inequality the strong Bernstein inequality for scalars \eqref{eq:strongbernscalarbis}. Together with \eqref{clp37} and the fact that $b>0$, we obtain:
\be\lab{clp37}
\norm{f}_{\ll{\infty}}\les \norm{\nabn f}_{\lhs{2}{b}}+\norm{\nabb^2f}_{\ll{2}}.
\ee

Next, we turn to the case where $F$ is a tensor. Using \eqref{clp37} with $\frac{b}{2}$ instead for $b$, and with $f=|F|^2$, we obtain:
\bee
\norm{F}^2_{\ll{\infty}}&\les& \norm{F\c \nabn F}_{\lhs{2}{\frac{b}{2}}}+\norm{F\nabb^2F}_{\ll{2}}+\norm{\nabb F}^2_{\ll{4}}\\
&\les& \norm{F\c \nabn F}_{\lhs{2}{\frac{b}{2}}}+\norm{F}_{\ll{\infty}}\norm{\nabb^2F}_{\ll{2}}+\norm{\nabb^2F}^2_{\ll{2}}+\norm{\nabn F}^2_{\ll{2}},
\eee
where we used in the last inequality Proposition \ref{p2bis} to estimate $\norm{\nabb F}_{\ll{4}}$. This yields:
\be\lab{clp38}
\norm{F}^2_{\ll{\infty}}\les \norm{F\c \nabn F}_{\lhs{2}{\frac{b}{2}}}+\norm{\nabb^2F}^2_{\ll{2}}+\norm{\nabn F}^2_{\ll{2}}.
\ee

Next, we estimate the first term in the right-hand side of \eqref{clp38}. We have:
\be\lab{clp39}
\norm{P_j(F\c \nabn F)}_{\ll{2}}\les \sum_{l\geq 0}\norm{P_j(F\c P_l\nabn F)}_{\ll{2}}.
\ee
In the case $l>j$, the boundedness of $P_j$ on $\lp{2}$ yields:
\bea\lab{clp40}
2^{\frac{bj}{2}}\norm{P_j(F\c P_l\nabn F)}_{\ll{2}}&\les& 2^{\frac{bj}{2}}\norm{F\c P_l\nabn F}_{\ll{2}}\\
\nn&\les& 2^{\frac{bj}{2}}\norm{F}_{\ll{\infty}}\norm{P_l\nabn F}_{\ll{2}}\\
\nn&\les & 2^{\frac{bj}{2}-bl}\norm{F}_{\ll{\infty}}\norm{\nabn F}_{\lhs{2}{b}}.
\eea
In the case $l\leq j$, we use the finite band property for $P_j$. We have:
\bee
P_j(F\c P_l\nabn F)&=&2^{-2j}P_j(\lap(F\c P_l\nabn F))\\
&=&2^{-2j}P_j(\divb(\nabb F\c P_l\nabn F))+2^{-2j}P_j(\divb(F\c \nabb P_l\nabn F)).
\eee
Together with \eqref{lbz14bis} -note that $\nabb F\c P_l\nabn F$ is a 1-form - and the finite band property for $P_j$, we obtain:
\bee
&&2^{\frac{bj}{2}}\norm{P_j(F\c P_l\nabn F)}_{\ll{2}}\\
\nn&\les& 2^{\frac{bj}{2}-\frac{j}{2}}\norm{\nabb F\c P_l\nabn F}_{\l{2}{\frac{4}{3}}}+2^{\frac{bj}{2}-j}\norm{F\c \nabb P_l\nabn F}_{\ll{2}}\\
\nn&\les& 2^{\frac{bj}{2}-\frac{j}{2}}\norm{\nabb F}_{\l{\infty}{2}}\norm{P_l\nabn F}_{\l{2}{4}}+2^{\frac{bj}{2}-j}\norm{F}_{\ll{\infty}}\norm{\nabb P_l\nabn F}_{\ll{2}}.
\eee
Using Bernstein and the finite band property for $P_l$, this yields for $l\leq j$:
\bea\lab{clp41}
&&2^{\frac{bj}{2}}\norm{P_j(F\c P_l\nabn F)}_{\ll{2}}\\
\nn&\les& 2^{\frac{bj}{2}}(2^{-\frac{j}{2}+\frac{l}{2}}\norm{\nabb F}_{\l{\infty}{2}}+2^{-j+l}\norm{F}_{\ll{\infty}})\norm{P_l\nabn F}_{\ll{2}}\\
\nn&\les& 2^{-\frac{1-b}{2}j+(\frac{1}{2}-b)l}(\norm{\nabb F}_{\l{\infty}{2}}+\norm{F}_{\ll{\infty}})\norm{\nabn F}_{\lhs{2}{b}}.
\eea
We may assume $b<\frac{1}{2}$. Then, using \eqref{clp39}, \eqref{clp40} for $l>j$, and \eqref{clp41} for $l\leq j$, we obtain:
$$2^{\frac{bj}{2}}\norm{P_j(F\c \nabn F)}_{\ll{2}}\les 2^{-\frac{b}{2}j}(\norm{\nabb F}_{\l{\infty}{2}}+\norm{F}_{\ll{\infty}})\norm{\nabn F}_{\lhs{2}{b}},$$
which yields:
$$\norm{F\c \nabn F}_{\lhs{2}{\frac{b}{2}}}\les (\norm{\nabb F}_{\l{\infty}{2}}+\norm{F}_{\ll{\infty}})\norm{\nabn F}_{\lhs{2}{b}}.$$
Together with \eqref{clp38}, we obtain:
$$\norm{F}^2_{\ll{\infty}}\les (\norm{\nabb F}_{\l{\infty}{2}}+\norm{F}_{\ll{\infty}})\norm{\nabn F}_{\lhs{2}{b}}+\norm{\nabb^2F}^2_{\ll{2}}+\norm{\nabn F}^2_{\ll{2}},$$
and thus:
$$\norm{F}_{\ll{\infty}}\les \norm{\nabb F}_{\l{\infty}{2}}+\norm{\nabn F}_{\lhs{2}{b}}+\norm{\nabb^2F}_{\ll{2}}.$$
Now, using Proposition \ref{p2bis} to estimate $\norm{\nabb F}_{\l{\infty}{2}}$, we finally get:
$$\norm{F}_{\ll{\infty}}\les \norm{\nabn F}_{\lhs{2}{b}}+\norm{\nabb^2F}_{\ll{2}}.$$
This concludes the proof of the corollary.

\subsection{Proof of Proposition \ref{greveinf7}}

In view of \eqref{clp1}, we have:
$$\norm{[\nabna, P_j]f}_{\ll{2}}\les \int_0^\infty m_j(\tau)\norm{V(\tau)}_{\ll{2}} d\tau,$$
where $V$ is given by:
\begin{equation}\label{clp42}
(\partial_{\tau}-\lap)V(\tau)=[\nabna,\lap]U(\tau)f,\,V(0)=0.
\end{equation}
Thus, to obtain \eqref{commLP3}, it suffices to show:
\be\lab{clp43}
\int_0^\tau\norm{V(\tau')}^2_{\ll{2}}d\tau'\lesssim \ep\norm{\La^{-(1-\delta)}f}^2_{\l{\infty}{2}}.
\ee
Indeed, \eqref{clp43} yields:
\bee
\norm{[\nabna, P_j]f}_{\ll{2}}&\les& \int_0^\infty m_j(\tau)\norm{V(\tau)}_{\ll{2}} d\tau\\
&\les& \left(\int_0^\infty m_j(\tau)^2 d\tau\right)^{\frac{1}{2}}\left(\int_0^\infty\norm{V(\tau)}^2_{\ll{2}} d\tau\right)^{\frac{1}{2}}\\
&\les& 2^j\ep\norm{\La^{-(1-\delta)}f}_{\l{\infty}{2}},
\eee
which is \eqref{commLP5}. From now on, we focus on proving \eqref{clp43}. 

In view of \eqref{clp42} and the heat flow estimate \eqref{eq:l2heat1bis}, we have:
$$\norm{\La^{-1}V(\tau)}^2_{L^2(\p)}+\int_0^\tau\norm{\nabb\La^{-1} V(\tau')}^2_{L^2(\p)}d\tau'\lesssim \int_0^\tau\int_{\p}\La^{-2}V(\tau')[\nabna,\lap]U(\tau')d\mu_ud\tau'.$$
Injecting the commutator formula \eqref{dj3}, integrating by parts, we obtain the following estimate:
\bea\lab{clp44}
&&\norm{\La^{-1}V(\tau)}^2_{L^2(\p)}+\int_0^\tau\norm{\nabb\La^{-1} V(\tau')}^2_{L^2(\p)}d\tau'\\
\nn&\lesssim& (\norm{a\nabb(\th)}_{\lp{2}}+\norm{\nabb(a)\th}_{\lp{2}}+\norm{aR}_{\lp{2}})\\
\nn&&\times\int_0^\tau\norm{\nabb U(\tau')}_{\lp{p}}\norm{\nabb\La^{-2} V(\tau')}_{\lp{2}}d\tau',
\eea
where 
$$2<p<4$$ 
will be chosen later. Now, we have in view of \eqref{La1}:
$$\norm{\nabb\La^{-2} V(\tau')}_{\lp{2}}\les \norm{\La^{-1} V(\tau')}_{\lp{2}}$$
which together with \eqref{clp44} implies:
\bea\lab{clp45}
&&\norm{\La^{-1}V(\tau)}^2_{L^2(\p)}+\int_0^\tau\norm{\nabb\La^{-1} V(\tau')}^2_{L^2(\p)}d\tau'\\
\nn&\lesssim& (\norm{a\nabb(\th)}^2_{\lp{2}}+\norm{\nabb(a)\th}^2_{\lp{2}}+\norm{aR}^2_{\lp{2}})\int_0^\tau{\tau'}^{1_-}\norm{\nabb U(\tau')}^2_{\lp{p}}d\tau'.
\eea
The Gagliardo-Nirenberg inequality \eqref{eq:GNirenberg} and the Bochner inequality \eqref{eq:Bochconseqbis} imply: 
\bee 
&&\int_0^\tau{\tau'}^{1_-}\norm{\nabb U(\tau')}^2_{\lp{p}}d\tau'\\
&\les& \int_0^\tau{\tau'}^{1_-}\norm{\nabb U(\tau')}^{\frac{4}{p}}_{\lp{2}}\norm{\lap U(\tau')}^{2(1-\frac{2}{p})}_{\lp{2}}d\tau'\\
&\les& \int_0^\tau{\tau'}^b\norm{\nabb U(\tau')}^2_{\lp{2}}d\tau'+\int_0^\tau{\tau'}^{1+b}\norm{\lap U(\tau')}^2_{\lp{2}}d\tau'
\eee
where $b$ is given by: 
\be\lab{clp46}
b=1_--1+\frac{2}{p}.
\ee
We have $0<b<1$ from the choice of $p$. Thus, we obtain in view of the heat flow estimates \eqref{eq:heat1} and \eqref{eq:heat3}:
$$\int_0^\tau{\tau'}^{\a_-}\norm{\nabb U(\tau')}^2_{\lp{p}}d\tau'\les \norm{\La^{-b_-}f}^2_{\lp{2}}.$$
Together with \eqref{clp25}, this yields:
\bee
&&\norm{\La^{-1}V(\tau)}^2_{L^2(\p)}+\int_0^\tau\norm{\nabb\La^{-1} V(\tau')}^2_{L^2(\p)}d\tau'\\
\nn&\lesssim& (\norm{a\nabb(\th)}^2_{\lp{2}}+\norm{\nabb(a)\th}^2_{\lp{2}}+\norm{aR}^2_{\lp{2}}) \norm{\La^{-b_-}f}^2_{\lp{2}}.
\eee
Integrating in $u$, this yields:
\bea\lab{clp47}
&&\norm{\La^{-1}V(\tau)}^2_{\ll{2}}+\int_0^\tau\norm{\nabb\La^{-1}V(\tau')}^2_{\ll{2}}d\tau'\\
\nn&\lesssim& (\norm{a\nabb(\th)}_{\ll{2}}+\norm{\nabb(a)\th}_{\ll{2}}+\norm{aR}_{\ll{2}}) \norm{\La^{-b_-}f}_{\l{\infty}{2}}\\
\nn&\les&\ep\norm{\La^{-b_-}f}_{\l{\infty}{2}},
\eea
where we used in the last inequality the estimate \eqref{thregx1} for $a$ and $\th$, and the smallness assumption \eqref{small1} for $R$. Now, in view of the definition \eqref{clp46} of $b$, and since $\delta>0$, we may choose $p>2$ close enough to 2 such that $b_->1-\delta$, which together with \eqref{clp47} implies \eqref{clp43}. This concludes the proof of the proposition. 

\section{Product estimates}\lab{sec:proofpropprod}

In this section, we prove the commutator estimates stated in section \ref{sec:statepropprodprod}.

\subsection{Proof of Proposition \ref{greveinff1}}

We have:
\be\lab{fichtre1}
\norm{P_j(F\c G\c H)}_{\lp{2}}\les \sum_{l\geq 0}\norm{P_j(F\c G\c P_lH)}_{\lp{2}}.
\ee
We first consider the case where $l\leq j$. Since $0<b<\frac{1}{2}$, there exists a real number $p$ such that:
\be\lab{fichtre2}
\frac{2}{\frac{3}{2}-b}<p<2.
\ee 
We have:
$$\norm{P_j(F\c G\c P_lH)}_{\lp{2}}=2^{-2j}\norm{P_j(\lap (F\c G\c P_lH))}_{\lp{2}}=2^{-2j}\norm{P_j(\divb (\nabb( F\c G\c P_lH)))}_{\lp{2}}.$$
Since $F\c G\c H$ is a scalar, we may use \eqref{lbz14bis}, and we obtain:
\bee
&&2^{jb}\norm{P_j(F\c G\c P_lH)}_{\lp{2}}\\
&\les& 2^{jb}2^{-2j}2^{\frac{2j}{p}}\norm{\nabb( F\c G\c P_lH)}_{\lp{p}}\\
&\les& 2^{j(b-2+\frac{2}{p})}(\norm{\nabb F}_{\lp{2}}\norm{G}_{\lp{r}}\norm{P_lH}_{\lp{r}}+\norm{F}_{\lp{r}}\norm{\nabb  G}_{\lp{2}}\norm{P_lH}_{\lp{r}}\\
&&+\norm{F}_{\lp{r}}\norm{G}_{\lp{r}}\norm{\nabb P_lH}_{\lp{2}}),
\eee
where $4<r<+\infty$ is given by:
$$\frac{2}{r}=\frac{1}{p}-\frac{1}{2}.$$
Together with the finite band property and Bernstein for $P_l$, and using the Gagliardo-Nirenberg inequality \eqref{eq:GNirenberg}, we obtain in the case $l\leq j$:
\bea\lab{fichtre3}
2^{jb}\norm{P_j(F\c G\c P_lH)}_{\lp{2}}&\les& 2^{j(b-2+\frac{2}{p})}\norm{F}_{\hs{1}}\norm{G}_{\hs{1}}2^l\norm{P_lH}_{\lp{2}}\\
\nn&\les & 2^{-\max(j,l)(\frac{3}{2}-b-\frac{2}{p})}\norm{F}_{\hs{1}}\norm{G}_{\hs{1}}\norm{H}_{\hs{\frac{1}{2}}},
\eea
where we used in the last inequality the fact that $l\leq j$ and the choice of $p$ \eqref{fichtre2}. 

Next, we consider the case $l>j$. Since $0<b<\frac{1}{2}$, there exists a real number $q$ such that:
\be\lab{fichtre4}
2<q<\frac{2}{b+\frac{1}{2}}.
\ee 
Then, let $4<r<+\infty$ such that:
$$\frac{2}{r}+\frac{1}{q}=\frac{1}{2}.$$
Using the boundedness of $P_j$ on $\lp{2}$, Bernstein for $P_l$, and the Gagliardo-Nirenberg inequality \eqref{eq:GNirenberg}, we have:
\bea\lab{fichtre5}
2^{jb}\norm{P_j(F\c G\c P_lH)}_{\lp{2}}&\les&  2^{jb}\norm{F\c G\c P_lH}_{\lp{2}}\\
\nn&\les&  2^{jb}\norm{F}_{\lp{r}}\norm{G}_{\lp{r}}\norm{P_lH}_{\lp{q}}\\
\nn&\les& 2^{jb}\norm{F}_{\hs{1}}\norm{G}_{\hs{1}}2^{l(1-\frac{2}{q}}\norm{P_lH}_{\lp{2}}\\
\nn&\les& 2^{-\max(j,l)(\frac{2}{q}-\frac{1}{2}-b)}\norm{F}_{\hs{1}}\norm{G}_{\hs{1}}\norm{H}_{\hs{\frac{1}{2}}},
\eea
where we used in the last inequality the fact that $l> j$ and the choice of $q$ \eqref{fichtre4}.

Let $\delta$ given by:
$$\delta=\min\left(\frac{3}{2}-b-\frac{2}{p}, \frac{2}{q}-\frac{1}{2}-b\right).$$
Then, we have $\delta>0$ in view of \eqref{fichtre2} and \eqref{fichtre4}. Now, in view of \eqref{fichtre1}, \eqref{fichtre3} and \eqref{fichtre5}, we have:
\bee
\sum_{j\geq 0}2^{2jb}\norm{P_j(G\c G\c H)}^2_{\lp{2}}&\les& \norm{F}^2_{\hs{1}}\norm{G}^2_{\hs{1}}\norm{H}^2_{\hs{\frac{1}{2}}}\sum_{j\geq 0}\left(\sum_{l\geq 0}2^{-\delta\max(l,j)}\right)^2\\
&\les & \norm{F}^2_{\hs{1}}\norm{G}^2_{\hs{1}}\norm{H}^2_{\hs{\frac{1}{2}}},
\eee
since $\delta>0$. This concludes the proof of the Proposition.

\subsection{Proof of Proposition \ref{greveinff2}}

We have:
\be\lab{kei1}
\norm{P_j(G\c H)}_{\lp{2}}\les \sum_{l, m\geq 0}\norm{P_j(P_lG\c P_mH)}_{\lp{2}}.
\ee
By symmetry, we may assume:
$$l\leq m.$$
We first consider the case where $l\leq m\leq j$. Then, we have:
\bee
\norm{P_j(P_lG\c P_mH)}_{\lp{2}}&=& 2^{-2j}\norm{P_j(\lap (P_lG\c P_mH))}_{\lp{2}}\\
&=&2^{-2j}\norm{P_j(\divb (\nabb( P_lG\c P_mH)))}_{\lp{2}}.
\eee
Since $G\c H$ is a scalar, we may use \eqref{lbz14bis}, and we obtain:
\bea\lab{kei2}
&&\norm{P_j(P_lG\c P_mH)}_{\lp{2}}\\
\nn&\les & 2^{-2j}2^{\frac{4j}{3}}\norm{\nabb (P_lG\c P_mH)}_{\lp{\frac{3}{2}}}\\
\nn&\les & 2^{-\frac{2j}{3}}\norm{\nabb P_lG}_{\lp{2}}\norm{P_mH}_{\lp{6}}+ 2^{-\frac{2j}{3}}\norm{P_lG}_{\lp{6}}\norm{\nabb P_mH}_{\lp{2}}\\
\nn&\les & 2^{-\frac{2j}{3}}(2^{l+\frac{2m}{3}}+2^{m+\frac{2l}{3}})\norm{P_lG}_{\lp{2}}\norm{P_mH}_{\lp{2}}\\
\nn&\les & 2^{-\frac{|j-m|}{6}-\frac{|j-l|}{6}}2^{\frac{l}{2}}\norm{P_lG}_{\lp{2}}2^{\frac{m}{2}}\norm{P_mH}_{\lp{2}},
\eea
where we used the finite band property and Bernstein for $P_l$ and $P_m$, and the fact that $l\leq m\leq j$.

Next, we consider the case where $l\leq j<m$. Then, we use the boundedness of $P_j$ on $\lp{2}$ which yields:
\bea\lab{kei3}
\norm{P_j(P_lG\c P_mH)}_{\lp{2}}&\les& \norm{P_lG\c P_mH}_{\lp{2}}\\
\nn&\les& \norm{P_lG}_{\lp{6}}\norm{P_mH}_{\lp{3}}\\
\nn&\les& 2^{\frac{2l}{3}}\norm{P_lG}_{\lp{2}}2^{\frac{m}{3}}\norm{P_mH}_{\lp{2}}\\
\nn&\les& 2^{-\frac{|j-m|}{12}-\frac{|j-l|}{12}}2^{\frac{l}{2}}\norm{P_lG}_{\lp{2}}2^{\frac{m}{2}}\norm{P_mH}_{\lp{2}}
\eea
where we used Bernstein for $P_l$ and $P_m$, and the fact that $l\leq j<m$.

Finally, we consider the case where $j<l\leq m$. Then, we use Bernstein for $P_j$ which yields:
\bea\lab{kei4}
\norm{P_j(P_lG\c P_mH)}_{\lp{2}}&\les& 2^{\frac{j}{3}}\norm{P_lG\c P_mH}_{\lp{\frac{3}{2}}}\\
\nn&\les& 2^{\frac{j}{3}}\norm{P_lG}_{\lp{3}}\norm{P_mH}_{\lp{3}}\\
\nn&\les& 2^{\frac{j}{3}}2^{\frac{l}{3}}\norm{P_lG}_{\lp{2}}2^{\frac{m}{3}}\norm{P_mH}_{\lp{2}}\\
\nn&\les& 2^{-\frac{|j-m|}{6}-\frac{|j-l|}{6}}2^{\frac{l}{2}}\norm{P_lG}_{\lp{2}}2^{\frac{m}{2}}\norm{P_mH}_{\lp{2}},
\eea
where we used Bernstein for $P_l$ and $P_m$, and the fact that $j<l\leq m$.

Finally, we have in view of \eqref{kei1}, \eqref{kei2}, \eqref{kei3} and \eqref{kei4}:
\bee
\sum_{j\geq 0}\norm{P_j(G\c H)}^2_{\lp{2}} &\les& \sum_{j\geq 0}\left(\sum_{l, m\geq 0}\norm{P_j(P_lG\c P_mH)}_{\lp{2}}\right)^2\\
&\les& \sum_{j\geq 0}\left(\sum_{l, m\geq 0}2^{-\frac{|j-m|}{12}-\frac{|j-l|}{12}}2^{\frac{l}{2}}\norm{P_lG}_{\lp{2}}2^{\frac{m}{2}}\norm{P_mH}_{\lp{2}}\right)^2\\
&\les& \left(\sum_{l\geq 0}2^l\norm{P_lG}^2_{\lp{2}}\right)\left(\sum_{m\geq 0}2^m\norm{P_mH}^2_{\lp{2}}\right)\\
&\les& \norm{G}^2_{\hs{\frac{1}{2}}}\norm{H}_{\hs{\frac{1}{2}}}^2.
\eee
This yields \eqref{kei} which concludes the proof of the proposition.

\subsection{Proof of Proposition \ref{greveinff3}}

We have:
\be\lab{bale1}
\norm{P_j(f h)}_{\lp{2}}\les \sum_{l\geq 0}\norm{P_j(fP_lh)}_{\lp{2}}.
\ee
If $l\leq j$, we use the boundedness of $P_j$ on $\lp{2}$ to obtain:
\bea\lab{bale2}
2^{-\frac{j}{2}}\norm{P_j(fP_lh)}_{\lp{2}}&\les & 2^{-\frac{j}{2}}\norm{fP_lh}_{\lp{2}}\\
\nn&\les & 2^{-\frac{j}{2}}\norm{f}_{\lp{\infty}}\norm{P_lh}_{\lp{2}}\\
\nn&\les& 2^{-\frac{|j-l|}{2}}\norm{f}_{\lp{\infty}}2^{-\frac{l}{2}}\norm{P_lh}_{\lp{2}},
\eea
where we used in the last inequality the fact that $l\leq j$.

If $l>j$, we use the following identity:
$$P_j(fP_lh)=2^{-2l}P_j(fP_l\lap h)=2^{-2l}P_j(\divb(f\nabb P_l h))+2^{-2l}P_j(\nabb f\c \nabb P_lh).$$
Together with the finite band property for $P_j$, the strong Bernstein inequality \eqref{eq:strongbernscalarbis} for scalars, and the finite band property for $P_l$, we obtain:
\bea\lab{bale3}
\nn2^{-\frac{j}{2}}\norm{P_j(fP_lh)}_{\lp{2}}&\les & 2^{-\frac{j}{2}-2l}(\norm{P_j\divb(f\nabb P_lh)}_{\lp{2}}+\norm{P_j(\nabb f\c\nabb P_lh)}_{\lp{2}})\\
\nn&\les & 2^{\frac{j}{2}-2l}(\norm{f\nabb P_lh}_{\lp{2}}+\norm{\nabb f\c\nabb P_lh}_{\lp{1}})\\
\nn&\les & 2^{\frac{j}{2}-2l}(\norm{f}_{\lp{\infty}}+\norm{\nabb f}_{\lp{2}})\norm{\nabb P_lh}_{\lp{2}}\\
\nn&\les & 2^{\frac{j}{2}-l}(\norm{f}_{\lp{\infty}}+\norm{\nabb f}_{\lp{2}})\norm{P_lh}_{\lp{2}}\\
&\les & 2^{-\frac{|j-l|}{2}}(\norm{f}_{\lp{\infty}}+\norm{\nabb f}_{\lp{2}})2^{-\frac{l}{2}}\norm{P_lh}_{\lp{2}},
\eea
where we used in the last inequality the fact that $l> j$.

Finally, \eqref{bale1}, \eqref{bale2} and \eqref{bale3} imply:
\bee
\sum_{j\geq 0}2^{-j}\norm{P_j(fh)}^2_{\lp{2}}&\les& (\norm{f}^2_{\lp{\infty}}+\norm{\nabb f}^2_{\lp{2}})\sum_{j\geq 0}\left(2^{-\frac{|j-l|}{2}}2^{-\frac{l}{2}}\norm{P_lh}_{\lp{2}}\right)^2\\
&\les& (\norm{f}^2_{\lp{\infty}}+\norm{\nabb f}^2_{\lp{2}})\sum_{l\geq 0}2^{-l}\norm{P_lh}_{\lp{2}}^2\\
&\les& (\norm{f}^2_{\lp{\infty}}+\norm{\nabb f}^2_{\lp{2}})\norm{h}^2_{\hs{-\frac{1}{2}}}.
\eee
This concludes the proof of the proposition. 

\subsection{Proof of Proposition \ref{greveinff4}}

We estimate $\norm{P_j(G\c P_lH)}_{\lp{2}}$ starting with the case where $j\geq l$. Using the boundedness of $P_j$ on $\lp{2}$ of $P_j$, we have:
\bee
\norm{P_j(G\c P_lH)}_{\ll{2}}&\les& \norm{G}_{\l{\infty}{4}}\norm{P_lH}_{\l{2}{4}}\\
\nn&\les& 2^{\frac{l}{2}}\norm{G}_{H^1(S)}\norm{P_lH}_{\ll{2}},
\eee
where we used in the last inequality the Proposition \ref{p1} and the Bernstein inequality for $P_l$. This yields in the case $j\geq l$:
\be\lab{clp18}
2^{-\frac{j}{2}}\norm{P_j(G\c P_lH)}_{\ll{2}}\les 2^{-\frac{|l-j|}{2}}\norm{G}_{H^1(S)}\norm{P_lH}_{\ll{2}}.
\ee

Next, we consider the case where $l>j$, and we estimate $\norm{P_j(P_mG\c P_lH)}_{\lp{2}}$ starting with the case where $m\geq l$. Using the sharp Bernstein inequality \eqref{eq:strongbernscalarbis}, we have:
\be\lab{clp19}
\norm{P_j(P_m G\c P_lH)}_{\ll{2}}\les 2^j\norm{P_mG}_{\l{\infty}{2}}\norm{P_lH}_{\ll{2}}.
\ee

Finally, we consider the case where $l>j$ and $l>m$. Using the finite band property for $P_l$, we have:
\bee
\norm{P_j(P_m G\c P_lH)}_{\ll{2}}&\les& 2^{-2l}\norm{P_j(P_m G\c \lap P_lH)}_{\ll{2}}\\
&\les& 2^{-2l}\norm{P_j(\nabb P_m G\c \nabb P_lH)}_{\ll{2}}+2^{-2l}\norm{P_j(\divb(\nabb P_m G\c P_lH))}_{\ll{2}}.
\eee
Using the sharp Bernstein inequality \eqref{eq:strongbernscalarbis} for the first term and the estimate \eqref{lbz14bis}  with $p=4/3$ for the second term, we obtain:
\bea\lab{clp20}
&&\norm{P_j(P_m G\c P_lH)}_{\ll{2}}\\
\nn&\les& 2^{j-2l}\norm{\nabb P_m G}_{\l{\infty}{2}}\norm{\nabb P_lH}_{\ll{2}}+2^{\frac{3j}{2}-2l}\norm{\nabb P_m G}_{\l{\infty}{2}}\norm{P_lH}_{\l{2}{4}}\\
\nn&\les& (2^{j-l+m}+2^{\frac{3j}{2}-\frac{3l}{2}+m})\norm{P_m G}_{\l{\infty}{2}}\norm{P_lH}_{\ll{2}},
\eea
where we used in the last inequality Bernstein for $P_l$ and $P_m$. \eqref{clp19} and \eqref{clp20} yield in the case $l>j$:
\be\lab{clp21}
2^{-\frac{j}{2}}\norm{P_j(P_m G\c P_lH)}_{\ll{2}}\les 2^{-\frac{|j-l|}{4}-\frac{|j-m|}{4}}(2^{\frac{m}{2}}\norm{P_m G}_{\l{\infty}{2}})\norm{P_lH}_{\ll{2}}.
\ee
Finally, \eqref{clp18} and \eqref{clp21} imply:
\bee
\sum_{j\geq 0}2^{-j}\norm{P_j(G\c H)}^2_{\ll{2}}&\les& \left(\norm{G}_{H^1(S)}^2+\left(\sum_{m\geq 0}2^m\norm{P_mG}^2_{\l{\infty}{2}}\right)\right)\left(\sum_{l\geq 0}\norm{P_lH}^2_{\ll{2}}\right)\\
&\les& \norm{G}^2_{H^1(S)}\norm{H}^2_{L^2(S)},
\eee
where we used in the last inequality Corollary \ref{cor:commLP1} for $G$ and the Bessel inequality for $H$. This concludes the proof of the proposition. 

\subsection{Proof of Lemma \ref{greveinff5}}

We have:
\be\lab{prod2}
\norm{P_j(F\c G)}_{\lp{2}}\les \sum_{l, m\geq 0}\norm{P_j(P_lF\c P_mG)}_{\lp{2}}.
\ee
If $j=\max(j,l,m)$, we use the boundedness of $P_j$ on $\lp{2}$ and the Bernstein inequality for $P_l$ and $P_m$ to obtain:
\bea\lab{prod3}
2^{-j}\norm{P_j(P_lF\c P_mG)}_{\lp{2}}&\les& 2^{-j}\norm{P_lF\c P_mG}_{\lp{2}}\\
\nn&\les& 2^{-j}\norm{P_lF}_{\lp{6}}\norm{P_mG}_{\lp{3}}\\
\nn&\les& 2^{-j}2^{\frac{2l}{3}}2^{\frac{m}{3}}\norm{P_lF}_{\lp{2}}\norm{P_mG}_{\lp{2}}\\
\nn&\les& 2^{-\frac{|l-m|}{6}}2^{\frac{l}{2}}\norm{P_lF}_{\lp{2}}2^{-\frac{m}{2}}\norm{P_mG}_{\lp{2}},
\eea
where we used in the last inequality the fact that $j=\max(j,l,m)$. 

If $l=\max(j,l,m)$, we use for $P_j$ the strong Bernstein inequality for scalars \eqref{eq:strongbernscalarbis} which yields:
\bea\lab{prod4}
2^{-j}\norm{P_j(P_lF\c P_mG)}_{\lp{2}}&\les& \norm{P_lF\c P_mG}_{\lp{1}}\\
\nn&\les& \norm{P_lF}_{\lp{2}}\norm{P_mG}_{\lp{2}}\\
\nn&\les& 2^{-\frac{|l-m|}{2}}2^{\frac{l}{2}}\norm{P_lF}_{\lp{2}}2^{-\frac{m}{2}}\norm{P_mG}_{\lp{2}},
\eea
where we used in the last inequality the fact that $l=\max(j,l,m)$.

If $m=\max(j,l,m)$, we use the following identity:
\bee
P_j(P_lF\c P_mG)&=& 2^{-2m}P_j(P_lF\c \lap P_mG)\\
&=& 2^{-2m}(P_j(\lap(P_lF\c P_mG))+P_j(\lap(P_lF)\c P_mG)+P_j(\divb(\nabb(P_lF)\c P_mG)))\\
&=&  2^{-2m}(2^{2j}P_j(P_lF\c P_mG)+2^{2l}P_j(P_lF\c P_mG)+P_j(\divb(\nabb(P_lF)\c P_mG))).
\eee
Together with the boundedness of $P_j$ on $\lp{2}$, the strong Bernstein inequality for scalars \eqref{eq:strongbernscalarbis}, and the estimate \eqref{lbz14bis}, we obtain:
\bea\lab{prod5}
&&2^{-j}\norm{P_j(P_lF\c P_mG)}_{\lp{2}}\\
\nn&\les& 2^{-j-2m}(2^{2j}\norm{P_lF\c P_mG}_{\lp{2}}+2^{2l}2^j\norm{P_lF\c P_mG}_{\lp{1}}+2^{\frac{2j}{3}}\norm{\nabb(P_lF)\c P_mG}_{\lp{\frac{4}{3}}}\\
\nn&\les& 2^{-j-2m}(2^{2j}\norm{P_lF}_{\lp{6}}\norm{P_mG}_{\lp{3}}+2^{2l+j}\norm{P_lF}_{\lp{2}}\norm{P_mG}_{\lp{2}}\\
\nn&&+2^{\frac{3j}{2}}\norm{\nabb(P_lF)}_{\lp{2}}\norm{P_mG}_{\lp{4}})\\
\nn&\les& 2^{-j-2m}(2^{2j+\frac{2l}{3}+\frac{m}{3}}+2^{2l+j}+2^{\frac{3j}{2}+l+\frac{m}{2}})\norm{P_lF}_{\lp{2}}\norm{P_mG}_{\lp{2}}\\
\nn&\les& 2^{-\frac{|l-m|}{6}}2^{\frac{l}{2}}\norm{P_lF}_{\lp{2}}2^{-\frac{m}{2}}\norm{P_mG}_{\lp{2}},
\eea
where we used Bernstein for $P_m$, the finite band property and Bernstein for $P_l$, and the fact that $m=\max(j,l,m)$. 

Finally, \eqref{prod2}, \eqref{prod3}, \eqref{prod4} and \eqref{prod5} imply for all $j\geq 0$:
\bee
2^{-j}\norm{P_j(F\c G)}_{\lp{2}}&\les& \sum_{l, m\geq 0}2^{-\frac{|l-m|}{6}}2^{\frac{l}{2}}\norm{P_lF}_{\lp{2}}2^{-\frac{m}{2}}\norm{P_mG}_{\lp{2}}\\
&\les& \left(\sum_{l\geq 0}2^l\norm{P_lF}_{\lp{2}}^2\right)^{\frac{1}{2}}\left(\sum_{m\geq 0}2^{-m}\norm{P_mG}_{\lp{2}}^2\right)^{\frac{1}{2}}\\
&\les& \norm{F}_{\hs{\frac{1}{2}}}\norm{G}_{\hs{-\frac{1}{2}}}.
\eee
This concludes the proof of the lemma.

\subsection{Proof of Lemma \ref{greveinff6}}

We have:
\be\lab{prod7}
\norm{P_j(\divb(f G)}_{\lp{2}}\les \sum_{l\geq 0}\norm{P_j(\divb(P_l(f)G))}_{\lp{2}}.
\ee
If $l\leq j$, we used the boundedness of $P_j$ on $\lp{2}$, and the strong Bernstein inequality for scalars \eqref{eq:strongbernscalarbis} and the finite band property for $P_l$. We obtain:
\bea\lab{prod8}
&& 2^{j(b-1)}\norm{P_j(\divb(P_l(f) G))}_{\lp{2}}\\
\nn&\les & 2^{j(b-1)}\norm{\divb(P_l(f) G)}_{\lp{2}}\\
\nn&\les& 2^{j(b-1)}(\norm{\nabb(P_lf) G}_{\lp{2}}+\norm{P_lf \nabb G}_{\lp{2}})\\
\nn&\les& 2^{j(b-1)}(\norm{\nabb(P_lf)}_{\lp{2}}\norm{G}_{\lp{\infty}}+\norm{P_lf}_{\lp{\infty}}\norm{\nabb G}_{\lp{2}})\\
\nn&\les& 2^{j(b-1)}2^l(\norm{G}_{\lp{\infty}}+\norm{\nabb G}_{\lp{2}})\norm{P_lf}_{\lp{2}}\\
\nn&\les& 2^{-|j-l|(1-b)}(\norm{G}_{\lp{\infty}}+\norm{\nabb G}_{\lp{2}})2^{lb}\norm{P_lf}_{\lp{2}}
\eea
where we used in the last inequality the fact that $l\leq j$ and $b<1$. 

If $l>j$, we use the following identity:
\bee
P_j(\divb(P_l(f) G))&=& 2^{-2l}P_j(\divb(\lap P_l(f) G))\\
&=& 2^{-2l}(P_j(\divb(\divb(\nabb P_l(f) G)))+P_j(\divb(\nabb P_l(f)\c \nabb G))).
\eee
Together with the estimate \eqref{lbz14bis} for $P_j$, we obtain:
\bee
&&2^{j(b-1)}\norm{P_j(\divb(P_l(f) G))}_{\lp{2}}\\
&\les & 2^{j(b-1)-2l}(\norm{P_j(\divb(\divb(\nabb P_l(f) G)))}_{\lp{2}}+\norm{P_j(\divb(\nabb P_l(f)\c \nabb G))}_{\lp{2}})\\
&\les & 2^{j(b-1)-2l}(\norm{P_j\divb\divb}_{\mathcal{L}(\lp{2})}\norm{\nabb P_l(f)G}_{\lp{2}}+2^{\frac{2j}{p}}\norm{\nabb P_l(f)\c \nabb G}_{\lp{2}}),
\eee
where $p$ satisfies:
\be\lab{prod9}
1<p<\frac{2}{1-b},
\ee
which is possible since $-1<b<1$. Together with the Bochner inequality for scalars \eqref{eq:Bochconseqbis} and the finite band property for $P_j$, this yields:
\bee
&&2^{j(b-1)}\norm{P_j(\divb(P_l(f) G))}_{\lp{2}}\\
&\les & 2^{j(b-1)-2l}(2^{2j}\norm{\nabb P_lf}_{\lp{2}}\norm{G}_{\lp{\infty}}+2^{\frac{2j}{p}}\norm{\nabb P_lf}_{\lp{r}}\norm{\nabb G}_{\lp{2}}),
\eee
where $r$ is given by:
$$\frac{1}{r}+\frac{1}{2}=\frac{1}{p}.$$
Together with the Gagliardo-Nirenberg inequality \eqref{eq:GNirenberg},  the Bochner inequality for scalars \eqref{eq:Bochconseqbis}, and the finite band property for $P_l$, we obtain:
\bea\lab{prod10}
&&2^{j(b-1)}\norm{P_j(\divb(P_l(f) G))}_{\lp{2}}\\
\nn&\les & 2^{j(b-1)-2l}(2^{2j}2^l+2^{\frac{2j}{p}}2^{l(2-\frac{2}{r})})(\norm{G}_{\lp{\infty}}+\norm{\nabb G}_{\lp{2}})\norm{P_lf}_{\lp{2}}\\
\nn&\les& (2^{-|j-l|(1+b)}+2^{-|j-l|(b-1+\frac{2}{p})})(\norm{G}_{\lp{\infty}}+\norm{\nabb G}_{\lp{2}})2^{lb}\norm{P_lf}_{\lp{2}}
\eea
where we used in the last inequality \eqref{prod9}, the fact that $b+1>0$, and the fact that $l>j$. 

Finally, \eqref{prod7}, \eqref{prod9} and \eqref{prod10} imply:
\bee
&&\sum_{j\geq 0}2^{2(b-1)j}\norm{P_j(\divb(P_l(f) G))}^2_{\lp{2}}\\
&\les& \sum_{j\geq 0}\left(\sum_{l\geq 0}2^{-|j-l|\min(1-b, 1+b, b-1+\frac{2}{p})}(\norm{G}_{\lp{\infty}}+\norm{\nabb G}_{\lp{2}})2^{lb}\norm{P_lf}_{\lp{2}}\right)^2\\
&\les&(\norm{G}_{\lp{\infty}}^2+\norm{\nabb G}_{\lp{2}}^2) \sum_{l\geq 0}2^{2lb}\norm{P_lf}_{\lp{2}}^2\\
&\les&(\norm{G}_{\lp{\infty}}^2+\norm{\nabb G}_{\lp{2}}^2)\norm{f}_{\hs{b}}^2.
\eee
This concludes the proof of the lemma.

\subsection{Proof of Lemma \ref{greveinff7}}

We have:
\be\lab{prod12}
\norm{P_j(\divb(f G)}_{\ll{2}}\les \sum_{l\geq 0}\norm{P_j(\divb(P_l(f)G))}_{\ll{2}}.
\ee
If $l\leq j$, we use the finite band property for $P_j$ to obtain:
\bee
&&2^{j(b-1)}\norm{P_j(\divb(P_l(f)G))}_{\ll{2}}\\
&\les& 2^{j(b-2)}\norm{\nabb\divb(P_l(f)G)}_{\ll{2}}\\
&\les& 2^{j(b-2)}(\norm{\nabb^2(P_l(f))G}_{\ll{2}}+\norm{\nabb(P_l(f))\divb(G)}_{\ll{2}}+\norm{P_l(f)\nabb\divb(G)}_{\ll{2}}\\
&\les& 2^{j(b-2)}(\norm{\nabb^2(P_l(f))}_{\ll{2}}\norm{G}_{\ll{\infty}}+\norm{\nabb(P_l(f))}_{\l{2}{p}}
\norm{\divb(G)}_{\l{2}{q}}\\
&&+\norm{P_l(f)}_{\ll{\infty}}\norm{\nabb^2(G)}_{\ll{2}}),
\eee
where $p$ and $q$ are such that:
$$\frac{2}{p}+\frac{2}{q}=\frac{1}{2},\,\, 2< q<p<+\infty.$$
Together with the Bochner inequality for scalars \eqref{eq:Bochconseqbis}, the Gagliardo-Nirenberg inequality \eqref{eq:GNirenberg}, the finite band property for $P_l$, and the strong Bernstein inequality for scalars \eqref{eq:strongbernscalarbis}, we obtain:
\bee
&&2^{j(b-1)}\norm{P_j(\divb(P_l(f)G))}_{\ll{2}}\\
&\les& 2^{j(b-2)}(2^{2l}(\norm{G}_{\ll{\infty}}+\norm{\divb(G)}_{\l{2}{q}})\norm{P_lf}_{\ll{2}}\\
&&+2^l\norm{\nabb^2(G)}_{\ll{2}}\norm{P_l(f)}_{\l{\infty}{2}})\\
&\les& 2^{-|j-l|(2-b)}(\norm{G}_{\ll{\infty}}+\norm{\divb(G)}_{\l{2}{q}}+\norm{\nabb^2(G)}_{\ll{2}})\\
&&\times(2^{lb}\norm{P_lf}_{\ll{2}}+2^{l(b-1)}\norm{P_l(f)}_{\l{\infty}{2}}),
\eee 
where we used in the last inequality the fact that $l\leq j$ and $b<2$. Since this holds for any $q>2$, we finally obtain:
\bea\lab{prod13}
&&2^{j(b-1)}\norm{P_j(\divb(P_l(f)G))}_{\ll{2}}\\
\nn&\les& 2^{-|j-l|(2-b)}(\norm{G}_{\ll{\infty}}+\norm{\divb(G)}_{\l{2}{2_+}}+\norm{\nabb^2(G)}_{\ll{2}})\\
\nn&&\times(2^{lb}\norm{P_lf}_{\ll{2}}+2^{l(b-1)}\norm{P_l(f)}_{\l{\infty}{2}}).
\eea

If $l>j$,  the finite band property for $P_j$ yields:
\bea\lab{prod14}
2^{j(b-1)}\norm{P_j(\divb(P_l(f)G))}_{\ll{2}}&\les & 2^{jb}\norm{P_l(f)G}_{\ll{2}}\\
\nn&\les& 2^{jb}\norm{G}_{\ll{\infty}}\norm{P_l(f)}_{\ll{2}}\\
\nn&\les& 2^{-|j-l|b}\norm{G}_{\ll{\infty}}2^{lb}\norm{P_l(f)}_{\ll{2}},
\eea
where we used in the last inequality the fact that $l>j$ and $b>0$. 

Finally, \eqref{prod12}, \eqref{prod13} and \eqref{prod14} imply:
\bee
&&\sum_{j\geq 0}2^{2(b-1)j}\norm{P_j(\divb(fG))}^2_{\ll{2}}\\
&\les& (\norm{G}_{\ll{\infty}}^2+\norm{\divb(G)}^2_{\l{2}{2_+}}+\norm{\nabb^2(G)}^2_{\ll{2}})\\
&&\times\sum_{j\geq 0}\left(\sum_{l\geq 0}2^{-|l-j|\min(2-b,b)}(2^{lb}\norm{P_l(f)}_{\ll{2}}+2^{l(b-1)}\norm{P_l(f)}_{\l{\infty}{2}})\right)^2\\
&\les& (\norm{G}_{\ll{\infty}}^2+\norm{\divb(G)}^2_{\l{2}{2_+}}+\norm{\nabb^2(G)}^2_{\ll{2}})\\
&&\times\sum_{l\geq 0}(2^{2lb}\norm{P_l(f)}_{\ll{2}}^2+2^{2l(b-1)}\norm{P_l(f)}^2_{\l{\infty}{2}})\\
&\les& (\norm{G}_{\ll{\infty}}^2+\norm{\divb(G)}^2_{\l{2}{2_+}}+\norm{\nabb^2(G)}^2_{\ll{2}})(\norm{f}^2_{\lhs{2}{b}}+\norm{f}_{\lhs{\infty}{b-1}}).
\eee
This concludes the proof of the lemma.

\subsection{Proof of Lemma \ref{greveinff8}}

We have:
\be\lab{prod16}
\norm{P_j(F h)}_{\lp{2}}\les \sum_{l\geq 0}\norm{P_j(FP_l(h))}_{\lp{2}}.
\ee
If $l\leq j$, we use the finite band property for $P_j$, and the strong Bernstein inequality for scalars \eqref{eq:strongbernscalarbis} and the finite band property for $P_l$, which yields:
\bea\lab{prod17}
2^{jb}\norm{P_j(FP_l(h))}_{\lp{2}}&\les & 2^{j(b-1)}\norm{\nabb(FP_l(h))}_{\lp{2}}\\
\nn&\les & 2^{j(b-1)}(\norm{\nabb F}_{\lp{2}}\norm{P_l(h)}_{\lp{\infty}}+\norm{F}_{\lp{\infty}}\norm{\nabb P_l(h)}_{\lp{2}})\\
\nn&\les & 2^{j(b-1)}(\norm{\nabb F}_{\lp{2}}2^l\norm{P_l(h)}_{\lp{2}}+\norm{F}_{\lp{\infty}}2^l\norm{P_l(h)}_{\lp{2}})\\
\nn&\les & 2^{-|j-l|(1-b)}(\norm{\nabb F}_{\lp{2}}+\norm{F}_{\lp{\infty}})2^{lb}\norm{P_l(h)}_{\lp{2}},
\eea
where we used in the last inequality the fact that $l\leq j$ and $b<1$.

If $l>j$, we use the boundedness of $P_j$ on $\lp{2}$ which yields:
\bea\lab{prod18}
2^{jb}\norm{P_j(FP_l(h))}_{\lp{2}}&\les & 2^{jb}\norm{FP_l(h)}_{\lp{2}}\\
\nn&\les & 2^{jb}\norm{F}_{\lp{\infty}}\norm{P_l(h)}_{\lp{2}})\\
\nn&\les & 2^{-|j-l|b}\norm{F}_{\lp{\infty}}2^{lb}\norm{P_l(h)}_{\lp{2}},
\eea
where we used in the last inequality the fact that $l>j$ and $b>0$. 

Finally, \eqref{prod16}, \eqref{prod17} and \eqref{prod18} imply:
\bee
&&\sum_{j\geq 0}2^{2jb}\norm{P_j(FP_l(h))}^2_{\lp{2}}\\
&\les& (\norm{F}_{\lp{\infty}}^2+\norm{\nabb F}^2_{\lp{2}})\sum_{j\geq 0}\left(\sum_{l\geq 0}2^{-|j-l|b}2^{lb}\norm{P_l(h)}_{\lp{2}}\right)^2\\
&\les& (\norm{F}_{\lp{\infty}}^2+\norm{\nabb F}^2_{\lp{2}})\sum_{l\geq 0}2^{2lb}\norm{P_l(h)}_{\lp{2}}^2\\
&\les & (\norm{F}_{\lp{\infty}}^2+\norm{\nabb F}^2_{\lp{2}})\norm{h}^2_{\hs{b}}.
\eee
This concludes the proof of the Lemma.

\end{document}